\documentclass{amsart}

\usepackage{color}
\usepackage{amssymb, amsmath, amsthm}
\usepackage{amsrefs}
\usepackage{fullpage}
\usepackage{multicol}
\usepackage{lineno}

\usepackage{graphicx}
\usepackage{pinlabel}
\usepackage[colorlinks=true,allcolors=blue]{hyperref}

\usepackage{soul} 

\newtheorem{theorem}{Theorem}[section]
\newtheorem*{theorem*}{Theorem}
\newtheorem{corollary}[theorem]{Corollary}
\newtheorem{notation}[theorem]{Notation}
\newtheorem{lemma}[theorem]{Lemma}
\newtheorem{proposition}[theorem]{Proposition}
\newtheorem{assumption}[theorem]{Assumption}
\newtheorem{convention}[theorem]{Convention}
\newtheorem{construction}[theorem]{Construction}

\newtheorem*{additivitytheorem}{Additivity Theorem}
\newtheorem*{cor-link version}{Corollary \ref{Schubert Cor}}

\theoremstyle{definition}
\newtheorem{definition}[theorem]{Definition}
\newtheorem{example}[theorem]{Example}
\newtheorem{remarkx}[theorem]{Remark}
 \newenvironment{remark}
  {\pushQED{\qed}\remarkx}
  {\popQED\endremarkx}

\newcommand{\N}{\mathbb{N}}

\newcommand{\R}{\mathbb{R}}

\renewcommand{\H}{\mathbb{H}}

\newcommand{\ob}[1]{\overline{#1}}

\newcommand{\nil}{\varnothing}
\newcommand{\up}{\uparrow}
\newcommand{\down}{\downarrow}
\newcommand{\wihat}[1]{\widehat{#1}}
\newcommand{\wh}[1]{\widehat{#1}}

\usepackage{bm} 
\newcommand{\weighted}[1]{\bm{#1}}

\newcommand{\defn}[1]{\textbf{#1}}

\newcommand{\vpH}{vp\mathbb{H}}
\newcommand{\vpoH}{\overrightarrow{\vpH}}
\newcommand{\punct}[1]{\mathring{#1}}
\newcommand{\g}{\operatorname{g}}
\newcommand{\x}{\operatorname{x}}
\renewcommand{\b}{\frak{b}}

\newcommand{\netchi}{\operatorname{net}\chi} 
\newcommand{\netg}{\operatorname{net g}} 
\newcommand{\netx}{\operatorname{net x}} 
\newcommand{\netw}{\operatorname{net w}} 

\newcommand{\boundary}{\partial}
\newcommand{\bdd}{\boundary}
\newcommand{\mc}[1]{\mathcal{#1}}

\newcommand{\cpt}{\sqsubset}

\newcommand{\spacing}{
\parskip 6.6pt
\parindent 0pt
}

\spacing

\begin{document}

\title{The genus 1 bridge number of satellite knots}
   \author{Scott A. Taylor}
   \address{Scott A. Taylor \\ Colby College\\5832 Mayflower Hill\\Waterville, ME 04901}
   \email{scott.taylor@colby.edu}
   \author{Maggy Tomova}
   \address{Maggy Tomova \\ University of Central Florida\\4000 Central Florida Blvd.\\Orlando, FL 32816}
   \email{Maggy.Tomova@ucf.edu}
   

   \begin{abstract}
   Let $T$ be a satellite knot, link, or spatial graph in a 3-manifold $M$ that is either $S^3$ or a lens space. Let $\b_0$ and $\b_1$ denote genus 0 and genus 1 bridge number, respectively.  Suppose that $T$ has a companion knot $K$ (necessarily not the unknot) and wrapping number $\omega$ with respect to $K$. When $K$ is not a torus knot, we show that $\b_1(T)\geq \omega \b_1(K)$. There are previously known counter-examples if $K$ is a torus knot. Along the way, we generalize and give a new proof of Schubert's result that $\b_0(T) \geq \omega \b_0(K)$. We also prove versions of the theorem applicable to when $T$ is a ``lensed satellite'' and when there is a torus separating components of $T$.
\end{abstract}

 \maketitle  

\begin{multicols}{2}
\tableofcontents
\end{multicols}

\newpage

\section{Introduction}

In the mid-1950s, Horst Schubert introduced two ideas which turned out to have lasting impact in knot theory: satellite knots \cite{Schubert-satellites} and (genus 0) bridge number $\b_0$ \cite{Schubert}. Satellite knots are those knots which are tied in the shape of other knots; we define them below and generalize the definition in Section \ref{satellites}. Bridge number is an invariant for knots and links in the 3-sphere $S^3$ with several equivalent definitions. 

Schubert famously showed:
\begin{theorem*}[Schubert]
For a knot $T \subset S^3$:
\begin{itemize}
    \item (Additivity Theorem) If $T = T_0 \# T_1$ is composite, then $\b_0(T) = \b_0(T_0) + \b_0(T_1) - 1$, and
    \item (Satellite Theorem) If $T$ is a satellite knot with companion $K$ and wrapping number $\omega$ with respect to that companion, then $\b_0(T) \geq \omega \b_0(K).$
\end{itemize}
\end{theorem*}

Subsequently, $\b_0$ has been connected to numerous topological, geometric, combinatorial, and algebraic knot invariants \cites{AS, BKVV, BCTT, BJW, EH, KM, Milnor, Murasugi, Scharlemann, Yokota}.  Schultens has written a brief surveys \cite{Schultens-BridgeNumSurvey, Schultens-SatSurvey} on bridge number and satellite knots. She also has particularly elegant proofs \cites{Schultens1, Schultens2} of Schubert's results. Her work provided some of the intuition for our methods. 

To what extent do Schubert's results hold for knots in other 3-manifolds? The most natural next class of 3-manifolds to consider are the lens spaces, introduced by Tietze in 1908 \cite{Gordon}. They are the 3-manifolds having genus 1 Heegaard splittings (although, as is common, we exclude $S^3$ and $S^1 \times S^2$). In 1992, Doll \cite{Doll} introduced the higher genus bridge numbers $\b_g$. Knots $K$ in $S^3$ or a lens space with $\b_1(K) = 1$ in particular, have attracted significant interest \cites{Baker, BGL,CS, EMMGRL,EM, EMRL,GLV}. Doll himself proved a version of Schubert's Additivity Theorem that applies to lens spaces, however he discovered certain complications. In Section \ref{Trivial Cases}, we state and prove the most natural version of this theorem for lens spaces.   

Primarily, however, we are concerned with extending Schubert's Satellite Theorem to \emph{genus one} bridge number $\b_1$ for knots, links, and spatial graphs in $S^3$ and lens spaces (other than $S^1 \times S^2$). For clarity (especially when working in manifolds other than $S^3$), we allow the term ``satellite'' to apply to any knot, link, or spatial graph $T$ in any 3-manifold $M$  contained in a solid torus $V \subset M$ such that the torus $Q = \boundary V$ is essential (i.e. incompressible and not $\boundary$-parallel) in $M\setminus T$; we also require $V\setminus T$ to be irreducible as a matter of convenience. The torus $Q$ is a \defn{companion torus} and a core loop of $V$ is a \defn{companion knot} $K$ for $T$. The \defn{wrapping number} $\omega$ of $T$ in $V$ is the minimum number of times $T$ intersects an essential disc in $V$. The requirement that $Q$ be essential in $M\setminus T$ means that necessarily $\omega \geq 1$. We can also establish an inequality for ``lensed satellite knots,'' which amounts to letting $V$ be the connected sum of a lens space and a solid torus, instead of just a solid torus. Later on, we will expand the definitions of ``companion torus'', ``companion knot'', and ``wrapping number'' to allow the exterior of $V$ to contain components of $T$, but we will always reserve the term ``satellite'' and ``lensed satellite'' for the case when $T \subset V$ and with $V\setminus T$ irreducible. According to our definition, a satellite or lensed satellite knot, link, or spatial graph never has trivial companion or companion that is a core loop of a lens space.

We prove 

\begin{theorem}[Main Theorem]\label{Main Theorem}
Suppose that $M\neq S^1 \times S^2$ is either $S^3$ or a lens space, and that $T \subset M$ is a  satellite knot, link, or spatial graph with nontorus companion knot $K$ and wrapping number $\omega\geq 1$ with respect to $K$. Then
\[
\b_1(T) \geq \omega \b_1(K).
\]
If $T$ is instead a lensed satellite, then $\b_1(T) \geq \omega(\b_1(K) - 1)$.
\end{theorem}

The exclusion of torus knots as companion knots cannot be removed; there are satellite knots $T$ in $S^3$ with $\b_1(T)= 1$ \cite{Hayashi, GHS}. For such knots, Morimoto and Sakuma \cite{Morimoto-Sakuma} showed that the companion knot $K$ is a torus knot. See \cite{Saito} for further discussion. However, it is only in the proof of Theorem \ref{Main Thm complicated} (which concludes the proof of the \href{Main Theorem}{Main Theorem})  in Section \ref{sec:Concluding Arguments} that we use of this restriction. 

In Section \ref{Comparisons} below, we give a variety of examples showing that the inequality in the \href{Main Theorem}{Main Theorem} is sharp.

\begin{remark}\label{quibbles}
The literature contains three different conventions regarding what it means for a knot $K$ in a 3-manifold with a genus $g$ Heegaard surface to have $\b_g(K) = 0$. In Doll's original article \cite{Doll}, every $K$ has $\b_g(K) \geq 1$. Other authors say that $\b_g(K) = 0$ if and only if $K$ is isotopic into a Heegaard surface of genus $g$. The third convention is that $\b_g(K) = 0$ if and only if $K$ is isotopic to a core loop in one of the handlebodies of some genus $g$ Heegaard splitting; consequently, if $K$ is isotopic into a genus $g$ Heegaard surface but is not isotopic to such a core, then $\b_g(K) = 1$. We adopt the third convention, although the advantage of the second convention is that it would allow us to remove the ``nontorus'' hypothesis concerning $K$ in our \href{Main Theorem}{Main Theorem}.

Our preference for the third convention arises not just from our techniques (where core loops and arcs in handlebodies play a significant role) but also from our fondness for Schubert's Additivity Theorem and its generalizations. With the third convention, Schubert's Additivity Theorem for $\b_0$ remains true as stated and it can be generalized to $\b_1$ (Corollary \ref{Schubert Cor} below). Under the second convention the Additivity Theorem for $\b_0$ becames much more complicated. In his paper, wherein he uses the first convention, Doll presents counterexamples to a generalization of Schubert's Additivity Theorem to $\b_1$, but those counterexamples disappear under the third convention. In the setting of this paper, where we generalize the Additivity Theorem and use it to generalize and reprove the Satellite Theorem, the third convention makes for the cleanest statements, the additional hypothesis on $K$ notwithstanding. Of course, in other settings, one of the other conventions may make more sense.

Every knot contained in a 3-ball in $M$ can be isotoped into a regular neighborhood of the unknot (i.e. a knot bounding a (tame) disc), in such a way that $\omega$ is arbitrarily large. When $M$ is a lens space, the genus 1 bridge number of the unknot is 1, so we cannot relax our convention discussed above that a companion knot $K$ for a satellite is nontrivial. On the other hand, when $K$ is a core loop of a genus 1 Heegaard splitting for $M$ (whether $M$ is $S^3$ or a lens space) then $\b_1(K) = 0$, so in our \href{Main Theorem}{Main Theorem}  we could allow $K$ to be a core loop of $M$ when $M$ is a lens space or the unknot when $M = S^3$.
\end{remark}

Our proof is split into two portions; a good portion of each applies much more widely than just in the context of the Main Theorem. In particular, in Corollary \ref{Schubert Cor}, we give a new proof of Schubert's Satellite Theorem for $\b_0$, one that applies to spatial graphs, as well as to links with essential tori in their exteriors. (We note, however, that extending Schubert's Satellite Theorem, to toroidal links requires some care. See Example \ref{link ex} below.) Once the machinery is developed, our proof is significantly less involved than Schulten's (already very nice) proof. One additional feature of our work is that the proofs are algorithmic in nature and can likely be translated into \emph{bona fide} algorithms for positioning a companion torus nicely with respect to a knot in minimal bridge position. Specifying such an algorithm is beyond the scope of this paper, however we note that normal surface theory \cite{Matveev} provides the framework for algorithmically finding surfaces including the essential discs and discs that we make extensive use of. Our techniques are based on those of Heegaard splitting theory, much of which has been made algorithmic (eg. \cites{Lackenby, Li-algorithm, CGK}), though those algorithms have not yet been extended to bridge surfaces.

\begin{remark}\label{generalizing history}
To what extent can our results be generalized to $\b_g$ for $g \geq 2$, or to other 3-manifolds? To what extent can they be extended to higher genus satellites? In Section \ref{pot gen}, we discuss the extent to which our methods can or cannot be extended to higher genus. By way of exploration, consider the case of large $g$. For any $g$, it is the case that  $\b_g(K) = 0$ if and only if $g$ is at least the Heegaard genus of the exterior of $K$. Suppose that $T$ is a satellite knot with companion $K$ and wrapping number $\omega$. If $\b_g(T) \geq \omega \b_g(K)$ for $g = \g(S^3 \setminus T)$, one would conclude that the tunnel number $t(T)$ (which is one less than the Heegaard genus of the knot exterior) of $T$ is at least $t(K)$. Schirmer \cite{Schirmer} showed that this is true when the wrapping number is 1; that is, for composite knots. Wang and Zou \cite{WZ} showed that the tunnel number of a cable knot is at least the tunnel number of its companion. Li \cite{Li} has shown that $t(T)$ is at least that of its pattern; however, it seems to be an open problem whether or not it is at least that of the companion $K$. Our work points towards some of the difficulties of proving this and the existence of counter-examples to Schubert's inequality when $g = 1$ perhaps indicates some skepticism is in order. 

Naively, one might hope that if the exterior of $T$ contains an essential genus 2 surface $Q$ bounding a handlebody containing $T$, such that each compressing disc for $S$ intersects $T$ at least $\omega \geq 2$ times, then $\b_g(T)$ is at least $\omega \b_g(G)$ where $G$ is some spine for the handlebody. The introduction to \cite{EM2} gives a helpful overview of what is known about higher genus surfaces in the complement of knots of low bridge number. 

If the naive hope were to hold, then if $T$ has $\b_0(T) = 3$, we must have $\b_0(G) \leq 3/2$. In \cite[Theorem 6.1 (2)]{TT-Genus2Graphs}, we showed that such a $G$ would be an unknotted $\theta$-curve. In \cite{Ozawa-3bridge}, Ozawa shows that each incompressible and meridionally incompressible (henceforth, \defn{c-incompressible}) genus two surface $Q$ in the exterior of a knot $T \subset S^3$ with $\b_0(T) = 3$ takes one of three standard forms. Knots with surface of the third form lie in the regular neighborhood of the handlebody having designation $4_1$ in the table of genus 2 handlebody knots \cite{handlebodyknottable}. Since such a handlebody is knotted, it does not have any spine that is an unknotted $\theta$-curve. Thus, if there is such a 3-bridge knot assuming Ozawa's third form (which is almost certainly the case), we have a counterexample to the naive hope for $\b_0$. 

Concerning $\b_1$, in \cite[Section 3]{NC-EM}, Neumann-Coto and Eudave-Mu\~noz construct examples of hyperbolic knots $T$ in $S^3$ such that $\b_1(T) = 2$ and $T$ lies in a knotted genus 2 handlebody $V$ such that the the genus 2 surface $Q = \boundary V$ is c-incompressible in the exterior of $T$. In \cite{EM2}, Eudave-Mu\~noz gives additional constructions of a similar ilk and shows that those constructions classify all c-incompressible genus 2 surfaces $Q$ in the complement of knots $T$ in $S^3$ with $\b_1(T) = 2$.  If the naive conjecture held in this case, it would follow that $V$ had some spine $G$ for which $\b_1(G) \leq 1$. Such spines have a cycle that is a core loop of the genus 1 bridge splitting.  It is not completely evident that the  Neumann-Coto--Eudave-Mu\~noz or Eudave-Mu\~noz examples are counterexamples to the naive hope, since although they are described using a spine of $V$, there might \emph{a priori} be some other spine which does satisfy the naive hope. However, it seems likely that many, if not all, of them will be counterexamples. Nevertheless, one might still hope for some suitable extension of Schubert's theorem to apply to higher genus satellites, perhaps by excluding handlebodies containing spines with certain sorts of cycles.
\end{remark}

Finally, we remark that in \cite{TT-Additive}, we introduced an invariant we called \emph{net extent}. It provides a lower bound on bridge number (with regard to any genus) and has very nice additivity properties. We make some minor modifications to its definition in this paper and use its additivity properties to create new techniques for studying the relationship of essential surfaces to bridge surfaces. 

\subsection{Comparisons and Consequences}\label{Comparisons}

In Remark \ref{generalizing history}, we mentioned work of Eudave-Mu\~noz on c-incompressible genus 2 surfaces. Another facet of that work in \cite{EM2} is worth mentioning. That paper also addresses the situation when $Q$ is a torus. He proves the $\b_1(T) = 2$ version of the following theorem and provides an explicit construction; we can reprove that result and extend it to $\b_1(T) = 3$:

\begin{corollary}[{Eudave-Mu\~noz, see {\cite[Section 2.10, Theorem 3.25]{EM2}}}]
Suppose that $T \subset S^3$ is a satellite knot with companion $K$ and wrapping number $\omega \geq 2$. If $\b_1(T) \in\{2,3\}$, then $\b_1(K) = 1$, that is $K$ is a (1,1) knot. If $K$ is not a torus knot, then also $\omega \leq \b_1(T)$.
\end{corollary}
\begin{proof}
If $K$ is not a torus knot, then by our \href{Main Theorem}{Main Theorem}, $3 \geq \b_1(T) \geq \omega\b_1(K) \geq 2\b_1(K)$. Thus, $\b_1(K) \leq 1$. If $\b_1(K) = 0$, then $K$ is the unknot, contradicting the definition of satellite knot. Thus, $\b_1(K) = 1$ and $\omega \leq \b_1(T)$. If $K$ is a torus knot, then automatically $\b_1(K) = 1$. 
\end{proof}

The next two corollaries of our \href{Main Theorem}{Main Theorem} are modelled on theorems of Schubert concerning genus 0 bridge number.

\begin{corollary}\label{Cor: Whitehead Double}
If $K$ is a nontrivial, noncore, nontorus,  knot in $S^3$ or a lens space, and if $T$ is an $n^\text{th}$ iterated, possibly twisted, Whitehead double of $K$, then $\b_1(T) = 2^n\b_1(K)$.
\end{corollary}
\begin{proof}
Suppose that $L$ is an $n^\text{th}$ iterated, possibly twisted, Whitehead double of $K$. We can construct the exterior of $L$ by gluing the exterior of $K$ to $n$ copies of the exterior of the Whitehead link. As the Whitehead link is hyperbolic \cite{Thurston}, its exterior is anannular. Thus, as the exterior of $K$ is anannular, so is the exterior of any $(n-1)^\text{st}$ iterated twisted Whitehead double of $K$. The fact that $\b_1(L) \geq 2^n \b_1(K)$ follows by induction from our \href{Main Theorem}{Main Theorem}. 

To show the other inequality, suppose that $K'$ is a knot in minimal bridge position with respect to a Heegaard torus $H$ such that $\b_1(K') \geq 1$ (i.e. $K'$ is not a core of one of the solid tori on either side of $H$.) This is depicted in Figure \ref{fig: Whitehead}. Let $V$ be a regular neighborhood of $K'$ and let $V_0$ be a component of $V \setminus H$. Inside $V_0$ place two clasped arcs as in the top left of Figure \ref{fig: Whitehead}. Inside all other components of $V \setminus H$ place two arcs as in the top right of Figure \ref{fig: Whitehead}. Of course, arrange the endpoints of the arcs so that they glue up to be a knot. Twist around a component of $\boundary V \cap H$ as desired to create a twisted Whitehead double $K''$. It is straightforward to see that $H$ is a bridge surface for $K''$. Hence, $\b_1(K'') \leq 2\b_1(K')$. Our result follows by induction on the number of iterates.
\end{proof}

\begin{figure}[ht!]
\centering
\includegraphics[scale=0.35]{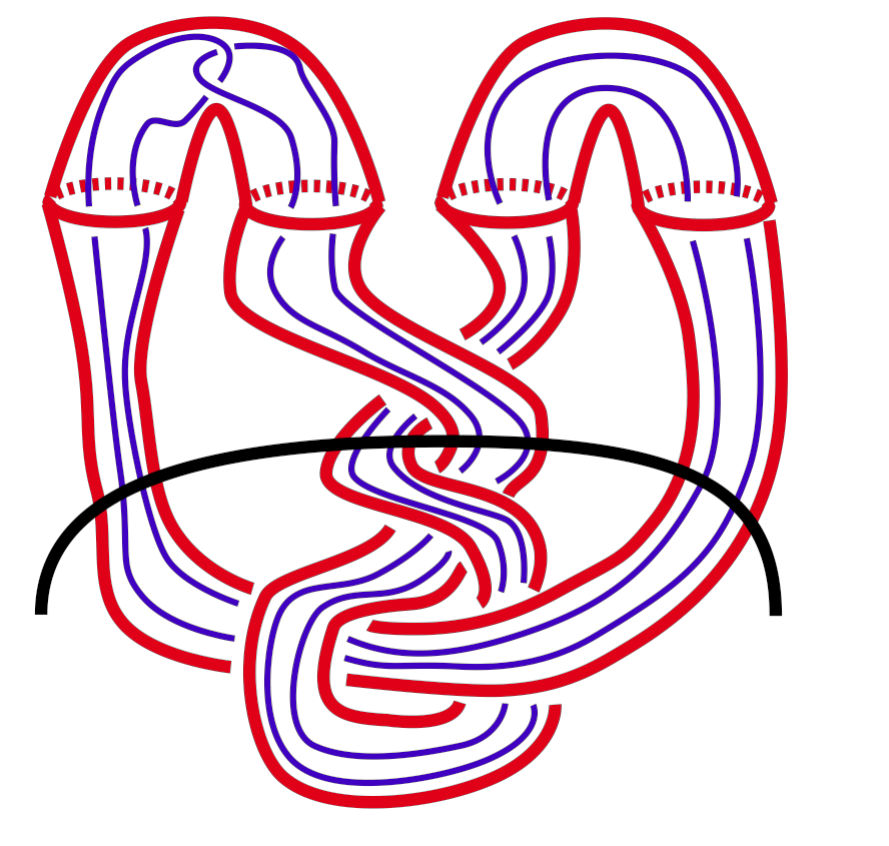}
\caption{A schematic depiction of finding the bridge number of a Whitehead double, described in Corollary \ref{Cor: Whitehead Double}. The companion knot is put in bridge position and then the Whitehead double is placed inside a regular neighborhood. Since it is easier to show, the figure in actuality depicts the example when the companion is a figure eight knot and the bridge surface (in black) is a sphere, but the principle is the same for when the bridge surface is a torus.}
\label{fig: Whitehead}
\end{figure}

\begin{corollary}\label{cables}
Suppose that $T$ is a $(p,q)$ cable of a nontrivial, nontorus, noncore knot $K$ in $S^3$ or a lens space (using the convention that $T$ lies on the boundary of a regular neighborhood of $K$ and minimally intersects a meridian of $K$ $p$ times). Then
\[
\b_1(T) = p\b_1(K).
\]
\end{corollary}
\begin{proof}
The proof is nearly identical to that of Corollary \ref{Cor: Whitehead Double}. By our convention (which is the standard one) $\omega = p$. Since $K$ is not a torus knot or unknot, by the \href{Main Theorem}{Main Theorem}, $\b_1(T) \geq \omega \b_1(K) \geq pb_1(K)$. To show the other inequality, let $H$ be minimal genus 1 bridge surface for $K$. Since $K$ is not a core loop, $\b_1(K) \geq 1$. Isotope $T$ so that it lies on the boundary of a small regular neighborhood of $K$. Each bridge arc of $K\setminus H$ then gives rise to $p$ bridge arcs of $T\setminus H$, so $\b_1(T) \leq p \b_1(K)$, as desired.
\end{proof}

\begin{remark}
Using different methods, Zupan \cite{Zupan} determines $\b_g(T)$ for many iterated cables of torus knots. We also note that, in Corollary \ref{cables}, if $K$ is a core loop of a lens space or the unknot in $S^3$, then the cable $T$ is a torus knot and has $\b_1(T) = 1$, while $\b_1(K) = 0$. 
\end{remark}

\begin{remark}
As can be seen from the proof of Corollaries \ref{Cor: Whitehead Double} and \ref{cables}, the basic strategy can be applied to compute the genus 1 bridge number of broad classes of satellite knots. For example, if $(V, \tau)$ is the mapping torus of a homeomorphism of an $\omega$-punctured disc to itself, then if we embed $(V, \tau)$ in $S^3$ or a lens space so that the core of $V$ is a nontrivial, noncore, nontorus knot $K$ having $\b_1(K) \geq 1$ and if the knot or link $T$ is the image of $\tau$, then $\b_1(T) = \omega \b_1(K)$.  This can be generalized further to allow clasps, as with the Whitehead double.
\end{remark}

\begin{remark}
Corollary \ref{Cor: Whitehead Double} provides an interesting contrast to a theorem of Baker \cite{Baker}, who shows that ``small genus knots in lens spaces have small bridge number.'' For a knot $T$ in a lens space, let $s$ be its order when considered as an element of the fundamental group of the lens space. Considering such knots admitting generalized Seifert surfaces with a single boundary component, Baker shows that if $s \geq 4g - 1$, where $g$ is the Seifert genus, then  $\b_1(T) \leq 1$. All Whitehead doubles have Seifert genus $g = 1$ and order $s = 0$; Corollary \ref{Cor: Whitehead Double} shows that under repeated doublings $\b_1$ is arbitrarily large. This is certainly not the only way to show that some restriction on $s$ is necessary in Baker's theorem. For instance, careful choice of annulus twists and a careful application of \cite{BGL} can also be used to construct such examples including (presumably) hyperbolic ones. On the other hand, the Whitehead doubling construction is a very simple construction.
\end{remark}

\subsection{Overview of the methods}
Considering not just knots and links but also spatial graphs $T \subset S^3$, for the purposes of this paper, we define $\b_0$ as follows; for knots and links this definition is equivalent to all other standard definitions.

\begin{definition}
 Suppose $T \subset S^3$ is a spatial graph (possibly a knot or link). A \defn{bridge sphere} for $T$ is a (tame) 2-sphere $H \subset S^3$ transverse to $T$ and dividing $T$ into two acyclic subgraphs each isotopic relative to its endpoints in $S^3 \setminus H$ into $H$. (See Figure \ref{SatGraph} for an example.) The \defn{genus zero bridge number} $\b_0(T)$ of a spatial graph (possibly a knot or link) $T$ in $S^3$ is the minimum of $|H \cap T|/2$ over all bridge spheres $H$ for $T$.
\end{definition}

\begin{figure}[ht!]
\centering
\labellist
\small\hair 2pt
\pinlabel {$T$} [l] at 415 364
\pinlabel {$H$} [tl] at 371 40
\endlabellist
\includegraphics[scale=0.5]{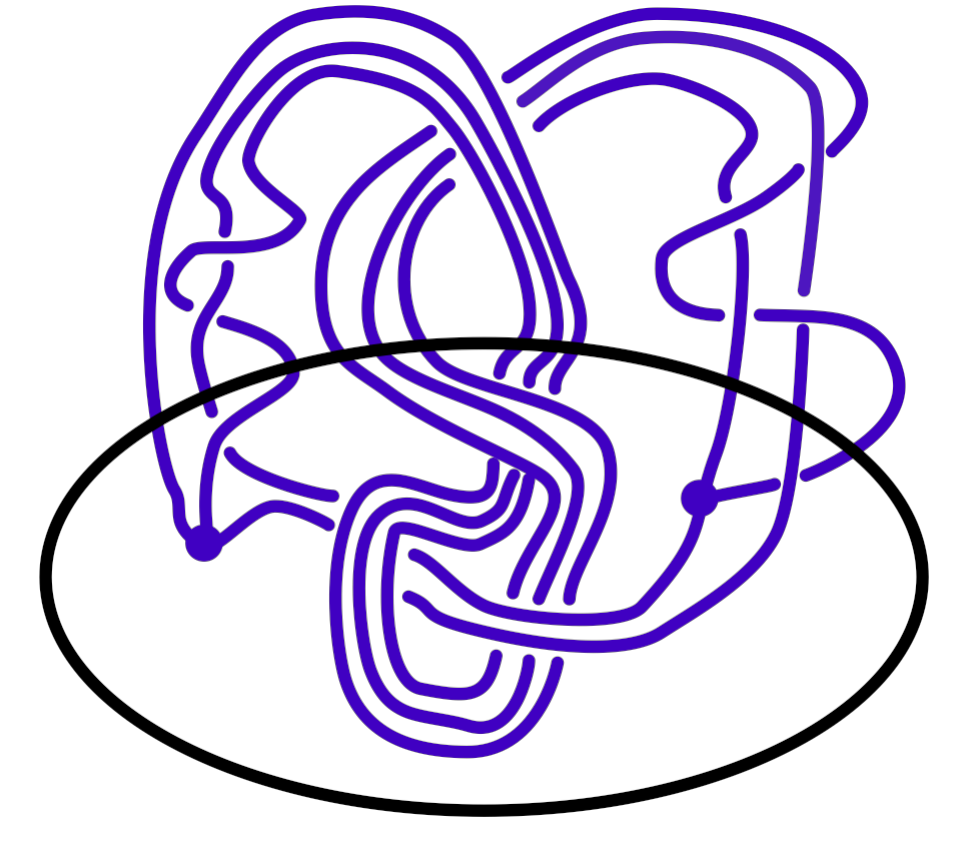}
\caption{The black ellipse represents a sphere $H \subset S^3$ dividing the spatial graph $T$ into two acyclic graphs, each parallel into $H$. Such a sphere is a bridge sphere. In this example, the spatial graph is a satellite graph with companion a figure-eight knot $K$. $T$ has wrapping number 2 with respect to $K$.}
\label{SatGraph}
\end{figure}

The definition of $\b_g$ for $g \geq 1$ is a generalization of this: we place $T$ into a certain kind of ``bridge position'' with respect to a genus $g$ Heegaard surface $H$ and then minimize $|H \cap T|/2$ over all such positions and genus $g$ Heegaard surfaces. We give the details in Section \ref{bridge def}.

\begin{remark}
Bridge position for spatial graphs has been considered previously by Motohashi \cite{Motohashi}, Ozawa \cite{Ozawa-bridgegraph} and others. See \cite[page 455]{TT-Genus2Graphs} for a discussion of the history and competing definitions of (genus zero) bridge number of spatial graphs.
\end{remark}

In approaching the problem of determining lower bounds for $\b_g(T)$ for a satellite knot $T$, a natural approach is as follows: Put $T$ into minimal bridge position with respect to a genus $g$ Heegaard surface $H$. Without changing the fact that $T$ is in minimal bridge position, perform an isotopy so that the solid torus $V$ containing $T$ intersects the bridge surface $H$ in a pairwise disjoint collection of discs. Using the fact that the bridge surface divides $T$ into two trivial tangles, it is easily seen that it divides $Q = \boundary V$ into annuli having the property that any simple closed curve on $Q$ intersecting each annulus in a spanning arc is a companion knot $K$ for $T$. As each disc must intersect $T$ at least $\omega$ (the wrapping number) times, $\b_g(T) \geq \omega \b_g(K)$. (See, for example, \cites{Blair1, Zupan}.) The tricky step is moving $Q$ to the right position. 

In her proof of Schubert's theorem for $\b_0$, Schultens studies the saddles on $Q$ arising from a foliation on $Q$ induced by a Heegaard sphere. She shows that $T$ and $Q$ can be isotoped so that inessential saddles can be canceled with center singularities. Once all inessential saddles have been eliminated, it is easy to construct a minimal bridge sphere $H$ for $T$ so that $V \cap H$ consists of disjoint discs.

For our proof, rather than analyzing a foliation of $Q$, we analyze the bridge surface. In fact, we make significant use of a technique called thin position, where the given bridge surface $H$ is decomposed into the union of thick and thin surfaces. This technique applies to higher genus bridge surfaces as well. The union of these thick surfaces $\mc{H}^+$ and thin surfaces $\mc{H}^-$ is a ``multiple bridge surface'' $\mc{H}$. These are generalizations of bridge surfaces and are a knot-theoretic version of Scharlemann and Thompson's \emph{generalized Heegaard surfaces}. We define them in Section \ref{bridge defs}. There is a partial order on these multiple bridge surfaces and the minimal elements in the partial order have very nice properties; in particular, in \cite{TT-Thin} we showed that the thin surfaces exhibit a prime decomposition for the pair $(M,T)$. In \cite{TT-Additive}, we used this to define new additive invariants for knots, links, and spatial graphs. In \cite{TT-Genus2Graphs}, we used it to study tunnel numbers of genus two spatial graphs and in \cite{Taylor-Equivariant}, Taylor used it to study equivariant Heegaard genus. 

As we discuss in Section \ref{sec:amalg}, a multiple bridge surface $\mc{H}$, when we ignore $T$, can be amalgamated to a Heegaard surface $H$ of genus $g$. We also call $g$ the \emph{net genus} of $\mc{H}$. In the case when $T$ is a knot and $g \in \{0,1\}$, the amalgamation can be done so that $H$ is a genus $g$ bridge surface for $M$. 

\begin{remark}\label{amalgamation rem}
If $g \geq 2$, amalgamation may not be possible. If $g = 0$ and $\boundary M = \nil$, $M$ is necessarily $S^3$; when $g = 1$ and $\boundary M = \nil$, $M$ is necessarily $S^3$ or a lens space. This is one reason (but not the only one) why our \href{Main Theorem}{Main Theorem} concerns $\b_0$ and $\b_1$, rather than $\b_g$ for $g \geq 2$. See Section \ref{pot gen} for more on what can and cannot be generalized to $g \geq 2$.
\end{remark}

We also define the \emph{net weight} of a multiple bridge surface $\mc{H}$ to be $\netw(\mc{H}) = |\mc{H}^+ \cap T| - |\mc{H}^- \cap T|.$ When $\mc{H}$ is a bridge surface (i.e. has a single thick surface and no thin surfaces), $\netw(\mc{H})$ is equal to twice its bridge number. We can turn net weight into a knot invariant $\netw_g(T)$ by minimizing over all multiple bridge surfaces for $T$ of net genus $g$. Details concerning both the net genus and the net weight are given in Section \ref{sec: invar}.

As we explain in Section \ref{sec: thinning}, for any $g \geq 0$, we find a ``locally thin'' multiple bridge surface $\mc{H}$ of net genus $g$ and net weight $\netw(\mc{H}) = 2\b_g(\mc{H})$. Such surfaces have a number of very nice properties, including the fact that (in most cases) they can be isotoped so that the intersection with $\mc{H}$ of \emph{any} given ``c-essential'' surface $Q$ transverse to $T$ will consist of simple closed curves essential on both surfaces. When $Q$ is a torus (the main subject of this paper), either $Q$ is disjoint from $\mc{H}$ or $Q\setminus \mc{H}$ is the union of annuli. As discussed above, if these annuli all have ends that bound disjoint discs in $\mc{H}$, then it is possible to derive an inequality between the net weight of $\mc{H}$ with respect to $T$ and the net weight of $\mc{H}$ with respect to any simple closed curve $K$ on $Q$ intersecting the boundaries of those discs exactly once each (for example, the companion knot). However, there is no \emph{a priori} reason that $Q$ should be positioned in such a way and for $g \geq 1$ it seems extremely difficult, potentially impossible, to ensure that it can be done without the addition of unreasonable assumptions on $T$ or $K$. We proceed instead as follows.

Our proof is divided into three parts. For simplicity, in this overview, we will assume that $Q$ is an essential separating torus in $M$ containing $T$ to one side $V$, which is a solid torus with core $K$, although throughout the paper we do allow $Q$ to separate components of $T$ and some of our arguments allow $Q$ to be of higher genus or have punctures. First (Section \ref{Trivial Cases}), we address various ``reducible'' situations: when there is a sphere intersecting $T$ in two or fewer points and separating $T$ into two nontrivial pieces. Dealing with these cases separately allows for the main arguments to be somewhat less involved. This includes the situation when the wrapping number $\omega = 1$, for in that case the torus $Q$ is not c-essential; it can be compressed to an essential twice-punctured sphere. In particular, if $\omega = 1$, then Schubert's Satellite Theorem follows from his Additivity Theorem. We use this observation as starting point for our work. 

In the second part of our discussion (Section \ref{sec: crushing}), we show that by passing from knots to spatial graphs we do not need to arrange for \emph{all} components of $Q \cap \mc{H}$ to bound disjoint discs in $V \cap \mc{H}$. It is (nearly) enough to find a \emph{single} annulus component of $Q \setminus H$ whose ends bound disjoint discs and which, in the absence of $T$, has a compressing disc (necessarily in $V$) disjoint from the rest of $Q$. We call such an annulus a \emph{crushable handle}.  When there is a crushable handle $A$, we convert $T$ into a spatial graph $\wh{T}$ as in Figure \ref{fig:crushing} by crushing the portion of $T$ inside the annulus to a single edge. We weight that edge $e$ according to the wrapping number $\omega$. The torus $Q$ then compresses to a twice-punctured annulus $\wh{Q}$ and the graph $\wh{T}$ is the connected sum of a certain graph with a companion knot $K$. The knot $K$ is, in effect, the edge $e$ and so it has weight $\omega$. In previous work \cite{TT-Thin}, we showed that a multiple bridge surface for a composite spatial graph can be thinned to a multiple bridge surface realizing a prime decomposition (technically, an \emph{efficient decomposition}) for the spatial graph. In Section \ref{sec: thinning}, we repurpose those techniques, taking into account the edge weights and making a few other minor adjustments. The locally thin multiple bridge surface $\mc{H}$ then realizes the graph $\wh{T}$ as a connected sum $K \# T_1$ where $K$ has weight $\omega$, so all of its punctures are counted with that weight. A relatively brief, but somewhat intricate argument, allows us to drop the $T_1$ factor without increasing net weight, arriving at the inequality:
\[
\b_g(T) \geq \omega \netw_g(K)/2.
\]
(See Theorem \ref{thm: handle crush} for the full statement.) This is enough to reprove Schubert's Satellite Theorem. We do that in Corollary \ref{Schubert Cor}. 

The third step, which occupies the remainder of the paper, is devoted to finding a crushable handle. The annuli of $Q\setminus \mc{H}$ are all either bridge annuli (annuli having both ends on the same thick surface) or vertical annuli (annuli having one end on a thick surface and one end on a thin surface). The bridge annuli are all $\boundary$-compressible in $(M,T)\setminus \mc{H}$ and we sort them into two disjoint types: curved and nested according to the side on which a $\boundary$-compressing disc lies. A curved annulus and a nested annulus form a \emph{matched pair} if they are separated in $Q$ only by vertical annuli. Inspired by Schulten's work where bridge positions are rearranged in order to cancel handles on $Q$, we use techniques stemming from Heegaard splitting theory to show that in many cases we can cancel matched pairs. When $g = 1$, we can cancel enough matched pairs to show that either there is a crushable handle, or $Q$ takes a very particular form, implying that the companion knot $K$ is a torus knot.

Below is the overall structure of the paper. Much of our work applies in far greater generality than to satellite graphs in lens spaces, so we introduce hypotheses only as needed. Readers well versed in thin position techniques, could easily skim Sections \ref{bridge defs} and \ref{sec: thinning}.

\begin{enumerate}
    \item Section \ref{notation} establishes notation and definitions, some classical and some new.
    \item Section \ref{bridge defs} summarizes previous work concerning ``multiple bridge surfaces'' $\mc{H}$ and defines some invariants for them. These are surfaces that are the union of ``thin surfaces'' $\mc{H}^-$ and ``thick surfaces'' $\mc{H}^+$. For most of this section, we assume only that $(M,T)$ is \emph{standard}; the definition is given in Convention \ref{std convention}. 
    \item Section \ref{sec: thinning} explains the process of thinning a bridge surface to create a multiple bridge surface. Thinning is a partial order on the set of multiple bridge surfaces; minimal elements are said to be ``locally thin.'' We establish some lower bounds on the invariants and elucidate the structure of these surfaces when our 3-manifold is either $S^3$ or a lens space and the ``net genus'' of these surfaces is 0 or 1.  We show that in most cases, a locally thin multiple bridge surface can be made to intersect a c-essential surface in curves that are essential in both surfaces. Also important is Section \ref{sec:amalg}, which shows how a multiple bridge surface can sometimes be \emph{amalgamated} back into a bridge surface.
    
    For most of this section, we assume only that $(M,T)$ is standard.

    \item Section \ref{Trivial Cases} discusses some trivial cases which serve as a model of our general approach and are used later to eliminate certain annoyances. We give new proofs of Schubert's and Doll's Additivity Theorems, as well as our Main Theorem in the case when $\omega = 1$. For most of this section, we assume only that $(M,T)$ is \emph{standard}.

    \item In Section \ref{sec: crushing} we explain the operation of crushing a handle and prove the properties of objects after crushing. We prove that if there is a crushable handle, then the conclusion of the \href{Main Theorem}{Main Theorem} holds. We use this to give a new proof of Schubert's Satellite Theorem (for genus zero bridge number). In this section, we assume that $(M,T)$ is standard and irreducible and that $Q \subset (M,T)$ is a c-essential separating torus. The remainder of the paper is devoted to finding a crushable handle.

    \item Annuli and how they are embedded in various 3-manifolds play a very important role in our work. We give some definitions and basic properties of annuli in Section  \ref{sec:annuli}. In this section, we assume only that $(M,T)$ is standard.

\item In Section \ref{sec:subsurface amalgamation} we adapt the operations of thinning and amalgamation to a setting where only certain parts of surfaces are considered. The goal is to ensure that thinning and amalgamation work nicely with respect to $Q$. We assume only that $(M,T)$ is standard and that $Q$ is c-essential and separating; it need not be a torus.
    
    \item Section \ref{sec:nested} considers two types of annulus components of $Q\setminus \mc{H}$, called ``curved'' and ``nested''.  A curved annulus and a nested annulus separated in $Q$ by vertical annuli are said to be a ``matched pair.'' This section develops some basic properties of matched pairs.   We assume that $(M,T)$ is standard. Although, in principle, the techniques can be applied to any closed, non-separating c-essential surface $Q$, most of our results require additional hypotheses on $Q$, for instance requiring it to be an unpunctured torus or four-punctured sphere.
    
    \item Section \ref{sec:Qcomplex} defines the notion of a \emph{cancellable} matched pair. This is similar in spirit to how certain handles can be cancelled in handle theory. This section uses the results of the previous two sections to show how a sequence of amalgamations and thinning operations can be used to eliminate a cancellable matched pair. We assume that $(M,T)$ is standard and irreducible and that $Q$ is a c-essential separating torus or four-punctured sphere. 

    \item Section \ref{sec:noodles} considers the structure that $Q$ can take when there are no matched pairs. In particular, we consider certain annuli known as \emph{tubes} and define a certain surface, called a \emph{spool}, whose existence causes spiraling in $Q$. Throughout we assume that $(M,T)$ is standard and that every closed surface in $M$ is separating. $Q$ can be any c-essential separating surface.

    \item Section \ref{sec:Concluding Arguments} contains the concluding arguments for the \href{Main Theorem}{Main Theorem}. Throughout we assume that $(M,T)$ is standard and irreducible, that $M$ is $S^3$ or a lens space, and that $Q$ is a c-essential torus. We also consider only multiple bridge surfaces $\mc{H}$ with \emph{net genus} equal to 1. We assume that $Q$ does not contain a cancellable matched pair (which we can do by the results in Section \ref{sec:Qcomplex}) and show that it does not contain any matched pair. The results of Section \ref{sec:noodles} allow us to either find a crushable handle or discover that the companion knot $K$ is a torus knot. We conclude by giving the proof of the \href{Main Theorem}{Main Theorem} when $\omega \geq 2$; the case when $\omega = 1$ already having been done in Section \ref{Trivial Cases}.
    
\end{enumerate}

\subsection{Acknowledgements}
Taylor is partially supported by a Colby College Research Grant and NSF Grant DMS-2104022. Tomova is partially supported by NSF Grant DMS-2104026. We are grateful to Ken Baker for several very useful conversations over the course of this project. The referees also provided numerous helpful comments.

\section{Notation and Definitions}\label{notation}

We work in the PL or smooth category. See \cite{FriedlHerrmann} for an interpretation of this in the context of spatial graphs. All surfaces and 3-manifolds we consider are compact and orientable. Throughout this paper, $(M,T)$ will denote a pair where $M$ is a 3-manifold and $T \subset M$ is a properly embedded finite graph. The graph $T$ may be disconnected and may have components that are closed loops. In particular, $T$ may contain components that are knots or itself be a knot or a link. We adopt the following convention throughout the paper:

\begin{convention}\label{std convention}
$M$ is a compact, connected, orientable, 3-manifold without spherical boundary components and $T \subset M$ is a properly embedded finite graph without vertices of degree 2. We insist that every sphere in $M$ is separating. We call $(M,T)$ a \defn{standard} (3-manifold, graph) pair. For a standard 3-manifold graph pair $(M,T)$, we let $(\punct{M}, \punct{T})$ denote the result of removing a regular neighborhood of each vertex of $T$ having degree at least 3 from both $M$ and $T$; i.e. we convert vertices into spherical boundary components. $M$ is a \defn{lens space} if it has Heegaard genus 1 and is not $S^1 \times S^2$ or $S^3$.
\end{convention}

To justify the convention, observe that if $S \subset \boundary M$ is a sphere that does not contain exactly one endpoint of $T$, we may attach a 3-ball to $S$ containing a $\boundary$-parallel graph that is the cone on the points $|S \cap T|$ and thereby convert $S$ into a vertex of the graph. In the case when $|S \cap T| = 2$, we are attaching a 3-ball containing a single unknotted arc. Conversely, if $v$ is a vertex of the graph that has degree at least 2, we may remove a regular neighborhood of $v$, thereby converting $v$ into a spherical boundary component. Finally, if $v$ is a vertex of $T$ of degree 2, we may merge the endpoints of the edges incident to $v$ and thereby absorb $v$ into an edge or loop of the graph. 

We will use $\eta()$ to denote a regular open neighborhood. If $X$ is a space, we let $Y \cpt X$ mean that $Y$ is a (path) component of $X$. The Heegaard genus of a 3-manifold $M$ is denoted $\g(M)$. 

If $S$ is surface in $M$, we write $S \subset (M,T)$ to mean that $S$ is a submanifold of $M$ and is transverse to $T$, possibly a component of $\boundary M$. We call the points $S \cap T$ \defn{punctures} on $S$. All isotopies of a surface $S \subset (M,T)$ are transverse to $T$. An \defn{end} of an annulus $A$ is a component of $\boundary A$. The genus of $S$ is denoted $\g(S)$ and its Euler characteristic is $\chi(S)$. The genus of a disconnected surface is the sum of the genera of its components.

If $S \subset (M,T)$ is a surface, we let $(M,T)\setminus S = (M\setminus \eta(S), T\setminus \eta(S))$ be the result of removing an open regular neighborhood of $S$ from both $M$ and $T$. Since $S$ and $M$ are orientable, $(M,T)\setminus S$ inherits two copies of $S$ in its boundary; these are the \defn{scars} corresponding to $S$. If $S \subset (M,T)$ is the union of pairwise disjoint spheres, none of which intersects $T$ exactly once, we can cap off each of the scars in $(M,T)\setminus S$ corresponding to $S$ with 3-balls containing $\boundary$-parallel spatial graphs each of which is the cone on the points of intersection between a component of $S$ and $T$. We denote the result by $(M,T)|_S$ and say that we have \defn{surgered} $(M,T)$ along $S$. Similarly, if $D \subset (M,T)$ is a disc, we say that $(M,T)|_D = (M,T)\setminus \eta(D)$ is the result of \defn{$\boundary$-reducing} $(M,T)$ along $D$. If $S \subset (M,T)$ and if $D$ is a disc with $\boundary D \subset S$, then the surface $S|_D$ obtained by pasting in parallel copies of $D$ to the scars of $\boundary D$ in $S\setminus \eta(\boundary D)$ is the result of \defn{compressing} $S$ along $D$. We will usually apply this construction when $|D \cap T| \in \{0,1\}$.

A simple closed curve $\gamma \subset S$ disjoint from $T \cap S$ is \defn{essential} if it does not bound a disc or a once-punctured disc in $S$. Suppose that $D \subset (M,T)$ is an embedded disc with $\boundary D \subset S$ disjoint from $S \cap T$ and the interior of $D$ is disjoint from $S$. Assume also that $|D \cap T| \leq 1$. If $\boundary D$ is essential in $S$, we say that $D$ is a \defn{compressing disc} if $|D \cap T| = 0$ and a \defn{cut disc} if $|D \cap T| = 1$. If $\boundary D$ is inessential, but $D$ is not parallel to a disc in $S$ by an isotopy transverse to $T$, then $D$ is a \defn{semi-compressing disc} if $|D \cap T| = 0$ and a \defn{semi-cut disc} if $|D \cap T| = 1$.  An \defn{sc-disc} is any of these four types of discs. 

If $S \subset (M,T)$ has no compressing discs, it is \defn{incompressible}, if it has no c-discs, it is \defn{c-incompressible}. If $S$ is incompressible; $S \setminus \eta(T)$ is not $\boundary$-parallel in $M \setminus \eta(T)$; and $S$ is not a 2--sphere bounding a 3-ball disjoint from $T$, then $S$ is \defn{essential}. If $S$ is essential and is also c-incompressible, then it is \defn{c-essential}. If $S$ admits c-discs $D_1$ and $D_2$ on opposites sides with disjoint boundaries, then $S$ is \defn{c-weakly reducible}, otherwise it is \defn{c-strongly irreducible}. 

We say that $(M,T)$ is \defn{0-irreducible} if every sphere in $M\setminus T$ bounds a 3-ball in $M\setminus T$. We say it is \defn{1-irreducible} if there is no sphere in $M$ intersecting an edge of $T$ exactly once and \defn{2-irreducible} if there is no essential twice-punctured sphere in $(M,T)$. We say that $(M,T)$ is \defn{irreducible} if it is 0-irreducible and 1-irreducible. (This terminology is modelled on  terminology from Heegaard splittings.) 

\begin{convention}
For much of the paper (though not its entirety), $Q \subset (M,T)$ will be a separating surface with one component $V$ of $M \setminus Q$ designated as the \emph{inside} of $Q$ and the other the \emph{outside}. If $X \subset Q$ and if $P$ is a submanifold (of any dimension) of $M$ such that there is a nonempty neighborhood of $X \cap P$ in $P$ that is on the inside (resp. outside) of $Q$, then $P$ is \defn{inside} (resp. \defn{outside}) $X$. For example, if $S \subset (M,T)$ is a surface and a loop $\gamma \cpt S \cap Q$ separates $S$ into two subsurfaces $S_1$ and $S_2$, then one of $S_1$ and $S_2$ is inside $\gamma$ and the other is outside $\gamma$ although both may contain other intersections with $Q$.
\end{convention}

\subsection{Compressionbodies and bridge surfaces}\label{bridge def}

We define a vertex-punctured compressionbody (henceforth, VPC) which can be thought of as a blending of the usual notions of a compressionbody, a trivial tangle, and the notion of ``core loop'' of a handlebody. See Figure \ref{fig:vpcompressionbody}. This definition can be traced back to \cite{Tomova} and is closely related to that of \emph{orbifold handlebodies} \cite{Zimmermann}. The terminology is perhaps a little unfortunate, but we find it useful to distinguish our definition from others in the literature. See Figure \ref{fig:vpcompressionbody} below for an example.

\begin{definition}
A \defn{trivial ball compressionbody} is a pair homeomorphic to $(B^3, \tau)$ that is the cone on finitely many points (possibly zero) in $S^2 = \boundary B^3$. Note that, if nonempty, $\tau$ is either an unknotted arc or a tree with exactly one interior vertex (having degree at least 3) and which is parallel in $B^3$ relative to its endpoints into $\boundary B^3$. We set $\boundary_+ B^3 = \boundary B^3$ and $\boundary_- B^3 = \nil$.

A \defn{trivial product compressionbody} is a pair homeomorphic to $(F \times [0,1], \text{ points } \times [0,1])$ where $F$ is a closed, connected, orientable surface. We set $\boundary_+ (F \times [0,1]) = F \times \{1\}$ and $\boundary_- (F \times [0,1]) = F \times \{0\}$. A \defn{trivial compressionbody} is either a trivial ball compressionbody or a trivial product compressionbody.

A \defn{vertex-punctured compressionbody} (henceforth \defn{VPC}) $(C, T_C)$ is a pair such that $C$ is connected, and where one component of $\boundary C$ (which is necessarily nonempty) has been designated as $\boundary_+ C$; furthermore, there is a collection of pairwise disjoint sc-discs $\Delta \subset (C, \tau)$ for $\boundary_+ C$ such that $(C, \tau)|_\Delta$ is the disjoint union of trivial compressionbodies with $\boundary_+$ designations inherited from $\boundary_+ C$. We call $\Delta$ a \defn{complete collection of sc-discs for $(C, T_C)$}. We let $\boundary_- C = \boundary C \setminus \boundary_+ C$.
\end{definition}

\begin{remark}
As with the compressionbodies of Heegaard splitting theory, there are other ways of formulating the definition of VPC. For example, we can start with trivial compressionbodies and attach 1-handles, possibly containing their cores, to $\boundary_+$ in such a way that the cores of the 1-handles (if included) are attached to the endpoints of the graphs in the trivial compressionbodies and the result is connected. Dually, we can start with a pair $(H, p)$ consisting of a closed orientable connected surface $H = \boundary_+ C$; thicken to $(H \times [0,1], p \times [0,1])$, with $\boundary_+ C = H \times \{1\}$); and then attach 2-handles (possibly containing their co-cores) to $(H\setminus p) \times \{0\}$; finally cap off some or all sphere components not in $\boundary_+ C$ with trivial ball compressionbodies.
\end{remark}

\begin{remark}
Observe that if $(B^3, \tau)$ is a trivial ball compressionbody such that $\tau$ has an internal vertex, then $(\punct{B}^3, \punct{\tau})$ is a trivial product compressionbody. 
\end{remark}

\begin{remark}\label{disjoint bits}
If follows from the construction that if $(C, T_C)$ is a VPC and if $T_C$ is a 1-manifold, then any component of $T_C$ disjoint from $\boundary_+ C$ is either a core loop for $C$ or the core of a 1-handle for $C$.
\end{remark}

The following definitions are essentially due to Doll \cite{Doll}.

\begin{definition}[Bridge surface, bridge number]
If $(M,T)$ is a (3-manifold, graph) pair a surface $H \subset (M,T)$ is a \defn{bridge surface} for $(M,T)$ if $(M,T) \setminus H$ is the disjoint union of two VPCs $(C_1, T_1)$ and $(C_2, T_2)$, such that after regluing, $H = \boundary_+ C_1 = \boundary_+ C_2$. 

The \defn{genus g bridge number} $\b_g(T)$ of $(M,T)$ (or of $T$ in $M$) is defined to be the minimum of $|H \cap T|/2$ over all bridge surfaces $H$ for $(M,T)$ of genus $g$. 
\end{definition}

\begin{remark}\label{std result}
For a knot or link $T$, $\b_g(T)$ is a nonnegative integer. For a spatial graph it may be a half integer. It is a standard result, and easily proven, that if $g \geq \g(M)$ (where $\g(M)$ is the Heegaard genus of $M$), then $\b_g(T)$ is defined.  

Observe that it follows from Remark \ref{disjoint bits} that $\b_g(T) = 0$ if and only if there is a genus $g$ Heegaard surface for the exterior of $T$ in $M$. As in Remark \ref{quibbles}, this differs from Doll for whom $\b_g > 0$ and from some other authors who would say that if $T$ is isotopic into some bridge surface than it has bridge number 0.

Given a genus $g$ bridge surface $H$ for $(M,T)$, if there is an arc component of $T\setminus H$, we can attach a tube along it to obtain a genus $g + 1$ bridge surface $H'$ for $K$ with $|H \cap K| = |H' \cap K| + 2$. Thus, $\b_g(K) \geq \b_{g+1}(K) + 1$, if $K$ is a link with $\b_g(K) \geq 1$. If there is no such arc, we can still attach a tube to $H$, but one that is disjoint from $K$, to obtain a genus $g+1$ for $K$, showing that $\b_g(T) \geq \b_{g+1}(T)$ in all situations. 
\end{remark}

\subsection{Weights}
We will have occasion to consider graphs $T$ with weighted edges. In particular, as we shall see, introducing weights allows for an elegant extension of Schubert's Satellite Theorem to links having an essential torus in their exterior.

\begin{definition}
A positive integer-valued function on the edges and loops of $T$ is a \defn{weight system} on $T$. A spatial graph $T$ is a \defn{weighted spatial graph} if we have some weight system in mind.
\end{definition}

\begin{convention}
The weights on $T$ play no topological role, they are only important for the arithmetic. In our various calculations, we will need to pass back and forth between considering each edge and loop to have weight 1 and considering the possibility that some edge or loop has weight greater than 1. For clarity, we will use boldface to emphasize that the weights are playing a role. That is, $T$ means that each edge or loop has weight 1 or (equivalently) that the weights are irrelevant; while $\weighted{T}$ means that some weight is greater than one or that the weights matter.
\end{convention}

\begin{definition}\label{def: weights}
Suppose $\weighted{T}$ is a weighted spatial graph. If $p$ is in the interior of an edge or loop of $T$, we let $w(p) = w(p;\weighted{T})$ denote the weight of the edge containing it. If $S \subset (M,T)$ is a surface, the \defn{weight} of $S$ is $w(S) = w(S;\weighted{T}) = \sum\limits_{p \in S \cap T} w(p)$. If $v$ is a vertex of $T$, we let $w(v) = w(v;\weighted{T})$ be the weight of a sphere that is the boundary of a regular neighborhood of $v$ in $M$. Note that converting a vertex to a spherical boundary component or vice-versa does not change its weight.
\end{definition}

\begin{remark}
Observe that if $S|_D$ is obtained by compressing a surface $S \subset (M,T)$ along a disc $D$, then $-\chi(S|_D) = -\chi(S) - 2$ and $w(S|_D) = w(S) + 2w(D)$, where $\chi$ is Euler characteristic.
\end{remark}

\subsection{Sums}\label{sec: sums}
Let $\mathbb{S}(k)$ be the (3-manifold, graph) pair that is the suspension of a $k$-punctured $S^2$. Thus, $\mathbb{S}(0) = (S^3, \nil)$ and $\mathbb{S}(2) = (S^3, \text{unknot})$. Given standard (3-manifold, graph) pairs $(M_i, T_i)$ for $i = 1,2$ and $k \in \{0,2,3\}$ we can form a sum $(M_1, T_1) \#_k (M_2, T_2)$ by taking the connected sum of $M_1$ and $M_2$ along points in $M_i$ of degree $k$ in $T_i$. When $k = 0$, the sum is performed at points disjoint from $T_1$ and $T_2$ and is called a \defn{distant sum}. When $k = 2$, the sum is performed at points interior to edges or loops of $T_1$ and $T_2$ and is called a \defn{connected sum}. When $k = 3$ the sum is performed at trivalent vertices of $T_1$ and $T_2$ and is called a \defn{trivalent vertex sum}.  If $T_1$ and $T_2$ have weights and $k = 2,3$, we insist that edges of the same weight are joined under the sum. The pair $\mathbb{S}(k)$ is an identity element for $\#_k$. It is possible to make corresponding definitions when $k \geq 4$, but the sums are not as well-behaved \cite{Wolcott}. The image of the boundary of a regular neighborhood of the summing points is a \defn{summing sphere} in $(M,T)$. Conversely, if there is a separating unpunctured or twice- or thrice-punctured separating sphere $S \subset (M,T)$, we can realize $(M,T)$ as the sum of two other (3-manifold, graph) pairs whose union is $(M,T)|_S$.

\begin{definition}
Assume that $(M,T)$ is standard and 1-irreducible. Suppose that $S \subset (M,T)$ is the union of pairwise disjoint spheres, each intersecting $T$ at most thrice (recalling that none can intersect it exactly once).  If $(M,T)|_S$ contains no essential sphere with three or fewer punctures but if for every proper subset of components $S' \subset S$ the surgered pair $(M,T)|_{S'}$ does contain an essential sphere with three or fewer punctures, we call $S$ an \defn{efficient summing system}.
\end{definition}

\begin{remark}
This terminology is drawn from \cite{Petronio}. In \cite[Theorem 2.4]{Taylor-Equivariant} it is extended to the case when $M$ contains nonseparating spheres, but we will not need it. Additionally, in this paper we do not need to consider thrice-punctured spheres, however we anticipate future work that will make use of them and little effort is expended by their inclusion here. We omit a proof of the following theorem; see \cites{Petronio, Taylor-Equivariant} for the essential ideas, albeit in a slightly different setting. 
\end{remark}

\begin{theorem}\label{thm: efficient}
Suppose that $(M,T)$ is standard and 1-irreducible. Suppose that $S, S'$ are efficient summing systems. Then $(S, T\cap S)$ and $(S', T \cap S')$ are homeomorphic as are $(M,T)|_S$ and $(M,T)|_{S'}$. Furthermore, if $\weighted{T}$ is weighted and does not contain an essential twice-punctured sphere with punctures of distinct weights, then the homeomorphisms also preserve weights.
\end{theorem}

Note the theorem does not claim that $S$ and $S'$ are isotopic in $(M,T)$.

\begin{remark}
Of particular importance to us are sums of the form $(M,T) = (M_0, T_0) \#_2 (M_1, T_1)$ where $T_0$ is a knot. For such sums, there is a ``swallow-follow torus'' obtained by tubing the summing sphere to itself along the arc that is the remnant of $T_0$ in $T$. The torus ``follows'' a portion of $T_0$ and ``swallows'' $T_1$. Conversely, if there exists an essential, unpunctured separating torus $Q \subset (M,T)$ such that $Q$ is cut-compressible into a side $V \cpt M\setminus Q$, then $(M,T) = (M_0, T_0) \#_2 (M_1, T_1)$. The arc of $T$ intersecting the cut-disc is converted into a knot component of either $T_0$ or $T_1$. If it is $T_0$, say, the other components of $T_0$ are exactly the components of $T\setminus V$.
\end{remark}

Finally, if $(M,T)$ is a standard pair, a \defn{cut edge} for $T$ is an edge of $T$ intersecting a once-punctured sphere $S \subset (M,T)$ exactly once. If $(M_0, T_0)$ and $(M_1, T_1)$ are standard pairs and if $e$ is any edge joining $T_0$ to $T_1$ in $(M_0, T_0) \#_0 (M_1, T_1)$ and intersecting a summing sphere for the distant sum exactly once, we say that $(M,T_0 \cup T_1 \cup e) = (M_0, T_0) \#_e (M_1, T_1)$ is a \defn{cut edge sum}.

\begin{lemma}\label{cut uniqueness}
    Suppose that $(M_0, T_0)$ and $(M_1, T_1)$ are standard pairs and that for $i = 0,1$, $e_i$ is an edge such that $(M, T'_i) = (M_0, T_0) \#_{e_i} (M_1, T_1)$ is a cut edge sum. If $e_0$ and $e_1$ have identical endpoints then there is an isotopy of $T'_0$ to $T'_1$ taking $e_0$ to $e_1$ and preserving the factors of the sum (although the factors may move during the isotopy). 
\end{lemma}
\begin{proof} For simplicity, in this proof, we drop the requirement that any degree one vertices of $T$ lie in $\boundary M$. For $i = 0,1$, let $S_i \subset (M,T)$ be a sphere intersecting $T'_i$ in a single point $p_i \in e_i$. For $j = 0,1$, consider the (non-standard) pairs $(\wh{M}_0, \wh{T}_{0j})$ and $(\wh{M}_1, \wh{T}_{1j})$ obtained from $(M,T'_j)\setminus S_j$ by collapsing the remnants of $S_j$ to vertices $v_{0j}$, and $v_{1j}$ of degree 1 where $\wh{T}_{ij}$ contains both $v_{ij}$ and $T_i$ for $i = 0,1$. Each edge $e_j$ restricts to edges $e_{0j}$ and $e_{1j}$ in $(\wh{M}_0, \wh{T}_{0j})$ and $(\wh{M}_1, \wh{T}_{1j})$ respectively. Note that each of $e_{00}$ and $e_{01}$ have an endpoint (namely $v_{01}$ and $v_{11}$ respectively) of degree 1 and not contained in $\boundary M$. There is, therefore an isotopy fixing $T_0$ that takes $e_{00}$ to $e_{01}$. Considering $(\wh{M}_1, \wh{T}_{10})$ to lie in a neighborhood of $v_{01}$, this is an isotopy taking $(M, T'_0)$ to a graph $(M, T'')$ which coincides with $(M, T'_1)$ on the side of $S$ containing $T'_j$. Furthermore, the isotopy takes $T_1$ to itself and takes $e_{10}$ to itself. A similar argument on the other side of $S$ produces an isotopy taking $e_{10}$ to $e_{11}$, $T_0$ to $T_0$, $T_1$ to $T_1$, and $e_{01}$ to $e_{01}$. The composition of these isotopies is the isotopy we desire.
\end{proof}

\subsection{Satellites and Lensed Satellites}\label{satellites}

Suppose that $(M,T)$ is a (3-manifold, graph) pair with $T \neq \nil$ and that $Q \subset (M,T)$ is an unpunctured essential separating torus compressible to a side $V \cpt M\setminus Q$. We call $Q$ a \defn{companion torus} for $T$. A simple closed curve $K \subset Q$ transversally intersecting the boundary of a compressing disc for $Q$ in $V$ exactly once is a \defn{companion knot} for $T\cap V$ (relative to $V$). The \defn{wrapping number} of $T\cap V$ (with respect to $V$ or $K$) is the minimal number of times $T$ intersects a compressing disc for $Q$ in $V$. These definition generalize those given in the introduction, since it may be that $T \setminus V \neq \nil$.  Observe that any two companion knots for $T$ with respect to $V$ are isotopic in $M$.  

\begin{remark}
A particular satellite $T$ can have different companion knots and tori, each with their own wrapping number.
\end{remark}

Generalizing further, a \defn{weighted companion} for $T$ (with respect to $V$) is  $\weighted{L} = (T \setminus V) \cup \weighted{K}$ where $\weighted{K}$ is a companion knot $K$ for $T\cap V$ that has been given weight $\omega$, which is the wrapping number of $T\cap V$ in $V$. As indicated by the notation, we give $L\setminus K = T\setminus V$ weight 1. If $Q$ is compressible to both sides, this weighted companion depends on the choice of $V$.

It would be natural to require the companion torus $Q$ to be incompressible in $M\setminus V$. We don't include that in the definition, although many of our results require it; we find it informative to keep track of where that hypothesis is used. For instance, one might think that it should be possible to immediately apply Schulten's arguments in \cite{Schultens} to deduce a version of Schubert's Satellite Theorem for the genus zero bridge number of toroidal links, including those where the essential torus in the link complement is isotopic to a Heegaard torus. However, Example \ref{link ex} shows that the theorem cannot be generalized that far. 

Before giving the example, we show how considering weighted graphs allows for particularly elegant statements of a version of Schubert's Satellite Theorem and our \href{Main Theorem}{Main Theorem}. We also anticipate that this will be useful terminology in future work concerning bridge numbers of knots with Conway spheres. 

\begin{definition}
Suppose that $\weighted{T}$ is a weighted spatial graph in a 3-manifold $M$. For $g \geq \g(M)$, the \defn{genus $g$ bridge number} of $T$ is $\b_g(\weighted{T}) = \min w(H;\weighted{T})/2$ where the minimum is over all genus $g$ bridge surfaces $H$ for $(M,T)$.
\end{definition}

Suppose that $T$ is a satellite knot in $S^3$ having companion knot $K$ and wrapping number $\omega \geq 1$ with respect to $K$. Give $K$ the weight $\omega$. Then the conclusion of Schubert's Satellite Theorem can be phrased as 
\[
\b_0(T) \geq \b_0(\weighted{K}).
\]
As a consequence of our work, we prove the following version of Schubert's theorem, which also applies to links:
\begin{cor-link version}[Schubert's Satellite Theorem]
Suppose that $T \subset S^3$ is a spatial graph and that $Q \subset (S^3,T)$ is an essential unpunctured torus compressible to a unique side $V \cpt S^3\setminus Q$. Suppose that $T\setminus V$ is a (possibly empty) link and that the wrapping number of $T\cap V$ in $V$ is $\omega \geq 1$. Then  
\[
\b_0(T) \geq \b_0(\weighted{L}) \geq \omega\b_0(K).
\]
where $\weighted{L} = (T\setminus V) \cup \weighted{K}$ is a weighted companion for $T$ with respect to $V$ and $\weighted{K}$ is the core loop of $V$ weighted by the wrapping number $\omega$.
\end{cor-link version}

Note that this incorporates Schubert's Satellite Theorem for knots. Here is an example to show that we cannot drop the hypothesis that $Q$ is incompressible in the complement of $V$. 

\begin{example}\label{link ex}
Let $\lambda$ be a Hopf link in $S^3$ in minimal bridge position with respect to a sphere $H \subset S^3$. Let $U$ be a regular neighborhood of one component $\lambda_U$ of $\lambda$ and let $\lambda_V = \lambda \setminus \lambda_U$. Let $U_0 \subset U_1 \subset U$ be solid tori with $\boundary U_0$, and $\boundary U_1$ parallel to each other and to $\boundary U$. Also choose them so that $|U_i \cap H| = 2$, for $i = 0,1$. Let $T_i \subset \boundary U_i$ be a $(p_i, q_i)$-torus knot which wraps $p_i$-times longitudinally around $U_i$ and $q_i$ times meridionally. (This means that there is a compressing disc for $\boundary U_i$ in $U_i$ whose boundary intersects $T_i$ $p_i$ times transversally.) Choose the parameters so that $1 < p_i < q_i$ and $p_i$ and $q_i$ are relatively prime, for $i = 0,1$. Let $T = T_0 \cup T_1$. Note that $|H \cap T|/2 = p_0 + p_1$ and that $H$ is a bridge sphere for $T$. Since $\b_0(T_i) = p_i$, we have $\b_0(T) = p_0 + p_1$. Let $V$ be the exterior of $U_0$, so that $T_1 \subset V$ and $T_0 \cap V = \nil$. Note that the wrapping number of $T_1$ in $V$ is equal to $q_1$. Let $\weighted{L} = T_0 \cup \weighted{K}$ where $K$ is a core loop for $V$ having weight $q_1$. We might as well take $K = \lambda_V$.  In which case, $H$ is a bridge sphere for $L$ intersecting $K$ twice and $T_0$ a total of $2p_0$ times. Consequently, $\b_0(\weighted{L}) = p_0 + q_1$. Thus, $\b_0(T) < \b_0(\weighted{L})$, contrary to what we might expect from a naive restatement of Schubert's Satellite Theorem for toroidal links.
\end{example}

Before moving on, we generalize the notion of wrapping number.

\begin{definition}\label{wrapping number}
Suppose that $(V,T_V)$ is a standard pair with $\weighted{T_V}$ a weighted spatial graph and $(\boundary V) \setminus T_V$ compressible in $V$ (but not necessarily in $V \setminus T_V$). The \defn{weighted wrapping number} of $(V,\weighted{T_V})$ is the minimum of $w(D)$ over all essential discs $D \subset (V, T_V)$. If all edge weights are 1, this is the \defn{wrapping number} of $(V, T_V)$.
\end{definition}

\subsection{Properties of compressionbodies}
We will need some additional terminology concerning VPCs and a reference for some basic properties. Throughout this section, suppose that $(C, T_C)$ is a VPC.

\begin{definition}\label{arc types}
An arc component of $\punct{T}_C$ that has both endpoints on $\boundary_- \punct{C}$ is called a \defn{ghost arc}; an arc component with both endpoints on $\boundary_+ C$ is a \defn{bridge arc}; an arc component with one end on $\boundary_+ C$ and one end on $\boundary_- \punct{C}$ is a \defn{vertical arc}; and a loop component of $T$ disjoint from $\boundary C$ is a \defn{core loop}. Each component of $\punct{T}_C$ is a ghost arc, bridge arc, vertical arc, or core loop. We call an edge or loop of $T$ a ghost arc, bridge arc, vertical arc, or core loop depending on the type of its image in $\punct{T}_C$.  For a VPC $(C, T_C)$ we define the \defn{ghost arc graph} $\Gamma$ as follows. The vertices of $\Gamma$ are the vertices of $T_C$ and the components of $\boundary_- C$. The edges are the ghost arcs of $T_C$.
\end{definition}

\begin{figure}[ht!]
\centering
\includegraphics[scale=0.35]{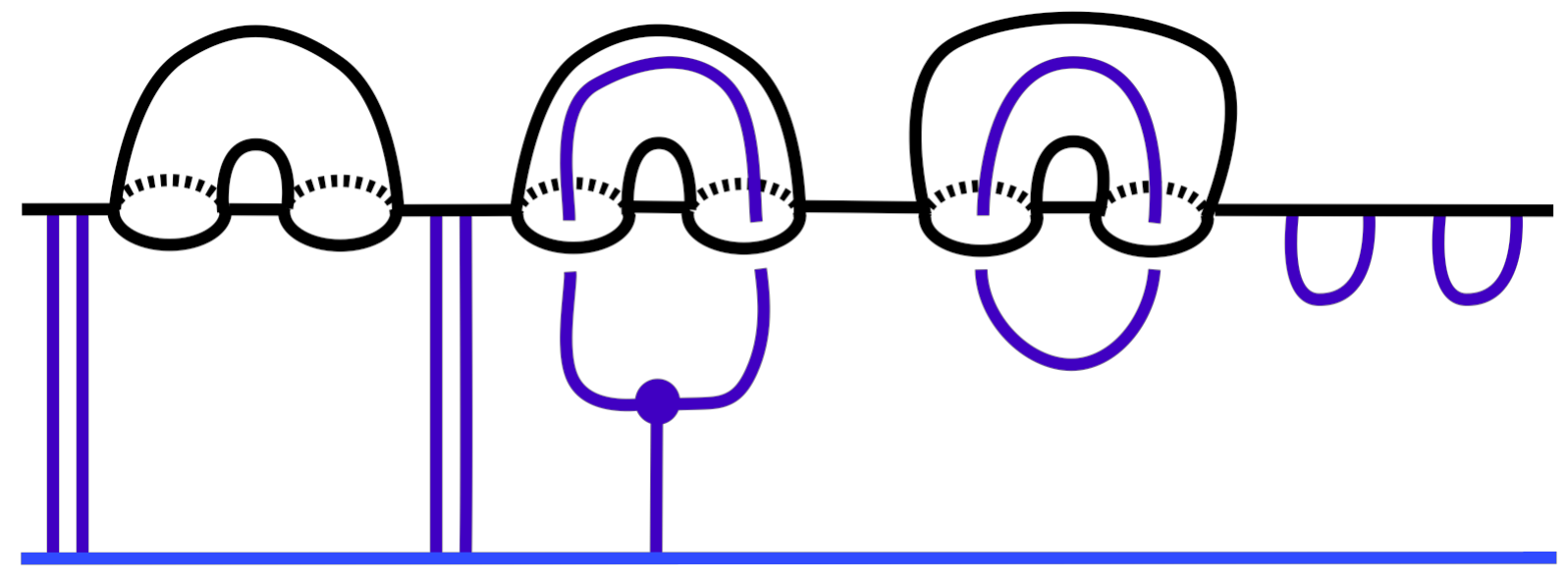}
\caption{Illustrating Definition \ref{arc types}, here is an example of a VPC with (from left to right) four vertical arcs, two ghost arcs, one core loop, and two bridge arcs. In this example, the positive boundary (black) has genus three more than the genus of the negative boundary (blue), each of which is depicted schematically as a horizontal line. One of the ghost arcs joins a vertex to itself and the other joins the vertex to the negative boundary. In this example, the ghost arc graph has two vertices: one is the negative boundary surface and the other is the vertex $v$ of the graph. The ghost arc graph has two edges (the ghost arcs).}
\label{fig:vpcompressionbody}
\end{figure}

\begin{definition}\label{punct prod}
A \defn{punctured product compressionbody} $(D, T_D)$ is obtained from a trivial product compressionbody $(C, T_C)$ by removing  a regular neighborhood of a finite collection of points (not vertices of $T_C$) from the interior of $C$. We let $\boundary_+ D = \boundary_+ C$ and $\boundary_- D = \boundary D \setminus \boundary_+ D$. Note that each component of $\boundary_- D \setminus \boundary_- C$ is an unpunctured or twice-punctured sphere. We also say that $(D, T_D)$ is a a punctured product compressionbody \defn{between} $\boundary_+ C$ and $\boundary_- C$.
\end{definition}

\begin{remark}
It is straightforward to confirm that a punctured product compressionbody is a VPC.
\end{remark}

\begin{lemma}\label{all spheres}
If $\boundary_+ C$ is a sphere, then all components of $\boundary_- C$ are spheres and the ghost arc graph $\Gamma$ is acyclic. Similarly, if $g(\boundary_+ C) = g(\boundary_- C)$, then $\Gamma$ is also acyclic.
\end{lemma}
\begin{proof}
Each ghost arc is dual to a compression, so if $\Gamma$ had a cycle $g(\boundary_+ C) \geq g(\boundary_- C) + 1$. Also, the union of spheres results from compressing a sphere, so if $\boundary_+ C$ is a sphere, so is each component of $\boundary_- C$.
\end{proof}

We also need the following:
\begin{lemma}\label{hidden spheres}
Suppose that $(C, T_C)$ is a VPC such that no component of $\boundary_- C$ is a once-punctured sphere. If $F \cpt \boundary_- C$ admits an sc-disc $D$, then $D$ is either a semi-compressing disc or a semi-cut disc. Let $E \subset F$ be a zero or once-punctured disc with boundary $\boundary D$. If $D$ is a semi-compressing disc, then there exists an unpunctured sphere component of $\boundary_- C$ in one of the components of $C\setminus (D \cup E)$. If $D$ is a semi-cut disc, then there exists an unpunctured or twice-punctured sphere component of $\boundary_- C$ in one of the components of $C \setminus (D \cup E)$.
\end{lemma}
\begin{proof}
Let $\Delta$ be a complete set of sc-discs for $\boundary_+ C$, chosen to intersect $D$ minimally. If $\Delta \cap D = \nil$, then $D$ lies in a product VPC component of $(C, T_C)|_\Delta$. The disc $D$ must then be parallel into $F$, a contradiction. Thus there exists $\zeta \cpt \Delta \cap D$. Choose $\zeta$ to be innermost in $D$. Let $D' \subset D$ be the innermost disc it bounds. Then $D'$ lies in a component $(C', T')$ of $(C, T_C)|_\Delta$ and $\boundary D' \subset \boundary_+ C'$. Thus, $\boundary D'$ bounds a disc $E' \subset \boundary_+ C'$. Furthermore, $E'$ must be contained in a scar of the surgery along $\Delta$. Since we minimized $|\Delta \cap D|$, the sphere $E' \cup D'$ neither bounds a 3-ball nor a (3-ball, trivial arc). Thus, $E' \cup D'$ is parallel in $(C', T')$ to a component of $\boundary_- C' \subset  \boundary_- C$. The result follows.
\end{proof}

Finally, we observe the following. We leave the proof to the reader, though it is nearly identical to the proofs in \cite[Section 3.1]{TT-Thin}.

\begin{lemma}\label{disc existence}
Suppose $(C, T_C)$ is a VPC. Then:
\begin{enumerate}
\item $(C, T_C)$ is 0-irreducible if and only if $\boundary_- C$ contains no unpunctured spheres.
\item $(C, T_C)$ is 1-irreducible if and only if $\boundary_- C$ contains no once-punctured spheres.
\item $(C, T_C)$ is irreducible and 2-irreducible if and only if $\boundary_- C$ contains no spheres with two or fewer punctures. 
\item If $(C,T_C)$  contains no sc-disc for $\boundary_+ C$, it is a trivial ball compressionbody or a product compressionbody.
\item If $(C,T_C)$  contains an sc-disc for $\boundary_+ C$ but no c-disc, it is a punctured product compressionbody with $|\boundary_- C| \geq 2$ or a punctured trivial ball compressionbody with $|\boundary_- C| \geq 1$.
\item If $\boundary_- C$ contains no spheres with two or fewer punctures, then every sc-disc for $\boundary_+ C$ is a c-disc and $\boundary_- C$ does not admit an sc-disc in $(C, T_C)$.
\end{enumerate}
\end{lemma}

\section{Multiple bridge surfaces and their invariants}\label{bridge defs}
We decompose standard pairs $(M,T)$ into pieces that are particularly easy to understand. The surfaces used in the decomposition will be our ``multiple bridge surfaces.''

\subsection{Multiple bridge surfaces}

Multiple bridge surfaces were first introduced in \cite{TT-Thin} where their properties are developed in detail. (In that paper, they were called ``multiple v.p.-bridge surfaces''.) They are adaptations of earlier constructions by Gabai \cite{Gabai}, Scharlemann-Thompson \cites{ScharlemannThompson1, ScharlemannThompson2}, Hayashi-Shimokawa \cite{HS}, and others. We introduce only the relevant definitions and results here and refer the reader interested in more details to \cite{TT-Thin}. 

\begin{definition}\label{Def:multiple bridge surfaces}
A  \defn{multiple bridge surface} for $(M,T)$ is a closed (possibly disconnected) surface $\mc{H} \subset (M,T)$ such that:
\begin{itemize}
\item $\mc{H}$ is the disjoint union of $\mc{H}^-$ and $\mc{H}^+$, each of which is the union of components of $\mc{H}$;
\item $(M,T)\setminus \mc{H}$ is the union of embedded VPCs $(C_i, T_i)$ with $\mc{H}^- \cup \boundary M= \bigcup \boundary_- C_i$ and $\mc{H}^+ = \bigcup \boundary_+ C_i$;
\item Each component of $\mc{H}$ is adjacent to two distinct VPCs.
\end{itemize}
The components of $\mc{H}^-$ are called \defn{thin surfaces} and the components of $\mc{H}^+$ are called \defn{thick surfaces}.  If $T = \nil$, then $\mc{H}$ is also called a \defn{multiple (or generalized) Heegaard surface} for $M$. 
\end{definition}

\begin{remark}
A connected multiple bridge surface is a bridge surface. A bridge surface for $(M,\nil)$ is a Heegaard surface. 
\end{remark}

Suppose that $\mc{H}$ is a multiple bridge surface for $(M,T)$ and suppose that each component of $\mc{H}$ is given a transverse orientation so that for every component $(C,T_C)$ of $(M,T) \setminus \mc{H}$ if the transverse orientation of $\boundary_+ C$ points into (respectively, out of) $C$, then the transverse orientations of all components of $\boundary_- C \cap \mc{H}^-$ point out of (respectively, into) $C$. With these orientations, the edges of the dual graph to $\mc{H}$ in $M$ inherit an orientation, making the dual graph into a digraph, called the \defn{dual digraph}. (This is a version of the \emph{fork complexes} of \cite{SSS}.)

The multiple bridge surface $\mc{H}$ is an \defn{oriented acyclic} multiple bridge surface if the dual digraph is acyclic. (This is easily seen to be equivalent to the definition of ``oriented'' in \cite{TT-Thin}.) The underlying undirected graph may have cycles; we require it to be acyclic as a \emph{directed} graph. We let $\vpoH(M,T)$ denote the set of oriented acyclic multiple bridge surfaces of $(M,T)$ up to isotopy transverse to $T$. 

In the dual digraph to $\mc{H} \in \vpoH(M,T)$, each edge corresponds to a component $\mc{H}$ and each vertex to a VPC of $(M,T) \setminus \mc{H}$. If $x,y$ are edges or vertices (or their corresponding surfaces or VPCs) we say that $y$ is \defn{above} $x$ if there is an edge path in the dual digraph from $x$ to $y$ that is consistent with the orientations. If there is an edge path from $x$ to $y$ that is always inconsistent with the orientations, then $y$ is \defn{below} $x$. Two VPCs are \defn{adjacent} if they share a component of $\mc{H}$. Two components of $\mc{H}$ are \defn{adjacent} if they lie in the same VPC and a VPC and a component of $\mc{H}$ are \defn{adjacent} if the surface is a boundary component of the VPC.

\begin{definition}
Recall that for $\mc{H} \in \vpoH(M,T)$, the dual digraph is acyclic. A VPC $(C, T_C) \cpt (M,T)\setminus \mc{H}$ is an \defn{outermost} VPC with some property if there is no nontrivial directed path in the dual digraph from the VPC $(C, T_C)$ with that property to another VPC with that property. Similarly, a VPC $(C, T_C)$ is an \defn{innermost} VPC with some property if there is no nontrivial directed path from some other VPC with the property to $(C, T_C)$.
\end{definition}

For example, we can refer to the outermost VPC intersecting $T$. If we don't specify a property, then the outermost VPCs correspond to sinks in the dual digraph to $\mc{H}$ and the innermost VPCs correspond to sources in the dual digraph.

\begin{definition}
Suppose that $\mc{H} \in \vpoH(M,T)$ and that $Q \subset (M,T)$ is a surface. We say that $\mc{H}$ is \defn{adapted} to $Q$ if for each component $Q_0 \cpt Q$, either $Q_0 \subset \mc{H}^-$ or all of the following hold:
\begin{enumerate}
    \item $\mc{H}$ is transverse to $Q_0$
    \item $\mc{H} \cap Q_0 \neq \nil$
    \item each component of $Q_0 \cap \mc{H}$ is essential in each of $Q_0$ and $\mc{H}$, and
    \item no component of $Q_0 \setminus \mc{H}$ with boundary is parallel to a subsurface of $\mc{H}$ by a proper isotopy keeping $Q_0$ transverse to $T$.
\end{enumerate}
Note that unless $Q_0 \subset \mc{H}^-$, since $\mc{H}$ is closed, the intersection $Q \cap \mc{H}$ consists of simple closed curves. 

We let $\H(Q)$ be those elements of $\vpoH(M,T)$ that are adapted to $Q$. We consider multiple bridge surfaces in $\H(Q)$ that differ by an isotopy relative to $Q$ and transverse to $T$ to be equivalent. This means that surfaces equivalent in $\H(Q)$ are also equivalent in $\vpoH(M,T)$, but not necessarily vice-versa.
\end{definition}

\subsection{Invariants} \label{sec: invar}
Let $(M,T)$ be a standard pair with $T$ weighted and $S \subset (M,T)$ a surface. For $m \in \N$, we define
\[\begin{array}{rclcl}
\x(S) &=& \x(S;\weighted{T}) &=& (-\chi(S) + w(S;\weighted{T})/2 \\
\x_m(S) &=& x_m(S;\weighted{T}) &=& (-m\chi(S) + w(S;\weighted{T}))/2
\end{array}
\]
As in Definition \ref{def: weights}, for a vertex $v$ of $T$, we define $\x(v)$ and $\x_m(v)$ to be equal to $\x(S)$ and $\x_m(S)$ where $S$ is the sphere that is the boundary of a regular neighborhood of $v$.

For a multiple bridge surface $\mc{H}$, define:
\[
\begin{array}{rcl}
\netchi(\mc{H}) &=&  -\chi(\mc{H}^+) + \chi(\mc{H}^-)\\
\netg(\mc{H}) &=& g(\mc{H}^+) - g(\mc{H}^-) + |\mc{H}^-| - |\mc{H}^+| + 1 \\
\netw(\mc{H}) &=& \netw(\mc{H};\weighted{T})= w(\mc{H}^+) - w(\mc{H}^-)\\
\netx(\mc{H}) &=& \netx(\mc{H};T) = \x(\mc{H}^+) - \x(\mc{H}^-)\\
\netx_m(\mc{H}) &=& \netx_m(\mc{H};\weighted{T}) = \x_m(\mc{H}^+) - \x_m(\mc{H}^-)\\
\end{array}
\]
We call $\netg$ the \defn{net genus}. In its definition, recall that the genus of a disconnected surface is the sum of the genera of its components. The quantity $|\mc{H}^-| - |\mc{H}^+| + 1$ is equal to the genus of the dual digraph to $\mc{H}$. That is, the genus of a handlebody having the dual digraph as its spine. Several of our results are nicer to state with regard to genus, but easier to prove using euler characteristic. 

\begin{remark}\label{dropping weights}
Note that:
\[
\netx_m(\mc{H}) = m\netchi(\mc{H})/2 + \netw(\mc{H})/2
\]
\end{remark}

\begin{definition}\label{def: net w}
For a pair $(M,T)$ with $T$ weighted and $g \geq \g(M)$ define
\[
\netw_g(\weighted{T}) = \min \netw(\mc{H})
\]
where the minimum is taken over all $\mc{H} \in \vpoH(M,T)$ with $\netg(\mc{H}) = g$.
\end{definition}

\begin{remark}\label{rem:various}
Lemma \ref{weight bounded below} shows that the minimum in Definition \ref{def: net w} exists.

If $H$ is a bridge surface for $(M,T)$, then $\netg(H) = g(H)$. Consequently, $\b_g(T) \geq \netw_g(M,T)/2$. Conversely, it turns out that if $\netg(\mc{H}) = g$ and $\netchi(\mc{H}) = x$, then $\mc{H}$ can be ``amalgamated'' (Section \ref{sec:amalg}) to a Heegaard surface for $M$ of genus $g$ with $x = 2g - 2$. If this amalgamation can always be done in a way so that the result is a bridge surface for $(M,T)$, then $\b_g(M,T) = \netw_g(M,T)/2$. However, this may not be possible. Here is an example.

Let $K_1$ and $K_2$ be knots in $S^3$, each of tunnel number $t$ (so their exteriors have Heegaard genus $t + 1$). Let $T = K_1 \# K_2$. Let $\mc{H} \in \vpoH(S^3, T)$ have two thick surfaces, each of genus $t + 1$ that are disjoint from $T$ and a single thin surface which is the summing sphere. Note that $\netg(\mc{H}) = 2t + 2$. Then $\netw(\mc{H})/2 = -1$ but $b_{2t+2}(T) \geq 0$.
\end{remark}

\subsection{Lower bounds on net weight}
For the purposes of attaining bounds, it is also useful to adapt the definitions to VPCs. For a VPC $(C, T_C) \cpt (M,T)\setminus \mc{H}$ with $T_C$ weighted, let $U$ be the (possibly empty) union of some vertices of $T$. We define
\[
\delta_m(C, \weighted{T_C}) = \x_m(\boundary_+ C) - \x_m(\boundary_- \punct{C} \setminus U)- \x(U).
\]
Although we omit $U$ from the notation of $\delta_m$, it does affect the result. We think of $U$ as the collection of vertices we wish to consider as ``unweighted.''

\begin{lemma}\label{negative delta}
Assume that $(M,T)$ is standard, that $T$ is weighted, and that if $v \in U$, then $v$ is incident only to edges of weight 1. Let $(C, T_C) \cpt (M,T)\setminus \mc{H}$. Suppose that $\delta_m (C, \weighted{T_C}) < 0$ and that $m$ is at least the maximal weight of any edge or loop in $\weighted{T}$. Then, $C$ is a 3-ball and $T_C$ contains only vertices in $U$.
\end{lemma}
\begin{proof}
Assume $\delta_m (C, \weighted{T_C}) < 0$. Let $\Delta$ be a complete collection of sc-discs for $(C, T_C)$. We prove the result by induction on $|\Delta|$. If $|\Delta| = 0$, then $(C, T_C)$ is a trivial VPC and, as we explain, it is easy to verify that the result holds. For if $(C, T_C)$ is a trivial product compressionbody, then $\x_m(\boundary_+ C) = \x_m(\boundary_- C)$, so $\delta_m(C, \weighted{T_C}) = 0$. Suppose $(C, T_C)$ is a trivial ball compressionbody. The graph $T_C$ has at most one vertex $v$. If $v$ does not exist or is in $U$, the conclusion holds. If $v \not\in U$, then $\delta_m(C, \weighted{T_C}) = -m + w(\boundary_+ C)/2 + (m - w(v)/2) = 0$. 

We now proceed to the inductive step. Let $D$ be an sc-disc of weight $p$ in $\Delta$ and let $(C', T'_C) = (C, T_C)|_D$. We have:
\[
0 > \delta_m(C, \weighted{T_C}) = m - p + \delta(C', \weighted{T'_C}) \geq \delta_m(C', \weighted{T'_C})
\]
By the inductive hypothesis, at least one component of $(C', T'_C)$ is a 3-ball whose intersection with $T'_C$ contains only vertices in $U$. 

If $(C', T'_C)$ is itself such a 3-ball, then $C$ is a solid torus. Suppose such is the case. If $T'_C$ is empty or a bridge arc, then $T_C$ is either empty or a core loop for $C$, in which case $\delta_m(C, \weighted{T_C}) = 0$, a contradiction. Suppose, therefore, that $T_C$ contains vertices; by the inductive hypothesis, they all belong to $U$. Thus,
\[
\delta_m(C, \weighted{T_C}) = w(\boundary_+ C)/2 - \x(U) = w(\boundary_+ C)/2 + |U| - w(U;T)/2.  
\]
The weight of each vertex in $U$ is its degree, so
\[
\delta_m(C, \weighted{T_C}) = w(\boundary_+ C)/2 + \sum\limits_{v \in U}(2 - \deg(v))/2.
\]
If there is a ghost arc in $T_C$ with distinct endpoints, we can contract it without changing $\delta_m(C, \weighted{T_C})$. By the definition of VPC, we may, therefore, assume that $T_C$ has at most one ghost arc and all other arcs are vertical. The ghost arc, if it exists, has its endpoints at the same vertex of $T_C$. An easy calculation shows that $\delta_m(C, \weighted{T_C}) \geq 0$. Thus, we may assume that $(C', T'_C)$ is the union of two VPCs: $(C_1, T_1)$ and $(C_2, T_2)$. Each of them has a complete collection of sc-discs containing fewer discs than $\Delta$.

Without loss of generality, assume that $(C_1, T_1)$ is a 3-ball and that $T_1$ has only vertices in $U$. We desire to show that this is also true for $(C_2, T_2)$. As in the previous paragraph, if $T_1$ or $T_2$ has a ghost arc edge with both ends in $U$, we may contract it without changing $\delta_m$. As $C_1$ is a 3-ball, this means we may assume that $T_1$ contains at most one vertex; if it does, the vertex is in $U$. 

As $T_1$ contains no ghost arcs (as we have contracted them), each arc of $T_1$ is a vertical arc or bridge arc. The bridge arcs contribute their weight to $w(\boundary_+ C_1;\weighted{T_1})$. Since each vertex of $U$ is incident only to edges of weight 1, each vertical arc in $T_1$ contributes $0$ to $w(\boundary_+ C_1;\weighted{T_1})/2 - w(\boundary_- \punct{C}_1; T_1)$. Thus,
\[
\delta_m(C_1, \weighted{T_1}) = -m + |U \cap T_1| + b
\]
where $b$ is the total weight of all the bridge arcs in $\weighted{T_1}$. Consequently,
\[\begin{array}{rcl}
0 &>& \delta_m(C, \weighted{T_C}) \\
&=& m - p + \delta_m(C_1, \weighted{T_1}) + \delta_m(C_2, \weighted{T_2})\\
&=& m - p - m + |U| + b + \delta_m(C_2, \weighted{T_2}) \\
&=& - p + |U| + b + \delta_m(C_2, \weighted{T_2})
\end{array}
\]
If $p \geq 1$, then some arc of $T_1$ is incident to the scar from $D$. If that arc is a vertical arc, then $p = 1$ and $|U| \geq 1$. If that arc is a bridge arc, then $b \geq p$. Thus, in either case, $-p + |U| + b \geq 0$. Thus, $\delta_m(C_2, \weighted{T_2}) < 0$.

We may, therefore, also apply the inductive hypothesis to $(C_2, T_2)$. Reconstructing $(C, T_C)$, we see that the result also holds for $(C, T_C)$. 
\end{proof}

\begin{lemma}\label{weighted lower bound}
Suppose that $\mc{H}$ is a multiple bridge surface for a standard weighted $(M,T)$. Let $U$ be a subset of $V(T)$, the vertices of $T$. Then
\[\begin{array}{rcl}
2\netx_m(\mc{H}) - \x_m(\boundary M) - \x_m(V(T)\setminus U) - \x(U) = \sum\limits_{(C, T_C)} \delta_m(C, \weighted{T_C})
\end{array}
\]
\end{lemma}
\begin{proof}
This is a version of \cite[Corollary 4.8]{TT-Additive}. It follows from the fact that each component of $\mc{H}$ is incident to two VPCs.
\end{proof}

The next corollary uses a similar counting argument to show that $\netw_g(\weighted{T})$ is well-defined, although as we observed above, it can be negative. A more sophisticated version of this argument is employed in the proof of Theorem \ref{thm: handle crush} below where we need to show that in certain important situations, $\netw(\mc{H})$ is non-negative.

\begin{corollary}\label{weight bounded below}
If $(M,\weighted{T})$ is standard and $g \geq \g(M)$, then $\netw_g(\weighted{T})$ exists and is bounded below by a constant depending only on $(M, \weighted{T})$ and $g$.
\end{corollary}
\begin{proof}
Let $\mu$ be the maximal weight of an edge of $\weighted{T}$. Let $\mc{H} \in \vpoH(M,T)$ have $\netg(\mc{H}) = g$. Consider a VPC $(C, \weighted{T}_C) \cpt (\punct{M},\punct{T}) \setminus \mc{H}$. Let $n(C, T_C)$ be the number of ghost arcs in $T_C$ (equivalently $\punct{T}_C$). As each $\boundary$-reduction along an sc-disc reduces $\delta_1(C, \nil)$ by 2 and decreases the number of ghost arcs by at most 1,
\[
-\chi(\boundary_+ C) + \chi(\boundary_- C)  \geq 2n(C, T_C)
\]
Thus, as bridge arcs contribute positively to $w(\boundary_+ C) - w(\boundary_- C)$ and vertical arcs contribute equally to $w(\boundary_+ C)$ and $w(\boundary_- C)$, and each end of a ghost arc contributes to $w(\boundary_- C)$ and not at all to $w(\boundary_+ C)$, we have
\[\begin{array}{rcl}
\mu(-\chi(\boundary_+ C) + \chi(\boundary_- C)) + w(\boundary_+ C;\weighted{T}) - w(\boundary_- C;\weighted{T}) &\geq&\\
2\mu n(C, T_C) - 2\mu n(C, T_C) &\geq& 0.
\end{array}
\]
Summing over all components $(C, T_C)$ of $(M,T) \setminus \mc{H}$, we have:
\[\begin{array}{rcl}
2\mu \netchi(\mc{H}) + \mu\chi(\boundary M) + 2\netw(\mc{H}) + w(\boundary M) & \geq&\\
\sum\left( \mu(-\chi(\boundary_+ C) + \chi(\boundary_- C)) + w(\boundary_+ C;\weighted{T}) - w(\boundary_- C;\weighted{T})\right) &\geq 0
\end{array}
\]
Consequently,
\[
\netw(\mc{H};\weighted{T}) \geq -\mu(2g - 2) - \mu\chi(\boundary M)/2 -  w(\boundary M; \weighted{T})/2
\]
The right-hand side is independent of $\mc{H}$. As $\netw(\mc{H})$ is always an integer or half-integer, $\min \netw(\mc{H};\weighted{T})$ exists and is bounded below.
\end{proof}

Our last corollary is another lower bound that will be useful later on. For brevity, we omit the arithmetic which is simpler than that of the previous lemma.

\begin{lemma}\label{connected boundary bound}
Suppose that $(M,T)$ is a (not necessarily standard) pair and that $\boundary M$ is nonempty and connected. Then if $\mc{H} \in \vpoH(M,T)$, $\x(\boundary M) \leq \netx(\mc{H})$.    
\end{lemma}
\begin{proof}
Observe that each VPC $(C, T_C)$ of $(M,T)\setminus \mc{H}$ is either above or below its positive boundary $\boundary_+ C$. Call it an \defn{upper} VPC in the former case and a \defn{lower} VPC in the second case. Without loss of generality, assume that the VPC containing $\boundary M$ is a lower VPC. Each thick and thin surface in $\mc{H}$ is contained in exactly one lower VPC. Thus, summing over $(C, T_C) \cpt (M,T)\setminus \mc{H}$, we have
\[
\netx(\mc{H}) - x(\boundary M) = \sum \delta_1(\punct{C}, \punct{T}_C).
\]
It is straightforward to verify that each $\delta_1(C,T_C) \geq 0$. The result follows.
\end{proof}

\subsection{Net genus one multiple bridge surfaces}
We investigate the multiple bridge surfaces of $S^3$ and lens spaces.

\begin{lemma}\label{lens space structure}
Suppose that $\mc{H} \in \vpoH(M,T)$ with $(M,T)$ standard and irreducible and with $M$ having a genus 1 Heegaard surface. Suppose that $\g(M) \leq \netg(\mc{H}) = g  \leq 1$. Then, every component of $\mc{H}$ is a sphere or torus. (There is a torus if and only if $g = 1$.) Furthermore, no sphere separates the collection of tori. Ignoring the spheres, all tori are parallel and each is a Heegaard torus for $M$. 
\end{lemma}

An example of such a multiple bridge surface (without $T$) is shown in Figure \ref{fig:lensspacehs}.

\begin{figure}[ht!]
\centering
\includegraphics[scale=0.35]{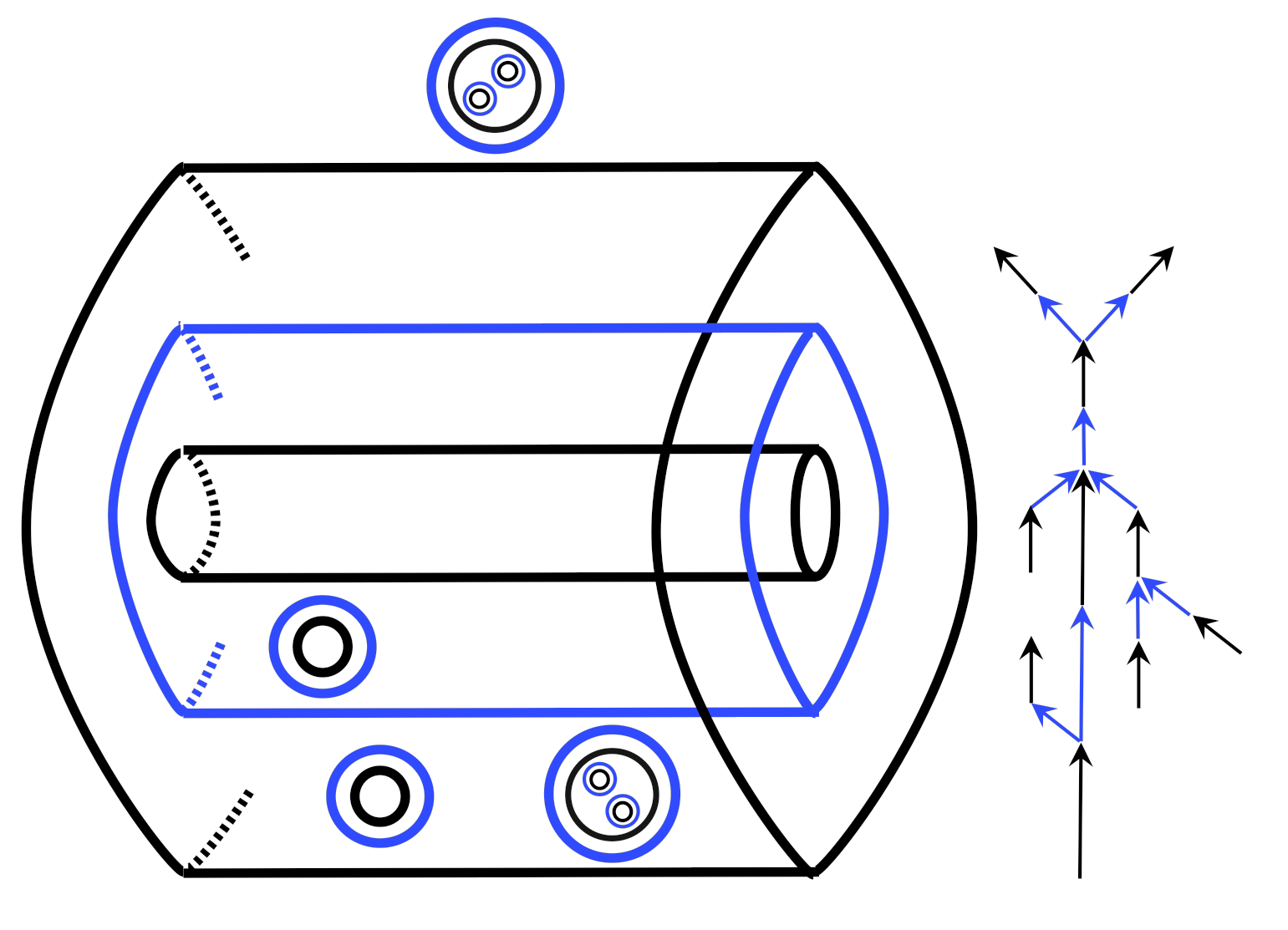}
\caption{As in Lemma \ref{lens space structure}, we depict an example of a multiple bridge surface $\mc{H}$ (the knot or spatial graph $T$ is not shown) of net genus 1. The tori are depicted as annuli with a boundary identification implied. The thick surfaces are black and the thin surfaces blue. With correct choice of orientation, the dual digraph is shown to the right. The long arrows represent tori and the short arrows represent spheres.}
\label{fig:lensspacehs}
\end{figure}

\begin{proof}
It follows from the definition of $\netg$ and the fact that every closed surface in $M$ is separating, that each component of $\mc{H}$ is a sphere or torus and that dual digraph to $\mc{H}$ is a tree. Furthermore, in $M$, $\mc{H}$ amalgamates to a Heegaard torus for $M$. It follows that no sphere component of $\mc{H}^-$ separates two tori components of $\mc{H}^+$. Ignoring the sphere components of $\mc{H}$, all toroidal components are parallel. Thus, each toroidal component of $\mc{H}$ is a Heegaard torus for $M$.
 
Consequently, there exist exactly two compressiobodies $(C, T_C)$ and $(D, T_D)$ of $M\setminus \mc{H}$ with $\boundary_+ C$ and $\boundary_+ D$ tori and $\boundary_- C$ and $\boundary_- D$ the (possibly empty) union of spheres. 
\end{proof}

\begin{corollary}\label{net weight bound}
    Suppose that $(M,T)$ is an irreducible standard pair with $M$ either $S^3$, a lens space, solid torus, or $T^2 \times I$. Let $g \geq \g(M)$ be either 0 or 1. Then
    \[
    \netw_g(T) \geq 0.
    \]
\end{corollary}
\begin{proof}
Let $\mc{H} \in \vpoH(M,T)$ have $\netg(\mc{H}) = g$. As we will explain in Section \ref{wr consol}, if there is a VPC $(C, T_C) \cpt (M,T)\setminus \mc{H}$ that is a trivial product VPC with $\boundary_+ C \cpt \mc{H}^+$ and $\boundary_- C \cpt \mc{H}^-$, we may remove $\boundary_+ C$ and $\boundary_- C$ from $\mc{H}$, without affecting the fact that $\mc{H} \in \vpoH(M,T)$ and without changing $\netg(\mc{H})$ or $\netw(\mc{H})$. For simplicity, we may, therefore, assume that there is no such VPC.

We induct on $|\mc{H}^+|$. Suppose that $|\mc{H}^+| = 1$. In this case $\mc{H}^- = \nil$ and so $\netw(\mc{H}) = |\mc{H} \cap T| \geq 0$. If equality holds, then $T$ is disjoint from $\mc{H}$. If $T \neq \nil$, this means that $\mc{H}$ must be a torus and so $T$ must be a core loop or Hopf link. 

Suppose, therefore, that the result holds for bridge surfaces with up to $n\geq 1$ thick surfaces and that $|\mc{H}^+| \geq n+1$. 

If $F \cpt \mc{H}^-$ is an outermost or innermost thin surface bounding a submanifold $W \cpt M\setminus F$ disjoint from $\mc{H}^- \setminus F$. Let $H = \mc{H}^+ \cap W$.  Let $\Gamma$ be the ghost arc graph for the VPC $(X, T_X)$ bounded by $C$ and $H$. One vertex of $\Gamma$ is $F$; all others are vertices of $T$, each of which has degree at least 3. Let $b$ be the number of bridge arcs of $T_X$; $v_0$ the number of vertical arcs of $\punct{T}_X$ not incident to $F$; and $g_F$ the number of ghost arcs of $\punct{T}_X$ with an endpoint at $F$. 

\textbf{Case 1:} $\mc{H}^-$ does not contain a sphere.

By Lemma \ref{lens space structure}, all components of $\mc{H}$ are tori. Let $F$ be an innermost or outermost such $F$ as above and consider $W$ and $H$ as above. If $\mc{H}^- = F$, then there are two possible choices for $H$. The ghost arc graph $\Gamma$ is acyclic by Lemma \ref{all spheres}. Each isolated vertex of $\Gamma$ (other than $F$) is incident to at least 3 vertical arcs and each leaf of $\Gamma$ (other than $F$) is incident to at least two vertical arcs. Note that
\[
|H \cap T| - |F \cap T| = 2b + v_0 - g_F.
\]
If $F$ is an isolated vertex of $\Gamma$, then $g_F = 0$ and either $b \geq 1$ or $v_0 \geq 3$, as no component is a product compressionbody between a thick and thin surface. Thus, $|H \cap T| - |F \cap T| \geq 2$. If $F$ is incident to at least two ghost arcs, then each subtree of $\Gamma \setminus F$ contains a leaf (not equal to $F$) incident to at least two vertical arcs. There are at least two such leaves, and so
\[
|H \cap T| - |F \cap T| = 2b + v_0 - g_F \geq 2b + 2 \geq 2.
\]
Finally, if $F$ is a leaf of $\Gamma$, then there is at least one other leaf of $\Gamma$ and so
\[
|H \cap T| - |F \cap T| = 2b + v_0 - g_F \geq 2b + 1 \geq 1.
\]

If $|\mc{H}^-| \geq 2$, we can find distinct tori $F_1, F_2 \cpt \mc{H}^-$ such that one is innermost and the other outermost and they bound disjoint submanifolds $W_1, W_2$ in $M$ with interiors disjoint from $M \setminus \mc{H}^-$. Let $H_1, H_2$ be the corresponding thick surfaces.  Let $\mc{H}'$ be the restriction of $\mc{H}$ to $M\setminus (W_1 \cup W_2)$. Then
\[
\netw(\mc{H}) = \netw(\mc{H}') + (|H_1 \cap T| - |F_1 \cap T|) + (|H_2 \cap T| - |F_2 \cap T|) \geq 2
\]
where the last inequality follows from our previous calculations and the inductive hypothesis applied to $\mc{H}'$.

If $|\mc{H}^-| = 1$, there is a unique choice for $F$ and there are tori $H_1, H_2 \cpt \mc{H}^+$ on either side of $F$. Then
\[
\netw(\mc{H}) = (|H_1 \cap T| - |F \cap T|) + (|H_2 \cap T| - |F \cap T|) + |F \cap T| \geq 2
\]
by the previous calculations. Thus, in all cases $\netw(\mc{H}) \geq 2$.

\textbf{Case 2:} $\mc{H}^-$ contains a sphere.

By Lemma \ref{lens space structure}, there exists an innermost or outermost sphere $F \cpt \mc{H}^-$ such that the thick surface $H$ (defined above) is a sphere. Again, the ghost arc graph $\Gamma$ is acyclic. An identical calculation to that above shows that
\[|H \cap T| - |F \cap T| \geq 1\]
and if equality holds, then $b = 0$, $F$ is a leaf of $\Gamma$, and there is a unique other vertex $v$ of $\Gamma$ with $\deg(v) = 3$. Let $(M', T') \cpt (M,T)|_F$ be the component not containing $\mc{H}$. Let $\mc{H}' = \mc{H} \cap M'$. We have:
\[
\netw(\mc{H}) = \netw(\mc{H}') + |H \cap T| - |F \cap T| \geq \netw(\mc{H}') + 1 \geq 1.
\]
The last inequality follows from the inductive hypothesis. 
\end{proof}

\section{Thinning}\label{sec: thinning}
Suppose that $(M,T)$ is standard and $Q \subset (M,T)$. The main goal of this section is to find $\mc{H} \in \vpoH(M,T)$ such that $\mc{H} \cap Q$ consists of loops that are essential in both surfaces. This is a standard application of ``thin position'' techniques, but there are some details special to our situation.

In \cite{TT-Thin}, building on work of many authors (e.g. \cites{Gabai,ScharlemannThompson1, ScharlemannThompson2, HS }) we defined a partial order on $\vpoH(M,T)$. The partial order was defined using a certain, somewhat large, collection of moves on multiple bridge surfaces. In this paper, we restrict attention to only two moves. The corresponding partial order will be denoted $\rightsquigarrow$. In what follows we briefly revisit some of the earlier arguments in this new context; were it not for our need to consider weighted graphs, we could simply use the earlier partial order and appeal to results from \cites{TT-Thin, TT-Additive}; however, the simplifications we present here are likely to be useful in the future.

\subsection{Weak reduction and consolidation}\label{wr consol}

\begin{definition}\label{def: consol}
Suppose that $\mc{H} \in \vpoH(M,T)$ and that some VPC $(C, T_C) \cpt (M,T) \setminus \mc{H}$ is a punctured product compressionbody between $H = \boundary_+ C$ and $F \cpt \boundary_- C \cap \mc{H}^-$. Then $\mc{J} = \mc{H} \setminus (H \cup F)$ is obtained by \defn{consolidating} $\boundary_+ C$ and $F$. 
\end{definition}

\begin{remark}\label{rem: orientations and consolidations}
In the context of Definition \ref{def: consol}, if $P \cpt \boundary_- C \setminus F$, then $P$ is a separating sphere in $M$. Any component of $\mc{J} \setminus \boundary_- C$ on the same side of $P$ as $H \cup F$ retains its orientation, but $P$ and any component on the opposite side of $P$ from $H \cup F$ has its orientation reversed.
\end{remark}

\begin{lemma}
If $\mc{J}$ is obtained by consolidating $\mc{H} \in \vpoH(M,T)$, then $\mc{J} \in \vpoH(M,T)$.
\end{lemma}
\begin{proof}
Let $(C, T_C)$ be the punctured product compressionbody between the thick surface $H$ and the thin surface $F$ that are removed from $\mc{H}$ when forming $\mc{J}$. If $(C, T_C)$ is a product compressionbody (equivalently $|\boundary_- C| = 1$), the result follows immediately from \cite[Lemma 5.5]{TT-Thin}. We briefly address the case when $|\boundary_- C| \geq 2$. 

As we noted in Remark \ref{punct prod}, each component of $P = \boundary_- C \setminus F$ is a zero or twice-punctured sphere. Since $(M,T)$ is standard, each such sphere is separating in $M$. Surger $(M,T)$ along $P$ to get $(M,T)|_P$. Let $\mc{H}_P = \mc{H} \setminus P$. In each component of $(M,T)|_P$, $\mc{H}_P$ is an acyclic oriented multiple bridge surface. Let $(M_0, T_0)$ be the component containing $H \cup F$ and let $\mc{H}_0 = \mc{H} \cap (M_0, T_0)$. Consolidating $\mc{H}_0$ by removing $H \cup F$ preserves the fact it is an oriented multiple bridge surface as they cobound a product compressionbody. Removing a regular neighborhood of a finite collection of points from the interior of a  VPC preserves the fact it is a VPC. Reversing the surgery along $P$, and taking into account the orientations as specified in Remark \ref{rem: orientations and consolidations}, we see that $\mc{J} = \mc{H} \setminus (H \cup F)$ is an oriented acyclic multiple bridge surface.
\end{proof}

The following definition can be traced back to \cite{CassonGordon,ScharlemannThompson2}. 

\begin{definition}\label{def: untelescope}
Suppose that $\mc{H} \in \vpoH(M,T)$ and that $H \cpt \mc{H}^+$ is c-weakly reducible, with disjoint c-discs $D_+$ and $D_-$ on opposite sides of $H$. Let $H_\pm$ be the result of compressing $H$ using $D_\pm$ respectively and let $F$ be the result of compressing $H$ using $D_+ \cup D_-$. Isotope these surfaces to be pairwise disjoint and so that $F$ lies between $H_+$ and $H_-$. Let $\mc{J}$ be the result of removing $H$ from $\mc{H}$ and adding the components of $H_+$ and $H_-$ as thick surfaces and the components of $F$ as a thin surface. See Figure \ref{fig:untel}. We say that $\mc{J}$ is obtained from $\mc{H}$ by \defn{untelescoping}. We call $D_- \cup D_+$ a \defn{weak reducing pair}. 

After untelescoping, there may be either one or two trivial product VPCs between components of $H_- \cup H_+$ and components of $F$. The untelescoping, followed by consolidating those surfaces (if possible) is an \defn{elementary thinning move on $\mc{H}$}. Not all components of the thin surface $F$ will be consolidated under an elementary thinning move. (See \cite[Lemma 5.8]{TT-Thin}.)
\end{definition}

\begin{figure}[ht!]
\labellist
\small\hair 2pt
\pinlabel{$D_-$} [l] at 105 91
\pinlabel{$D_+$} [br] at 136 171
\pinlabel{$H$} [r] at 10 132
\endlabellist
\centering
\includegraphics[scale=0.6]{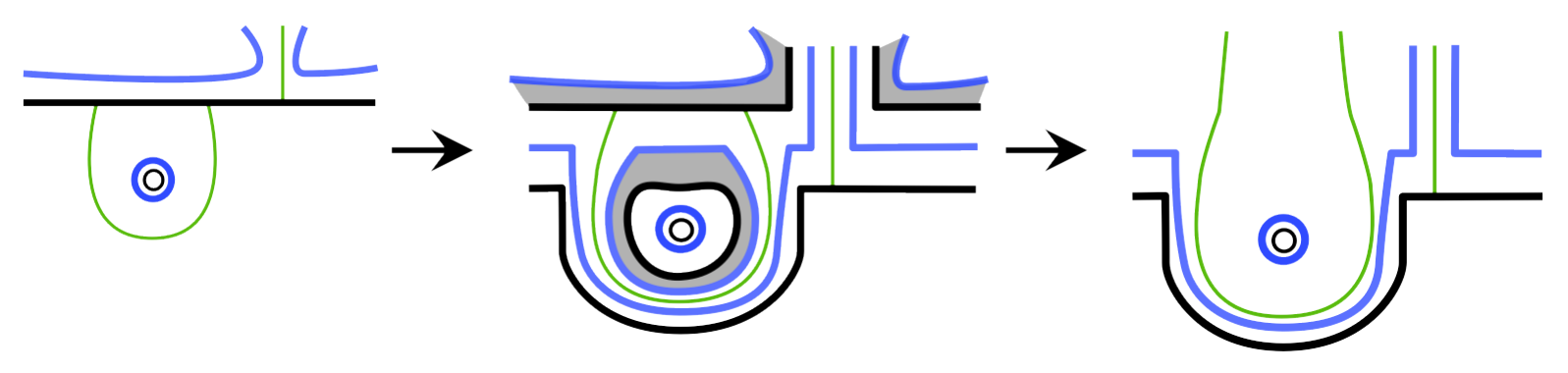}
\caption{An example of an elementary thinning move (Definition \ref{def: untelescope}). In the example the discs $D_+$ and $D_-$ are in green and are above and below the thick surface $H$. $D_+$ is non-separating (consider the left and right sides of $H$ to be identified) and $D_-$ is separating. There are two thin surfaces depicted in blue, one above $H$ and one below. One other thick surface is also shown in black. The graph $T$ is not shown. After untelescoping, there is a product VPC between a new thin surface and a new thick surface and also a product VPC between an old thin surface and a new thick surface. These are depicted in grey. Note that $\boundary D_\mp$ can be extended to lie on $H_\pm$. After consolidating, these boundaries can be further extended, as necessary to lie on other thick surfaces.}
\label{fig:untel}
\end{figure}

In \cite{TT-Thin} we showed that untelescoping and consolidating an oriented acyclic multiple bridge surface results in another oriented acyclic multiple bridge surface. 

\begin{definition}
We say that $\mc{H}$ \defn{uc-thins to} (or just \defn{thins to}) $\mc{J}$ and write $\mc{H} \rightsquigarrow \mc{J}$ if $\mc{J}$ can be obtained from $\mc{H}$ by a (possibly empty) sequence of elementary thinning moves and consolidations. $\mc{H}$ is \defn{locally thin} if $\mc{H} \rightsquigarrow \mc{J}$ implies $\mc{J} = \mc{H}$.
\end{definition}

\begin{remark}\label{any order}
The relation $\rightsquigarrow$ is a partial order on $\vpoH(M,T)$. In \cite{TT-Thin}, we defined a similar partial order $\to$ using not only consolidation (of thick and thin surfaces bounding product, rather than punctured product, compressionbodies) and untelescoping but also a variety of other moves, each somewhat analogous to the process of destabilization in the theory of Heegaard splittings. We also allowed untelescoping using semi-cut and semi-compressing discs. We proved that each of these moves decreases a certain complexity taking values in a well-ordered set and so any sequence of such moves must eventually terminate. In this paper, we will avoid using those additional moves and by allowing ourselves to consolidate punctured product compressionbodies, we can use a much simpler complexity and avoid having to use semi-cut and semi-compressing discs when untelescoping.
\end{remark}

\begin{theorem}\label{lem:Thinning invariance}
 Suppose that $(M,T)$ is standard and 1-irreducible and that $T$ is weighted. Then the following hold:
 \begin{enumerate}
    \item There is no infinite sequence of elementary thinning moves and consolidations.
     \item For every $\mc{H} \in \vpoH(M,T)$, there exists a locally thin $\mc{J} \in \vpoH(M,T)$ such that $\mc{H} \rightsquigarrow \mc{J}$.
      \item $\netchi$, $\netw$, and $\netx_m$ are each invariant under $\rightsquigarrow$.
\end{enumerate}
\end{theorem}
\begin{proof}
Conclusions (1) and (2) are similar to Sections 6.2.2 and 6.2.3 and Remark 6.15 of \cite{TT-Thin} and to \cite{Tomova}. However, we can handle things more simply here. Let $c(\mc{H})$ be the finite sequence of non-negative integers whose entries are 
\[
2\x_2(H) + 4 = -2\chi(H) + |H \cap T| + 2 
\]
for each thick surface $H \cpt \mc{H}^+$, arranged in non-increasing order. We compare these complexities lexicographically: $c(\mc{H}) < c(\mc{J})$ if and only if there is some $k$ such that either the first $k$ entries of $c(\mc{H})$ and $c(\mc{J})$ coincide and the $(k+1)$st entry of $c(\mc{H})$ is strictly smaller than the $(k+1)$st entry of $\mc{J}$ or if $c(\mc{H})$ has exactly $k$ entries and they coincide with the first $k$ entries of $c(\mc{J})$ and $c(\mc{J})$ has at least $(k+1)$ entries. Note that these complexities are totally ordered. As with the complexity in \cite{Tomova}, it is straightforward to show that untelescoping using a pair of c-discs strictly decreases complexity. Consolidation removes a thick surface (and a thin surface) so it also strictly decreases complexity. As complexities are well-ordered, Conclusions (1) and (2) follow immediately.
    
For the proof of Conclusion (3), we show that neither consolidation nor untelescoping affect the invariants. Consolidation removes a single thick surface and a single thin surface. If $(C, T_C)$ is a punctured product VPC with $\boundary C \subset \mc{H}$, then $\x$, $\x_m$, $-\chi$, and $w$ are the same for both $\boundary_+ C$ and the component of $\boundary_- C$ that we remove. Hence, consolidation does not change $\netchi$, $\netg$, $\netw$, $\netx_m$, or $\netx$. 

We consider untelescoping. Adopt the notation of Definition \ref{def: untelescope}. Let $p_\pm = w(D_\pm)$. We record the calculation for $\netx_m$ only; the others are similar, although the proof for $\netg$ requires considering the possibilities for how $\boundary D_+$ and $\boundary D_-$ separate $H$. Observe that (up to isotopy), $H_\pm$ is obtained by compressing $H$ by $D_\pm$ and $F$ is obtained by compressing $H$ by $D_+ \cup D_-$.

Thus,
\[
\begin{array}{rcl}
x_m(H_+)&=& x_m(H) - m + p_+\\
x_m(H_-) &=& x_m(H) - m + p_- \\
x_m(F) &=& x_m(H) - 2m + p_- + p_+\\
\end{array}
\]
Thus,
\[\begin{array}{rc}
x_m(H_+) + x_m(H_-) - x_m(F) = x_m(H).
\end{array}
\]
\end{proof}


\begin{theorem}\label{thm:essential twicepunct sphere}
Suppose that $(M,T)$ is standard and 1--irreducible. If $\mc{H} \in\vpoH(M,T)$ is locally thin, then the following hold:
     \begin{enumerate}
         \item Each component of $\mc{H}^+$ is c-strongly irreducible in $M\setminus \mc{H}^-$
         \item Each component of $\mc{H}^-$ is c-incompressible
         \item If $F\cpt \mc{H}^-$ is $\boundary$-parallel in $M \setminus T$, then either $F$ bounds a trivial ball compressionbody in $(M,T)$ or there exists $S \cpt \boundary \punct{M}$ such that $F$ and $S$ cobound a trivial product compressionbody in $(\punct{M}, \punct{T})$.
        \item $\mc{H}^-$ contains an efficient system of summing spheres for $(M,T)$.
         \item Suppose $(M,T)$ is irreducible and $Q \subset (M,T)$ is a c-essential connected surface. If $|Q \cap T| \neq 0$, assume that $(M,T)$ is 2-irreducible. Then $\mc{H}$ can be isotoped (transversally to $T$) such that $\mc{H}$ is adapted to $Q$.
     \end{enumerate}
\end{theorem}
\begin{proof}
Conclusion (1) follows from the definition of ``locally thin''. For Conclusion (2), assume that $F \cpt \mc{H}^-$ is c-compressible via a c-disc $D$. Choose $F$ and $D$ so as to minimize $|D \cap \mc{H}^-|$. We claim that the interior of $D$ is disjoint from $\mc{H}^-$. If not, choose an circle $\zeta \subset D \cap \mc{H}^-$ that is innermost in $D$. If $\zeta$ is essential in $G \cpt \mc{H}^-$, we may replace $F$ with $G$, contradicting our initial choices. If $\zeta$ is inessential in $\mc{H}^-$, we may surger $D$ along the unpunctured or once-punctured disc in $\mc{H}^-$ bounded by $\zeta$ and again contradict our initial choices. Applying \cite[Corollary 7.5]{TT-Thin} to the component of $M\setminus \mc{H}^-$ containing $D$ contradicts either Conclusion (1) or the fact that $\mc{H}$ cannot be consolidated. Thus $\mc{H}^-$ is c-incompressible. Hence, Conclusion (2) holds.

Suppose that $F \cpt \mc{H}^-$ is $\boundary$-parallel in $M\setminus T$. For simplicity, we may convert all vertices of $T$ into spherical boundary-components. Then there exists the union of components $S \subset \boundary \punct{M}$ such that $F \cup S$ bounds a submanifold $W$ of $\punct{M}$ and $W \setminus \punct{T}$ is a product. In particular $S \cup (\punct{T} \cap W)$ is connected. If $S$ is disconnected, there must be an edge $e$ of $\punct{T} \cap W$ with one end on $S_0 \cpt S$ and the other end on $S\setminus S_0$. The product structure on $W\setminus \punct{T}$ implies that there is an annulus in $\punct{M} \setminus \punct{T}$ with one end on $F$ and the other on a meridian of $e$. This annulus caps off to a once-punctured disc $D \subset W$ intersecting $e$. Compressing $F$ along $D$ creates two components, each $\boundary$-parallel in $W|_D$. Since no component of $\boundary \punct{M}$ is a twice-punctured sphere, $\boundary D$ is essential in $F$. This contradicts the c-incompressibility of $\mc{H}^-$. Thus, $S$ is connected. If $S \neq \nil$, a similar argument shows that every edge of $\punct{T} \cap W$ must have one endpoint on $S$ and one on $F$; in particular $\punct{T} \cap W$ consists of arcs vertical in the product structure. In which case, $(W, \punct{T} \cap W)$ is a product VPC, as claimed. If $S = \nil$, then $W$ is a 3-ball and $T \cap W$ is a single unknotted arc. This proves Conclusion (3).

Conclusion (4) is essentially the content of \cite[Section 8]{TT-Thin} and is closely related to the content of \cite[Corollary 4.8]{Taylor-Equivariant}. Although, the first of those references does not consider weights on the edges of $T$ and the second is phrased for orbifolds, where the edge weights are handled differently, the argument is essentially the same. We will sketch the proof in our context momentarily, but first consider Conclusion (5).

Assume $(M,T)$ is irreducible and that $Q \subset (M,T)$ is c-essential. Isotope $Q$ to intersect $\mc{H}^-$ minimally. By Conclusion (3), we may assume that $Q$ is disjoint from all components of $\mc{H}^-$ that are $\boundary$-parallel in $M\setminus T$. Suppose that $\zeta \cpt Q \cap \mc{H}^-$ is inessential in $\mc{H}^-$. Without loss of generality, we may assume it is an innermost inessential loop in $\mc{H}^-$, bounding an unpunctured or once-punctured disc $D \subset \mc{H}^-$. Since $Q$ is c-essential, $\zeta$ is also inessential in $Q$. Since $(M,T)$ does not contain a once-punctured sphere, the surface $Q|_D$ is the union of a surface $Q'$ and a zero or twice-punctured sphere $P$. At least one of the two is essential. It is easy to show that $Q'$ remains c-incompressible, although it may be $\boundary$-parallel. If $P$ is inessential it either bounds a 3-ball disjoint from $T$ or a 3-ball intersecting $T$ in a trivial arc. Therefore, if $P$ is inessential, $Q$ and $Q'$ are isotopic and we have contradicted our choice of $Q$. Such must be the case if $D$ is unpunctured, as $(M,T)$ is irreducible. If $(M,T)$ does not contain any essential twice punctured sphere, then again $P$ is inessential and $Q$ and $Q'$ are isotopic. In the setting of Conclusion (5), we have shown that all loops of $Q \cap \mc{H}^-$ are essential in $\mc{H}^-$. In the setting of Conclusion (4), take $Q$ to be an essential sphere with at most three punctures. We note that $Q'$ and $P$ each are spheres with at most 3 punctures. (Indeed, the number of punctures on each does not exceed the number of punctures on $Q$.) Since $(M,T)$ is standard, at least one of these is essential. 

Returning to the proof of Conclusion (5), suppose that $\zeta \cpt Q \cap \mc{H}^-$ is inessential in $Q$. Without loss of generality, we may assume that $\zeta$ is an innermost inessential loop in $Q$. Since $\mc{H}^-$ is c-incompressible, $\zeta$ must also be inessential in $\mc{H}^-$. However, we have already ruled out such loops. Thus, under the hypotheses of (4) we have produced an essential sphere with at most three punctures disjoint from $\mc{H}^-$. We continue to call the sphere $Q$. In the setting of Conclusion (5), $Q \cap \mc{H}^-$ consists of loops essential in both $Q$ and $\mc{H}^-$. A further isotopy of $Q$ or $\mc{H}^-$ eliminates any components of $Q \setminus \mc{H}^-$ with boundary and which are isotopic to a subsurface of $\mc{H}^-$. 

We now consider the intersections of each $Q_0 \cpt Q\setminus \mc{H}^-$ with $\mc{H}^+$. Applying \cite[Theorem 7.2]{TT-Thin} to $Q_0$, we see that, after an isotopy, either $Q_0 \cap \mc{H}^+ = \nil$ or $Q_0 \cap \mc{H}^+$ is nonempty and consists of loops that are essential in both surfaces. The latter case does not occur if $Q = Q_0$ is a sphere with at most 3 punctures as $Q$ does not contain any essential curves. In the context of Conclusion (5), this is what we seek.

Suppose, therefore, that $Q_0 \cap \mc{H}^+ = \nil$. Let $(C, T_C) \cpt (M,T)\setminus \mc{H}$ be the VPC containing $Q_0$. Let $\Delta$ be a complete set of sc-discs for $\boundary_+ C$, chosen to intersect $Q$ minimally. An argument nearly identical to the one concerned with $Q \cap \mc{H}^-$ shows that either $\boundary_- C$ contains a zero or twice-punctured sphere or $Q_0$ is isotopic to a component of $\boundary_- C$. In the setting of Conclusion (5), we can continue to surger $Q_0$ to make it disjoint from $\Delta$, eventually concluding that the resulting surface is isotopic to a component of $\boundary_- C$. This concludes the proof of Conclusion (5).

For the remainder, suppose $Q$ to be an essential sphere with three or fewer punctures. Assume $Q$ is such a sphere, minimizing the number of possible punctures. The argument above produces an essential sphere component of $\mc{H}^-$, with at most the same number of punctures as $Q$. We may as well assume that $Q \cpt \mc{H}^-$. Let $\mc{H}_1$ and $\mc{H}_2$ be the restrictions of $\mc{H}$ to each component of $(M,T)|_Q$. It is easy to check that they are locally thin. Thus, if $(M,T)|_Q$ contains an essential sphere with three or fewer punctures, there is one in $\mc{H}^-_1$ or $\mc{H}^-_2$. Continuing in this vein, we deduce that $\mc{H}^-$ contains a collection of essential spheres $S$, each with three or fewer punctures, such that $(M,T)|_S$ does not contain any such sphere. A subset of $S$ is then an efficient summing system for $(M,T)$.
\end{proof}

\begin{corollary}\label{adapted}
Suppose that $(M,T)$ is standard, irreducible, that $Q \subset (M,T)$ is c-essential and that  $\mc{H} \in \vpoH(M,T)$. Then the following hold:
\begin{enumerate}
\item If $Q$ is unpunctured and $\mc{H}$ is locally thin, then $\mc{H}$ can be isotoped to be adapted to $Q$.
\item If $Q$ is punctured, $(M,T)$ is 2-irreducible and $\mc{H}$ is locally thin, then $\mc{H}$ can be isotoped to be adapted to $Q$
\item In both of the previous conclusions, we can replace the hypothesis that $\mc{H}$ is locally thin with the hypothesis that every curve of $\mc{H} \cap Q$ is essential in both $\mc{H}$ and $Q$.
\end{enumerate}
\end{corollary}
\begin{proof}
If $\mc{H}$ is locally thin, by Theorem \ref{thm:essential twicepunct sphere}, we may assume that all curves of intersection between $\mc{H}$ and $Q$ are essential in both surfaces. Suppose that $Q_0 \cpt Q \setminus \mc{H}$ is a component contained in $(C, T_C) \cpt (M,T) \setminus \mc{H}$. If $Q_0$ is properly isotopic to a subsurface of $\mc{H}$ via a proper isotopy keeping $Q_0$ transverse to $T$, then we may isotope that thin or thick surface across $Q_0$ reducing $|Q \cap \mc{H}|$ and preserving the property that all curves of intersection are essential in both surfaces. Applying as many of these isotopies as needed shows that $\mc{H}$ can be isotoped to be adapted to $Q$. 
\end{proof}

We note the following fact. We do not use it until Section \ref{sec:Qcomplex}; it is used to ensure that a certain complexity does not increase under certain kinds of thinning. See Figures \ref{fig:untel} and \ref{fig: pos error}.

\begin{lemma}\label{discs persist}
Suppose that an elementary thinning move is applied to $\mc{H}$ using discs $D_\pm$. Let $\mc{J}$ be the resulting multiple bridge surface. Then each of $D_\pm$ persists as an sc-disc for a thick surface of $\mc{J}$. Furthermore, suppose that $\mc{K}$ is obtained by applying a sequence of elementary thinning moves and consolidations to $\mc{J}$, such that each untelescoping in the sequence uses $D_-$ or $D_+$ and no consolidations of $\mc{K}$ are possible. Let $\alpha$ be an oriented path in $M$ transverse to $\mc{H}$ such that if it intersects a component of $\mc{H}$ it does so in a single point and with positive sign of intersection. If $\alpha$ is disjoint from all the weak reducing pairs used in the creation of $\mc{K}$, then $\alpha$ intersects at most one more thick surface of $\mc{K}$ than $\mc{H}$.
\end{lemma}
\begin{proof}
Let $H \cpt \mc{H}^+$ be the thick surface containing $\boundary D_- \cup \boundary D_+$. Let $H_\pm = H|_{D_\pm}$ be the new thick surfaces resulting from the untelescoping. Let $F = H|_{D_- \cup D_+}$ be the new thin surfaces. As in the first step of Figure \ref{fig:untel}, it follows from the construction that $\boundary D_\pm$ can be extended to lie on $H_\mp$. Even if $\boundary D_\pm$ is inessential in $H_\mp$, the disc $D_\pm$ is an sc-disc for $H_\mp$. Consequently, the components of $H_\mp$ containing $\boundary D_\pm$ cannot be consolidated with components of $F$. As noted in Lemmas 5.7 and 5.8 of \cite{TT-Thin}, a component of $H_\pm$ can be consolidated with a component of $F$ if and only if $\boundary D_\pm$ separates $H$. If the component $H'$ of $H_\pm$ containing $\boundary D_\pm$ can be consolidated with a component $S \cpt \mc{H}^-$, then we can extend $\boundary D_\mp$ through the product VPC bounded by $H' \cup S$ to lie on $S$. Since $S \cpt \mc{H}^-$, the structure of the adjacent VPC, allows us to further extend $\boundary D_\mp$ to lie on $\mc{H}^+\setminus H$. 

Let $\alpha$ be as in the statement of the lemma. The proof is by induction on the number of untelescopings used to generate $\mc{K}$ from $\mc{H}$. Suppose, first, that a single untelescoping is used. In which case, $\mc{K} = \mc{J}$. As $\alpha$ is disjoint from $D_- \cup D_+$, it intersects at most one component of each of $H_- \cup H_+$. Note that if $\alpha$ intersects $H$ then one of the following occurs:
\begin{enumerate}
    \item the components of $H_- \cup H_+$ intersecting $\alpha$ persist into $\mc{J}$ and $|\alpha \cap \mc{J}^+| = |\alpha \cap \mc{H}^+| + 1$
    \item one of the components of $H_- \cup H_+$ intersecting $\alpha$ is consolidated and the other persists into $\mc{J}$, in which case $|\alpha \cap \mc{J}^+| = |\alpha \cap \mc{H}^+|.$
    \item both components of $H_- \cup H_+$ intersecting $\alpha$ are consolidated and $|\alpha \cap \mc{J}^+| = |\alpha \cap \mc{H}^+| - 1.$
\end{enumerate} 

Suppose there are exactly two untelescopings. $\mc{K}$ is obtained from $\mc{J}$ by untelescoping a thick surface which is above or below a component of $F$ using a weak reducing pair containing $D_-$ or $D_+$ respectively. Without loss of generality, suppose it is a surface above a component of $F$. If it is a component of $H_+$, then as $F$ is isotopic in $(M,T)$ to components of $(H_+)|_{D_-}$, those components of $(H_+)|_{D_-}$ are consolidated. It follows that $|\alpha \cap \mc{K}^+| = |\alpha \cap \mc{J}| = |\alpha \cap \mc{J}^+|$. If, on the other hand, in the creation of $\mc{J}$, $H_+$ is consolidated, then we are in Case (2) or (3) above. The untelescoping of the component of $\mc{J}^+$ (which is also a component of $\mc{H}^+$) containing $\boundary D_-$ may create two additional thick surfaces intersecting $\alpha$. Consequently, $|\alpha \cap \mc{K}^+| \leq |\alpha \cap \mc{J}^+| + 1 \leq |\alpha \cap \mc{H}^+| + 1$. Call the $+1$ that arises in this argument, the \defn{positive error term}.

When there are more untelescopings, the argument proceeds in the same way: if $\boundary D_-$ remains on one of the surfaces resulting from an untelescoping, that surface, after compressing along $D_-$ will consolidate with the new thin surface. Whereas, if $\boundary D_-$ is moved to a new component of $\mc{H}^+$, we must have consolidated one of the thick surfaces created by the untelescoping with a component of $\mc{H}^-$. Consequently, the positive error terms do not accumulate. See Figure \ref{fig: pos error}.
\end{proof}

\begin{figure}[ht!]
\labellist
\small\hair 2pt
\pinlabel{$H$} [r] at 45 453
\pinlabel{$\alpha$} [b] at 141 584
\pinlabel{$D_+$} [bl] at 216 530
\pinlabel{$D_-$} [br] at 123 369
\pinlabel{$H_+$} [r] at 351 464
\endlabellist
\centering
\includegraphics[scale=0.4]{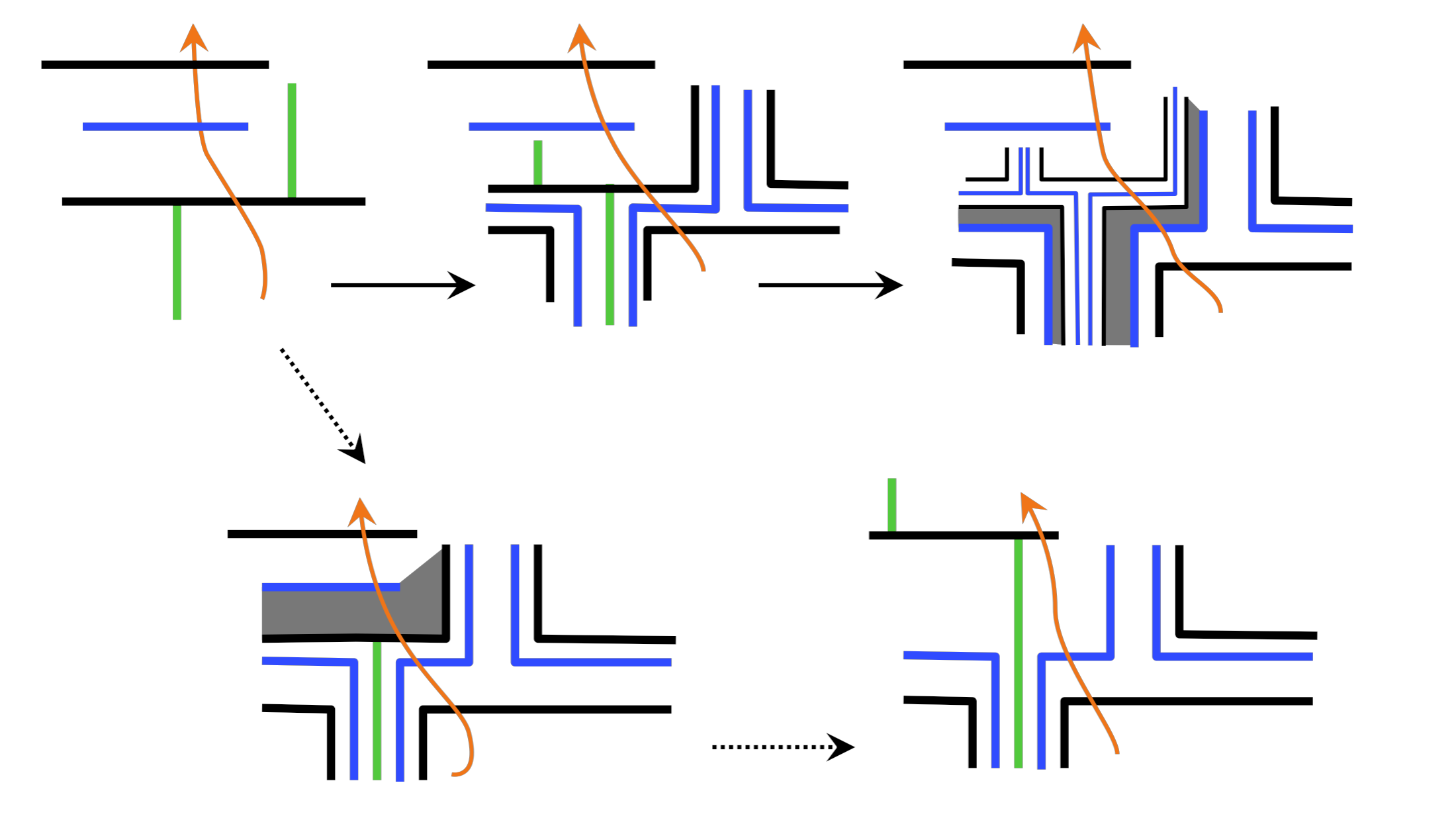}
\caption{In the top left we show a multiple bridge surface $\mc{H}$ with two thick surfaces and one thin surface. We also show in orange a path $\alpha$ consistent with the orientations of $\mc{H}$ (not shown) and a weak reducing pair of discs (in green) with boundary on a thick surface $H$. Following the solid arrows we untelescope $H$ and then find a weak reducing pair for $H_+$; one of the pair is $D_-$. Untelescoping $H_+$ we discover that one of the new thick surfaces intersecting $\alpha$ is parallel to a thin surface intersecting $\alpha$. After consolidating, we see that $|\alpha \cap \mc{K}^+| = |\alpha \cap \mc{H}^+| + 1$. Following the dashed black arrows, we see the case when the component of $H_+$ intersecting $\alpha$ is parallel to a component of $\mc{H}^-$. Consolidating those surfaces and then untelescoping another thick surface using a weak reducing pair containing $D_-$ we see that $|\alpha \cap \mc{K}^+| = |\alpha \cap \mc{H}^+| + 1$. There are of course other choices for $\alpha$, but for each if we keep track of the consolidations, we see that Lemma \ref{discs persist} holds.}
\label{fig: pos error}
\end{figure}

\subsection{Amalgamation}\label{sec:amalg}

As we can see from the definition of untelescoping, thinning generally leads to the creation of a multiple bridge surface that has more components than the original multiple bridge surface did. Amalgamation (originally defined in \cite{Schultens}) is the inverse operation where we begin with a multiple bridge surface with many components and obtain a new multiple bridge surface that has fewer (but generally more complex) components. The presence of a knot, link, or spatial graph $T$ makes this operation more subtle than in the case of a 3-manifold without a graph and in fact the operation is not always allowed. The next lemma outlines the condition that needs to be met for amalgamation to result in a new multiple bridge surface.

\begin{lemma}[{\cite[Lemma 6.3]{Taylor-Equivariant}}]\label{lem:amalgamate}
Suppose that $\mc{H} \in \vpoH(M,T)$ and suppose $w$ is a weight system for $T$. Let $(C_i, T_i) \cpt (M,T) \setminus \mc{H}$ for $i = 1,2$ be VPCs  such that $F = \boundary_- C_1 \cap \boundary_- C_2 \neq \nil$. Let $H_i = \boundary_+ C_i$. Suppose there does not exist a ghost arc in $C_1$ sharing an endpoint with a ghost arc in $C_2$. Then there exists a surface $H \subset C_1 \cup C_2$ such that $\mc{J} = H \cup (\mc{H} \setminus (H_1 \cup H_2 \cup F))$ is an oriented multiple bridge surface for $(M,T)$ with $H \cpt \mc{J}^+$. Furthermore, $\netx$, $\netx_m$, $\netchi$, $\netg$, and $\netw$ are the same for $\mc{H}$ as for $\mc{J}$.
\end{lemma}
\begin{proof}
In \cite{Taylor-Equivariant} we show that $H$ and $\mc{J}$ exist. The equalities concerning the invariants follow from the construction or from the fact that $\mc{J}$ thins back to $\mc{H}$.
\end{proof}

\begin{definition}
Using the notation of the previous lemma we say that $H$ is obtained by \defn{amalgamating} $H_1$ and $H_2$ across $F$ and that $\mc{J}$ is obtained by an amalgamation of $\mc{H}$.
\end{definition}


The proof of the following is an easy exercise using Lemma \ref{lem:amalgamate}. The version for $g = 0$ can be found as \cite[Theorem 2.11]{TH-widthtree}. The proof for $g = 1$ is nearly identical. There is no analogous result for $g > 1$ or if $L$ is a spatial graph that is not a link.

\begin{lemma}\label{sphere eq}
Suppose that $(M,L)$ is a pair with $M$ either $S^3$ or a lens space and $L$ a weighted link. If $\mc{H} \in \vpoH(M,T)$ has $g = \netg(\mc{H}) \in \{0,1\}$, then $\mc{H}$ amalgamates to a genus $g$ bridge surface $H$ for $(M,L)$ with $\netw(\mc{H}) = w(H)$. Consequently, $\netw_g(\weighted{L})/2 = \b_g(\weighted{L})$.
\end{lemma}

\section{Simple Cases}\label{Trivial Cases}

As a prelude to the more involved work of most of the paper, in this section we dispense with several simple cases. At the conclusion of the section we give a new proof of Schubert's and Doll's additivity theorems, applicable in a much more general context. From there we prove the \href{Main Theorem}{Main Theorem} when $\omega = 1$; that proof contains one of the key ideas used in the general proof.

We start with two results that allow us to reduce to considering only irreducible pairs. Lemma \ref{ensure irreducible} is essentially due to Hayashi-Shimokawa \cite{HS}.

\begin{lemma}[Distant Sums]\label{ensure irreducible}
Suppose that for $i = 0,1$, the pair $(M_i, T_i)$ is standard with $T_i$ weighted and suppose $(M, T) = (M_0, T_0) \#_0 (M_1, T_1)$. Then for any $g \geq \g(M)$, there exist $g_0 \geq \g(M_0)$ and $g_1 \geq \g(M_1)$ such that $g_0 + g_1 = g$ and 
        \[ \b_g(M,\weighted{T}) = \b_{g_0}(M_0, \weighted{T_0}) + \b_{g_1}(M_1, \weighted{T_1})\],
\end{lemma}
\begin{proof}
Suppose $g_i \geq \g(M_i)$ for $i = 0,1$ and $g_0 + g_1 = g$. Also for $i = 0,1$, let $H_i$ be a genus $g_i \geq \g(M_i)$ bridge surface for $(M_i, T_i)$ such that $\b_{g_i}(\weighted{T}_i) = w(H_i)/2$.  By performing the connected sum of $(M_0, T_0)$ and $(M_1, T_1)$ along points $p_0 \in H_0 \setminus T_0$ and $p_1 \in H_1 \setminus T_1$, we can also form a Heegaard surface $H = H_0 \# H_1$ for $(M,T)$ with $g = g(H) = g_0 + g_1$ and $w(H) = w(H_0) + w(H_1)$. This Heegaard surface is also a bridge surface for $(M,T)$, so $\b_g(\weighted{T}) \leq \b_{g_0}(\weighted{T_0}) + \b_{g_0}(\weighted{T_1})$.

For the other inequality, let $H$ be a genus $g$ bridge surface for $(M,T)$ such that $\b_g(\weighted{T}) = w(H)/2$. By \cite[Theorem 1.3]{HS}, there exists a collection $S \subset (M,T)$ of essential unpunctured spheres, each intersecting $H$ in a single simple closed curve such that $(M,T)|_S$ is irreducible. The collection $S$ lies in an efficient summing system such that every other sphere in the system is twice or thrice-punctured. So up to homeomorphism, the collection $S$ and the factors $(M,T)|_S$ are unique by Lemma \ref{thm: efficient}. In $(M,T)|_S$, we may cap $H\setminus S$ off with discs to obtain bridge surfaces for each component of $(M,T)|_S$. As in the previous paragraph, for each $i = 0,1$, we may reassemble the bridge surfaces for the factors of $(M_i, T_i)$ to construct a bridge surface $H_i$ for $(M_i, T_i)$ of genus $g_i$ such that $g_0 + g_1 = g$.Consequently, $\b_{g_0}(\weighted{T_0}) + \b_{g_1}(\weighted{T_1}) \leq \b_g(\weighted{T})$. 
\end{proof}

The next lemma will allow us to assume 1-irreducibility throughout, since, as every sphere in $M$ is separating, if $(M,T)$ is not 1-irreducible, then it is the result of a cut edge sum.

\begin{lemma}[Cut edge sums]\label{1-irred}
Suppose that $(M,T) = (M_0, T_0) \#_e (M_1, T_1)$ is a standard pair.  Then for any $g \geq \g(M)$,  there exist $g_0 \geq \g(M_0)$ and $g_1 \geq \g(M_1)$ such that $g_0 + g_1 = g$ and 
        \[ \b_g(M,\weighted{T}) = \b_{g_0}(M_0, \weighted{T_0}) + \b_{g_1}(M_1, \weighted{T_1}).\]
\end{lemma}

\begin{proof}
Any edge of $T$ intersecting a sphere in $M$ exactly once transversally is a cut edge for $T$.  Let $E = \{e_1, \hdots, e_n\}$ be the set of cut edges intersecting once-punctured spheres $S \subset (M,T)$ exactly once each. For each $i$, let $p_\pm(i)$ be the endpoints of $e_i$. Note that $(M,T\setminus E)$ is 1-irreducible and that it can be formed by iterated distant sums with the spheres $S$ as summing spheres. For $j = 0, \hdots, n$ let $(M'_j, T'_j)$ be the factors corresponding to summing along $S$. 

Let $H'_j$ be a bridge surface for $(M'_j, T'_j)$ with $g(H'_j) = g_j$, $\sum\limits_j g_j = g$, and $\b_{g_j}(\weighted{T'_j}) = w(H'_j)/2$. As in Lemma \ref{ensure irreducible}, we may sum the $H'_j$ to get a bridge surface $H$ for $(M,T\setminus E)$ with $g(H) = g$ and $w(H) = \sum w(\weighted{T'_j})$. Lemma \ref{ensure irreducible} tells us that $w(H)/2 = \b_g(\weighted{T}\setminus E)$. 

We take a closer look at the construction. The points $p_\pm(i)$ lie on the graphs $T'_j$, with $p_-(i)$ and $p_+(i)$ on different graphs. A small isotopy ensures that no $p_\pm(i)$ lies in any $H'_j$. Thus, each $p_\pm(i)$ is either above or below the surface $H'_j$ with $T'_j$ containing $p_\pm(i)$. Since all the summing spheres are separating, the dual graph to $S$ in $M$ is a tree. Choosing a root for that tree and working outward, we may reverse orientations on the $H'_j$ as needed so that either all $p_\pm(i)$ lie above the $H'_j$ with $p_\pm(i) \in T'_j$ or all lie below. Without loss of generality, assume that all lie above. Thus, in $T\setminus E$, all $p_\pm(i)$ lie above the surface $H$. We may then insert edges $e'_i$ joining $p_-(i)$ to $p_+(i)$ such that each $e'_i$ is disjoint from $H$ and intersects $S$ exactly once and so that $H \in \vpoH(M, T')$ where $T' = (T\setminus E) \cup \bigcup_i e'_i$. By Lemma \ref{cut uniqueness}, $T'$ and $T$ are isotopic. Thus, $H \in \vpoH(M,T)$ and $\b_g(\weighted{T}) \leq \b_g(\weighted{T}\setminus E)$. For the other inequality, note that if $H$ is a bridge surface for $(M,\weighted{T})$ then it is also a bridge surface for $(M, \weighted{T}\setminus E)$. Thus, $\b_g(\weighted{T}) = \b_g(\weighted{T}\setminus E)$. The stated equality follows from induction on the number of cut edges.
\end{proof}

We now prove Schubert's and Doll's Additivity Theorems\footnote{We note that in \cite[Section 6]{Doll}, Doll claims to give a counter-example to the theorem when $g = 1$ and $K_2$ is either link in $S^3$ of two components or a knot in a lens space. However, both these counter-examples rely on considering the core loop $\lambda$ of a genus 1 Heegaard splitting to have $\b_1(\lambda) = 1$. Taking $\b_1(\lambda) = 0$, as we do, breaks the counterexample.} \cite{Doll}. This is similar to the work in \cite{TT-Additive}. We include it both for completeness and because it encapsulates the central philosophy of our arguments. We take the opportunity to generalize the classical theorems. We could generalize further to the situation when $T_0$ and $T_1$ are not links; however, the statements become less clean as one has to take into account ghost arcs and it does not seem possible to give a clean statement that specifies in all cases exactly how $\b_0(\weighted{T})$ relates to $\b_{0}(\weighted{T_1})$ and $\b_{0}(\weighted{T_2})$ (for example).

\begin{additivitytheorem}\label{thm: add}
Suppose that $(M_0, T_0)$ and $(M_1, T_1)$ are weighted standard pairs with $M_0 = S^3$ and $M_1$ either $S^3$ or a lens space. Assume that there is no twice-punctured sphere in either $(M_i,T_i)$ with punctures of different weights. Suppose $(M,T) = (M_0, T_0) \#_2 (M_1, T_1)$, with the sum respecting the weights. Let $u$ be the weight of the edges where the sum is performed. If both $T_0$ and $T_1$ are links, then
\begin{enumerate}
    \item If $M_2 = S^3$, $\b_0(\weighted{T}) = \b_0(\weighted{T}_0) + \b_0(\weighted{T}_1) - u$, and
    \item There exist $g_0, g_1 \in \{0,1\}$ with $g_0 + g_1 = 1$, and so that $\b_1(\weighted{T}) = \b_{g_0}(\weighted{T_0}) + \b_{g_1}(\weighted{T_1}) - u$. Furthermore, if $M_1 \neq S^3$, then $g_1 = 1$.
\end{enumerate}
\end{additivitytheorem}

\begin{proof}
Note that $(M,T)$ is irreducible if and only if both $(M_0, T_0)$ and $(M_1, T_1)$ are. From Theorems \ref{ensure irreducible} and \ref{1-irred}, we may assume that all three pairs are irreducible. Also note that no twice-punctured sphere in $(M,T)$ has punctures of distinct weights.

We start by showing ``$\leq$'' for both equations. For $i \in \{1,2\}$, let $H_i$ be a bridge surface for $(M_i, \weighted{T_i})$ such that $g(H_i) = g_i$ and $w(H_i)/2 = \b_{g_i}(\weighted{T_i})$. Assume that $g_0 + g_1 = g \in \{0,1\}$. Let $p_0 \in T_0$ and $p_1 \in T_1$ be the points where the sum is performed. Without loss of generality, we may assume that $p_i \not\in H_i$. Form $(M,T) = (M_1, T_1) \#_2 (M_2,T_2)$ by performing a connected sum at $p_1$ and $p_2$. Let $S \subset (M,T)$ be the summing sphere. Assign an orientation to $S$ consistent with that on $H_1$ and, if needed, reverse the orientation on $H_2$ so that $\mc{H} = H_1 \cup H_2 \cup S \in \vpoH(M,T)$ with $\mc{H}^+ = H_1 \cup H_2$ and $\mc{H}^- = S$. Observe that $\netg(\mc{H}) = g_1 + g_2 = g$ and $\netw(\mc{H}) = 2\b_{g_0}(\weighted{T_1}) + 2\b_{g_1}(\weighted{T_2}) - 2u.$ As $T$ is a link, by Lemma \ref{lem:amalgamate}, $\b_g(\weighted{T}) \leq \b_{g_0}(\weighted{T_1}) + \b_{g_1}(\weighted{T_2}) - u$.

For the other inequalities, suppose $\g(M) \leq g \in \{0,1\}$. By Theorem \ref{lem:Thinning invariance}, there exists locally thin $\mc{H} \in \vpoH(M,T)$ such that $\netg(\mc{H}) = g$ and $\netw(\mc{H})/2 = \b_g(\weighted{T})$. By Proposition \ref{thm:essential twicepunct sphere}, $\mc{H}$ contains an efficient summing system $S$ for $(M,T)$. Discard from $S$ all thrice-punctured spheres. By Lemma \ref{thm: efficient}, $(M,T)|_S = (M'_1, T'_1) \cup \cdots \cup (M'_n, T'_n)$ is the disjoint union of all the factors of $(M_0, T_0)$ and $(M_1, T_1)$. Each $(M'_i, T'_i)$ is irreducible and is without an essential twice-punctured sphere. Some of these factors maybe copies of  $(S^3, \text{unknot})$, and one may be $(M_1, \nil)$. By Theorem \ref{thm: efficient}, each of these factors is a factor of either $(M_0, T_0)$ or $(M_1, T_1)$. Let $h$ be the homeomorphism taking $S$ to the union of summing spheres for $(M_0, T_0)$ with summing spheres for $(M_1, T_1)$ with a summing sphere separating $(M_0, T_0)$ from $(M_1, T_1)$. Let $n_0$ and $n_1$ be the number of factors of $(M_0, T_0)$ and $(M_1, T_1)$ respectively. Assume that for $i \leq n_0$, each $(M'_i, K'_i)$ is a factor of $(M_0, T_0)$. 

Let $\mc{H}'_i$ be the restriction of $\mc{H}$ to $(M'_i, T'_i)$ and $g'_i = \netg(\mc{H}'_i)$ and $w'_i = \netw(\mc{H}'_i)$. We have $g = g'_1 + \cdots + g'_n$ and $\netw(\mc{H}) = -w(S) + w'_1 + \cdots + w'_n$. By Lemma \ref{sphere eq}, each $w'_i/2 \geq \b_{g'_i}(T'_i)$. For $j = 0,1$, let $S_j \subset S$ the spheres taken into $(M_j, T_j)$ by the homeomorphism $h$. Hence,
\[\begin{array}{rcl}
\b_g(K) &\geq& -w(S)/2 + \sum\limits_{i=1}^n \b_{g'_i}(\weighted{T'_i}) \\
&=& -u - w(S_0)/2 + \sum\limits_{i=1}^{n_0} \b_{g'_i}(\weighted{T'_i}) - w(S_1)/2 + \sum\limits_{i=n_0 + 1}^{n} \b_{g'_i}(\weighted{T'_i}) \\
&\geq&  -u + \b_{g_1}(\weighted{T_0}) + \b_{g_2}(\weighted{T_1}).
\end{array}
\]
The final inequality is obtained by recombining the factors of $(M_0, T_0)$ and $(M_1, T_1)$ and applying the results of the first paragraph of the proof. We define $g_1 = g'_1 + \cdots + g'_{n_1}$ and $g_2 = g'_{n_1 + 1} + \cdots + g'_{n}$.

Since we have proved both inequalities, the proof of the lemma is complete.
\end{proof}

\begin{corollary}[{\href{Main Theorem}{Main Theorem} for $\omega = 1$}]\label{omega1}
Suppose that $(M,T)$ is a standard pair with $M$ either $S^3$ or a lens space. Suppose that there exists an essential unpunctured torus $Q \subset (M,T)$ which is compressible into a side $V \cpt M\setminus Q$ and that $T \setminus V$ is a link. Suppose that the weighted wrapping number of $T \cap V$ in $V$ is $\omega = 1$ and that $\weighted{L}$ is a weighted companion with respect to $V$ (albeit, all weights are equal to 1). Then for $g \in \{0,1\}$, we have
\[
\b_{g}(T) \geq \b_g(\weighted{L}) - \delta.
\]
where $\delta = 0$ if $V$ is a solid torus and $\delta = 1$ if $V$ is a lensed solid torus. 
\end{corollary}
\begin{proof}
Use a cut disc $D$ for $Q$ in $V$ to compress $Q$ into $V$, obtaining an essential-twice-punctured sphere $Q'$. The weight of $D$ is equal to 1. The sphere $Q'$ realizes $(M, T) = (M_0, T_0) \#_2 (M_1, T_1)$ where $T_1$ contains $(T \cap V) \setminus \eta(D)$. The proof is now similar to that of the \href{thm: add}{Additivity Theorem}. By Theorems \ref{ensure irreducible} and \ref{1-irred}, we may assume that $(M,T)$ is irreducible. Note that $T_0$ coincides with $\weighted{L}$ outside of the arc used in the surgery creating $(M,T)|_{Q'}$. 

Let $\mc{H} \in \vpoH(M,T)$ be such that $\netg(\mc{H}) = g$ and $\netw(\mc{H}) = \netw_g(T)$. By definition, $\netw_g(T) \leq \b_g(T)$. By Proposition \ref{thm:essential twicepunct sphere}, $\mc{H}^-$ contains an efficient summing system for $(M,T)$. Discard from $S$ all thrice-punctured spheres. As in the proof of the \href{thm: add}{Additivity Theorem}, we see that there exist $g_0 \geq \g(M_0)$ and $g_1 \geq \g(M_1)$ such that $g_0 + g_1 = g$ and:
\[
\b_g(T) \geq \netw_{g}(T)/2 \geq \netw_{g_0}(T_0)/2 + \netw_{g_1}(T_1)/2 - 1.
\]

Since $T_0$ is a link, by Lemma \ref{sphere eq}, $\netw_{g_0}(T_0)/2 = \b_{g_0}(T_0)$.

By Corollary \ref{net weight bound}, $\netw_{g_1}(T_1) \geq 0$. If $g_1 = 0$, then $\netw_{g_1}(T_1) \geq 1$. However, as $(M_1, T_1)$ is 1--irreducible (or by the properties of bridge spheres), we must actually have $\netw_{0}(T_1) \geq 2$. In which case, $\b_g(T) \geq \b_{g_0}(T_0)$. Suppose that $g_1 = 1$. If $M_2$ is a lens space, then $T$ is a lensed satellite and we have the desired inequality. If $M_2 = S^3$, then $M = S^3$. We have
\[
\b_1(T) \geq \b_0(\weighted{L}) - 1.
\]
However, as we pointed out in Remark \ref{std result}, since $L$ cannot be isotoped to be disjoint from a bridge sphere for $S^3$ and since all the weights are equal to 1. $\b_0(\weighted{L}) \geq \b_1(\weighted{L}) + 1$. In which case, $\b_1(T) \geq \b_1(\weighted{L})$ as desired.
\end{proof}

\section{Crushing}\label{sec: crushing}
Throughout this section we make the following assumption.

\begin{assumption}\label{crushing assumps}
Assume that $(M,T)$ is standard and irreducible,  that $Q \subset (M,T)$ is a c-essential torus, and that $\mc{H} \in \H(Q)$.
\end{assumption}

The goal of this section is to describe a certain operation, called ``crushing'', which replaces the pair $(M,T)$ with a new pair $(\wihat{M}, \wihat{T})$ and replaces a multiple bridge surface $\mc{H}$ with a new multiple bridge surface $\wihat{\mc{H}}$ for $(\wihat{M}, \wihat{T})$. This replacement will not increase net weight but will result in the torus $Q$ becoming cut-compressible. After performing the compression, $Q$ becomes a twice-punctured sphere and we will be able to apply the additivity theorem to bound net weight. We note that even if $T$ is a knot, $\wihat{T}$ will be a weighted spatial graph. 

We start by looking at a very simple example of crushing. We will then explain some aspects which require a more general definition and then embark on the details.

\begin{figure}[ht!]
\labellist
\small\hair 2pt
\pinlabel{$A$} [b] at 161 246
\pinlabel{$\omega$} [b] at 584 194
\pinlabel{$D_1$} [t] at 140 83
\pinlabel{$D_2$} [tl] at 234 83
\pinlabel{$B$} at 226 192
\pinlabel{$D$} at 582 148
\pinlabel{$\mc{H}$} [bl] at 25 98
\pinlabel{$\mc{H}$} [bl] at 452 98
\endlabellist
\centering
\includegraphics[scale=0.6]{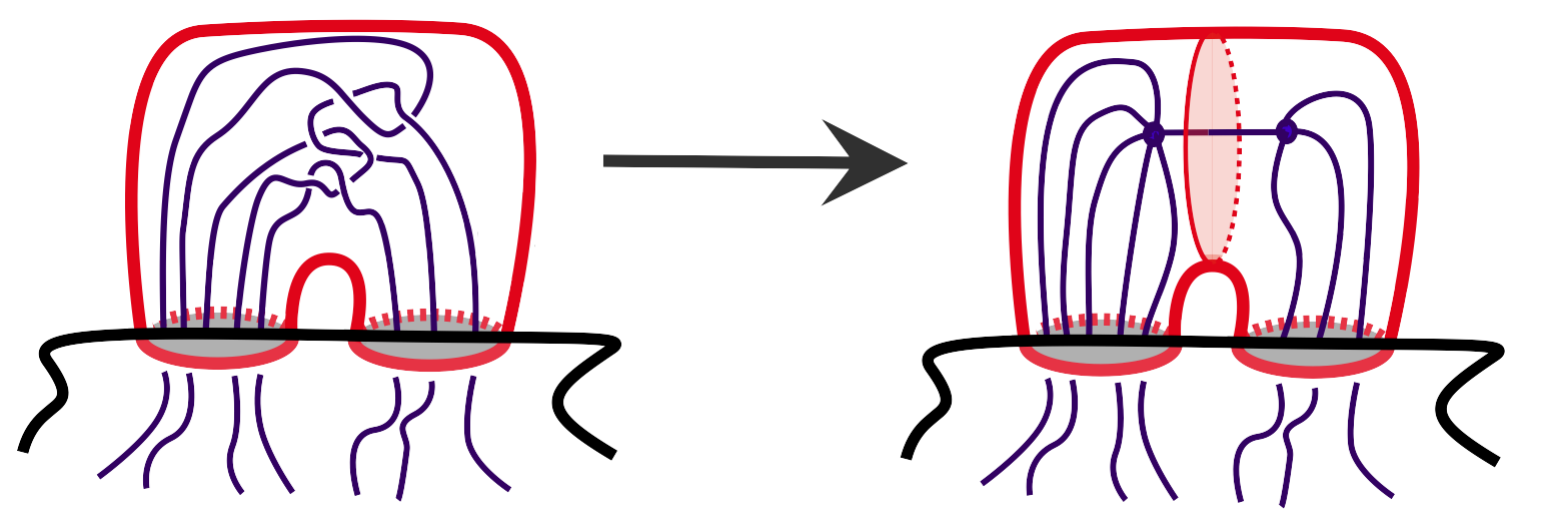}
\caption{In Example \ref{ex:crushing-prelim}, we provide an example of crushing in the simplest case when the discs $D_1$ and $D_2$ have interiors disjoint from $Q$. On the left we see the crushable handle $A \subset Q$ for the bridge surface $\mc{H}$ for $(M,T)$. On the right, we see the graph $\wihat{T}$ resulting from the crushing, which contains a new edge $e$ of weight $\omega$, the wrapping number. The tree $\wihat{T}$ is parallel into $A \cup D_1 \cup D_2$. There is also a crushing disc which is a cut-disc for $Q$ intersecting $e$. The bridge surface $\mc{H}$ has the same weight with respect to $T$ as with respect to $\wihat{T}$, since the edge $e$ is disjoint from $\mc{H}$.}
\label{fig:crushing-prelim}
\end{figure}

\begin{example}\label{ex:crushing-prelim}
Suppose that $\mc{H}$ is connected (i.e. is a bridge surface), that $Q$ is transverse to $\mc{H}$ and that $A \cpt Q \setminus \mc{H}$ is an annulus whose ends bound disjoint discs $D_1, D_2 \subset \mc{H}$ with interiors disjoint from $Q$, as on the left-hand side of Figure \ref{fig:crushing-prelim}. The annulus $A$ is a ``crushable handle.'' The sphere $A \cup D_1 \cup D_2$ bounds a 3-ball $B \subset V$. We replace $T \cap B$ with a tree having a single internal edge $e$ (called ``the newly crushed edge'') of weight $\omega$ (the wrapping number). One endpoint of $e$ is joined by edges to the points of $T \cap D_1$ and the other to the points of $T\cap D_2$, as on the right side of Figure \ref{fig:crushing-prelim}. The surface $\mc{H}$ persists as a bridge surface for $(\wihat{M}, \wihat{T}) = (M,\wihat{T})$ of the same weight. Note also the presence of a disc $D \subset B$ which, although it intersects $T$ in at least $\omega$ points, intersects $\wihat{T}$ exactly once. The disc $D$ is the ``crushing disc''. Compressing $Q$ using $D$ creates a twice-punctured summing sphere $\wihat{Q}$, realizing the fact that $\wihat{T}$ is the connected sum of a weighted spatial graph with the companion knot to $T$ (corresponding to $Q$) having weight $\omega$. After crushing $T$, the bridge surface $\mc{H}$ becomes c-weakly reducible, even it was not prior to the crushing.
\end{example} 

We will need a slightly more general version of crushing than that described in Example \ref{ex:crushing-prelim} and Figure \ref{fig:crushing-prelim}. Figure \ref{fig:crushing-forbidden} shows two examples of the kind of behavior that is forbidden. Figure \ref{fig:crushingnotation} shows one of the more general allowed types of crushable handles, along with some notation. Figure \ref{fig:crushing} shows the result of crushing, along with additional notation. We will of course need to allow $\mc{H}$ to be disconnected, in which case the crushable handle is an annulus of $Q \setminus \mc{H}$ having both ends on a thick surface $H$, with those ends bounding disjoint discs $D_1, D_2$. Secondly, we will need to relax the condition that $D_1$ and $D_2$ have interiors disjoint from $Q$; however we will forbid the intersection of $Q$ with $B$ from running all the way across $A$. Finally, since $\mc{H}^-$ may be empty, it is possible that there are components of $\mc{H}^-$ interior to $A$. It will turn out that such components must be spheres. Those spheres are separating in $M$, so we will surger $M$ along them, discarding the components not containing $A$. We will need to verify that this does not affect the fact that $\mc{H}$ is a multiple bridge surface and that it does not increase net weight.

\begin{figure}[ht!]
\labellist
\small\hair 2pt
\pinlabel{$A$} [b] at 168 249
\pinlabel{$A$} [b] at 545 249
\endlabellist
\centering
\includegraphics[scale=0.6]{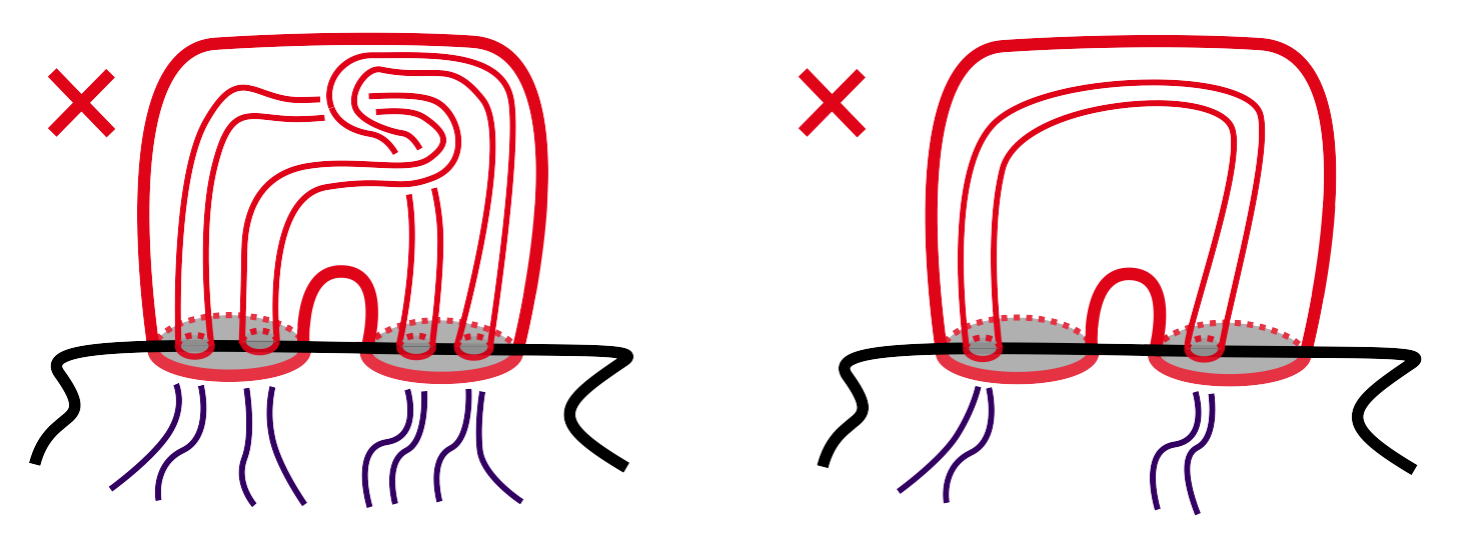}
\caption{Here are two examples of annuli $A$ that are \emph{not} crushable handles. (In each case, the portions of $T$ above the bridge surface are not shown.)}
\label{fig:crushing-forbidden}
\end{figure}

\begin{figure}[ht!]
\labellist
\small\hair 2pt
\pinlabel{$A$} [b] at 224 301
\pinlabel{$\omega$} [b] at 555 153
\pinlabel{$D_1$} [t] at 146 53
\pinlabel{$D_2$} [t] at 317 53
\pinlabel{$F$} [b] at  332 234
\pinlabel{$B$} [tl] at 116 276
\pinlabel{$(C, T_C)$} at 34 248
\endlabellist
\centering
\includegraphics[scale=0.6]{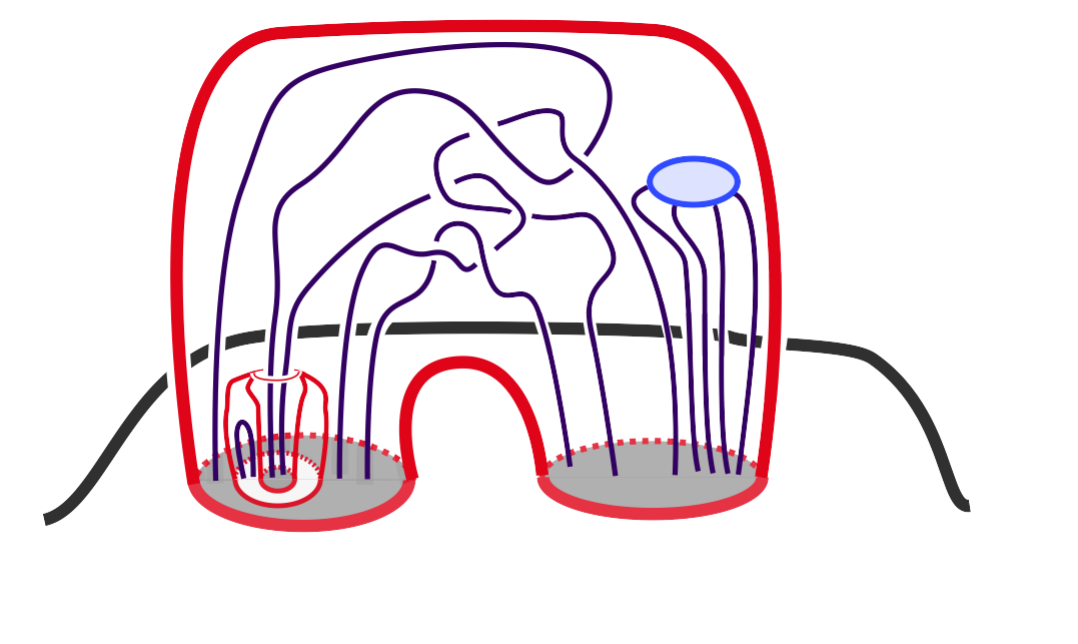}
\caption{The outer annulus $A$ is a crushable handle, as in Definition \ref{crushable handle}. The VPC $(C, T_C)$ contains $A$. As in Lemma \ref{disjt crushed spheres}, the sphere $A \cup D_1 \cup D_2$ bounds a submanifold $B \subset M$ which may contain thin surfaces, such as the sphere $F \subset \boundary_- C$. When crushing, as in Construction \ref{crushing inst}, we surger along such thin surfaces $F$; this converts $B$ into $\wihat{B}$. Figure \ref{fig:crushing} shows the result of crushing.}
\label{fig:crushingnotation}
\end{figure}

\begin{figure}[ht!]
\labellist
\small\hair 2pt
\pinlabel{$A$} [b] at 157 192
\pinlabel{$\omega$} [b] at 555 153
\pinlabel{$D_1$} [t] at 106 19
\pinlabel{$D_2$} [t] at 226 19
\pinlabel{$\wihat{B}$} [b] at 558 111
\endlabellist
\centering
\includegraphics[scale=0.6]{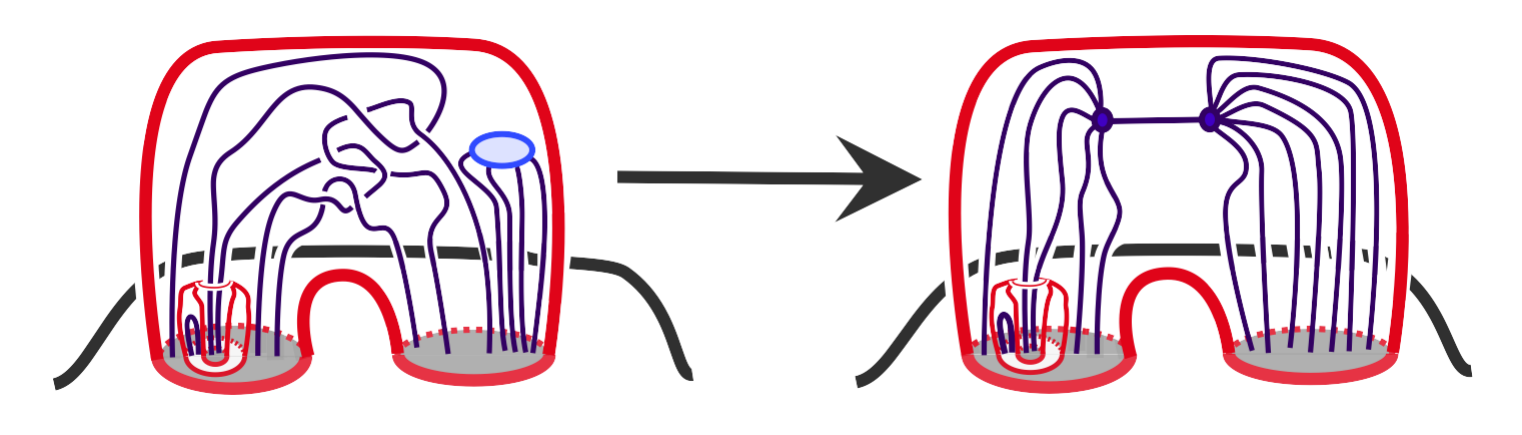}
\caption{ On the right, we show the result of crushing $T$ using the crushable handle as in Construction \ref{crushing inst}. On the right, the submanifold $\wihat{B} \subset \wihat{M}$ is bounded by $A \cup D_1 \cup D_2$, while $\Pi$ is the component of $\wihat{B}\setminus Q$ containing the crushing disc $D$ (not shown). $D$ is the cut disc for $A$ intersecting the newly crushed edge $e$ (which has weight $\omega$).}
\label{fig:crushing}
\end{figure}

\begin{definition}\label{crushable handle}
See Figure \ref{fig:crushingnotation} for a depiction. Suppose that $\mc{H}$ is a multiple bridge surface for $(M,T)$ and that $A \cpt Q \setminus \mc{H}$ is a bridge annulus and that the ends of $A$ bound disjoint discs $D_1$ and $D_2$ in $\mc{H}^+$. Then $A$ is a \defn{crushable handle}, if, for each $i = 1,2$, whenever a component $A'$ of $Q \setminus \mc{H}$ in the same VPC as $A$ has one end in $D_i$, then both ends of $A'$ lie in $D_i$ and they do not bound disjoint discs in $D_i$. 
\end{definition}

\begin{lemma}\label{disjt crushing disc}
 Suppose that $A \cpt Q\setminus \mc{H}$ is a crushable handle with ends bounding disjoint discs $D_1, D_2 \subset \mc{H}^+$. Then there exists a disc $D$ on the same side of $A$ as $D_1 \cup D_2$, with $\boundary D$ an essential loop in $A$, and with the interior of $D$ disjoint from $Q$. 
 \end{lemma}

We call the disc $D$ in the conclusion of Lemma \ref{disjt crushing disc} a \defn{crushing disc} for $A$.
 
\begin{proof} 
Let $(C, T_C) $be the VPC containing $A$. Suppose that $D \subset (C, T_C)$ is a disc with interior disjoint from $A$, on the same side of $A$ as the discs $D_1$ and $D_2$ and with $\boundary D \subset A$ parallel to each of $\boundary D_1$ and $\boundary D_2$, but disjoint from $\boundary C$. Observe that a pushoff of $D_1$ or $D_2$ is such a disc. Choose such a disc $D$ that intersects $T$ minimally and, out of all such, minimizes $|D \cap Q|$; such a disc may well intersect $T$ more than either $D_1$ or $D_2$. We claim that the interior of $D$ is disjoint from $Q$.

Suppose to the contrary that $D$ intersects some component $A'' \cpt Q \cap C$. One end of $A''$ must lie in either $D_1$ or $D_2$. Suppose, it is $D_1$. By definition, the other end of $A''$ also lies in $D_1$ and one end of $A''$ bounds a disc in $D_1$ containing the other end. An innermost disc argument shows that $D \cap A''$ consists of curves that are essential in $A''$. Each component of $A'' \setminus D$ is either an annulus with ends on $D$ or an annulus with one end on $D$ and one end also an end of $A''$. There must be a component $A''_0 \cpt A''\setminus D$ with ends on $D$ and with one end bounding a disc in $D$ containing the other end. An annulus swap converts $D$ to a disc that intersects $Q$ fewer times and intersects $T$ no more times than does $D$. This contradicts our choice of $D$.
\end{proof}

\begin{lemma}\label{disjt crushed spheres}
Let $B \subset V$ be the submanifold of $M$ with boundary the sphere $A \cup D_0 \cup D_1$ and which is inside $A$. Let $(C, T_C)$ be the VPC containing $A$. Then each component of $\boundary_- C$ contained in $B$ is a sphere belonging to $\mc{H}^-$.
\end{lemma}
\begin{proof}
The annulus $A$ $\boundary$-compresses to a disc $D$ such that $\boundary D$ bounds a disc $E \subset \boundary_+ C$. $\boundary$-reducing $(C, T_C)$ along $E$ produces two VPCs, one of which has spherical boundary containing $E$. Its negative boundary is exactly $\boundary_- C \cap B$. By Lemma \ref{all spheres}, it is the union of spheres.
\end{proof}

\begin{construction}\label{crushing inst}
Use the preceding notation and see Figures \ref{fig:crushingnotation} and \ref{fig:crushing} for a depiction and some of the notation. Let $\wh{M}$ be the result of surgering $M$ along the components of $\boundary_- C \cap B$; each of those components belongs to $\mc{H}^-$. Discard from $\wh{M}$ those components not containing $A$. Let $\wh{B}, \wh{V}$ be the surgery applied to $B, V$ respectively. Let $\wh{\mc{H}}$ be the result of discarding from $\mc{H}$ all components completely contained in $B$. Let $\Pi$ be the path component of $\wh{B}\setminus Q$ containing the interior of the crushing disc $D$ from Lemma \ref{disjt crushing disc}.

Let $\wh{T}$ be the graph constructed as follows.  Outside of $\Pi$, $\wh{T}$ coincides with $T$. Inside $\Pi$, the graph $\wh{T}$ is defined to be a tree $\tau$ with exactly two internal vertices $V(\tau)$ and with leaves equal to $ T \cap (D_1 \cup D_2) \cap \Pi$. There is exactly one edge having endpoints $V(\tau)$, and $\tau$ is isotopic, relative to its endpoints and in the complement of $T\setminus \Pi$, into $D_1 \cup D_2 \cup A$, with $e$ going to a spanning arc of $A$. We give $e$ weight $\omega$, where $\omega$ is the wrapping number of $T \cap V$ in $V$. We call $e$ the \defn{newly crushed edge}.

We say that $(\wh{M}, \wh{V}, \wh{T}, \wh{\mc{H}})$ is obtained from $(M,V, T,\mc{H})$ by \defn{crushing along handle $A$}.
\end{construction}

\begin{remark}
Observe that the crushing disc becomes a cut-disc. Figure \ref{fig:crushing notation} shows the result of compressing along that disc.
\end{remark}

\begin{remark}\label{high deg vertex}
By the definition of crushable handle, there are disc components of $D_1 \setminus Q$ and $D_2 \setminus Q$ which lie in $\Pi$. These discs are compressing discs for $Q$ in $V$ and thus intersect $T$ at least $\omega$ times each. Thus, each of the endpoints of $e$ has degree at least $\omega + 1$.
\end{remark}

\begin{lemma}\label{crushing lemma}
Suppose that $(\wh{M}, \wh{V}, \wh{T}, \wh{\mc{H}})$ is obtained from $(M, V, T, \mc{H})$ by crushing along a handle. Then the following hold:
\begin{enumerate}
    \item $\wh{\mc{H}} \in \vpoH(\wh{M}, \wh{T})$
    \item $\netx_\omega(\wh{\mc{H}};\weighted{\wh{T}}) = (\omega-1)\netchi(\wh{\mc{H}})/2 + \netx(\wh{\mc{H}};T)$.
    \item $Q$ is an essential unpunctured torus in $(\wh{M},\wh{T})$ and the weighted wrapping number of $\wh{T}\cap \wh{V}$ in $\wh{V}$ is equal to $\omega$. Furthermore, there is a cut compressing disc in $\wh{V}$ for $Q$ intersecting the newly weighted edge.
    \item $(\wh{M}, \wh{T})$ is irreducible and no essential twice-punctured sphere in $(\wh{M}, \wh{T})$ has punctures of different weights.
    \end{enumerate}
\end{lemma}

\begin{proof}
Throughout we use the notation of Construction \ref{crushing inst}. 

We first show Claim (1). Let $(C, T_C)$ be the VPC containing the handle $A$. As we remarked in Lemma \ref{disjt crushed spheres}, components of $\boundary_- C$ contained completely interior to $B$ are spheres, so surgering $(M,T)$ along those spheres and discarding the pairs that don't contain $Q$ does not change the fact that we have a multiple bridge surface. Consequently, we may assume that $B$ is a 3-ball and no component of $\boundary_- C$ is interior to $B$. Each component of $Q \cap B$ (other than $A$) is an annulus with both ends in $D_0$ or with both ends in $D_1$. Each separates $C$. Let $T'_C = T_C\setminus \Pi$ and note that $(C, T'_C)$ is a VPC and that the crushing disc $D$ is an unpunctured disc. Let $D'_0$ and $D'_1$ be the two unpunctured discs that result from compressing $A$ along $D$ in $(C, T'_C)$. We may, therefore, choose a complete set of sc-discs $\Delta$ for $(C, T'_C)$ that are disjoint from $A \cup D$ and which contain $D'_0, D'_1$. It remains a complete set of discs for $(C, T'_C \cup (\tau\setminus e))$, since each component of $\tau \setminus e$ is isotopic into $D_0$ or $D_1$. Putting in the edge $e$ turns $D'_0$ and $D'_1$ into cut discs, but preserves the fact that $\Delta$ is a complete set of discs. Consequently, $(C, \wh{T} \cap C)$ is a VPC. It follows that $\wh{\mc{H}} \in \vpoH(\wh{M}, \wh{T})$. 

We now show Claim (2). Since $\boundary_+ C \cap T = \boundary_+ C \cap \wh{T}$ and those points all have weight equal to 1, we need only compare the effect of the crushing on components of $\mc{H}$ completely internal to $B$. It suffices to consider the case when there is a single component $F \cpt \mc{H}^- \cap B$. (If there are more, simply iterate the following calculations.) Recall that $F$ is a punctured sphere, which must separate $M$. The sphere $F$ bounds a submanifold $W_F \subset B$. Let $\mc{H}_F$ be those components of $\mc{H}$ interior to $W_F$. Since (in the absence of $T$) $\mc{H}_F$ amalgamates to a Heegaard splitting for $W_F$, $\netchi(\mc{H}_F) \geq x(F)$. Thus, $\netchi(\wh{\mc{H}}) \leq \netchi(\mc{H})$. By Theorem \ref{connected boundary bound}, $\x(F) \leq \netx(\mc{H}_F)$.  Thus,
\[
\netx(\wh{\mc{H}};\wh{T}) = \netx(\mc{H};T) - (\netx(\mc{H}_F) - \x(F)) \leq \netx(\mc{H};T).
\]
The newly weighted edge $e$ of $\wh{T}$ is disjoint from $\wh{\mc{H}}$; it does not contribute anything to $\netx_\omega(\wh{\mc{H}};\wh{T})$. Consequently,
\[\begin{array}{rcl}
\netx_\omega(\wh{\mc{H}};\weighted{\wh{T}}) &=& (\omega-1)\netchi(\wh{\mc{H}})/2 + \netx(\wh{\mc{H}};\weighted{\wh{T}}) \\
&\leq& (\omega-1)\netchi(\mc{H})/2 + \netx(\mc{H};T),
\end{array}
\]
as desired.

We prove Claim (3). Since, by Lemma \ref{disjt crushing disc}, there is a cut-disc for $Q$ with boundary equal to $\boundary D$ and intersecting the edge $e$, the weighted wrapping number is at most $\omega$. Suppose that there exists a properly embedded disc $E$ in $V$ with essential boundary in $Q$ and of weight strictly less than $\omega$. Then $E \cap e = \nil$. The disc $E$ can then be properly isotoped to be disjoint from $\Pi$. That is, it is a compressing disc for $Q$ in $V$.  But this implies the wrapping number for $T\cap V$ in $V$ would be strictly less than $\omega$, a contradiction. Thus, the weighted wrapping number of $\wh{T} \cap \wh{V}$ in $\wh{V}$ is exactly $\omega$. 

Finally, we prove Claim (4). Suppose that $P \subset (\wh{M}, \wh{T})$ is an essential sphere which is unpunctured, once-punctured, or twice-punctured with punctures of different weights. In the latter case, one puncture on $P$ lies in $\wh{T}\setminus e$ and the other in $e$. We can reconstruct $(M,T)$ from $(\wh{M}, \wh{T})$, by removing a regular neighborhood of $e$ and re-inserting pieces of $T$, $M$, and $\mc{H}$. By doing so, we see that $P$ gives rise to a an essential sphere, or a compressing disc for $Q$ in $(M,T)$ or a cut-disc for $Q$ in $(M,T)$. All contradict our initial assumptions.
\end{proof}

The next theorem contains the essence of the proof of the \href{Main Theorem}{Main Theorem} in the case when we can find a crushable handle.

\begin{theorem}[Handle-Crushing Theorem]\label{thm: handle crush}
Assume that:
\begin{itemize}
\item $(M,T)$ is standard and irreducible,  $Q \subset (M,T)$ is a c-essential torus, $\mc{H} \in \H(Q)$ with $g = \netg(\mc{H})$
\item there is a crushable handle $A \cpt Q \setminus \mc{H}$ with crushing disc $D$ on side $V$ of $Q$, and $\weighted{L}$ is a weighted companion for $T$ relative to $V$, 
\item $(\wh{M}, \wh{V}, \wh{T}, \wh{\mc{H}})$ be obtained by crushing $(M, V, T, \mc{H})$ using $A$. 
\item $M_1$ is obtained by $\boundary$-reducing $\wh{V}$ along $D$.  
\item $M_0$ be the manifold obtained from the complement of $\wh{V}$ by attaching a 2-handle to $Q$ along $\boundary D$.
\end{itemize}
Then there exist integers $g_0, g_1$ such that:
\begin{enumerate}
    \item $g_0 + g_1 = g$, 
    \item $g_0 \geq \g(M_0)$ and $g_1 \geq \g(M_1)$, and 
    \item We have:
\[
\netw(\mc{H};T)/2 \geq  \netw_{g_0}(M_0;\weighted{L})/2 -\omega g_1 
\]
\end{enumerate}
\end{theorem}
\begin{proof}
See Figure \ref{fig:crushing notation} for a depiction of the notation in this paragraph. Let $\wh{Q}$ be the twice-punctured summing sphere obtained by cut-compressing $Q$ in $(\wh{M},\wh{T})$ using the crushing disc $D$ for $A$. Note that $\x_\omega(\wh{Q}) = 0$ and $\omega \geq 2$. The sphere $\wh{Q}$ defines a connected sum decomposition $(\wh{M}, \wh{T}) = (M_0, T_0) \#_2 (M_1, T_1)$. Choose the notation so that $M_1$ is the result of capping off the 3-manifold obtained by $\boundary$-reducing $\wh{V}$ using $D$ with a 3-ball. Then $T_0$ is easily seen to be isotopic to a weighted companion $\weighted{L} = (T\setminus V) \cup \weighted{K}$ for $T$ with respect to $V$. It has weight $\omega$ on the newly crushed edge $e$ and weight equal to 1 elsewhere. The graph $T_1$ contains two vertices $v_\pm$ that were not vertices of $T$ and one edge $e'$ that does not belong to $T$ and which has $v_\pm$ as its endpoints. The edge $e'$ is the union of the subarcs $e \setminus N(D)$ with an arc in the 3-ball used to cap off the boundary of $V|_D$. The edge $e'$  also has weight $\omega$, and all other edges of $T_1$ have weight equal to 1. If $\netg(\wh{\mc{H}}) < g$, stabilize some component of $\wh{\mc{H}}^+$ enough times so that $\netg(\wh{\mc{H}}) = g$. 

\begin{figure}[ht!]
\labellist
\small\hair 2pt
\pinlabel{$\wh{Q}$} [br] at 153 368
\pinlabel{$T_1$} [b] at 193 325
\pinlabel{$T_1$} [b] at 437 317
\pinlabel{$v_-$} [r] at 212 280
\pinlabel{$v_+$} [tr] at 390 278
\pinlabel{$\weighted{e'}$} [b] at 248 289
\pinlabel{$\weighted{e'}$} [b] at 375 289
\pinlabel{$\weighted{K}$} [b] at 308 289
\pinlabel{$M_1$} [t] at 425 383
\pinlabel{$M_0$} [b] at 425 395
\endlabellist
\centering
\includegraphics[scale=0.5]{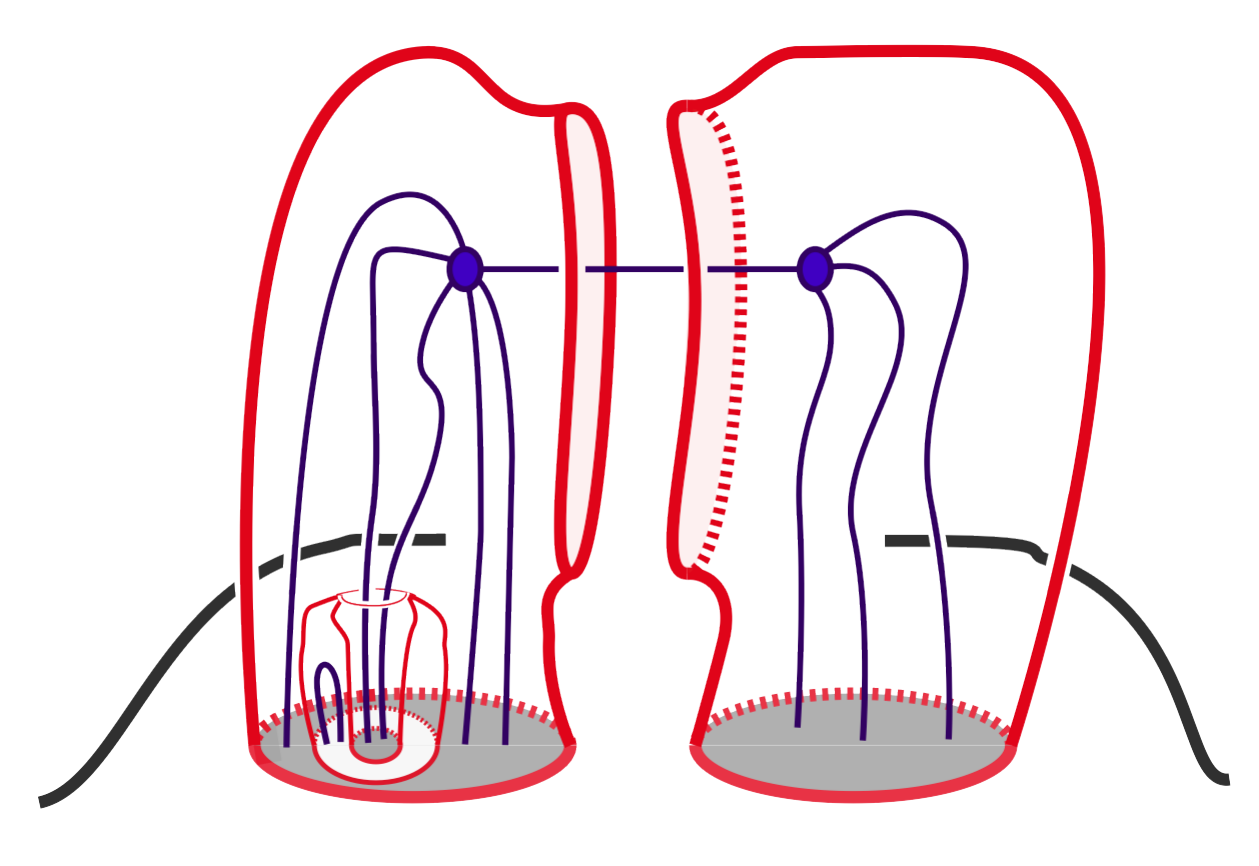}
\caption{We depict the notation used in the proof of \href{thm: handle crush}{The Handle-Crushing Theorem}. The pairs $(M_0, T_0)$ and $(M_1, T_1)$ are obtained by surgering along sphere $\wh{Q}$, with the labels chosen as indicated. In $T_1$, the edge $e'$ includes the two arcs of $\wh{T}$ labeled as $e'$, as well as an arc in the 3-ball used in surgering along $\wh{Q}$; all three have weight $\omega$. The knot $\weighted{K}$ has weight $\omega$ and is the union of the arc labelled as $\weighted{K}$ with an arc in the 3-ball used in the surgery along $\wh{Q}$, also having weight $\omega$.}
\label{fig:crushing notation}
\end{figure}

We follow the proof of the \href{thm: add}{Additivity Theorem}, abbreviating the arguments that are repetitions.

By Theorem \ref{lem:Thinning invariance}, $\wh{\mc{H}}$ thins to a locally thin $\mc{J} \in \vpoH(\wh{M}, \wh{T})$. By Theorem \ref{thm:essential twicepunct sphere}, since $\wh{Q} \subset (\wh{M}, \wh{T})$ is an essential twice-punctured sphere, there is a nonempty efficient summing system $S$ contained in $\mc{J}^-$. By the irreducibility of $(\wh{M}, \wh{T})$, $S$ is the union of twice and thrice-punctured spheres. Discard the thrice-punctured spheres. As $(\wh{M}, \wh{T})$ does not contain a twice-punctured sphere with punctures of different weights (Lemma \ref{crushing lemma}), by Theorem \ref{thm: efficient}, both $(\wh{M}, \wh{T})|_S$ and $S$ are unique up to homeomorphism. Thus, from the components of $(\wh{M}, \wh{T})|_S$ we can reconstruct $(M_0, T_0)$ and $(M_1, T_1)$. From the components of $\mc{J}$ in $(\wh{M}, \wh{T})|_S$ and the (new) summing spheres, we can construct $\mc{J}_i \in \vpoH(M_i, T_i)$ (for $i = 0,1$). Let $g_i = \netg(\mc{J}_i)$. This construction can be accomplished so that, letting $\chi_i = \netchi(\mc{J}_1)$, we have $\chi_0 + \chi_1 + 2 = \netchi(\mc{J}) = \netchi(\mc{H})$, $g = g_0 + g_1$, and 
\[
\netx_\omega(\mc{J}_0) + \netx_\omega(\mc{J}_1) = \netx_\omega(\mc{J}) \tag{$\dag$}
\] 
as $\x_\omega(\wh{Q}) = 0$.

We claim that $\netx_m(\mc{J}_1;\weighted{T_1}) \geq 0$. Recall that $e'$ is the unique edge of $\weighted{T_1}$ of weight $\omega$.

To make the arithmetic easier, we begin by collapsing certain submanifolds to vertices. Let $P \subset \mc{J}^-_1$ be the union of spheres disjoint from $e'$. Then $(M_1, T_1)|_P$ has a unique component containing $e'$. We discard all other components. As in the proof of Lemma \ref{crushing lemma}, this does not increase $\netx(\mc{J}_1)$ or $\netchi(\mc{J}_1)$; for the purposes of the arithmetic that follows, the discarded components do not matter. For simplicity, therefore, assume that $P = \nil$. Let $U$ be the union of all vertices of $T_1$ that are also vertices of $T$. (If, in fact, $P$ were nonempty, we would also include into $U$ all the new vertices of $T$ introduced by the aforementioned surgery.) Observe that all vertices of $U$ are adjacent only to edges of weight 1. 

Let $V = v_- \cup v_+$ (the endpoints of $e'$). Note that 
\[
\x_\omega(v_\pm) = 2(-\omega + |D_\pm \cap T|/2 + \omega) \geq 0.\]
Thus, by Lemma \ref{weighted lower bound},
\[
2\netx_\omega(\mc{J};\weighted{\wh{T}_1}) = \x(U) + \x_\omega(V) + \sum\limits_{(C, T_C)} \delta_\omega(C, T_C) \tag{$*$}
\]
where the sum is over all VPCs $(C, T_C) \cpt (M_1, T_1)\setminus \mc{J}_1$. The quantities $\x(U)$ and $\x_\omega(V)$ are non-negative. We desire to show that the sum is also. This is where we use the fact that we weighted $e$ by $\omega$, the wrapping number, rather than, say, the number of times the crushing disc intersects $T$.

We proceed by contradiction. Suppose the sum in ($*$) is negative, then there exists a VPC $(C_2, \tau_2)$ with $\delta_\omega (C_2, \tau_2) < 0$. We will establish a contradiction similarly to the proof of Lemma \ref{net weight bound}. Let $(C_1, \tau_1)$ be the other VPC such that $\boundary_+ C_1 = \boundary_+ C_2 = H \cpt \mc{J}^+_1$. 

By Lemma \ref{negative delta}, $(C_2, \tau_2)$ is a 3-ball and all vertices of $\tau_2$ lie in $U$. Also, every edge of $\tau_2$ has weight $1$. Since $H$ is a sphere, $\boundary_- C_1$ is the union of spheres; consequently, $\boundary_- C_1 \subset \mc{H}^-$, as $(M_1, T_1)$ is standard. Each component of $\boundary_- C_1$ intersects $e'$, since $P = \nil$ (alternatively, since the spheres in $P$ were collapsed to vertices). Since every puncture on $H$ lies in $\tau_2$ and since all edges of $\tau_2$ have weight 1, it must be the case that each subarc of $e'$ in $\tau_1$ is a ghost arc.

Let $\Gamma$ be the ghost arc graph for $(C_1, \tau_1)$; its edges are the ghost arcs of $\tau_1$ and its vertices are the vertices of $\tau_1$ together with the components of $\boundary_- C_1$. Since $H$ is a sphere, $\Gamma$ is acyclic. We just showed that $\Gamma$ has at least one edge of weight $\omega$.

\begin{figure}[ht!]
\labellist
\small\hair 2pt
\pinlabel{$(C_1, \tau_1)$} [b] at 154 81
\pinlabel{$e_\lambda$} [b] at 173 175
\endlabellist
\centering
\includegraphics[scale=0.35]{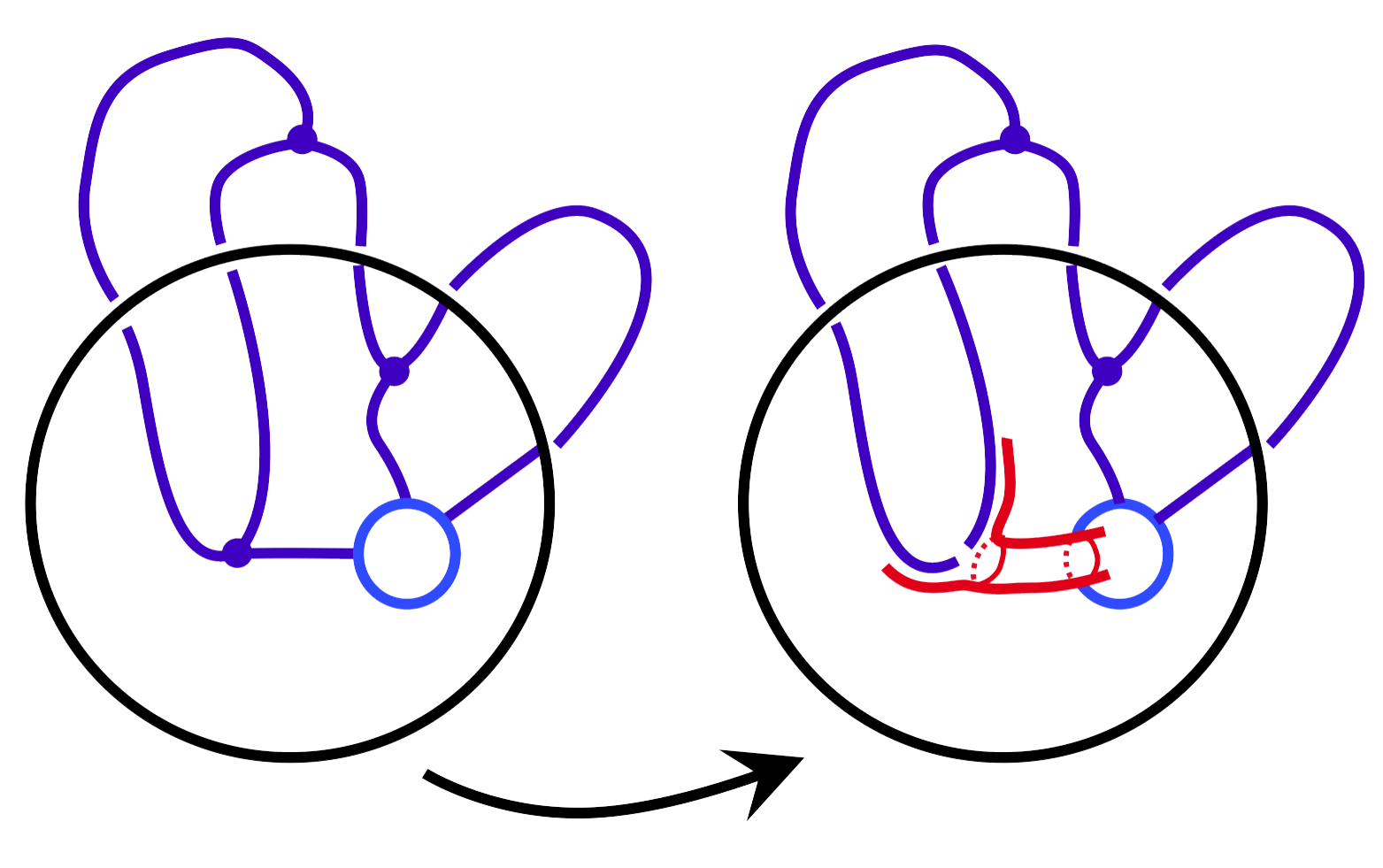}
\caption{We depict the situation in Cases 1 and 2 of the proof of Theorem \ref{thm: handle crush}. For Case 1, consider $\lambda$ to be the vertex on the left and for Case 2, consider $\lambda$ to be the thin sphere on the right in $(C_1, T_1)$. The arrow shows how when we recover $(M,T)$, the thin sphere on the right becomes a punctured disc in $V$ with boundary on $Q$.}
\label{fig: using wrapping number}
\end{figure}

See Figure \ref{fig: using wrapping number} for a depiction of the notation used in the next two claims. Let $\lambda$ be a leaf of the subgraph $\Gamma(\omega)$ of $\Gamma$ induced by edges of weight $\omega$. Let $e_\lambda$ be the sole edge of weight $\omega$ incident to $\lambda$. 

\textbf{Case 1:} The vertex $\lambda$ lies in $V$.

By Remark \ref{high deg vertex}, $\lambda$ is incident in $\tau_1$ to at least $\omega$ edges of weight 1. Other than $e_\lambda$, if $f$ is a ghost arc incident to $\lambda$, the subtree of $\Gamma\setminus \lambda$ containing $f$, contains edges of weight 1 and vertices in $U$. Other than $\lambda$, each leaf of that subtree is incident to at least two vertical arcs. Consequently, $\lambda$ produces at least $\omega$ punctures in $H$.

\textbf{Case 2:} $\lambda$ is a component of $\boundary_- C$. 

Since it is a sphere in $(M_1, T_1)$ intersecting $e'$ exactly once and as we can recover $\eta(\boundary D) \subset A$ from a regular neighborhood of $e$, the sphere $\lambda$ restricts to a punctured disc in $\wh{V}$. Since $\omega$ is the wrapping number of $T\cap V$ in $V$, the sphere $\lambda$ must again contain the end points of at least $\omega$ edges of $\tau_1$ of weight 1. As in Case 1, this implies that $\lambda$ again contributes at least $\omega$ punctures to $H$.

Since $\Gamma(\omega)$ has at least two leaves and since $\Gamma$ is acyclic, $H$ has at least $2\omega$ punctures. However, this implies that $\delta_\omega(C_2, \tau_2) \geq 0$, contrary to our initial choice. As desired, we conclude that 
\[
\netx_\omega(\mc{J}_1;\weighted{T_1}) \geq 0. \tag{$\ddagger$}
\]

Thus, 
\[\begin{array}{rcl}
\omega\chi/2 + \netw(\wh{\mc{H}};T)/2 &=& (\text{Lemma \ref{crushing lemma} and Remark \ref{dropping weights}})\\
\netx_\omega(\wh{\mc{H}};\weighted{\wh{T}}) &=& {\small (\text{Theorem \ref{lem:Thinning invariance}.3})}\\
\netx_\omega(\mc{J}; \weighted{T}) &=& {\small (\dag)}\\
\netx_\omega(\mc{J}_0; \weighted{T_0}) + \netx_\omega(\mc{J}_1;\weighted{T_1}) &\geq& {\small (\ddagger)}\\
 \netx_\omega(\mc{J}_0; \weighted{T_0})&=&{\small (\text{Theorem \ref{lem:Thinning invariance}.3})} \\
\omega\chi_0/2 + \netw(\mc{J}_0; \weighted{T_0})/2.
\end{array}
\]

Recalling that $\chi_0 + \chi_1 + 2 = \chi$, produces
\[
\netw(\wh{\mc{H}};T)/2 \geq -\omega(\chi_1 + 2)/2 + \netw(\mc{J}_0; \weighted{T_0})/2
\]

Hence,
\[
\netw(\mc{H};T)/2 \geq -\omega g_1 + \netw_{g_0}(M_0;\weighted{L})/2.
\]
\end{proof}
Although we can use the \href{thm: handle crush}{Handle-Crushing Theorem} to study bridge number for any $g$, in the interests of space we restrict attention to the cases when $g = 0,1$, as the ability to amalgamate gives us the strongest conclusions.

\begin{corollary}\label{low g}
Suppose that the hypotheses of the Handle Crushing Theorem hold, that $g \in \{0,1\}$, and $\netw(\mc{H}) \leq 2\b_g(T)$. Also suppose that $T \cap V^c$ is a (possibly empty) link. Then
    \[
    \b_g(T) \geq \b_g(\weighted{L}) - \omega\delta
    \]
    where both sides are computed in $M$ and $\delta = 0$ if $V$ is a solid torus and $\delta = 1$ if $V$ is a lensed solid torus.
\end{corollary}
\begin{proof}
    Recall that $g_0 + g_1 = 1$, so $\{g_0, g_1\} = \{0,1\}$. By Lemma \ref{sphere eq}, $\netw_{g_0}(\weighted{L})/2 = \b_{g_0}(\weighted{L})$. Thus, by  Theorem \ref{thm: handle crush}, $\b_g(T) \geq \b_{g_0}(\weighted{L}) - \omega g_1$. Finally, observe that $g_1 = 1$ if and only if $V$ is a lensed solid torus.
\end{proof}

The remainder of this section is devoted to giving our new proof of Schubert's Satellite Theorem. We phrase it for spatial graphs, in the interest of adding some spice.

\begin{lemma}\label{spheres and crushable handles}
Suppose that $\mc{H} \in \H(Q)$ and that $(C, T_C) \cpt (M,T)\setminus \mc{H}$ is an innermost or outermost VPC intersecting $Q$. Suppose that a component of $Q \cap C$ is a bridge annulus with ends bounding disjoint discs in $\boundary_+ C$. Then there is a crushable handle in $C$. 
\end{lemma}
\begin{proof}
Let $A_0 \cpt Q \cap C$ be a bridge annulus with ends bounding disjoint discs $D_1, D_2$ in $\boundary_+ C$. As $Q$ is a torus, each component $A \cpt Q \cap C$ is an annulus. By our choice of $C$, no such annulus is vertical. If $D_1, D_2$ are disjoint from $Q$ in their interiors, then $A_0$ is a crushable handle. Otherwise, we may choose $A \cpt Q \cap C$ to have the property that its ends bound disjoint discs $D'_1, D'_2$ in $D_1 \cup D_2$ and that its ends are innermost with that property. (That is, no component of $Q \cap C$ with ends in the interior of $D'_1 \cup D'_2$ has ends that bound disjoint discs in $D'_1 \cup D'_2$.) Possibly, $A = A_0$. As no component of $Q \cap C$ is vertical, $A$ is a crushable handle.
\end{proof}

\begin{corollary}[Schubert's Satellite Theorem]\label{Schubert Cor}
Suppose that $T \subset S^3$ is a spatial graph and that $Q \subset (S^3,T)$ is an essential unpunctured torus compressible to a unique side $V \cpt S^3\setminus Q$. Suppose that $T\setminus V$ is a (possibly empty) link and that the wrapping number of $T\cap V$ in $V$ is $\omega \geq 1$. Then  
\[
\b_0(T) \geq \b_0(\weighted{L}) \geq \omega\b_0(K).
\]
where $\weighted{L} = (T\setminus V) \cup \weighted{K}$ is a weighted companion for $T$ with respect to $V$ and $\weighted{K}$ is the core loop of $V$ weighted by the wrapping number $\omega$.
\end{corollary}
\begin{proof}
    If $\omega = 1$, then this is Corollary \ref{omega1}. Thus, we may assume that $\omega \geq 2$. By Theorems \ref{ensure irreducible} and \ref{cut uniqueness}, we may assume that $(M,T)$ is irreducible.
    
    Let $H$ be a minimal bridge surface for $T$. By Lemma \ref{lem:Thinning invariance} and Corollary \ref{adapted}, there exists $\mc{H} \in \H(Q)$ such that $\netg(\mc{H}) = 0$ and $\netw(\mc{H}) \leq 2\b_0(T)$. Each component of $\mc{H}$ is a sphere, so some outermost or innermost VPC intersecting $Q$ is a 3-ball. Every component of $Q$ in that VPC is an annulus whose ends bound disjoint discs in its positive boundary. By Lemma \ref{spheres and crushable handles}, there is a crushable handle. By Corollary \ref{low g},
    \[
    \b_0(T) \geq \netw_0(\weighted{L})/2.
    \]
    The result follows from the observation that discarding components of a link cannot increase bridge number.
\end{proof}

\section{Annuli in VPCs} \label{sec:annuli}

The remainder of the paper is devoted to the search for crushable handles. Much of our work applies beyond considering companion tori in $S^3$ or lens spaces. As long as it does not overly complicate the exposition, we work as generally as possible. In particular $Q$ need not be a torus and $M$ need not be $S^3$ or a lens space, though we do always assume $(M,T)$ is standard.

Annuli inside VPCs will play an important role in our discussions. We begin with some basic results about such annuli.

\begin{definition}
If $(C, T_C)$ is a VPC and $A \subset (C, T_C)$ is a c-essential annulus disjoint from $T_C$, then $A$ is \defn{bridge} if $\boundary A \subset \boundary_+ C$ and \defn{vertical} if one component of $\boundary A$ is in each of $\boundary_\pm C$. 
\end{definition}

We use the following lemma often and without comment. 

\begin{lemma}\label{bridge or vert}
Suppose $(C, T_C)$ is a VPC such that no component of $\boundary_- C$ is an unpunctured $S^2$. If $A \subset (C, T_C)$ is an unpunctured c-essential annulus, then $A$ is bridge or vertical.
\end{lemma}
\begin{proof}
An innermost disc argument shows that if  $\boundary A \subset \boundary_- C$, then $A$ is disjoint from a complete set of sc-discs for $(C, T_C)$. It thus lies in a trivial product compressionbody and must be $\boundary$-parallel.
\end{proof}

\begin{definition}\label{six types}
Suppose that $A \subset (C, T_C)$ is an unpunctured bridge or vertical annulus in a VPC. Assume that $\boundary_- C$ contains no unpunctured spheres. Let $H = \boundary_+ C$ and let $F \cpt \boundary_- C$ contain a component of $\boundary A$, if $A$ is vertical. Observe that if an end of $A$ is nonseparating in $F$ then its end in $H$ is also nonseparating in $H$. Harnessing Lemma \ref{bridge or vert}, we denote all the possibilities for $A$ (and depict them in Figure \ref{fig:annuli}), as follows:
\begin{enumerate}
    \item (Type BNN) A bridge annulus with both ends nonseparating in $H$. 
    \item (Type VNN) A vertical annulus with both ends nonseparating in $H \cup F$.
    \item (Type VSS) A vertical annulus with both ends separating in $H \cup F$.
    \item (Type VNS) A vertical annulus with one end nonseparating in $H$ and one end separating in $F$.
    \item (Type BSS) A bridge annulus with both ends separating in $H$
    \item (Type BNS) A bridge annulus with one end separating and the other nonseparating in $H$.
\end{enumerate}
\end{definition}

\begin{figure}[ht!]
\labellist
\small\hair 2pt
\pinlabel{BNN} [r] at 142 178
\pinlabel{VNN} [r] at 512 52
\pinlabel{VSS} [t] at 687 104
\pinlabel{VNS} [r] at 221 44
\pinlabel{BSS} [tl] at 310 126
\pinlabel{BNS} [r] at 89 77
\endlabellist
\centering
\includegraphics[scale=0.5]{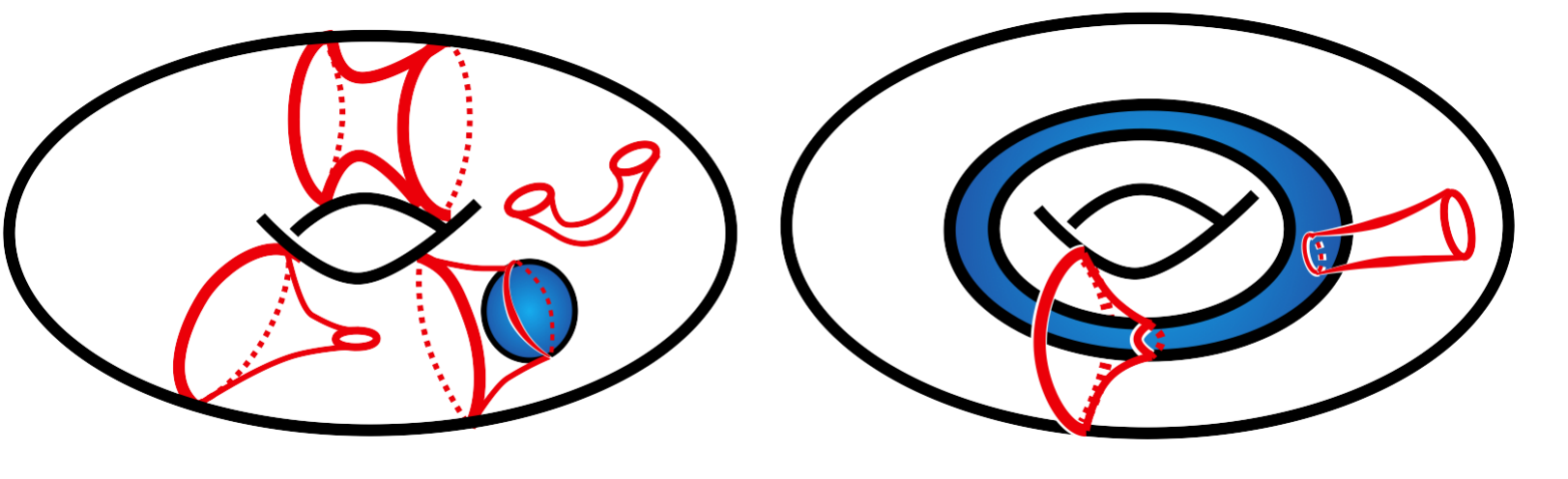}
\caption{The six types of annuli from Definition \ref{six types} are shown in red.}
\label{fig:annuli}
\end{figure}

\begin{lemma}\label{lem:annuli}
Suppose that $(C, T_C)$ is a VPC such that $\boundary_- C$ has no unpunctured spheres, and let $\mc{A} \subset (C, T_C)$ be a collection of disjoint c-essential unpunctured annuli and twice-punctured discs. If $\Delta$ is a collection of discs chosen to intersect $\mc{A}$ minimally, then:
\begin{enumerate}
\item $\Delta$ is disjoint from each vertical annulus in $\mc{A} \cap C$.
\item Each bridge annulus $A$ in $\mc{A} \cap C$ has intersection with $\Delta$ consisting of a possibly empty collection of arcs (no loops) each of which is a spanning arc in the annulus.
\item If $D$ is a twice-punctured disc then $\Delta \cap D$ is the union of properly embedded arcs, each separating the punctures in $D$. 
\end{enumerate}
\end{lemma}
\begin{proof}
This is another straightforward innermost disc/outermost arc argument.
\end{proof}

We use the following lemma to find sc-discs disjoint from $Q$. See Figure \ref{fig:corralling} for a schematic depiction.

\begin{lemma}\label{lem:disjoint vertical}
Suppose that $(C, T_C)$ is a VPC with no unpunctured sphere in $\boundary_- C$ and let $\mc{A} \subset (C, T_C)$ be a collection of disjoint c-essential unpunctured annuli and twice-punctured discs. Let $A \subset \mc{A}$ be the union of vertical annuli and assume that the ends of $A$ separate $\boundary_+ C$ into two distinct subsurfaces $X$ and $Y$. Suppose that there is an sc-disc $D$ for $\boundary_+ C$ with $\boundary D \subset X$. Then there exists an sc-disc $E$ with boundary in $X$ and which is also disjoint from $\mc{A}$.
\end{lemma}
\begin{proof}
An innermost disc argument shows that we may assume that $D \cap A = \nil$ and that $D \cap \mc{A}$ consists of arcs essential in the components of $\mc{A}$ containing them. If $D$ is disjoint from $\mc{A}$, then we are done, so suppose that $D \cap \mc{A} \neq \nil$. Let $\alpha$ be an arc of $D\cap \mathcal{A}$ bounding a sub-disc $E \subset D$ not containing a puncture of $D$ (if such exists) and which has interior disjoint from $\mathcal{A}$. Let $A'$ be the component of $\mc{A}$ containing $\alpha$. Observe that $\boundary A' \subset X$. Boundary compress $A'$ using $E$. By our choice of $E$, we obtain one or two once-punctured or unpunctured discs with boundary in $\boundary_+ C$. At least one of them must be essential and a small isotopy makes it disjoint from $\mc{A}$.
\end{proof}

\begin{figure}[ht!]
\labellist
\small\hair 2pt
\pinlabel{$Y$} [b] at 79 254
\pinlabel{$Y$} [b] at 613 254
\pinlabel{$X$} [b] at 330 254
\pinlabel{$A$} [r] at 130 117
\pinlabel{$A$} [l] at 513 117
\pinlabel{$\mc{A}$} [t] at 359 153
\pinlabel{$D$} [tl] at 241 147
\pinlabel{$E$} [t] at 254 122
\pinlabel{$(C, T_C)$} at 618 122
\endlabellist
\centering
\includegraphics[scale=0.5]{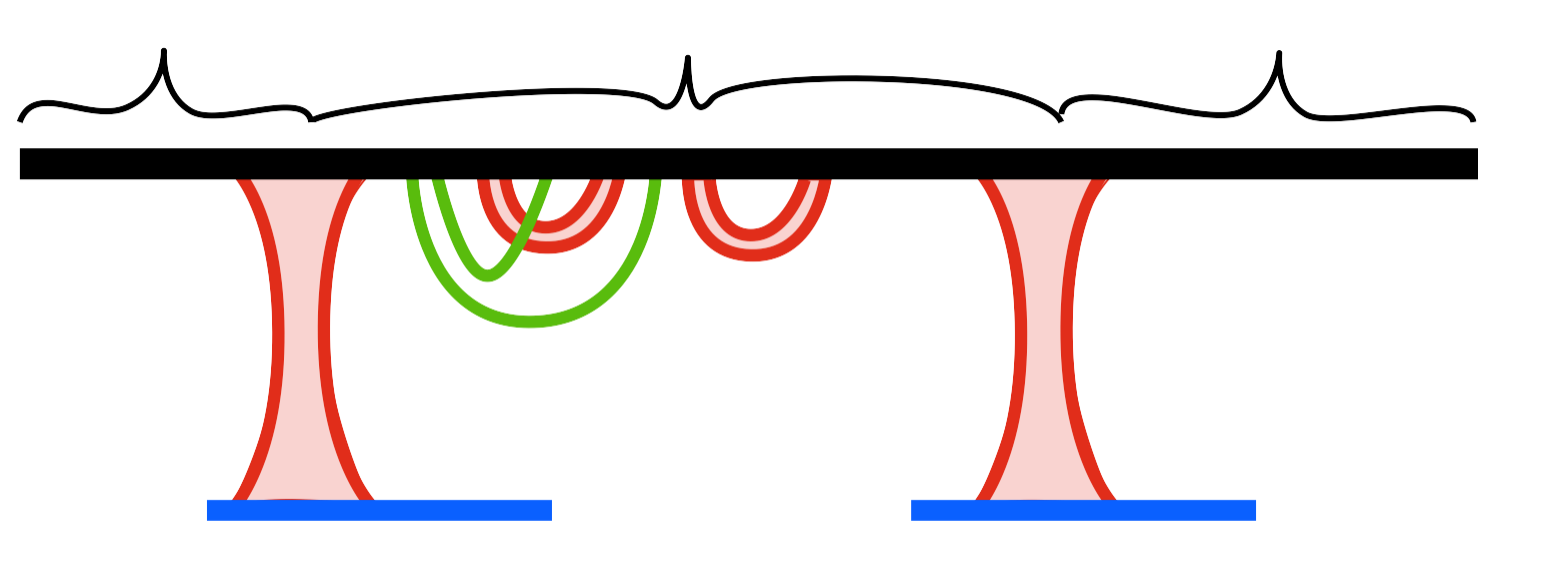}
\caption{Finding discs with boundary in a subsurface, using Lemma \ref{lem:disjoint vertical}. The annuli $\mc{A}$ are shown in red, with $A \subset \mc{A}$ the vertical annuli whose ends separate $\boundary_+ C$ into subsurfaces $X$ and $Y$. The disc $D$ intersects $\mc{A}$ and has boundary in $X$. The lemma guarantees the existence of a disc $E$ disjoint from $\mc{E}$ with boundary in $X$.}
\label{fig:corralling}
\end{figure}

\section{$Q$-thinning and Subsurface Amalgamation}\label{sec:subsurface amalgamation}

This section adapts the notions of thinning and amalgamation so that they interact nicely with a surface $Q$.

\begin{assumption}\label{subsurf amalg assump}
    Throughout this section, suppose that $Q \subset (M,T)$ is a c-essential, orientable properly embedded surface and that $(M,T)$ is irreducible. If $Q$ has punctures, also assume that $(M,T)$ is 2-irreducible. The graph $T$ may be weighted, although we will not make use of weights (other than 1) in the remainder of the paper.
\end{assumption}

\begin{definition}
A \defn{$Q$-disc} for a surface $F \subset (M,T)$ is a c-disc for $F$ that is disjoint from $Q$.
\end{definition}

\subsection{$Q$-thinning}

\begin{definition}\label{def: Q-seq}
Suppose that $Q \subset (M,T)$ and $\mc{H} \in \H(Q)$. An \defn{elementary thinning sequence relative to $Q$} (henceforth \defn{$Q$-sequence}) on $\mc{H}$ consists of the following in any order:
\begin{itemize}
    \item an elementary thinning move on $\mc{H}$ using c-discs $D_\up$ and $D_\down$ that are disjoint from $Q$;
    \item consolidating adjacent thick and thin surfaces bounding a punctured product VPC in $(M,T) \setminus \mc{H}$;
    \item an isotopy of $\mc{H}$ to remove all components of $Q \setminus \mc{H}$ with boundary that are parallel to subsurfaces of $\mc{H}$. 
\end{itemize}
Observe that the resulting multiple bridge surface lies in $\H(Q)$. 
\end{definition}

\begin{figure}[ht!]
\labellist
\small\hair 2pt
\pinlabel {$H$} [l] at 327 509
\pinlabel {$D_\up$} [b] at 101 567
\pinlabel {$D_\down$} [tl] at 165 470
\pinlabel {$Q$} [b] at 51 542
\pinlabel {untelescoping} [b] at 332 564
\pinlabel {consolidation} [t] at 329 315
\pinlabel {consolidation} [tl] at 332 29
\endlabellist
\centering
\includegraphics[scale=0.5]{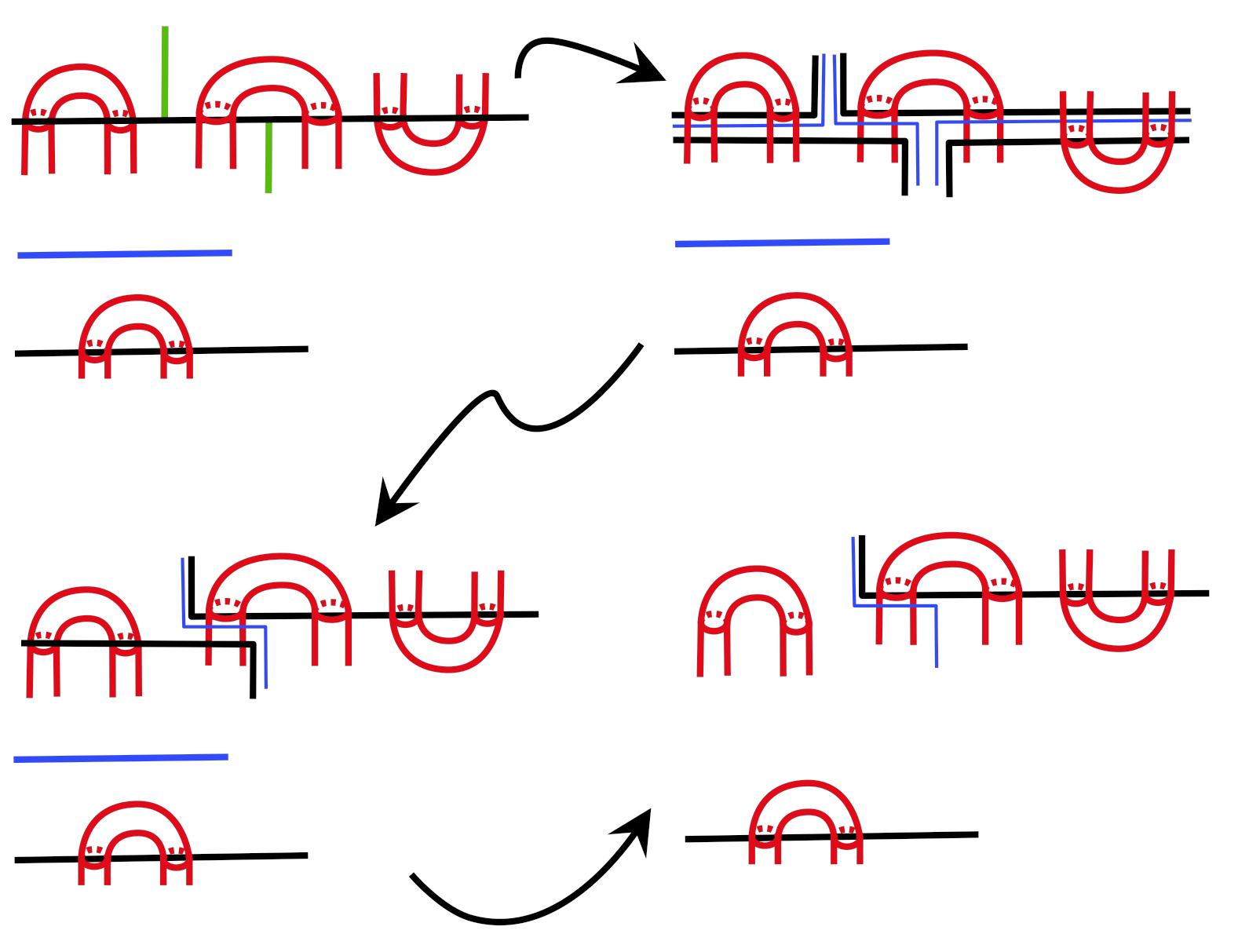}
\caption{We depict two steps of a $Q$-sequence (Definition \ref{def: Q-seq}). The top arrow shows a $Q$-weak reduction using discs $D_\down$ and $D_\up$, both disjoint from $Q$; the middle arrow shows the consolidation of new thin surfaces with new thick surfaces. Those two steps together form an elementary thinning move. the bottom arrow shows the consolidation of a new thick surface and an old thin surface. Throughout we show bits of the surface $Q$. Not shown is the step where we peform additional isotopies of $\mc{H}$ (if possible) to remove components of $Q\setminus \mc{H}$ parallel to subsurfaces of $\mc{H}$. }
\label{fig: Qthin}
\end{figure}

The next lemma follows immediately from Theorem \ref{lem:Thinning invariance}.
\begin{lemma}
For any $\mc{H} \in \H(Q)$, there is no infinite sequence of $Q$-sequences.
\end{lemma}

\begin{definition}
 If $\mc{H} \in \H(Q)$ cannot be thinned using $Q$-thinning moves, then $\mc{H}$ is \defn{$Q$-locally thin}. On the other hand, it is \defn{$Q$-weakly reducible} if there are disjoint $Q$-discs with boundaries on the same thick surface but on opposite sides.
\end{definition}

\begin{lemma}\label{Q disc exist}
Suppose that $\mc{H} \in \H(Q)$, with $Q$ an unpunctured torus or annulus, or four-punctured sphere. Assume also that for each component $H \cpt \mc{H}^+$ and each component $Q_0 \cpt Q\setminus H$ with $\boundary Q_0 \subset H$ and $\boundary Q_0 \neq \nil$, there is no isotopy of $Q_0$ to a subsurface of $H$. Then for every $H \cpt \mc{H}^+$, if $H$ has a c-disc on a particular side, it has a $Q$-disc on that side. Furthermore, if $(C, T_C)$ is a VPC without a $Q$-disc for $\boundary_+ C$, then $(C, T_C)$ is a punctured product compressionbody or a punctured trivial ball compressionbody and each component of $Q \cap C$ is a vertical annulus with one end in $\boundary M \cap \boundary_- C$.
\end{lemma}

\begin{proof}
Suppose that $(C, T_C) \cpt (M,T)\setminus \mc{H}$ and $H= \boundary_+ C$. Note, with our hypotheses on $Q$, each component of $Q \cap (C, T_C)$ is an unpunctured annulus or a twice-punctured disc. A twice-punctured disc component of $Q \cap (C, T_C)$ must have its boundary on $H$, as otherwise it could either be isotoped out of $(C, T_C)$ or be $\boundary$-compressed to create a c-disc for $\boundary_- C$ in $(C, T_C)$, which is not possible.

Recall from Lemma \ref{disc existence}, that if $H$ does not have a c-disc in $(C, T_C)$, then $(C, T_C)$ is either a punctured product between $H$ and a component of $\boundary M$ or is a punctured trivial ball VPC. In either case, by Assumption \ref{subsurf amalg assump}, any unpunctured sphere in $\boundary_- C$ bounds a 3-ball in $M$ and, if $Q$ has punctures, every twice punctured sphere in $\boundary_- C$ bounds a trivial ball compressionbody in $M$. Consequently, any component of $Q \cap C$ disjoint from $\boundary M$ can be isotoped out of $M$, a contradiction. We conclude that if $H$ does not have a c-disc in $(C, T_C)$, then every component of $Q \cap (C, T_C)$ is a vertical annulus with one end on $\boundary M \cap \boundary_- C$.

Suppose, therefore,  that $(C, T_C)$ contains a c-disc. Choose one such $D$ which intersects $Q \cap C$ minimally. If $D \cap Q =\nil$ it is a $Q$-disc. Assume, therefore, that it intersects $Q$. It is disjoint from all vertical annuli of $Q \cap (C, T_C)$ and each component of $Q \cap D$ is an arc. Let $\zeta$ be an outermost such arc, bounding an unpunctured outermost disc $E \subset D$. A $\boundary$-compression using $E$ of the component of $Q \cap C$ containing $\zeta$ produces a $Q$-disc, or else we could again isotope that component of $Q \cap C$ out of $C$, contradicting our hypothesis.
\end{proof}

\begin{remark}
Observe that if $\mc{H} \in \H(Q)$ is $Q$-locally thin, it may or not be locally thin. However, if $\mc{H} \in \H(Q)$ is locally thin, then it can be isotoped to be $Q$-locally thin.   
\end{remark}

\subsection{Subsurface Amalgamation}

In \cite{Schultens}, Schubert's theorem is proven by keeping track of the structure of maxima and minima through certain isotopies of tangles arising from the interaction between the companion torus and a height function. Some of these isotopies are shown in Figure \ref{fig: maxminmove}. In what follows, we develop a version of those isotopies that is phrased in terms of multiple bridge surfaces. Since our focus is on the structure of discs and annuli in compressionbodies, there is hope of an algorithmic implementation of our approach. (See \cite[Section 5]{Lackenby} for inspiration.)

\begin{figure}[ht!]
\centering
\includegraphics[scale=0.35]{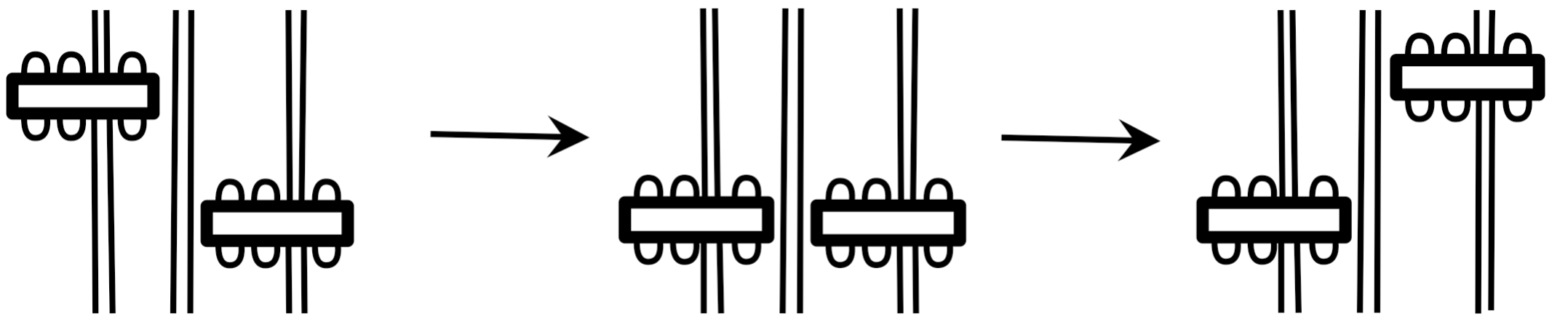}
\caption{The boxes represent arbitrary braids. Some isotopies of tangles that rearrange the order of maxima and minima but preserve the total number of maxima. Subsequently we develop moves on multiple bridge surfaces corresponding to these isotopies. Our moves are defined via amalgamation and untelescoping and thus are appropriate even in the absence of a height function. Compare to Figures \ref{fig: insulating}, \ref{subsurfaceamalg}, and \ref{fig:interchanging} in order. }
\label{fig: maxminmove}
\end{figure}

In this subsection, we define a new way of combining certain thick surfaces into a new thick surface. Throughout, unless otherwise specified, we assume only:
\begin{assumption}
Assume $Q \subset (M,T)$ is a separating surface with $V \cpt M\setminus Q$ designated as the ``inside'' and the other component designated as the outside. 
\end{assumption}

\begin{definition}\label{long annulus}
Suppose that $\mc{H} \in \H(Q)$. A \defn{long annulus} $\mc{A}$ is the union of unpunctured annular components of $Q \setminus \mc{H}$ such that they can be indexed as:
\[
\mc{A} = A_0 \cup A_1 \cup \cdots \cup A_n \cup A_{n+1}
\]
with $A_i \cap A_{i+1}$ a curve $\gamma_{i+1} \cpt Q \cap \mc{H}$ for all $i \in \{0, \hdots, n\}$. We let $\gamma_0 = \boundary A_0 \setminus \gamma_1$ and $\gamma_{n+2} = \boundary A_{n+1} \setminus \gamma_{n+1}$. A long annulus is \defn{vertical} if all of its annuli are vertical. 
\end{definition}
\begin{remark}
Throughout the rest of the paper we will consider long annuli in numerous contexts. We will always adopt the notation convention (or a small variation) of Definition \ref{long annulus} with regard to the numbering of the annuli and the curves of intersection with $\mc{H}$. In most cases, $\gamma_0 \neq \gamma_{n+2}$. 
\end{remark}

\begin{definition}\label{insulating}
Consider VPCs $(C_1, T_1), (C_2, T_2) \cpt (M,T)\setminus \mc{H}$ such that there is a component $F_2 \cpt \boundary_- C_1 \cap \boundary_- C_2$. Let $H_1 = \boundary_+ C_1$ and $H_3 = \boundary_+ C_2$. An \defn{insulating set} $\mc{A}$ is the union of vertical long annuli $\mc{A}_1, \hdots, \mc{A}_m \subset Q \cap (C_1 \cup C_2)$, such that for $i = 1,3$ $\mc{A} \cap H_i$ separates $H_i$ and each component of $H_i \setminus \mc{A}$ is on the same side (inside or outside) of each component of $\mc{A} \cap H_i$ to which it is adjacent and similarly for the components of $F \setminus \mc{A}$. In particular, $\mc{A} \cap H_i$ separates $H_i$ and $\mc{A} \cap F$ separates $F$. Observe that this means that in each of $H_1$, $H_3$, and $F$, we can refer to the inside or outside of $\mc{A}$. See Figure \ref{fig: insulating}. In this paper, our insulating sets will have at most two components.

In our definition of ``insulating set'', each $\mc{A}_i$ was the union of two vertical annuli (by construction). An \defn{extended insulating set} is a collection of vertical long annuli $\mc{A} \subset Q$ whose intersection with any  pair of VPCs that are adjacent along a thin surface intersecting $\mc{A}$, is an insulating set. The vertical long annuli in an extended insulating set may contain more than two vertical annuli.
\end{definition}

\begin{figure}[ht]
\labellist
\small\hair 2pt
\pinlabel {$H_3$} [l] at 375 209
\pinlabel {$F_2$} [l] at 375 129
\pinlabel {$H_1$} [l] at 375 48
\pinlabel {$\Delta_0$} [tl] at 94 32
\pinlabel {$\Delta_1$} [b] at 61 70
\pinlabel {$\Delta_2$} [tl] at 213 192
\pinlabel {$\Delta_3$} [bl] at 196 225
\pinlabel {$C_0$} [tl] at 327 32
\pinlabel {$C_1$} [tl] at 327 86
\pinlabel {$C_2$} [tl] at 327 168
\pinlabel {$C_3$} [tl] at 327 233
\pinlabel {$\mc{A}$} [l] at 125 91
\pinlabel {$\mc{A}$} [r] at 279 91
\endlabellist
\centering
\includegraphics[scale=0.6]{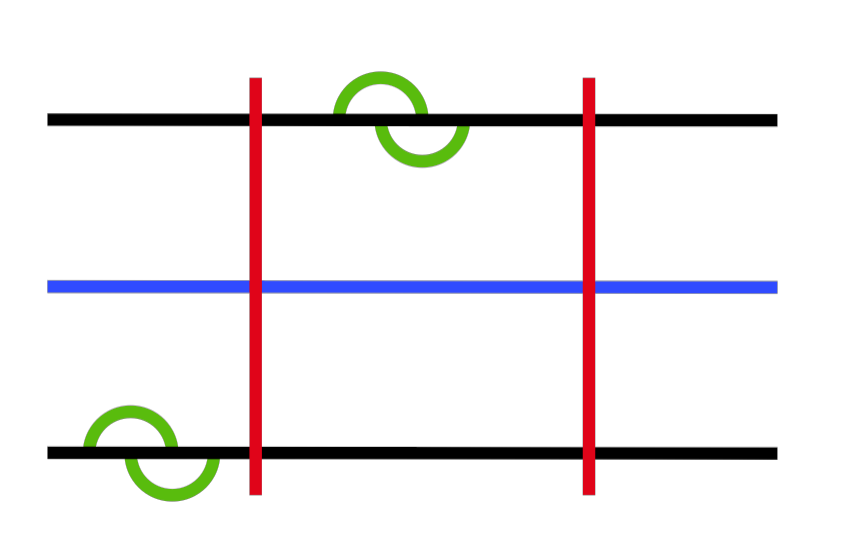}
\caption{Depiction of an insulating set (Definition \ref{insulating}) and the discs $\Delta_i$ in the definition of mergable (Definition \ref{mergeable def}).}
\label{fig: insulating}
\end{figure}

Let $(C_0, T_0)$ and $(C_3, T_3)$ be the components of $(M,T)\setminus \mc{H}$ on the opposite sides of $H_1$ and $H_3$ from $(C_1, T_1)$ and $(C_2, T_2)$ respectively. Let $\Delta_i \subset (C_i, T_i)$ be a complete set of sc-discs transverse to $Q$. We say that $\Delta_j$ and $\Delta_k$ are \defn{on opposite sides} of an insulating set $\mc{A}$ if $\boundary \Delta_j$ is on the inside of $\mc{A}$ while $\boundary \Delta_k$ is on the outside, or vice-versa. Otherwise, we say that they are \defn{on the same side}.

\begin{definition}\label{mergeable def}
As in Figure \ref{fig: insulating}, we say that $H_1$ and $H_3$ are \defn{mergable} via an insulating set $\mc{A}$ if the following conditions hold:
\begin{enumerate}
\item[(M1)] $\Delta_{1}$ and $\Delta_{2}$ are on opposite sides of $\mc{A}$
\item[(M2)] There is a $Q$-disc $D_0$ for $H_1$ in $(C_0, T_0)$, with $\boundary D_0$ contained completely on the same side of $\mc{A}$ as $\Delta_2$.
\item[(M3)] There is a $Q$-disc $D_3$ for $H_3$ in $(C_3, T_3)$,with $\boundary D_3$ contained completely on the same side of $\mc{A}$ as $\Delta_2$.
\end{enumerate}
\end{definition}

\begin{figure}[]
\labellist
\small\hair 2pt
\pinlabel {$H_3$} [l] at 392 205
\pinlabel {$F_2$} [l] at 392 123
\pinlabel {$H_1$} [t] at 221 40
\pinlabel {$\Delta_3$} [bl] at 217 214
\pinlabel {$\Delta_2$} [tl] at 232 190
\pinlabel {$\Delta_1$} [b] at 60 64
\pinlabel {$\Delta_0$} [tl] at 91 28
\pinlabel {$J$} [l] at 326 83
\pinlabel {$\mc{A}$} [l] at 144 83
\pinlabel {$\mc{A}$} [r] at 296 83
\endlabellist
\centering
\includegraphics[scale=0.6]{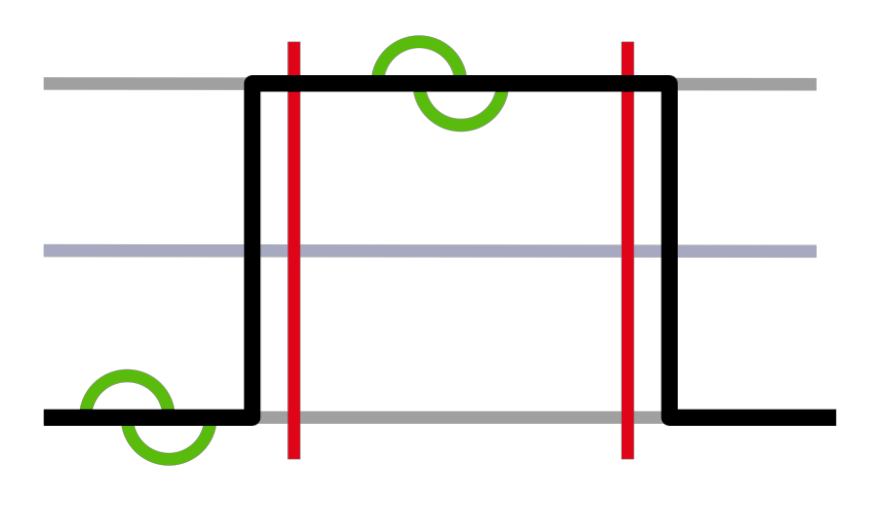}
\caption{Depiction of subsurface amalgamation, continued from previous figure.}
\label{subsurfaceamalg}
\end{figure}

Assuming (M1), (M2), (M3) hold we can create a new thick surface $J$ as follows. Let $\Delta_0$, $\Delta_3$ be complete set of discs for $(C_0, T_0)$ and $(C_3, T_3)$ respectively which contain $D_0$ and $D_3$ and, subject to that, intersect $Q$ minimally. Without loss of generality, we may assume that $\Delta_2$ lies on the inside of $\mc{A}$. Let $\mc{A}'$ be a parallel copy of the annuli $\mc{A}$ pushed slightly to the outside of $\mc{A}$. Let $X_1 \subset H_1 \subset \mc{A}'$ be the component not containing $\boundary \Delta_{1}$. Let $Y_1$ be the other component. Let $X_{3} \cpt H_3 \setminus \mc{A}'$ be the component containing $\boundary \Delta_{2}$. Let $Y_{3}$ be the other component. Let $J = Y_1 \cup \mc{A}' \cup X_3$. Let
\[
\mc{J} = (\mc{H} \setminus (H_1 \cup F_{2} \cup H_{3})) \cup J.
\]
We say that $\mc{J}$ is obtained by \defn{subsurface amalgamation} of $\mc{H}$. Refer to Figure \ref{subsurfaceamalg}; note that it is similar to the isotopy of height functions represented by the first arrow of Figure \ref{fig: maxminmove}. We now show that $\mc{J} \in \H(Q)$.

\begin{lemma}\label{merging}
If $\mc{H} \in \H(Q)$, then $\mc{J} \in \H(Q)$. Furthermore, $\netx_m$, $\netw$, $\netg$, and $\netchi$ are the same for $\mc{H}$ and $\mc{J}$.
\end{lemma}
\begin{proof}
Without loss of generality, we may assume that $\Delta_2$ is on the inside of $\mc{A}$. This assumption is used only to make a choice of how the annuli $\mc{A}$ intersect $\mc{J}$ and for ease of exposition. By our minimality assumption on $|\Delta_i \cap Q|$, the discs $\Delta_1$ and $\Delta_2$ are disjoint from $\mc{A}$. The situation is symmetric, so we show that the component $(E, T_E)$ of $(M,T)\setminus \mc{J}$ with $J \cpt \boundary E$ and containing $D_3$ is a VPC.

Compress $J$ using $\Delta_1$ to create a surface $J'$ containing the scars $\delta_1$ from the compressions. Since $\Delta_1$ is on the outside of $\mc{A}$, the components of $J' \setminus \mc{A}$ are either spheres bounding trivial VPCs or are parallel to the outside subsurface $H'_3$ of $H_3\setminus \mc{A}$. Extend $\boundary \Delta_3 \cap H'_3$ through the product structure to lie on $J'$ but missing the surgery scars. This converts $\Delta_3$ into a collection of sc-discs $\Delta'_3$ for $J'$,. Undoing the compressions, makes $\Delta_1 \cup \Delta'_3$ into a collection of sc-discs for $J$. Using them to compress $J$, produces the trivial VPCs from the compression of the outside subsurface of $H_1$ using $\Delta_1$ and components isotopic to the trivial VPCs obtained by compressing $H_3$ using $\Delta_3$. Consequently, $(E, T_E)$ is a VPC.

By construction, each component of $Q \cap \mc{J}$ is essential in both surfaces, since all such components correspond to components of $Q \cap \mc{H}$. If some component of $Q \setminus \mc{J}$ were parallel to a subsurface of $\mc{J}$, it would contain a component of $Q \setminus \mc{H}$ that was parallel to a subsurface of $\mc{H}$, contrary to hypothesis. Thus, $\mc{J} \in \H(Q)$.

The second claim follows easily from considering the Euler characteristics and weights of the subsurfaces involved and using the fact that $\mc{A}$ is the union of annuli.
\end{proof}

We can then go one step further, and, in effect, swap the original thick surfaces $H_1$ and $H_3$. See Figure \ref{fig:interchanging}. Note that it is analogous to the isotopy of height functions represented by the second arrow in Figure \ref{fig: maxminmove}.

\begin{figure}[ht!]
\labellist
\small\hair 2pt
\pinlabel {$H_3$} [l] at 528 248
\pinlabel {$F_2$} [l] at 528 167
\pinlabel {$H_1$} [t] at 304 84
\pinlabel {$K_3$} [b] at 304 313
\pinlabel {$K_0$} [l] at 528 46
\pinlabel {$F'$} [l] at 528 68
\pinlabel {$\mc{A}$} [l] at 233 123
\pinlabel {$\mc{A}$} [r] at 385 123
\pinlabel {$\Delta_3$} [bl] at 318 264
\pinlabel {$\Delta_2$} [tl] at 341 231
\pinlabel {$\Delta_1$} [b] at 87 111
\pinlabel {$\Delta_0$} [tl] at 117 72
\endlabellist
\centering
\includegraphics[scale=0.5]{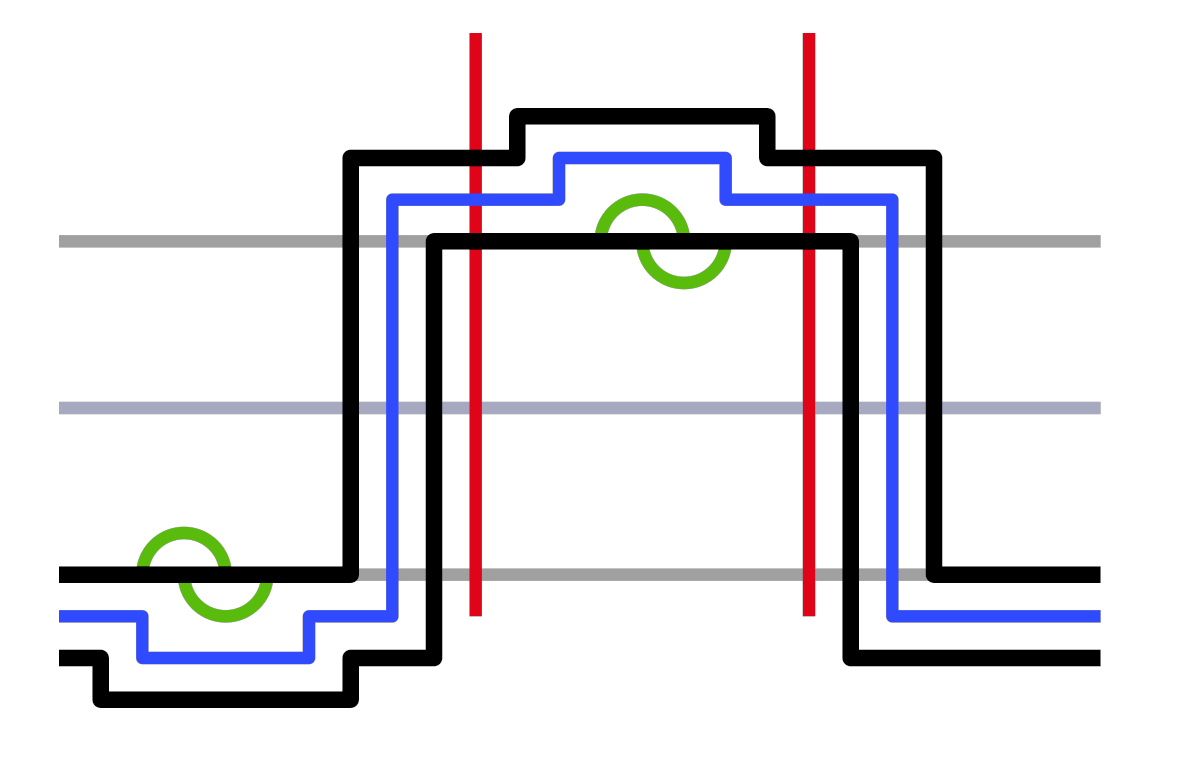}
\caption{Interchanging two $Q$-adjacent thick surfaces. We will assume $\Delta_1$ and $\Delta_0$ are outside $Q$.} 
\label{fig:interchanging}
\end{figure}

\begin{definition}\label{preliminterchange}
Let $\mc{H}, H_1, H_3, F, D_0, D_3$ be as above. Suppose that $H_1$ and $H_3$ are mergable via an insulating set $\mc{A}$ and let $J$ be the thick surface that results from subsurface amalgamation. Let $\mc{K}$ be the result of an elementary thinning move on $\mc{J}$ using the discs $D_0$ and $D_3$ (from conditions (M2) and (M3)). Let $K_1$ and $K_3$ be the new thick surfaces and $F'$ the new thin surface, so that $K_1$ is on the same side of $F'$ as $H_1$ is of $F_2$. We say that $\mc{K} = \mc{K}_0$ is obtained from $\mc{H}$ by \defn{a preliminary interchange} of the thick levels $H_1$ and $H_3$.
\end{definition}

Recall from Lemma \ref{discs persist} that sc-discs persist through elementary thinning sequences. (See Figures \ref{fig:untel} and \ref{fig: pos error}.) In particular, if $K_0$ admits a $Q$-disc $D'_0$ on the same side as $D_0$, and on the opposite side of $\mc{A}$ from $D_3$, then we may again weakly reduce using $D_3 \cup D'_0$ to obtain $\mc{K}_1$. As in Lemma \ref{discs persist}, which we will rely on, there will be consolidations that occur when we keep using the same disc for untelescoping. In Schultens' isotopy of height functions described in Figure \ref{fig: maxminmove}, this amounts to the difference of raising and lowering maxima and minimal one at a time versus all at the same time. In our context, it is easier to keep track of the bookkeeping if we use one weak reducing pair of discs at a time, versus using many discs at once.

We may perform a similar weak reduction and consolidation if $K_3$ has an sc-disk $D'_3$ on the same side of both $\mc{A}$ and $K_3$ as $D_3$ and disjoint from $Q$. We may in this way construct a sequence:
\[
\mc{K}_0, \mc{K}_1, \hdots, \mc{K}_n
\]
such that the following hold:
\begin{enumerate}
    \item $\mc{K}_{i+1}$ is obtained from $\mc{K}_i$ by a $Q$-move. The untelescoping is done using a pair of discs originally corresponding to an sc-disc for $K_3$ lying above $K_3$ and an sc-disc for $K_1$ lying below $K_1$. For $i < n-1$, the isotopy to eliminate parallel components of $Q \setminus \mc{K}_{i+1}$ only eliminates components of $Q \setminus \mc{K}_{i+1}$ that do not share a boundary component with $\mc{A}$.
    \item Any path from one VPC to another that is disjoint from the $Q$-discs and intersects $\mc{K}$ with consistent orientations will intersect $\mc{K}^+_i$ at most the same number of times it intersects $\mc{K}^+$. (Lemma \ref{discs persist})
    \item Either $\mc{K}_n$ is $Q$-locally thin or after the untelescoping and consolidation of $\mc{K}_{n-1}$ some component of $Q \setminus \mc{K}_n$ sharing a boundary component with $\boundary \mc{A}$ is parallel to a subsurface of $\mc{K}_n$. The completion of the $Q$-move eliminates all such components.
\end{enumerate}

\begin{definition}\label{def: interchange}
If $\mc{K}_n$ is as described above, we say that $\mc{K}_n$ is obtained from $\mc{H}$ by \defn{interchanging} the thick surfaces $H_i$ and $H_{i+1}$.
\end{definition}

\begin{lemma}\label{interchange result}
Suppose $\mc{H} \in \H(Q)$ and that $\mc{K}$ is obtained by interchanging thick surfaces, as above. Then the following hold:
\begin{enumerate}
     \item $\netchi$, $\netx_m$, $\netw$, $\netg$ are the same for $\mc{H}$ as for $\mc{K}$,
    \item No adjacent thick and thin surfaces of $\mc{K}$ cobound a punctured product compressionbody,
    \item Either $\mc{K}$ is $Q$-locally thin or at the final step of the interchange, a $Q$-move eliminated one or more bridge annuli sharing an end with $\boundary \mc{A}$.
    
\end{enumerate}
\end{lemma}
\begin{proof}
By construction the surfaces $K_3$, $K_0$, and $F$ are obtained by compressing $J \cpt \mc{J}^+$ along $Q$-discs. Thus, each component of $\mc{K} \cap Q$ is essential in $Q$. Since $Q$ is c-essential, each component is also essential in $\mc{K}$. Since we explicitly eliminated all parallelisms between subsurfaces of $Q$ and $\mc{K}$, we have $\mc{K} \in \H(Q)$.

Each compression along a disc on the positive side or negative side of $J$, contributes the same amount to both $\chi(K_i)$ and $\chi(F')$ and to $\x(K_i)$ and $\x(F')$ for $i = 0$ or $i = 3$. Eliminating parallel copies of $\mc{K}^+$ and $\mc{K}^-$ does not change $\netchi$ or $\netx$. Thus, (2) follows from Lemma \ref{merging}.

We recall from Lemma \ref{discs persist} how the discs used in an untelescoping operation persist to the new thick surfaces. The sequence terminates when we lose the ability to keep using $\mc{A}$ as an extended insulating set. Hence, (3) holds.
\end{proof}

\section{Matched Pairs}\label{sec:nested}
We now turn to the task of trying to arrange that $Q$ sits nicely with respect to a multiple bridge surface.

\begin{assumption}
    The pair $(M,T)$ is irreducible and standard, the surface $Q \subset (M,T)$ is c-essential and separating. The component  $V \cpt M\setminus Q$ is the \defn{inside} of $Q$. We have $\mc{H} \in \H(Q)$. If $Q$ is punctured, also assume that $(M,T)$ is 2-irreducible.
\end{assumption}

\begin{definition}
Every bridge annulus $A \cpt Q \setminus \mc{H}$ has a boundary compressing disk $E$ in the VPC containing it. If near the arc $\bdd E \cap Q$ the disk is contained inside $A$, we will call the annulus \defn{nested}. If a bridge annulus is not nested, we call it \defn{curved}.  See Figure \ref{fig: nestedcurved}. All curved annuli $A$ have a $\boundary$-compressing disc $E$ such that $E$ is outside $A$; but it is also possible for a nested annulus to have such a $\boundary$-compressing disc. However, by definition of ``curved,'' no annulus is both curved and nested.
\end{definition}

\begin{figure}
    \centering
    \includegraphics[scale=0.5]{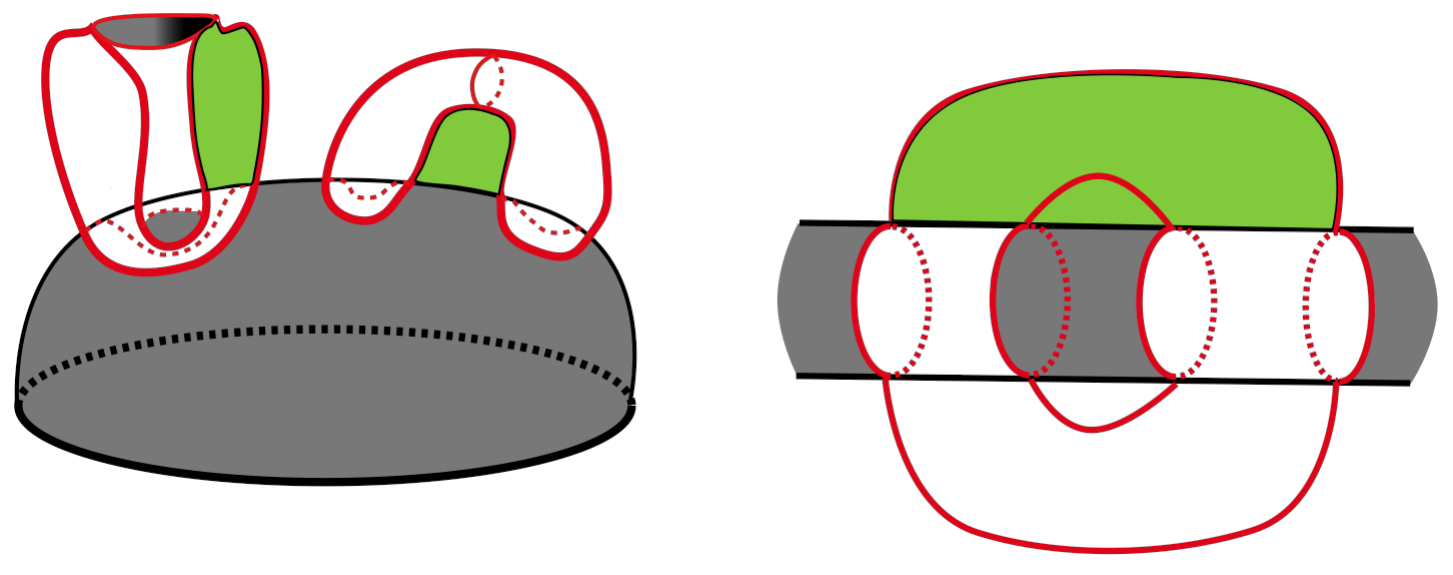}
    \caption{On both the left and the right we depict a nested annulus and a curved annulus. Which is which depends on whether the portion of the thick surface shaded gray is inside the ends of the annuli or outside the ends of the annuli. For each annulus, we depict a $\boundary$-compressing disc in green. The statement that one of each pair of annuli is curved does require the assumption that there is no $\boundary$-compressing disc on the side opposite the $\boundary$-compression shown. Assuming that the curves of intersection between $Q$ and $\mc{H}$ in the example on the left are separating in $\mc{H}$, both examples of matched pairs are cancellable in the sense of Definition \ref{def: cancellable} below.}
    \label{fig: nestedcurved}
\end{figure}

Anticipating future work, we extend the definition also to twice-punctured discs:

\begin{definition}
Every twice-punctured disc $D \cpt Q \setminus \mc{H}$ has a boundary compressing disk $E$ in the VPC containing it. If near the arc $\bdd E \cap D$ the disk is contained inside $D$, we will call the disc $D$ \defn{nested}. If $D$ is not nested, we call it \defn{curved}. 
\end{definition}

\begin{lemma}\label{insideoutside}
Suppose that $\mc{H} \in \H(Q)$ and that $Q$ is an unpunctured torus or annulus or a four-times punctured sphere. Suppose that $A \cpt Q \setminus \mc{H}$ is a bridge annulus or twice-punctured disc contained in a VPC $(C, T_C) \cpt (M,T)\setminus \mc{H}$. If $A$ is a curved (resp. nested), then there is a $Q$-compressing disk in $(C, T_C)$ for $\boundary_+ C$ outside (resp. inside) of $A$. 
\end{lemma}

If $\boundary A$ does not separate the component of $\mc{H}^+$ containing it, then the disc may be both inside and outside $A$, in which case it may or may not be helpful in what follows, depending on the specific circumstances. See Figure \ref{fig:qdiscconstruction} for an example of how $Q$-compressing disk may be constructed from an annulus component of $Q \setminus \mc{H}$.

\begin{figure}[ht]
\centering
\includegraphics[scale=0.3]{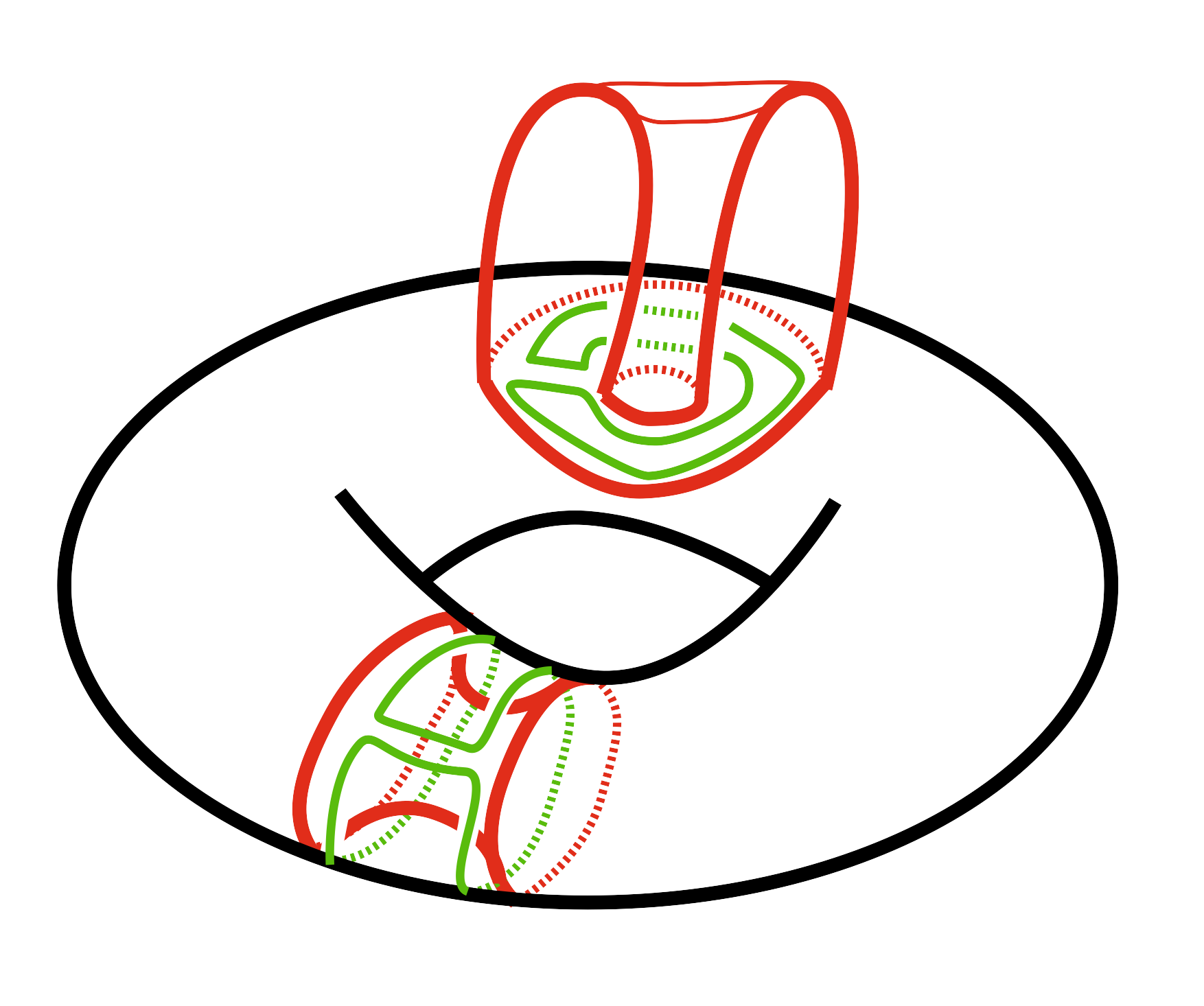}
\caption{An example of two annuli components (in red) of $Q \setminus \mc{H}$. Each has a $\boundary$-compressing disc which (in Lemma \ref{insideoutside}) we may use to $\boundary$-compress the annulus to a disc. We show the boundary of the discs in green. If the $\boundary$-compressing disc intersects some other component of $Q\setminus \mc{H}$, we pass to a component intersecting the $\boundary$-compressing disc in an outermost arc.}
\label{fig:qdiscconstruction}
\end{figure}

\begin{proof}
The proof is very similar to that of Lemma \ref{Q disc exist}; simply use a $\boundary$-compressing disc for $A$ on the appropriate side in place of the disc $D$.
\end{proof}

\begin{definition}\label{matched pair def}
A \defn{matched pair} $(A, A')$ consists of a curved annulus or twice-punctured disc $A$ and a nested annulus or twice-punctured disc $A'$ such that there is a long vertical annulus $A_1 \cup \cdots \cup A_n$ with $\gamma_1$ an end of $A$ and $\gamma_{n+1}$ an end of $A'$. We call 
\[
\mc{A} = \underbrace{A_0}_A \cup A_1 \cup \cdots A_n \cup \underbrace{A_{n+1}}_{A'}
\]
a \defn{matching sequence}. Its \defn{length} $\ell(\mc{A})$ is equal to $n$. We call $A_1 \cup \cdots \cup A_n$ the \defn{vertical part} of the matching sequence. 
\end{definition}

See Figure \ref{fig: matched pair} for two examples of matched pairs.

\begin{notation}\label{Indexing convention}
For matching sequences, we will always adopt the notation of Definition \ref{long annulus} and insist that $A= A_0$ is curved and $A' = A_{n+1}$ is nested. We also let $(C_i, T_i)$ be the VPC containing $A_i$. If $(A, A')$ is  matched pair then for each odd $i$, $\gamma_i$ lies in a thick surface $H_i$. Also for each even $i$ such that $0 < i < n+2$, the curve $\gamma_{i}$ lies in a thin surface $F_i$. 
\end{notation}

\begin{figure}[ht!]
\labellist
\small\hair 2pt
\pinlabel {$A' = A_5$} [b] at 226 408
\pinlabel {$A = A_0$} [t] at 328 31
\pinlabel {$A_1$} [r] at 197 134
\pinlabel {$A_2$} [r] at 205 186
\pinlabel {$A_3$} [r] at 221 241
\pinlabel {$A_4$} [r] at 229 294
\pinlabel {$\gamma_0$} [bl] at 300 107
\pinlabel {$\gamma_1$} [bl] at 354 109
\pinlabel {$\gamma_2$} [bl] at 338 163
\pinlabel {$\gamma_3$} [bl] at 322 219
\pinlabel {$\gamma_4$} [bl] at 294 273
\pinlabel {$\gamma_5$} [bl] at 287 329
\pinlabel {$\gamma_6$} [br] at 146 325
\pinlabel {$A$} [t] at 668 238
\pinlabel {$A'$} [b] at 565 344
\endlabellist
\centering
\includegraphics[scale=0.4]{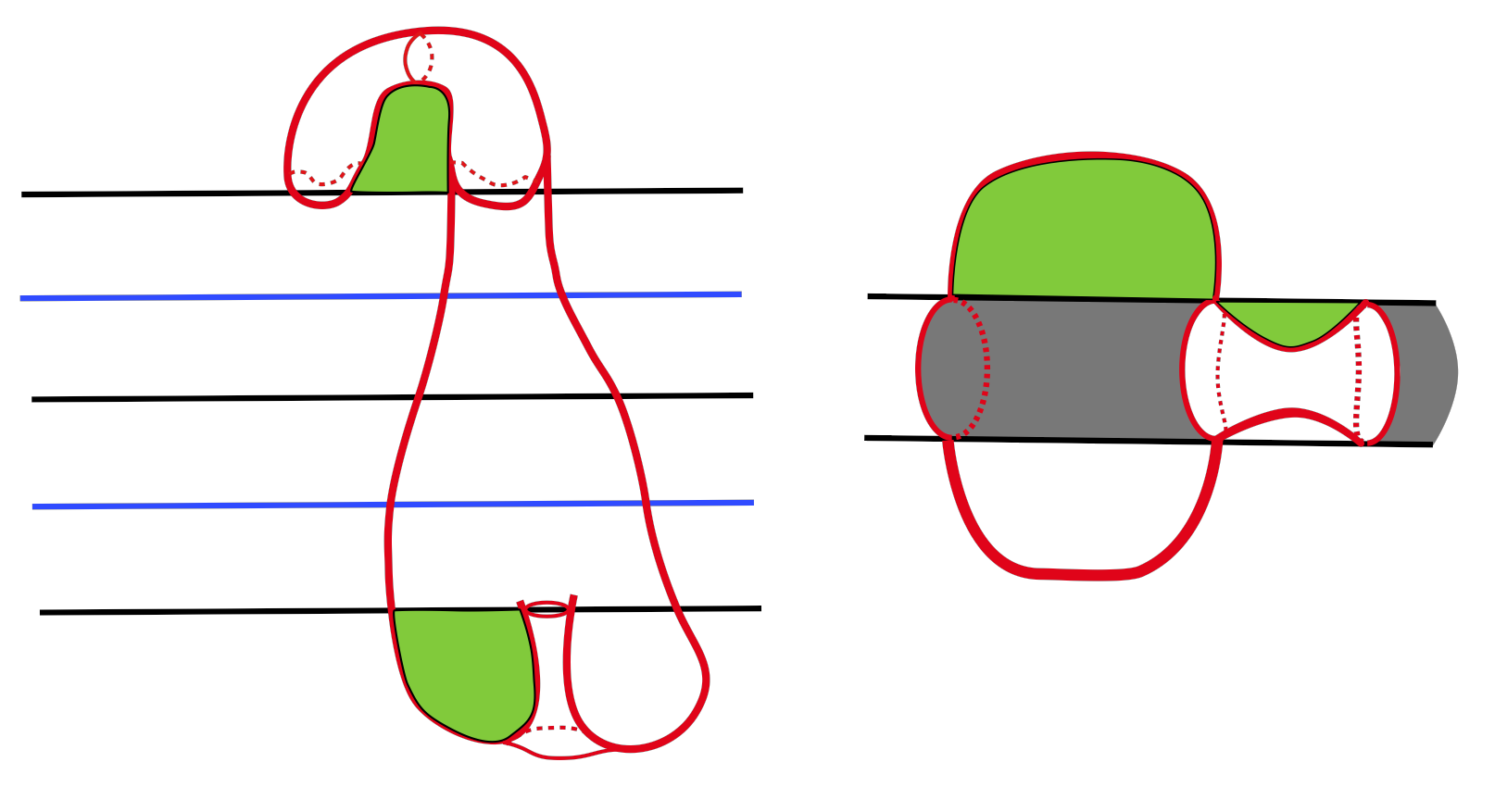}
\caption{Two examples of matched pairs (Definition \ref{matched pair def}). The example on the left has a matching sequence of length 4 and the one on the right's matching sequence has length zero. Boundary-compression discs for the matched pair are shown in green. To satisfy the definition of matched pair, assume that none of the bridge annuli have a boundary-compressing disc on the opposite side from the one indicated, although that requirement can be relaxed for the nested annuli. Thick surfaces are in black and thin surfaces in blue.}
\label{fig: matched pair}
\end{figure}

\begin{lemma}\label{matching props}
Suppose that $\mc{H} \in \H(Q)$ and that $(A, A')$ is a matched pair. The following hold:
\begin{enumerate}
\item The length of any matching sequence is always even.
\item Given a matching sequence $\mc{A}$ for $(A, A')$, each component of $\mc{H}$ is either disjoint from the vertical part of $\mc{A}$ or intersects it in a single simple closed curve.
\item If all components of $\mc{H}$ are separating, then any two matching sequences for $(A, A')$ have the same length.
\item If a matching sequence has length 0 and if $Q$ is an unpunctured or once-punctured torus or four-punctured sphere and if $\boundary A \cup \boundary A'$ separates $\mc{H}$, then $\mc{H}$ is $Q$-weakly reducible.
\end{enumerate}
\end{lemma}

\begin{proof}
To see that the length of a matching sequence is always even, recall that every annulus of $Q \setminus \mc{H}$ incident to a thin surface is always vertical, while nested and curved annuli have their boundaries on thick surfaces. There must, therefore, be an even number of vertical annuli in a matching sequence, making the length even.

Consider a matching sequence for $(A, A')$. Choose an oriented spanning arc $\alpha \subset A_1 \cup \cdots \cup A_{n}$ that is also a spanning arc for each $A_i$. As each annulus $A_i$ for $i = 1, \hdots, n-1$, is vertical, either the orientation of $\alpha$ coincides with the transverse orientation for each component of $\mc{H}$ it passes through or is opposite that orientation. By the definition of $\vpoH(M,T)$, this arc cannot pass through any component of $\mc{H}$ more than once. If all components of $\mc{H}$ are separating, then any two matching sequences for $(A, A')$ pass through the same thin surfaces and so have the same length. 

Finally, suppose that a matching sequence has length zero and that $Q$ is a separating unpunctured or once-punctured torus or four-punctured sphere. Let $V$ be a submanifold bounded by $Q$ and consider it to be on the inside of $Q$. Let $H \cpt \mc{H}^+$ contain $\boundary A \cup \boundary A'$. Since $A$ is curved, by Lemma \ref{insideoutside} there is a $Q$-disc for $H$ outside $\boundary A$. Similarly, there is a $Q$-disc for $H$ inside $\boundary A'$. If $\boundary A \cup \boundary A'$ separates $H$, these discs are disjoint and so $\mc{H}$ is $Q$-weakly reducible.
\end{proof}

\section{Eliminating Matched Pairs}\label{sec:Qcomplex}

\begin{assumption}
For the entirety of this section, assume that $Q \subset (M,T)$ is a c-essential unpunctured torus or annulus or four-punctured sphere, separating $M$ and bounding a submanifold $V$ which is the \defn{inside} of $Q$. \end{assumption}

Just as when applying Morse theory to knots, we can sometimes cancel a 1-dimensional 0-handle with a 1-dimensional 1-handle, so it is sometimes possible to cancel matched pairs. The next definitions lists some of the situations in which we can cancel a matched pair.

\begin{definition}\label{def: cancellable}
    A matched pair $(A, A')$ is \defn{cancellable} if one of the following holds:
    \begin{enumerate}
        \item $A, A'$ are both BSS annuli and all other annuli in a matching sequence are VSS.
        \item One of $A, A'$ is BNS and the other is BSS and all other annuli in a matching sequence are VSS
        \item Both of $A, A'$ are BNN annuli and, in addition to the matching sequence $\mc{A}$ for $(A, A')$ there is another long vertical annulus $\mc{B}$, disjoint from the interior of $\mc{A}$, with one end at $A$ or $A'$ and with $\mc{A} \cup \mc{B}$ separating each thick and thin surface it intersects. Additionally, whichever of $A$ or $A'$ is not incident to $\mc{B}$ has the property that it admits a bridge disc $E$ on the the appropriate side whose arc of intersection $\epsilon$ with $\mc{H}^+$ can be isotoped to have interior disjoint from $\mc{B} \cup (A \cup A')$. (``Appropriate side'' means on the inside if the annulus is nested and on the outside if it is curved.)
    \end{enumerate}
\end{definition}

The matched pairs in Figure \ref{fig: matched pair} are cancellable as is the matched pair on the left of Figure \ref{fig: cancellable}; however the matched pair on the right of Figure \ref{fig: cancellable} is not cancellable. The example on the right of Figure \ref{fig: matched pair} shows a cancellable matched pair satisfying (3) and of zero length. Figure \ref{fig: doubleinsulate} gives a schematic description of a cancellable matched pair of length 4 and satisfying (3). 

\begin{figure}[ht!]
\centering
\includegraphics[scale=0.5]{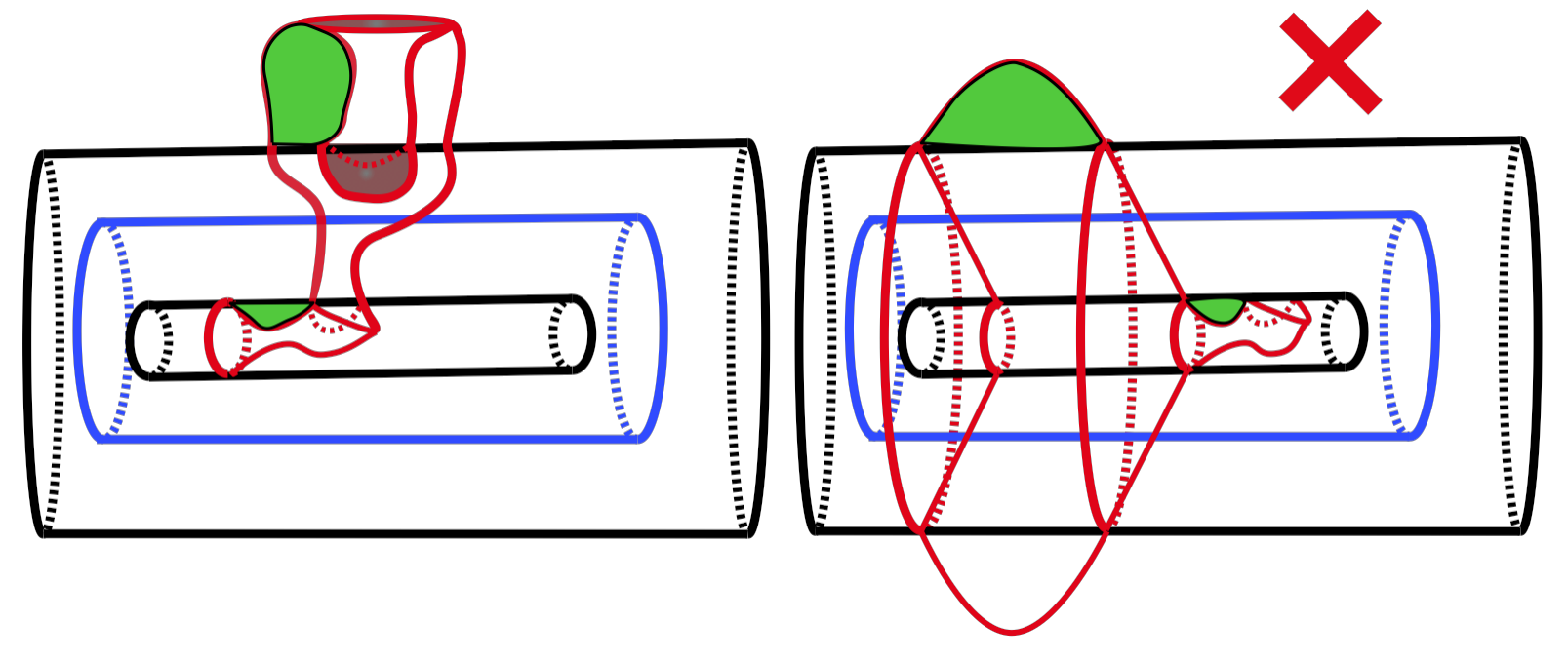}
\caption{Examples satisfying and not satisfying Definition \ref{def: cancellable}. On the left is an example of a cancellable matched pair where one annulus in the pair is a BNS and the other is a BSS. On the right is an example of a matched pair that is not cancellable: one of the pair is a BNN and the other is a BNS. The examples of matched pairs in Figure \ref{fig: matched pair} are also cancellable.}
\label{fig: cancellable}
\end{figure}

\begin{figure}[ht!]
\centering
\labellist
\small\hair 2pt
\pinlabel{$A'$} [tl] at 372 358
\pinlabel{$A$} [r] at 456 923
\pinlabel{$\mc{A}$} [r] at 413 692
\pinlabel{$\mc{B}$} [l] at 682 692
\pinlabel{$E$} at 333 372
\pinlabel{$\epsilon$} [b] at 325 415
\endlabellist
\includegraphics[scale=0.2]{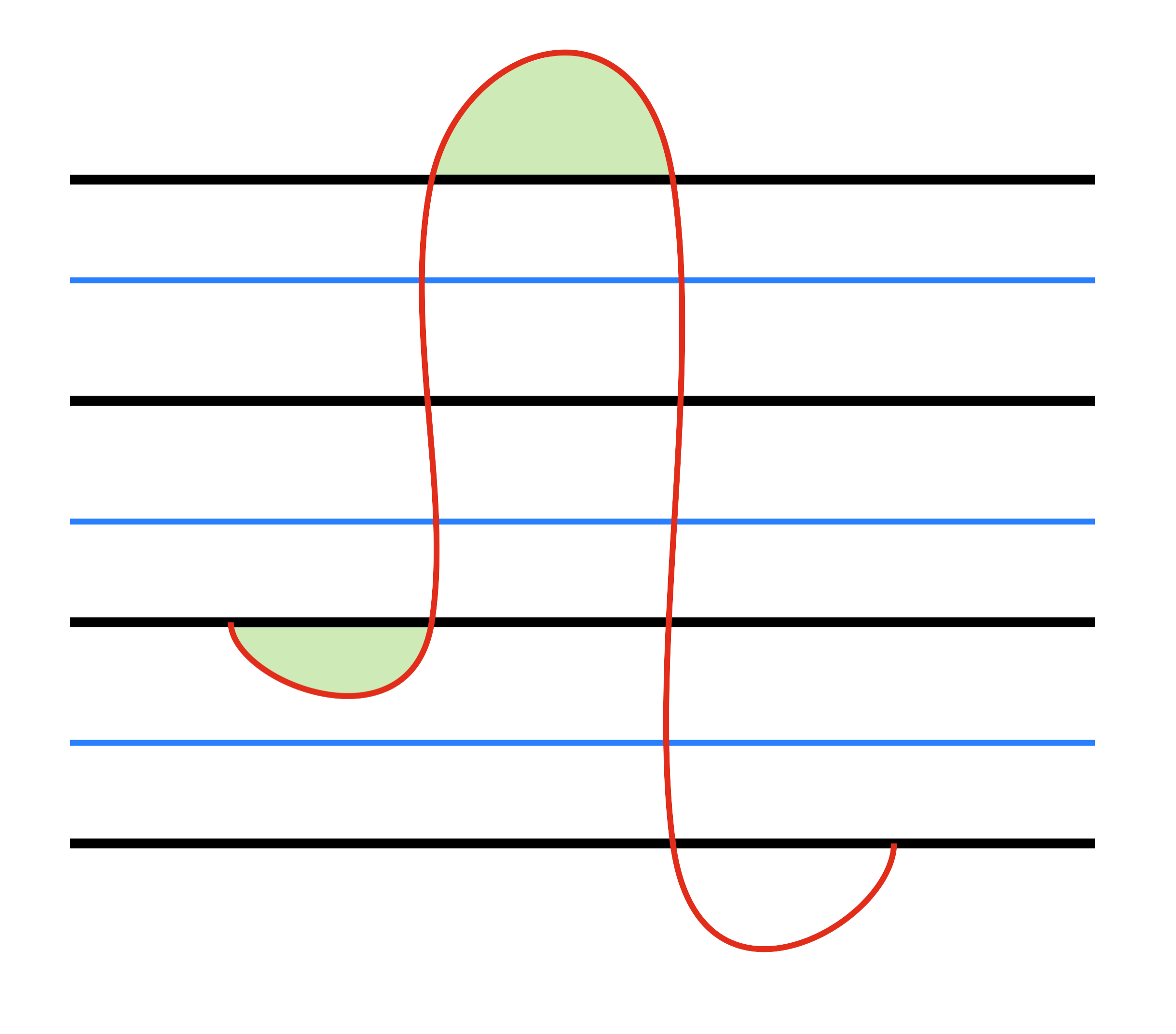}
\caption{A schematic description of a cancellable matched pair of length 4 and satisfying (3). The red curve represents a long annulus in $Q$. It contains two long vertical annuli $\mc{A}$ and $\mc{B}$. In this example, it is the annulus $A'$ which is disjoint from the long vertical annulus $\mc{B}$.  The bridge discs for the matched pair are shown in green; note they are on opposite sides of $\mc{A}$. The black horizontal lines represent thick surfaces and the blue horizontal lines represent thin surfaces. Together $\mc{A} \cup \mc{B}$ separate each thick and thin surface they intersect. In this example, the long vertical annulus $\mc{B}$ passes through the thick surface incident to $A'$; consequently, finding a bridge disc $E$ for $A'$ with arc $\epsilon$ disjoint from $\mc{A} \cup \mc{B}$ is a straightforward application of an innermost disc/outermost arc argument. If $\mc{B}$ were incident to a bridge annulus incident to the same thick surface, such a bridge disc may or may not exist; hence, the requirement in (3).}
\label{fig: doubleinsulate}
\end{figure}

The purpose of this section is to prove Proposition \ref{no cancellable matched pairs}, which guarantees that we can eliminate all cancellable matched pairs.

\begin{definition}
Let $N(\mc{H})$ denote the number of matched pairs in $Q$.
\end{definition}

\begin{lemma}\label{Q-moves and annuli}
Suppose that $\mc{H} \in \H(Q)$ and suppose that $\mc{H}$ $Q$-thins to $\mc{J}$, then $N(\mc{J}) \leq N(\mc{H})$.
\end{lemma}
\begin{proof}
Without loss of generality we may assume that we obtain $\mc{J}$ from $\mc{H}$ by untelescoping a thick surface $H$, consolidating (if possible) newly created thin and thick surfaces, then (if possible) consolidating some of the newly created thick surfaces with previously existing thin surfaces to create $\mc{J}'$, followed by an isotopy to remove components of $Q \setminus \mc{J}'$ that are parallel to a subsurface of $\mc{J}'$, arriving at $\mc{J}$. 

We start by showing that the number of bridge annuli does not increase. On $Q$, label each loop of $Q \cap \mc{H}$ with the component of $\mc{H}$ on which it lies. On $Q$, the untelescoping takes each loop labelled $H$ and replaces it with three parallel loops; two are labelled by new thick surfaces and they are separated by a loop labelled with a new thin surface. See Figure \ref{fig: Qthin} for a depiction. After untelescoping, the bridge annuli are in natural bijection with the bridge annuli before untelescoping. Consolidation removes adjacent loops one labelled by a thick surface and the other by a thin surface. Thus, again bridge annuli after the consolidations are in bijection with the bridge annuli before the consolidations. Finally, an isotopy eliminating a $\boundary$-parallel bridge annulus removes two adjacent loops labelled with the same thick surface and, possibly, two adjacent loops labelled with the same thin surface. Once again, we see that we do not create additional bridge annuli.

We now explain the finer result that neither the number of curved nor the number of nested annuli increases. Let $D_\up$ and $D_\down$ be the discs used in the untelescoping. Let $\mc{H}'$ be obtained from $\mc{H}$ by the untelescoping, so that $\mc{J}'$ is obtained from $\mc{H}'$ by consolidations. Consider a bridge disk $E$ for $A \cpt Q \setminus \mc{J}$.  Reversing the isotopies that create $\mc{J}$ from $\mc{J}'$, introduces intersections between $E$ and $\mc{J}'$. Since the components of $\mc{J}' \cap Q$ are essential in both surfaces, we may assume that $E \cap \mc{J}'$ consists of arcs. Reversing the consolidations may also introduce arcs of intersection between $E$ and $\mc{J}$. The bridge annulus $A'$ of $Q\setminus \mc{H}'$ corresponding to $A \cpt Q\setminus \mc{J}$ has a $\boundary$-compressing disc $E'$ that is a subdisc of $E$. Consequently, $E$ lies outside $A$ if and only if $E'$ lies outside $A'$. We reconstruct $\mc{H}$ from $\mc{H}'$ by adding tubes to thick surfaces of $\mc{H}'$ that are either dual to $D_\up$ or dual to $D_\down$ and then discarding the other components created by the untelescoping. These tubes are disjoint from $Q$, and so as a component of $Q\setminus \mc{H}$, $A'$ is nested if and only if $A$ is nested. Consequently, as we pass from $\mc{H}$ to $\mc{J}$, neither the number of curved or nested annuli increases nor does the number of matched pairs.
\end{proof}

We revisit the process of interchanging surfaces and prove a result which enables to control the number of matched pairs.

\begin{definition}
Suppose that $\mc{H} \in \H(Q)$ is $Q$-locally thin and that there is a matched pair $(A, A')$ with matching sequence
\[
\mc{A}: \underbrace{A_0}_A, \hdots, \underbrace{A_{n+1}}_{A'}.
\] 
Suppose that there is an extended insulating set $\ob{\mc{A}}$ containing the vertical part of $\mc{A}$. Recall that the curve $\gamma_i$ lies on a thick surface $H_i$ if $i$ is odd and a thin surface $F_i$ if $i$ is even (and not $0$ or $n+2$). If $H_i$ has a complete set of sc-discs $\Delta_1$, $\Delta_2$ on either side that are both disjoint from $\ob{\mc{A}}$ and have boundaries lying inside (resp. outside) $\ob{\mc{A}}$ we say that $H_i$ is \defn{mostly inside} (resp. \defn{mostly outside}) relative to $\ob{\mc{A}}$. 
\end{definition}

\begin{remark}
    The most natural setup is when the extended insulating set $\ob{\mc{A}}$ equals the vertical part of $\mc{A}$. However, we want to be able to eliminate cancellable matched pairs that satisfy the third condition of Definition \ref{def: cancellable}. To do that, we will need to allow $\ob{\mc{A}}$ to contain two disjoint vertical annuli; one of them is the vertical part of $\mc{A}$. That vertical annulus necessarily has both ends on the matched pair; the other component of $\mc{\ob{A}}$, will also have one of its ends at the matched pair. 
\end{remark}

\begin{proposition}\label{decreasing complexity}
Suppose that $\mc{H}$ is $Q$-locally thin, that $(A, A')$ is a matched pair with a matching sequence $\mc{A}$ of length at least 2. Suppose that for some odd $i$, $H_i$ and $H_{i+2}$ are mergeable using an insulating set $\ob{\mc{A}}$ containing $A_i \cup A_{i+1}$. Let $\mc{J}$ be obtained by interchanging them along $\ob{\mc{A}}$. Then:
\begin{enumerate}
    \item $N(\mc{J}) \leq N(\mc{H})$
    \item if $N(\mc{J}) = N(\mc{H})$, then $(A, A')$ persists as a matched pair for $\mc{J}$ and $\ell(A, A')$ has not increased. 
    \item if $(A, A')$ was cancellable and $\ob{\mc{A}}$ contains the union of the vertical annuli in Definition \ref{def: cancellable}, this remains true after the interchange.
\end{enumerate}
Furthermore, if the final step of interchanging $H_i$ and $H_{i+2}$ involves a $Q$-move that isotopes thick and thin surfaces across $A$ or $A'$, then there is a further sequence of $Q$-moves applied to $\mc{J}$ so that $\mc{J}$ becomes $Q$-locally thin, and (1)-(3) continue to hold, and additionally:
\begin{enumerate}
    \item[4.] if $N(\mc{J}) = N(\mc{H})$, then $\ell(A, A')$ has strictly decreased.
\end{enumerate}
\end{proposition}
\begin{proof}
According to Definition \ref{def: interchange}, interchanging $H_i$ and $H_{i+2}$ means that there is a sequence $\mc{K}_0, \hdots, \mc{K}_m$ with $\mc{K}_0$ obtained from $\mc{H}$ by a preliminary interchange of $H_i$ and $H_{i+1}$ and with $\mc{K}_{j+1}$ obtained from $\mc{K}_j$ by one of:
\begin{enumerate}
    \item a $Q$-weak reduction, or 
    \item by an isotopy to remove one or more components of $Q\setminus \mc{K}_{j}$ that is parallel to a subsurface of $\mc{K}_j$.
    \end{enumerate}
Without loss of generality, we assume that $\mc{K}_{j+1}$ is obtained from $\mc{K}_j$ by a $Q$-weak reduction if $j$ is odd and an isotopy if $j$ is even. For convenience, we split the analysis into two cases.


\textbf{Case 1:} $\ell(A, A') \geq 4$ and neither $H_i$ nor $H_{i+2}$ are incident to $A$ or $A'$

In this case $H_i$ and $H_{i+2}$ are incident to components, necessarily vertical annuli, of $\ob{\mc{A}}$ on both sides. By Definition \ref{def: interchange}, part (3), the result of the interchange is $Q$-locally thin and the annuli $A$ and $A'$ are unaffected by the interchange. By Lemma \ref{discs persist}, as we are repeatedly using the $Q$-discs on either side of $\ob{\mc{A}}$ to perform the weak reduction, none of the $Q$-weak reductions increase $\ell(A, A')$. By Lemma \ref{Q-moves and annuli}, none of the $Q$-moves increase the number of nested annuli. Conclusions (1) - (3) follow from Lemma \ref{discs persist} and the fact that $Q$-discs are disjoint from vertical annuli and the vertical annuli in the insulating set separate the surfaces intersecting them.

\textbf{Case 2:} $\ell(A, A') \geq 2$ and one, or both of $H_i$ or $H_{i+1}$ is incident to $A$ or $A'$.

As in Lemma \ref{Q-moves and annuli}, the annulus $A$ persists as a curved annulus and $A'$ persists as a nested annulus up until the final $Q$-move in the sequence $(\mc{K}_i)$. In this case, the final $Q$-move in the sequence $(\mc{K}_i)$ consists of isotoping a thick surface $\mc{K}_n(1)$, as well as the adjacent thin surface $F_{2}$ across $A$ or $A'$. As in Lemma \ref{Q-moves and annuli}, we preserve the existence of a matched pair corresponding to $(A, A')$; however, one or both of the annuli in this matched pair may be $\boundary$-parallel. In any case, we have strictly decreased $\ell(A, A')$.  As in Case 1, (1) - (3) hold. If one or both of $A, A'$ is now $\boundary$-parallel, we may isotope the thick surface across it. This causes the curved annulus $A$ to combine with the nested annulus $A'$, eliminating the matched pair $(A, A')$ without creating any new matched pairs. Figure \ref{fig: interchange case 2} shows the situation when only one of $H_i$, $H_{i+1}$ are incident to $A \cup A'$. Figure \ref{fig: interchange case 2-2} shows the case when both are.
\end{proof}

\begin{figure}[ht!]
\centering
\labellist
\small\hair 2pt
\pinlabel{$A$} [t] at 82 44
\pinlabel{$H_i$} [r] at 4 87
\pinlabel{$H_{i+1}$} [r] at 4 175
\pinlabel{$A$ becomes $\boundary$-parallel} [l] at 365 22
\pinlabel{interchange} [b] at 248 170
\pinlabel{isotope} [b] at 470 168
\endlabellist
\includegraphics[scale=0.65]{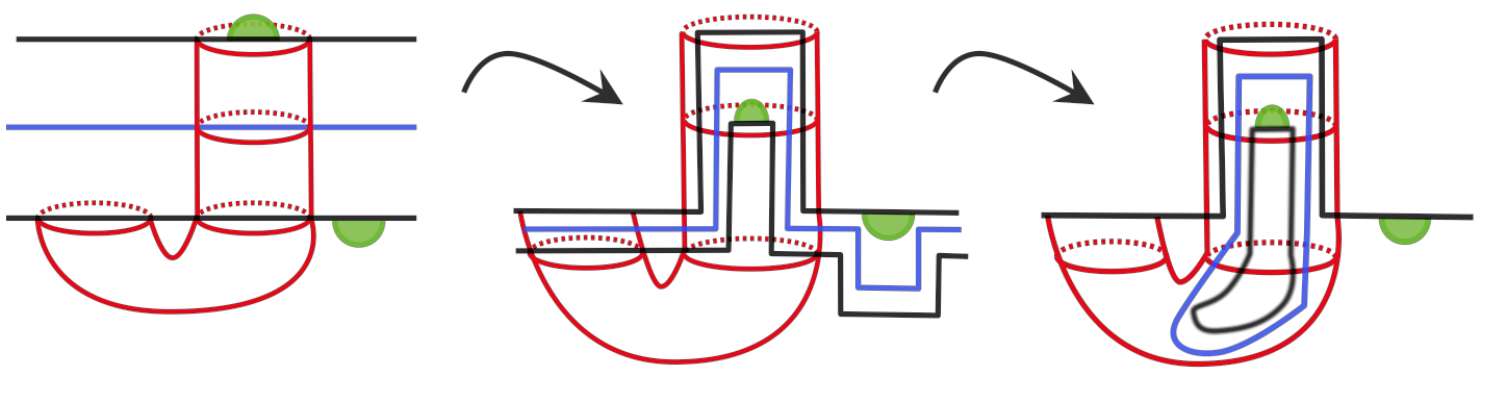}
\caption{This depicts one of the situations in Case (2) of the proof of Proposition \ref{decreasing complexity}: exactly one of $H_i \cup H_{i+1}$ is incident to $A \cup A'$. For concreteness we show the situation when $H_i$ is incident to $A$. The green half discs represent locations of the sc-discs for the thick surfaces. The figure depicts the case when $\ob{\mc{A}}$ consists of only annuli in the matching sequence; when $\ob{\mc{A}}$ contains other annuli the picture is not much different. Eventually one of $A$ or $A'$ (but not both) becomes $\boundary$-parallel. At the first arrow we perform the preliminary interchange and at the second arrow isotope the thick surface to remove the $\boundary$-parallelism. As drawn, this pops the lower thick and thin surfaces inside of $A$. Observe that the length of the matched pair decreases. }
\label{fig: interchange case 2}
\end{figure}

\begin{figure}[ht!]
\centering
\labellist
\small\hair 2pt
\pinlabel{$A$} [t] at 88 27
\pinlabel{$A'$} [b] at 126 184
\pinlabel{$H_i$} [r] at 5 66
\pinlabel{$H_{i+1}$} [r] at 5 140
\pinlabel{At least one of $A, A'$ is $\boundary$-parallel and $A, A'$ cancel.} [l] at 320 5
\pinlabel{interchange} [b] at 231 142
\pinlabel{isotope} [b] at 439 141
\endlabellist
\includegraphics[scale=0.65]{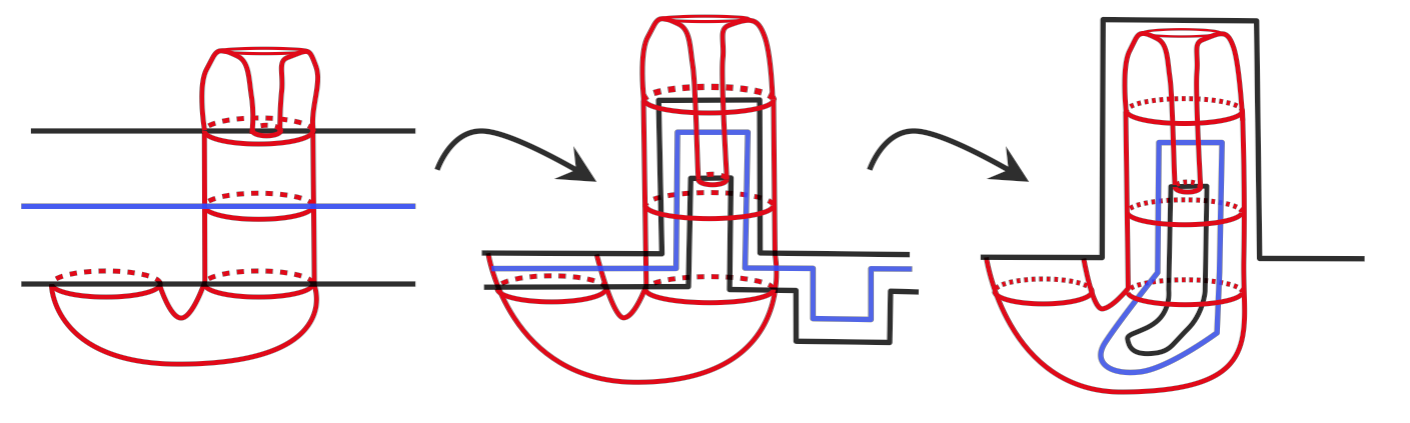}
\caption{This depicts one of the situations in Case (2) of the proof of Proposition \ref{decreasing complexity}: both of $H_i, H_{i+1}$ are incident to $A \cup A'$. For concreteness we show the situation when $H_i$ is incident to $A$ and $H_{i+1}$ to $A'$. In this depiction we have omitted the locations of the sc-discs for the thick surfaces. The figure depicts the case when $\ob{\mc{A}}$ consists of only annuli in the matching sequence; when $\ob{\mc{A}}$ contains other annuli the picture is not much different. At the first arrow we perform the preliminary interchange and at the second arrow isotope the thick surface to remove the $\boundary$-parallelism. As drawn this pops the lower thick and thin surfaces inside $A$.  The curved annulus and the nested annulus then cancel, becoming two vertical annuli (note how it runs from one thick surface, through a thin surface, to the other thick surface).}
\label{fig: interchange case 2-2}
\end{figure}

\begin{proposition}\label{no cancellable matched pairs}
Suppose $\mc{J} \in \H(Q)$. Then there exists $\mc{H} \in \H(Q)$ with:
\begin{itemize}
    \item $\netw(\mc{J}) = \netw(\mc{H})$ and $\netg(\mc{J}) = \netg(\mc{H})$
    \item $N(\mc{H}) \leq N(\mc{J})$
    \item there is no cancellable matched pair in $Q$ with respect to $\mc{H}$.
    \end{itemize}
\end{proposition}

\begin{proof}
Choose $\mc{H}$ so that:
\begin{enumerate}
    \item $\netg(\mc{H}) = \netg(\mc{J})$ and $\netx_m(\mc{H}) = \netx_m(\mc{J})$
    \item $\mc{H}$ is $Q$-locally thin
    \item Subject to all of the above $N(\mc{H}) \leq N(\mc{J})$ is minimal
    \item Subject to all of the above, we have minimized the minimal length of a cancellable matched pair with respect to $\mc{H}$.
\end{enumerate}

We must show that there is no cancellable matched pair for $Q$ with respect to $\mc{H}$. Suppose, for a contradiction, that there is a cancellable matched pair $(A, A')$. Choose it so that $\ell(A, A')$ is minimal among all such and let $\mc{A}$ be a matching sequence.

Suppose first that $\ell(A, A') \geq 2$. We recall from Definition \ref{def: cancellable} that there are three types of matched pairs. We will find an extended insulating set in each case.  If all annuli in $\mc{A} \setminus (A \cup A')$ are VSS, then $\mc{A}\setminus (A \cup A')$ is an insulating set, since each component of $\ob{\mc{A}} = (\mc{A} \setminus (A \cup A'))$ is connected, and each of its intersections with $\mc{H}$ separates $\mc{H}$. Hence, in the first two cases we have our insulating set. In the third case, the $\ob{\mc{A}} = (\mc{A} \setminus (A \cup A')) \cup \mc{B}$ is our insulating set, where $\mc{B}$ is as in Definition \ref{def: cancellable}. To see $\ob{\mc{A}}$ is an extended insulating set, suppose that $S \cpt \mc{H}$ intersects $\ob{\mc{A}}$. By hypothesis, $\ob{\mc{A}}$ separates $S$. Since $\mc{B}$ shares an end with $A$ or $A'$, each component of $S \setminus \ob{\mc{A}}$ lies on the same side of $\mc{A}$ as it does $\mc{B}$.  

Let $H_1, \hdots, H_{n}$ be the thick surfaces incident to $\ob{\mc{A}}$. They are indexed by odd numbers and are in order along $\ob{\mc{A}}$ from $H_0$ which is incident to $A$ to $H_n$ which is incident to $A'$. For expositional convenience, suppose that $H_0$ is below $H_n$. (This can always be attained by reversing the orientation of $\mc{H}$. Alternatively, in what follows one may replace ``lower'' with ``raise'', etc.).

Since $A$ is curved, it has a bridge disc outside $A$. By the definition of cancellable matched pair, there is such a bridge disc that is disjoint from $\ob{\mc{A}}$ and, therefore, lies outside $\ob{\mc{A}}$. By Lemma \ref{lem:disjoint vertical}, there is a $Q$-disc for $H_0$ on the same side of $H_0$ and which has boundary lying outside $\boundary \ob{\mc{A}}$. Since $\mc{H}$ is $Q$-locally thin each thick surface that intersects $\ob{\mc{A}}$ is $Q$-strongly irreducible. Thus, $H_0$ is mostly outside $\ob{\mc{A}}$. Similarly, $H_n$ is mostly inside $\ob{\mc{A}}$. Each thick surface intersecting $\ob{\mc{A}}$ is either mostly inside or mostly outside $\ob{\mc{A}}$. Suppose there is one such $H_i$ with $i \neq n$ that is mostly inside $\ob{\mc{A}}$. Choosing $i$ to be the smallest such index, we may then interchange $H_i$ with the $H_j$ below it as they are mergeable. This lowers the lowest thick surface that is mostly inside $\ob{\mc{A}}$, until (after renumbering) $H_2$ is mostly inside $\ob{\mc{A}}$. Similarly, if there was an $i \neq 0$ such that $H_i$ is mostly outside $\ob{\mc{A}}$, we may use interchanges to raise the highest such $H_i$ until, after renumbering, $H_{n-2}$ is mostly outside $\ob{\mc{A}}$. Proposition \ref{decreasing complexity} guarantees that these interchanges do not alter the fact that $(A, A')$ is a cancellable matched pair. They also do not increase $N(\mc{H})$ or $\ell(A, A')$. Another application of Proposition \ref{decreasing complexity} shows that if $H_2$ is mostly inside $\ob{\mc{A}}$ and $n \neq 2$, interchanging $H_2$ with $H_0$ will reduce $\ell(A, A')$ while preserving the fact that $(A, A')$ is a cancellable matched pair and $N(\mc{H})$ is minimal. This contradicts our choice of $\mc{H}$ and $(A, A')$. We similarly encounter a contradiction, if $n \neq 2$ and $H_{n-2}$ is mostly outside $\ob{\mc{A}}$. On the other hand, if $n = 2$, interchanging $H_0$ and $H_2 = H_n$ results in the elimination of the matched pair $(A, A')$ as in Proposition \ref{decreasing complexity}. We conclude that $\ell(A, A') = 0$. However, this implies the thick surface $H$ incident to $A$ and $A'$ is $Q$-weakly reducible. To see this, observe that, by the definition of cancellable matched pair, $A$ and $A'$ have $\boundary$-compressing discs $D, D'$ (respectively) with the interior of the arcs $\boundary D \cap H$ and $\boundary D' \cap H$ contained in disjoint components of $H \setminus \boundary (A \cup A')$. Lemma \ref{lem:disjoint vertical} then guarantees that there are disjoint $Q$-discs for $H$ on opposite sides. However, this contradicts the fact that $\mc{H}$ is $Q$-thin.
\end{proof}

\section{Tubes and Spools}\label{sec:noodles}
Informally speaking, after eliminating cancellable matched pairs as the torus $Q$ winds its way through the 3-manifold $M$, it is very difficult for it to turn around- it must (nearly) always turn in the same direction. The purpose of this section is to show that in most scenarios this implies that the torus must spiral infinitely. By way of analogy, consider a simple closed curve in the plane. If, as in Figure \ref{2Dspiral}, from some point on the tangent vector to the curve always turns in the same direction, then, with respect to some height function, either the curve has a single maximum and a single minimum, or it enters into an infinite spiral. One way to prove this topologically (as opposed to geometrically) is to use the Jordan curve theorem to find a region of the plane such that once the curve enters it cannot escape.

\begin{figure}[ht!]
\centering
\includegraphics[scale=0.15]{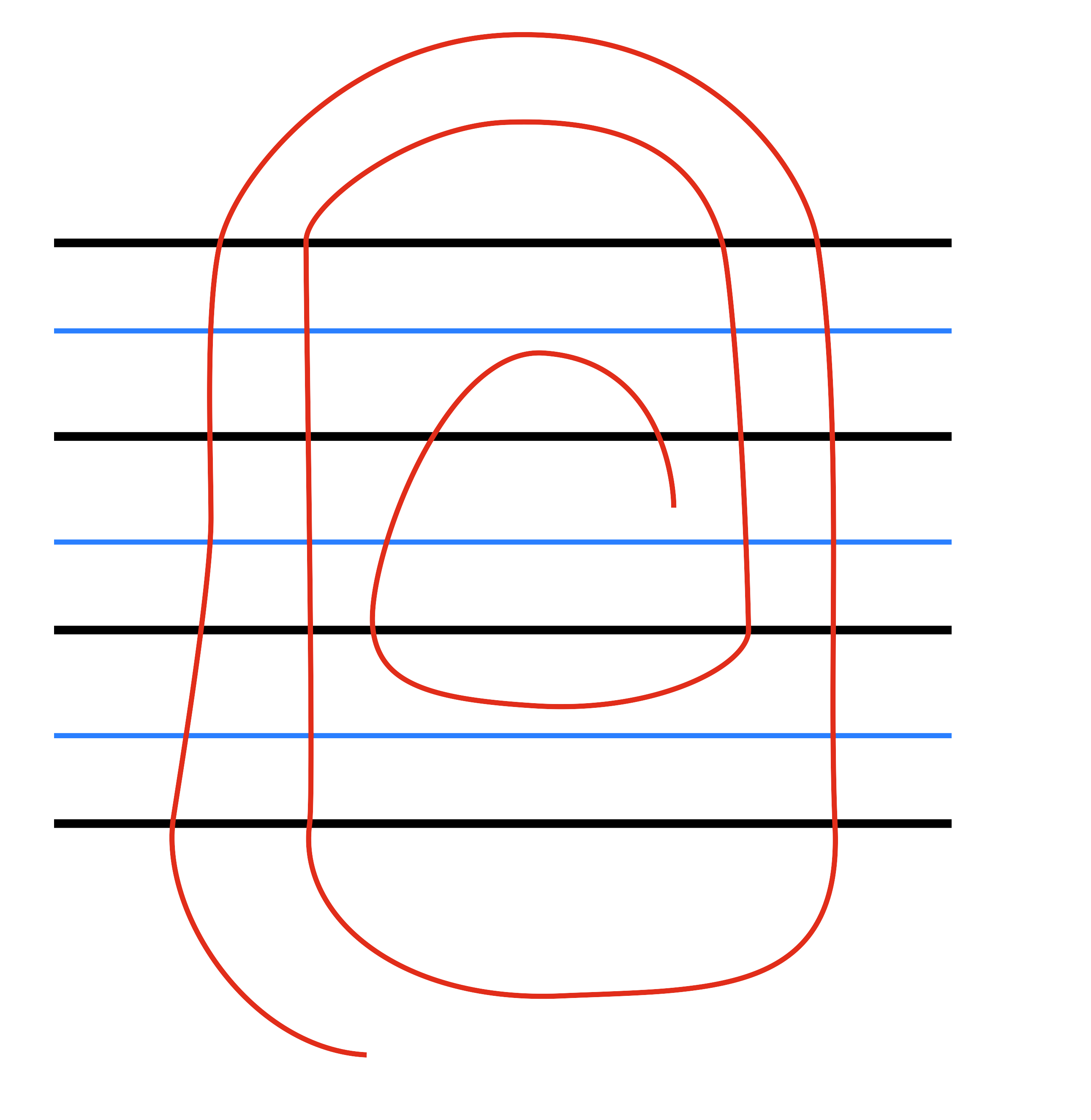}
\caption{This figure could be taken both as a literal depiction of a portion of a simple closed curve (in red) in the plane with the property that it always turns in the same direction or as a schematic depiction of a torus (in red) winding through a 3-manifold $M$. If the torus always turns in the same direction and if every closed surface separates then except in some very special situations it will end up entering into an infinite spiral, unable to escape from a submanifold of $M$. The horizontal lines depict either level curves of a height function on $\R^2$ or the thick and thin surfaces of a multiple bridge surface of $M$.}
\label{2Dspiral}
\end{figure}

Key to applying this perspective in our context, is a certain surface constructed as the union of a subsurface of a thick surface and a long annulus called a \emph{spool}. If a spool were to exist, the torus (in the absence of crushable handles and cancellable matched pairs) will enter a submanifold bounded by the spool and would have to spiral infinitely in that submanifold, unable to escape. Consequently, spools will not exit. The non-existence of a spool will force the ends of long annuli to behave in certain way which, in the next section, with additional hypotheses, will imply the existence of a crushable handle. We now embark on formalizing this perspective.

\begin{assumption}\label{Sep assump}
  Throughout this section we assume that $(M,T)$ is a standard (3-manifold, graph) pair such that every closed surface in $M$ separates $M$. Assume $Q \subset (M,T)$ is a c-essential closed separating surface with one component $V \cpt M\setminus Q$ designated as the ``inside''.  Let $\mc{H} \in \H(Q)$.  
\end{assumption}

\begin{definition}
 Suppose that $H \cpt \mc{H}^+$ and that $\mc{A} \cpt Q \setminus H$ is a long annulus containing a bridge annulus. It is \defn{curved} if all of the bridge annuli it contains are curved and \defn{nested} if the all the bridge annuli it contains are nested. 
\end{definition}

\begin{figure}[ht!]
\centering
\labellist
\small\hair 2pt
\pinlabel{$\Sigma$} [b] at 395 1034
\pinlabel{$\mc{A}$} [l] at 389 715
\pinlabel{$P$} [t] at 638 95
\pinlabel{$H$} [r] at 63 210
\endlabellist
\includegraphics[scale=0.15]{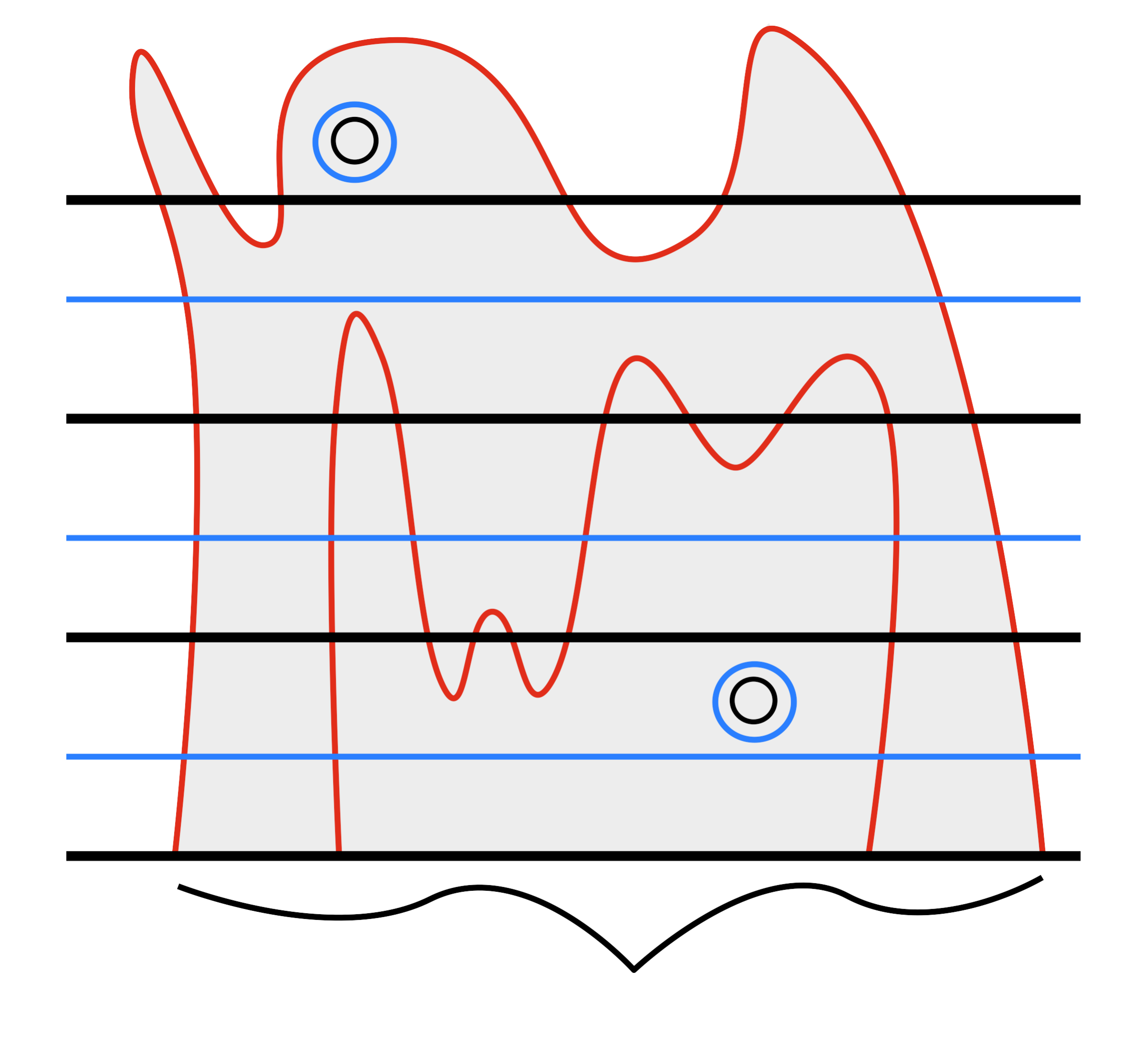}
\caption{A schematic depiction of the setup for Lemma \ref{long annulus nesting}. The shaded region is the submanifold $W$ bounded by $S = P \cup \Sigma$. The portions of the multiple bridge surface $\mc{H}$ intersecting $W$ are the surface $\mc{J}$. The ends of the long annulus $\mc{A}$ are $\alpha$ and $\beta$.}
\label{fig: Techlem}
\end{figure}

See Figure \ref{fig: Techlem} for a schematic depiction of the setup for the next lemma. 

\begin{lemma}\label{long annulus nesting}
Suppose that $H \cpt \mc{H}^+$ and that $\mc{A}\cpt Q\setminus H$ is a long annulus that does not contain a matched pair. Let its ends be $\alpha$ and $\beta$. Let $\Sigma$ be the (possibly empty) union of components of $Q\setminus (H \cup \mc{A})$ all on the same side of $H$ as $\mc{A}$ and such that each component of $H \setminus \Sigma$ is on the same side of each component of $\Sigma$. Let $P$ be the union of the components of $H\setminus \Sigma$ that are on the same side of $\Sigma$ as $\mc{A}$. The surface $S =P \cup \Sigma$ is a closed orientable surface in $M$ bounding a submanifold $W$ containing $\mc{A}$. Let $\mc{J} = \mc{H} \cap W$. (The components of $\mc{J}$ may have boundary.) Suppose that every component of $\mc{A} \cap \mc{J}$ is separating in the component of $\mc{J}$ containing it.

If $\mc{A}$ is curved (resp. nested) then $\beta$ is contained in the component of $H \setminus (\Sigma \cup \alpha)$ that is outside (resp. inside) $\alpha$.
\end{lemma}
\begin{proof}
    Without loss of generality, assume that $\mc{A}$ is curved.  By Assumption \ref{Sep assump}, $S$ separates $M$. Starting from $\alpha$, label the curves of $\mc{A} \cap \mc{J}$ as 
    \[
    \alpha = \gamma_0, \gamma_1, \hdots, \gamma_{n+2} = \beta
    \]
    as usual. By hypothesis, each $\gamma_i$ is separating in the component of $\mc{J}$ containing it. Let $i$ be the first index such that $A_i$ is a bridge annulus; it is curved with boundary in $J \cpt \mc{J}$. Thus, $\gamma_{i+1}$ lies in the component of $J\setminus \gamma_{i}$ that is outside $\gamma_i$. Since there is a path in $A_0 \cup \cdots \cup A_i$ from $\gamma_{i+1}$ to $\alpha$, this implies that if $J \subset H$, then $\gamma_{i+1}$ is also outside $\alpha$. Suppose $J \not \subset H$. The annulus $A_{i+1}$ is either vertical or bridge. If it is vertical, consider the annuli $A_{i+1}, A_{i+2}, \hdots, A_{i+k}$ such that either $\gamma_{i+k+1} = \beta$ or $A_{i+k+1}$ is bridge. Then $\gamma_{i + 1+ j}$ lies in the same component $J_j \cpt \mc{J}$ as $\gamma_{i-j}$ for each $1 \leq j \leq i+k+1$.  There is a path in $A_{0}\cup \hdots \cup A_i \cup  \hdots \cup A_{i+k}$ from $\gamma_{i+k+1}$ to $\gamma_{0}$. Thus, if $\gamma_{i+k+1} = \beta$, the result holds. If $A_{i+1}$ is bridge, it is curved. In this case, $\gamma_{i+2}$ lies in the subsurface of $\mc{J} \setminus \gamma_{i+1}$ with boundary component $\gamma_{i+1}$ and is outside $\gamma_{i+1}$. Since either $\gamma_i = \alpha$ or $A_{i-1}$ is vertical, it cannot lie in the subsurface that is inside $\gamma_i$. Thus, it lies in the subsurface that is outside $\gamma_i$. Continuing in this vein, we eventually see that $\beta$ lies in the subsurface of $H \setminus (\Sigma \cup \alpha)$ that is outside $\alpha$.
\end{proof}

\begin{definition}\label{def: tube filling}
Suppose $\mc{A} = A_0 \cup A_1 \cup \cdots \cup A_{n+1}$ is a long annulus such that each $A_i$ is a VSS or BSS, $\mc{A}$ does not contain a matched pair, and $\mc{A}$ does contain at least one BSS. 

    A connected subsurface $F_i \subset \mc{H}$ with boundary $\gamma_i$ is a \defn{filling surface} if it is inside $\gamma_i$ when $\mc{A}$ is curved and outside $\gamma_i$ when $\mc{A}$ is nested. $\mc{A}$ is a \defn{tube} if its filling surfaces are pairwise disjoint. See Figure \ref{fig:tubedef} for an example.
\end{definition}

\begin{figure}[ht!]
\centering
\labellist
\small\hair 2pt
\pinlabel{$A_0$} [r] at 110 308
\pinlabel{$A_1$} [r] at 117 378
\pinlabel{$A_2$} [b] at 197 490
\pinlabel{$A_3$} [b] at 278 369
\pinlabel{$A_4$} [l] at 441 465
\pinlabel{$A_5$} [l] at 441 372
\pinlabel{$A_6$} [l] at 441 305
\pinlabel{$A_7$} [l] at 441 225
\pinlabel{$A_8$} [l] at 441 146
\pinlabel{$A_9$} [t] at 287 44
\pinlabel{$A_{10}$} [l] at 212 146
\pinlabel{$A_{11}$} [l] at 182 223
\endlabellist
\includegraphics[scale=0.45]{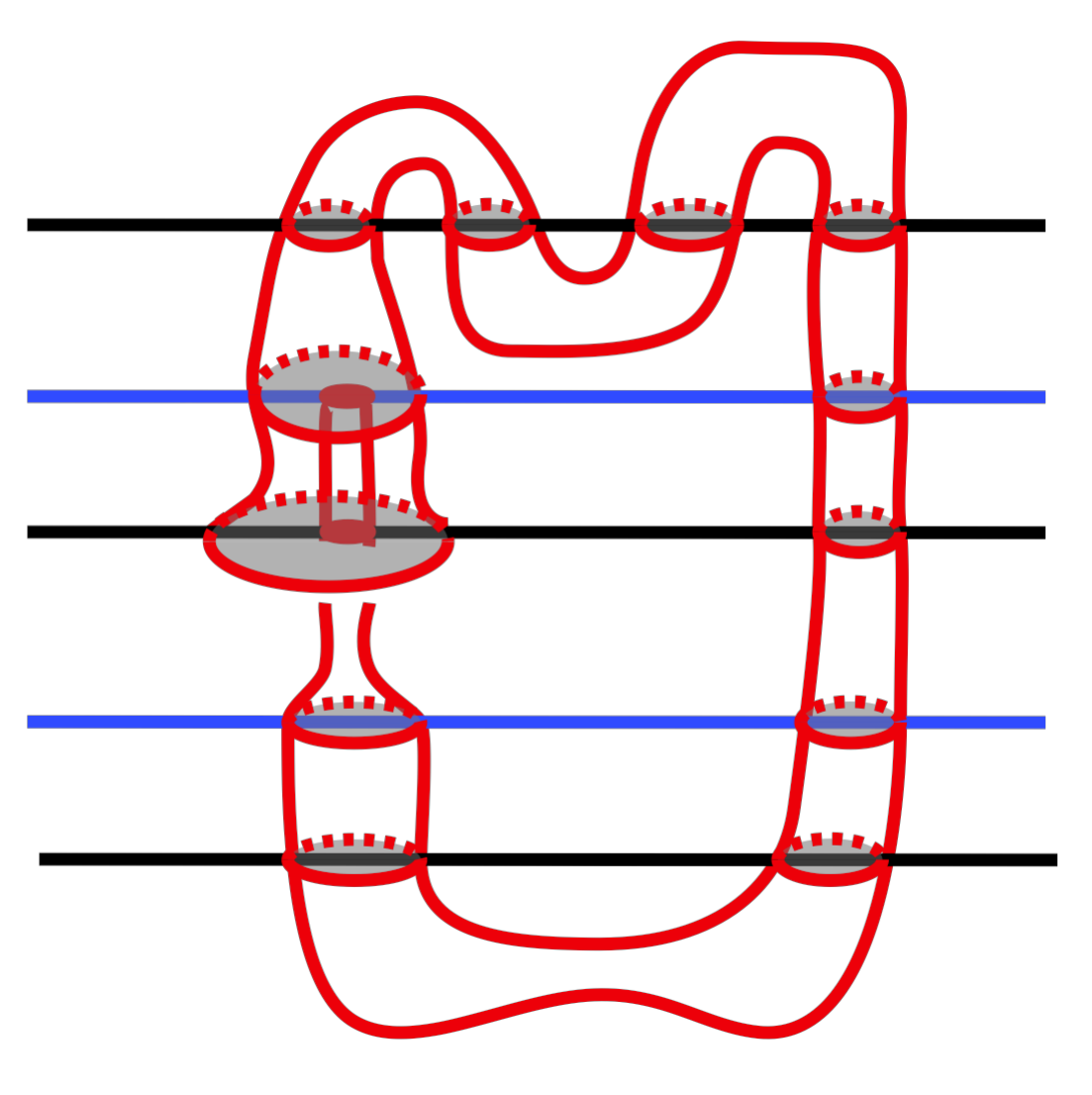}
\caption{A depiction of Definition \ref{def: tube filling}. The long annulus $A_0 \cup \cdots \cup A_{10}$ is a tube while the long annulus $A_0 \cup \cdots \cup A_{11}$ is not a tube. The filling surfaces are shaded; the filling surface $F_{12}$ is inside the filling surface $F_0$.}
\label{fig:tubedef}
\end{figure}

\begin{lemma}[Tube Lemma]\label{tube lemma}
Suppose that $\mc{A}$ is a long annulus that is a component of $Q\setminus H$ for some $H \cpt \mc{H}^+$ and which has the property that each annulus in $\mc{A}$ is a BSS or VSS. Assume that $\mc{A}$ contains a curved (resp. nested) annulus and that the ends of $\mc{A}$ bound disjoint subsurfaces $F_0, F_{n+1}$ of $H$ both inside (resp. outside) the ends of $\mc{A}$. Then $\mc{A}$ is a tube.
\end{lemma}
\begin{proof}
   This follows quickly from Lemma \ref{long annulus nesting} by induction on the number of components of $\mc{A} \cap \mc{H}$. When applying the Lemma, use $\Sigma = \nil$. Note that the requirement that $\mc{A}$ is a component of $Q \setminus H$ prevents the situation depicted in Figure \ref{fig:tubedef}.
\end{proof}

We now set about establishing criteria which will tell us when the surface $Q$ winds around another surface called a \emph{spool}. See Figure \ref{fig:spool} for depictions of two spools.

\begin{figure}[ht!]
\centering
\includegraphics[scale=0.3]{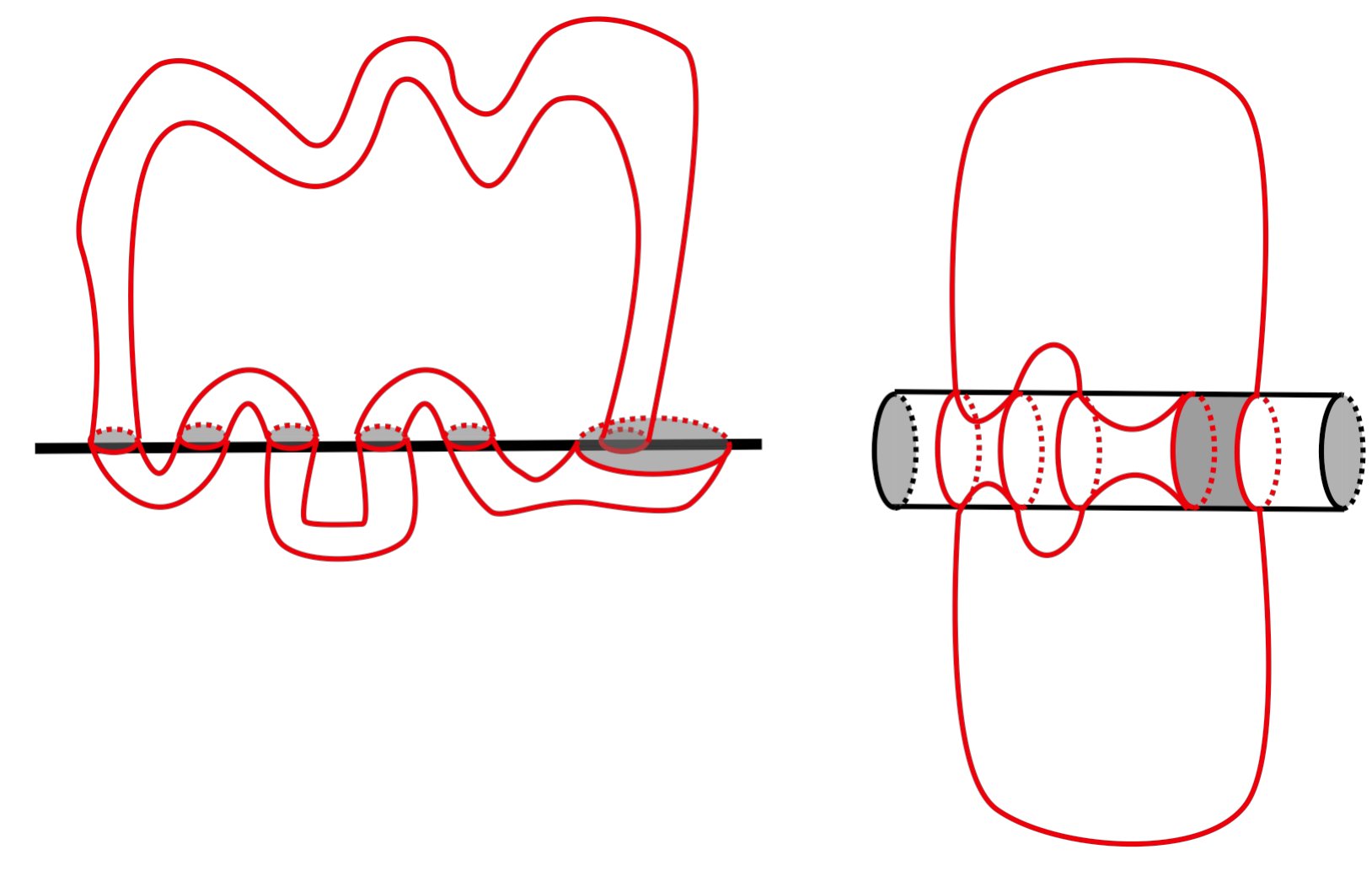}
\caption{Two spools (in red), as in Definition \ref{defn: spool}. The curves of intersection between the spool on the left and the thick and thin surfaces are all separating curves. Those on the right are all nonseparating curves. The black tube on the right represents a portion of a thick torus.}
\label{fig:spool}
\end{figure}

\begin{definition}\label{defn: spool}
A \defn{spool} is a closed orientable surface $P$ that is the union of a long annulus $\mc{A} = A_0 \cup \cdots \cup A_{n+1}$ with a surface $S$ such that:
\begin{enumerate}
    \item $S$ is a connected subsurface of some $H \cpt \mc{H}^+$ and $\boundary S = \boundary \mc{A}$
    \item $\mc{A}$ does not contain a matched pair
    \item There exists $1 \leq k < n+1$ such that the long annulus $\mc{A}' = A_0 \cup \cdots \cup A_k$ is a component of $Q \setminus H$.
    \item If $\mc{A}'$ is curved (resp. nested) then $\boundary \mc{A}' = \gamma_0 \cup \gamma_{k+1}$ is the boundary of a subsurface $F \subset H$ which is outside (resp. inside) the ends of $\mc{A}'$.
    \item Each curve $\gamma_i$ for $k+1 \leq i \leq n+1$ is separating in $F$ and $S$ has interior disjoint from $F$.
\end{enumerate}

By Assumption \ref{Sep assump}, the spool $P$ separates $M$. If $\mc{A}$ is curved, we let $W \cpt M\setminus Q$ be the component that is outside $\mc{A}$. If $\mc{A}$ is nested, we let $W \cpt M \setminus Q$ be the component that is inside $\mc{A}$. We call $W$ the \defn{spool room}. 

\end{definition}

The next lemma shows that in many cases if there is a spool, then $Q$ winds around it, similarly to what is depicted in Figure \ref{fig: spool2}.

\begin{figure}[ht!]
\centering
\includegraphics[scale=0.3]{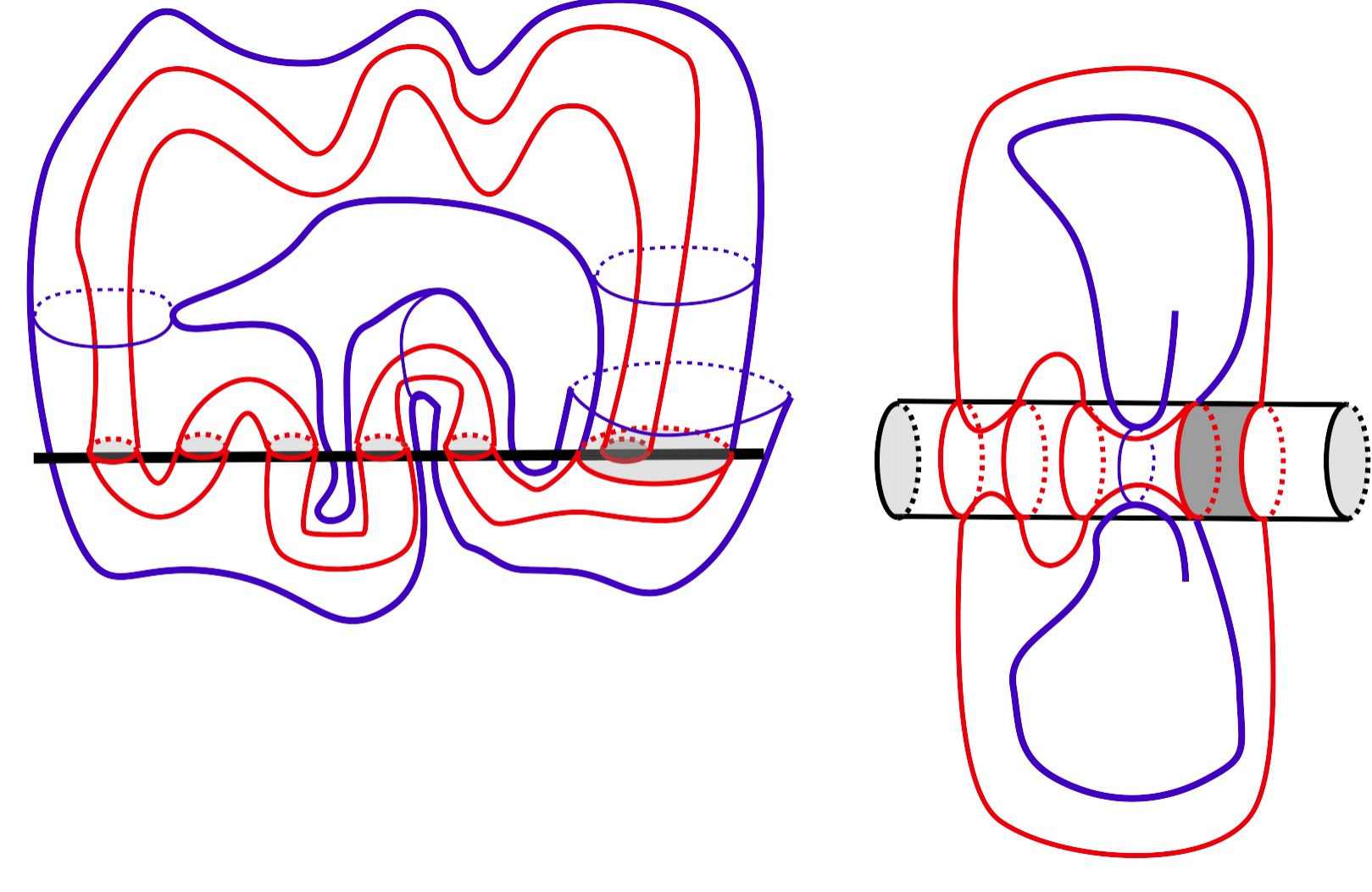}
\caption{We give an example of how spools, spool around each other as in the Spool Lemma (Lemma \ref{spool lemma}). The annulus $\mc{B}$ (the blue and red annulus) is the continuation of the annulus $\mc{A}$ (the red annulus).}
\label{fig: spool2}
\end{figure}

\begin{lemma}[Spool Lemma]\label{spool lemma}
Suppose that $P$ is spool as in Definition \ref{defn: spool}, with $W$ the spool room. Let $\mc{B}$ be the maximal long annulus $A_0 \cup \cdots$ containing $\mc{A}$ and having an end at $\gamma_0$ such that:
\begin{enumerate}
    \item $\mc{B}$ does not contain a matched pair
    \item Each curve of $(\mc{B} \setminus \mc{A}) \cap \mc{H} \cap W$ is separating in the component of $\mc{H} \cap W$ that contains it.
\end{enumerate}
Then $\mc{B}$ lies entirely in $W$. Furthermore, if $A$ is the component of $Q \setminus (\mc{B} \cup H)$ sharing the end $\boundary \mc{B} \setminus \gamma_0$ with $\mc{B}$, one of the following occurs:
\begin{enumerate}
    \item Either $A$ is not an annulus or $\mc{B} \cup A$ contains a matched pair.
    \item $A$ contains a curve that is nonseparating in the component of $\mc{J}$ containing it.
\end{enumerate}
\end{lemma}

\begin{proof}
The spool consists of a long annulus $\mc{A} \subset Q$ and a subsurface $S \subset H \cpt \mc{H}^+$. We write:
\[
\mc{A} = \underbrace{A_0 \cup A_1 \cup \cdots \cup A_k}_{\mc{A}'} \cup A_{k+1} \cup \cdots \cup A_{n+1}.
\]
We also have the (possibly empty) long annulus $\mc{B}$ which we can express as:
\[
\mc{B} = \underbrace{A_0 \cup \cdots \cup A_{n+1}}_{\mc{A}} \cup A_{n+2} \cup \cdots \cup A_{m}.
\]
Without loss of generality, we may assume that $\mc{A}'$ is curved, so that each bridge annulus in $\mc{B}$ is curved (as $\mc{B}$ does not contain a matched pair).

Recall that the ends of $A_i$ are $\gamma_i$ and $\gamma_{i+1}$. We are concerned with those $\gamma_i$ that lie on $H$. To that end, traversing the long annulus $\mc{B}$ beginning at $\gamma_0$, label the intersections $\mc{B} \cap H$ as
\[
\underbrace{\alpha_0}_{=\gamma_0}, \underbrace{\alpha_1}_{=\gamma_{k+1}}, \alpha_2, \hdots, \alpha_p
\]
See Figure \ref{fig: spool-labels} for an example. By the definition of $\mc{B}$, each $\alpha_i$ for $i \geq 2$ lies in the subsurface $F$ and separates $F$. If $\alpha_i$ does not separate the components of $\boundary F = \alpha_0 \cup \alpha_1$, let $\Phi_i \subset F$ be the subsurface with boundary $\alpha_i$; it must lie inside $\alpha_i$. If $\alpha_i$ separates $\boundary F$, let $\Phi_i \subset F$ be the subsurface bounded by $\alpha_i$ and $\gamma_0$.

\begin{figure}[ht!]
\labellist
\small\hair 2pt
\pinlabel{$\alpha_0$} [tl] at 484 141
\pinlabel{$\alpha_1$} [t] at 86 139
\pinlabel{$\alpha_2$} [t] at 147 141
\pinlabel{$\alpha_3$} [t] at 216 141
\pinlabel{$\alpha_4$} [t] at 333 138
\pinlabel{$\alpha_5$} [t] at 397 136
\pinlabel{$\alpha_6$} [t] at 490 122
\pinlabel{$\alpha_7$} [tr] at 41 139
\pinlabel{$\alpha_8$} [br] at 291 150
\pinlabel{$H$} [r] at 21 143
\endlabellist
\centering
\includegraphics[scale=0.6]{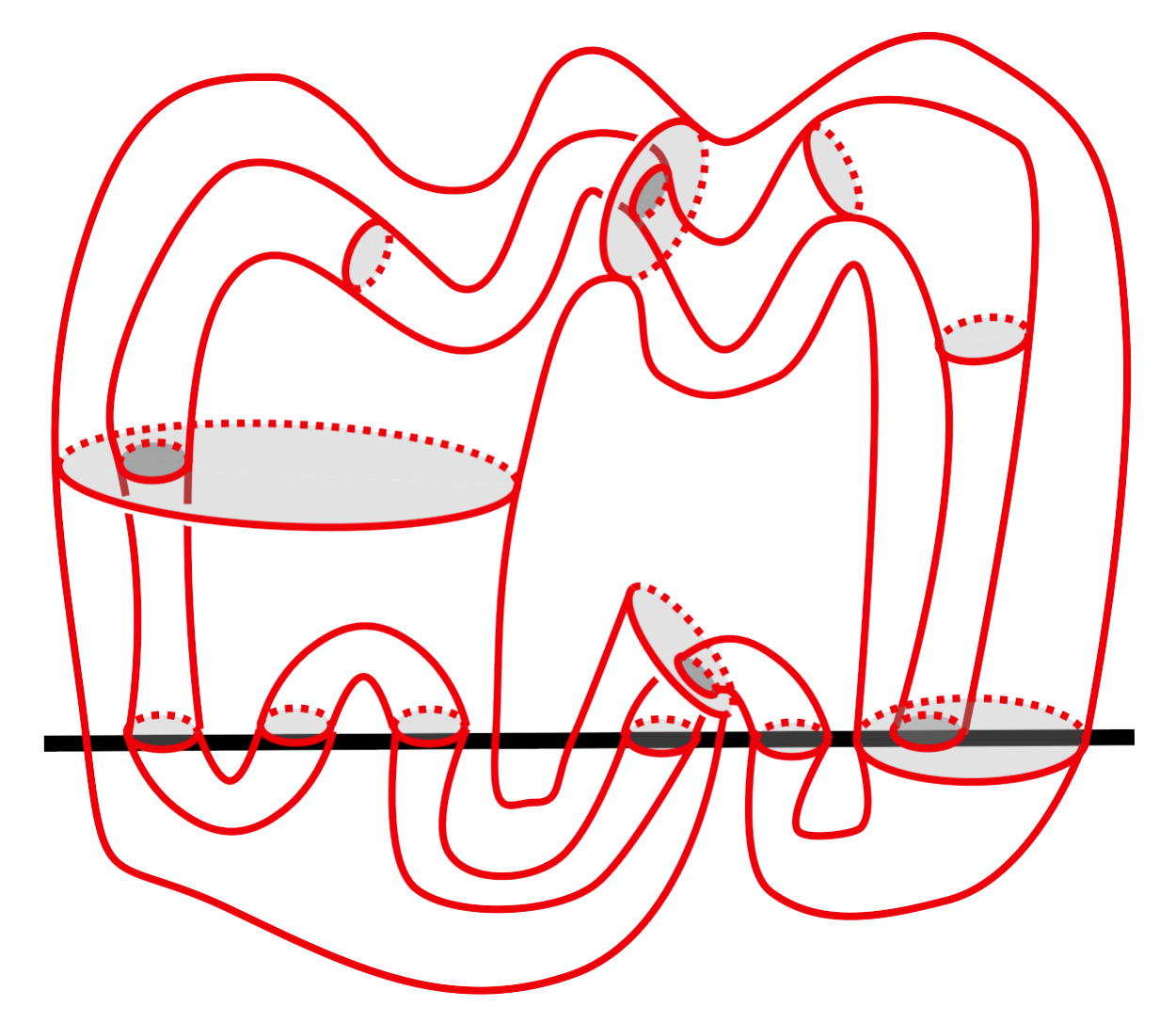}
\caption{The curves $\alpha_i$ on $\mc{B}$, as in the beginning of the proof of the Spool Lemma (Lemma \ref{spool lemma}). The only component of $\mc{H}$ that is shown is $H$.}
\label{fig: spool-labels}
\end{figure}

Assume that Conclusion (1) does not hold; that is, $A$ is an annulus with one end at $\alpha_p$ and the other a curve in $H$ and $\mc{B} \cup A$ does not contain a matched pair. In particular, $A$ is a curved long annulus. 

Number the components of $\mc{B} \setminus H$ as
\[
\mc{B}' = \mc{B}_0, \mc{B}_1, \hdots, \mc{B}_{p-1}
\]
Observe that they alternate which side of $H$ they lie on and that $\boundary \mc{B}_i = \alpha_i \cup \alpha_{i+1}$. Also, each $\mc{B}_i$ is curved. Let $n_0 = 0$, $n_1 = 1$, and let $n_2$ be the index such that $\alpha_{n_2} = \gamma_{n+2}$. We prove the result by induction on $p$. The Base Case and Inductive Step are nearly identical.

\textbf{Base Case:} $\mc{B}\setminus \mc{A}$ is empty (i.e. $p = n+2$.)

In this case, $A$ has one end at $\alpha_{n_2}$. By assumption, its other end (call it $\alpha_{n_3}$) lies in $H$. Since $A$ is disjoint from $\mc{A} = \mc{B}$ and since the spool surface $S = \mc{A} \cup P$ separates $M$, the curve $\alpha_{n_3}$ lies in $F$. We need to show it does not lie in $P$. Suppose, to the contrary, that it does lie in $P$. Observe that $P$ is inside $\alpha_{n_2}$. 

The surface $F \cup \mc{A}'$ separates $M$ and bounds a submanifold $W'$ containing $\mc{A}$. Let $\mc{J}' = \mc{H} \cap W'$. Either Conclusion (2) holds or $A \cap \mc{J}'$ separates the component of $\mc{J'}$ containing it. In which case, by Lemma \ref{long annulus nesting}, $A$ cannot have an end in the interior of $P$, but this contradicts the maximality of $\mc{B}$.

\textbf{Inductive Step}: Assume that $p > n+2$.

Apply the argument of the Base case to $\mc{A}_{n_2}$ in place of $A$. We conclude that $\alpha_{n_3} = \boundary \mc{A}_{n_2}\setminus \alpha_{n_2}$ lies in $F\setminus \Phi_{n_2}$. We then repeat the argument, applying it to $\mc{A}_{n_3}$ and conclude that its end $\alpha_{n_4} = \boundary \mc{A}_{n_3} \setminus \alpha_{n_3}$ lies in $F \setminus \Phi_{n_3}$. Continuing in this vein, we conclude that $\mc{B}$ lies entirely in $W$. The choice of $\mc{B}$ to be maximal then ensures that Conclusion (1) or (2) holds.
\end{proof}

\begin{definition}\label{def: tower}
A \defn{tower} is a long annulus $\mc{T} \cpt Q \setminus H$ with both ends on some thick surface $H$ and which is the union of VNN annuli and exactly one BNN annulus.
\end{definition}

Figure \ref{fig:towerstubes} depicts a torus $Q$ that is the union of two towers and some tubes. See Figure \ref{fig: tower twirl} for a depiction of the notation in the proof of Lemma \ref{no bendy towers}.

\begin{figure}[ht!]
\labellist
\small\hair 2pt
\pinlabel {$H$} [r] at 13 382
\endlabellist
\centering
\includegraphics[scale=0.3]{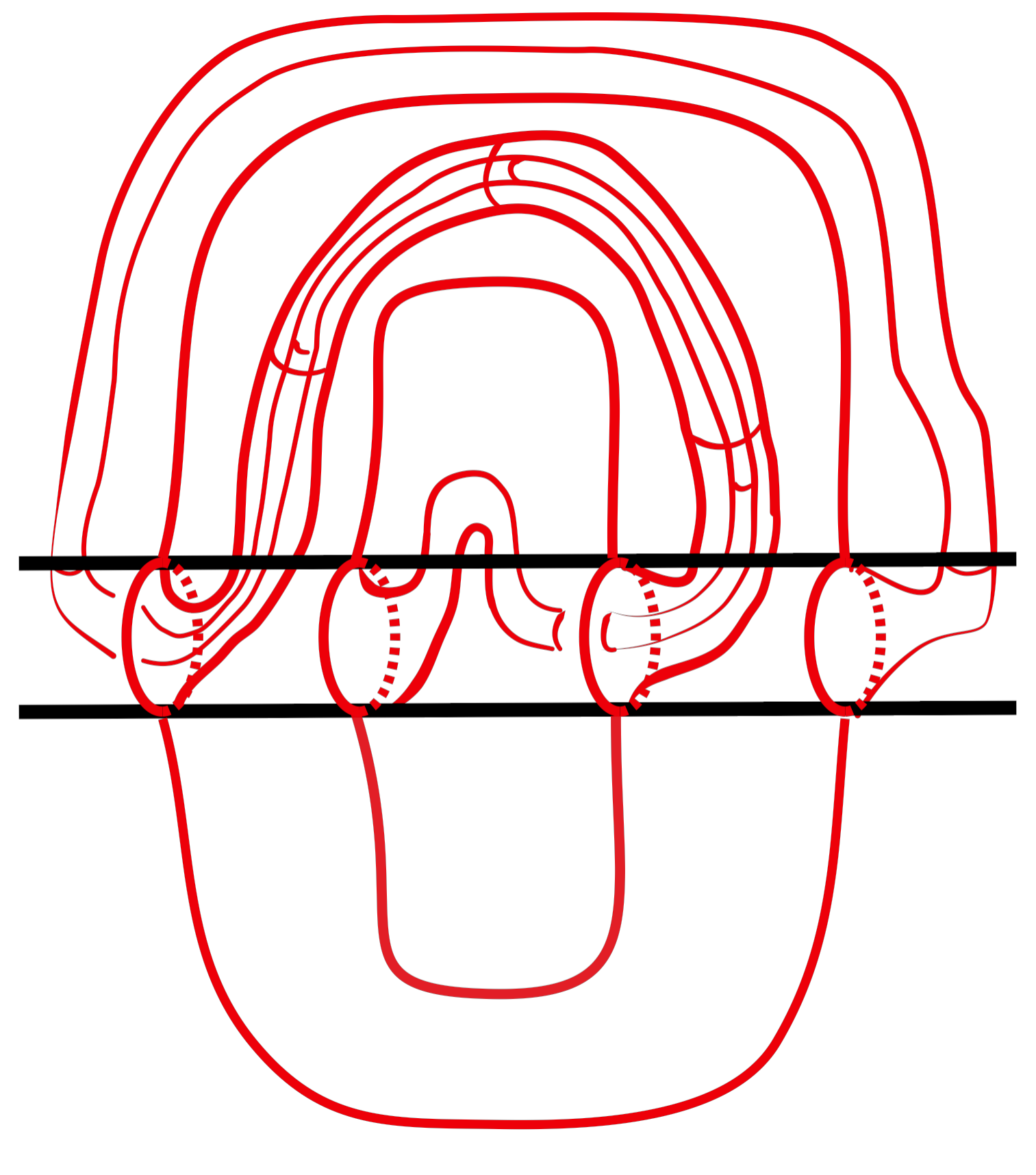}
\caption{A torus $Q$ that is the union of two towers (Definition \ref{def: tower}), some tubes, and two BNS annuli incident to a thick torus $H$. Note the existence of matched pairs that are not cancellable matched pairs. The graph $T$ is not shown, so it is not obvious that $H$ is weakly reducible. This configuration is what led us to the idea of crushable handles. The horizontal black lines together depict a single thick torus.}
\label{fig:towerstubes}
\end{figure}

\begin{lemma}\label{no bendy towers}
    Suppose that $\mc{H} \in \H(Q)$, that $H' \cpt \mc{H}^+$, and that $\mc{T}$ is a long annulus that is the union of VNN and BNN annuli, is a component of $Q\setminus H'$, and does not contain a matched pair. Then if every component of $\mc{H}^+$ intersecting $\mc{T}$ is a torus, then $\mc{T}$ is a tower.
\end{lemma}
\begin{proof}
Suppose, to the contrary, that $\mc{T}$ is not a tower. It must, therefore, contain at least two bridge annuli. Without loss of generality, assume that they are curved. 

Let $\beta \cpt \mc{T} \cap \mc{H}$ be the initial end of the second bridge annulus $A_{k+1}$ (determined by counting starting at either end of $\mc{T}$). Let $\alpha \cpt \mc{T} \cap \mc{H}$ be the loop occurring as the first encounter of $\mc{T}$ with the surface $H \cpt \mc{H}^+$ containing the ends of the second bridge annulus. Let $\mc{A}' = A_0 \cup \cdots \cup A_k$ be the initial $k+1$ annuli of $\mc{T}$ and let $\mc{A} = A_0 \cup \cdots \cup A_k \cup A_{k+1}$. 

\begin{figure}
\labellist
\small\hair 2pt
\pinlabel {$\beta$} [bl] at 372 719
\pinlabel {$A_{k+1}$} [b] at 525 667
\pinlabel {$\mc{A}'$} [l] at 734 851
\pinlabel {$H'$} [b] at 241 566
\pinlabel {$H \supset F$} [b] at 112 721
\pinlabel {$S$} [b] at 548 715
\endlabellist
    \centering
    \includegraphics[scale=0.2]{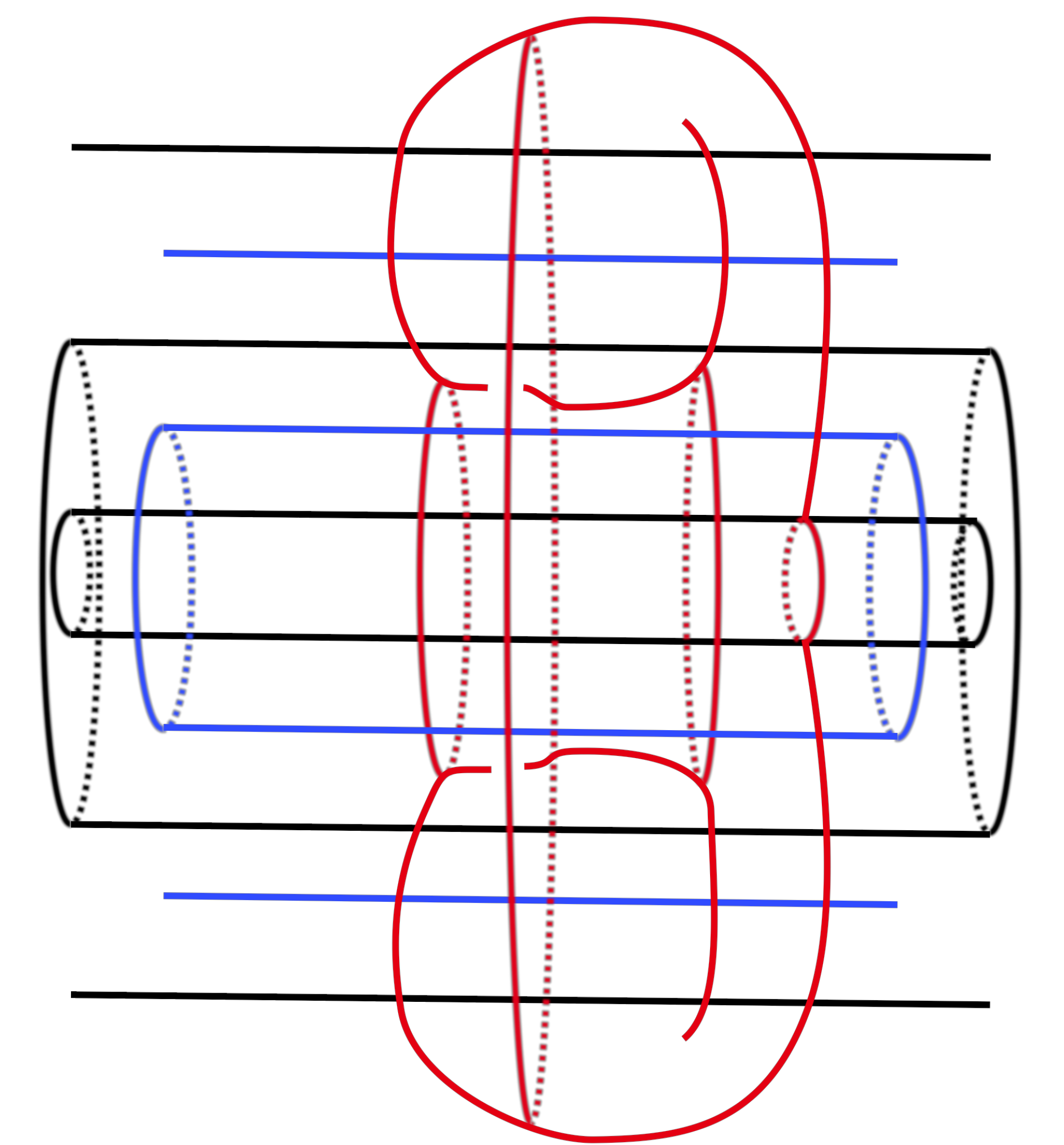}
    \caption{In the proof of Lemma \ref{no bendy towers} we show that a long enough annulus containing BNN annuli that intersects tori is either a tower or creates a spool. The black and blue lines represent thick and thin tori.}
    \label{fig: tower twirl}
\end{figure}

Since $\mc{A}$ is a component of $Q\setminus H'$, note that $H \neq H'$. Also note that since every component of $\mc{H}^+$ intersecting $\mc{T}$ is a torus, every component of $\mc{H}^-$ intersecting $\mc{T}$ is either a sphere or a torus. However, each annulus of $\mc{T}\setminus \mc{H}$ is a VNN or BNN and so each component of $\mc{H}^-$ intersecting $\mc{T}$ is also a torus. Let $F \cpt H \setminus \boundary \mc{A}'$ be the annulus that is outside $\boundary \mc{A}'$. Let $S \cpt H \setminus \boundary \mc{A}$ be the annulus contained in $F$. Then $P = \mc{A} \cup S$ is a spool. Let $W$ be the spool room; note that $H'$ is not contained in $W$. Let $\mc{B}$ be the maximal long annulus in $Q$ containing $\mc{A}$, having $\alpha$ as an end, not containing a matched pair, and with each component of $(\mc{B} \setminus \mc{A}) \cap \mc{H} \cap W$ separating the component of $\mc{H} \cap W$ containing it. If such exists, let $A \cpt Q \setminus (H \cup \mc{B})$. We claim that $A$ does not exist.

Consider a curve $\gamma$ of $\mc{B} \cap H'$ that is not an end of $\mc{A}$. The \href{spool lemma}{Spool Lemma} shows that $\gamma$ lies in $F \setminus S$. The bridge annulus $A_{k+1}$ separates the VPC $(C, T_C)$ containing it. Let $A' \cpt Q\cap C$ have $\gamma$ as an end. Observe that $A' \subset \mc{T}$, as $H \neq H'$. Consequently, $A'$ is an annulus.  It cannot be a vertical annulus in $(C, T_C)$, as $\gamma$ separates the ends of $A_{k+1}$. It cannot be nested as $\mc{T}$ does not contain a matched pair. It follows that the end of $\mc{B}$ that is not $\alpha$ lies on a component of $\mc{H} \cap W$. As $\mc{T}$ is disjoint from all sphere components of $\mc{H}$, and as if $(X,T_X)$ is a VPC with $\boundary_+ X$ a torus and $\boundary_- X$ containing a torus, then $X = T^2 \times I$, the end $\boundary \mc{B} \setminus \alpha$ lies in an annulus component of $\mc{H} \cap W$ and separates the ends of $W$. The first annulus of $A$ must therefore lie in $\mc{T}$ and cannot have an end lying in a torus component of $\mc{H} \cap W$. This contradicts the maximality of $\mc{B}$. Consequently, $A$ cannot exist. Hence, we have both that the end $\boundary \mc{B} \setminus \alpha$ lies in $F \setminus S$ and that $\mc{B} = \mc{T}$. This is a contradiction as the ends of $\mc{T}$ lie on $H'$. 
\end{proof}

\section{Concluding Arguments}\label{sec:Concluding Arguments}

\begin{assumption}\label{final assump}
$(M,T)$ is standard, irreducible and 1-irreducible. $M$ is either $S^3$ or a lens space and $Q$ is an essential unpunctured torus. Assume that $\mc{H} \in \H(Q)$ is $Q$-locally thin, has $\netg(\mc{H}) = 1$ and $\netw(\mc{H}) \leq \b_1(T)$. Also assume that there is no cancellable matched pair. 
\end{assumption}

\begin{lemma}\label{lens space structure 2}
If a VPC $(X, T_X) \cpt (M,T)\setminus \mc{H}$ contains a component of $Q \cap(X, T_X)$ with one end nonseparating in $\boundary_+ X$ and the other end separating a component of $\boundary C$, then $(X, T_X)$ is either the innermost or outermost VPC with $\boundary_+ X$ a torus.
\end{lemma}
\begin{proof}
    Recall from Lemma \ref{lens space structure} that $\mc{H}$ is the union of spheres and tori and all tori in $\mc{H}$ are parallel, once sphere components are ignored.  Consequently if some component of $\mc{H}$ is a torus, there exist exactly two compressionbodies $(C, T_C)$ and $(D, T_D)$ of $M\setminus \mc{H}$ with $\boundary_+ C$ and $\boundary_+ D$ tori and $\boundary_- C$ and $\boundary_- D$ the (possibly empty) union of spheres. Suppose $(X, T_X)$ is a compressionbody containing a component $A \cpt Q \cap X$ that has one end nonseparating in $\boundary_+ X$ and has the other end separating in $\boundary X$. Cap off the separating end and, after a small isotopy, arrive at an essential disc in $X$ with nonseparating boundary. Thus, $(X, T_X)$ must be either $(C, T_C)$ or $(D, T_D)$. 
\end{proof}

\begin{convention}
    Henceforth, $(X, T_X)$ and $(Y, T_Y)$ will be the innermost and outermost (respectively) VPCs with toroidal positive boundary.
\end{convention}

\begin{lemma}\label{no stretching annuli}
If both $X$ and $Y$ contain BNS or VNS annuli, then there does not exist a long annulus with each intersection with $\mc{H}$ nonseparating in $\mc{H}$ and with one end on $\boundary_+ X$ and one end on $\boundary_+ Y$.
\end{lemma}
\begin{proof}
If there were such an annulus, we could cap off the separating ends of the BNS and VNS annuli with discs and see that $M = S^1 \times S^2$, a contradiction.
\end{proof}

\begin{corollary}\label{name it and claim it}
    There is no matched pair $(A, A')$ with any of the following properties:
    \begin{enumerate}
        \item $A$ and $A'$ are BSS
        \item One of $A, A'$ is BNS and the other is BSS 
        \item $A$ and $A'$ are BNN
    \end{enumerate}
\end{corollary}

\begin{proof}
In each case, we will show that $(A, A')$ is cancellable, and thereby contradict Assumption \ref{final assump}. Let $\mc{A}$ be a matching sequence for $(A, A')$. Without loss of generality, we may assume that $(A, A')$ is the matched pair of one of the given types of shortest length.

Suppose that $A, A'$ are both BSS. If $\mc{A}$ contained a VNN or VNS, it must contain at least two VNS. Since $\mc{A}\setminus (A \cup A')$ is vertical, this contradicts Lemma \ref{no stretching annuli}. Thus, each annulus in $\mc{A}\setminus (A \cup A')$ is a VSS. Consequently, $(A, A')$ is cancellable, a contradiction.

If one of $A, A'$ is BNS and the other is BSS, each annulus of $\mc{A}\setminus (A \cup A')$ is VSS, as in the previous case. Consequently, $(A, A')$ is again cancellable, a contradiction.

Finally, suppose that $A, A'$ are both BNN. As in the first case, if $\mc{A}$ contained a VSS or VNS, we would contradict Lemma \ref{lens space structure} or \ref{no stretching annuli}. Thus, $\mc{A}\setminus (A \cup A')$ entirely consists of VNN annuli. By Lemma \ref{no stretching annuli}, at most one of $A$ or $A'$ is contained in $X \cup Y$. 

If $\ell(A, A') = 0$, let $H \cpt \mc{H}^+$ contain $\boundary D \cup \boundary D'$. We note that $\boundary (A \cup A')$ separates $H$ into three subsurfaces and if $D, D'$ are $\boundary$-compressing discs for $A, A'$ respectively, the arcs $\boundary D \cap H$ and $\boundary D' \cap H$ lie in distinct subsurfaces. This means that $(A, A')$ is cancellable, a contradiction. 

Assume, therefore, that $\ell(A, A') \geq 2$. This implies that $\boundary_+ X \neq \boundary_+ Y$. Suppose, for the moment, that one of $A$ or $A'$ is contained in $X \cup Y$. Without loss of generality, suppose it is $A$. Let $H$ be the thick surface containing $\boundary A$. (We have $H = \boundary_+ X$ or $H = \boundary_+ Y$.)  Let $A''$ be the annulus sharing an end with $A$ and not lying in $\mc{A}$. We claim that $A''$ is vertical (and thus a VNN). If it were bridge, it must be a BNN. Since $A$ is curved, $A''$ cannot be nested for then we contradict our choice of $(A, A')$ to minimize length. However, we will see it cannot be curved either. Let $S \cpt H \setminus \boundary (A \cup A'')$ be the annulus that is outside a component of $\boundary A$ and inside a component of $\boundary A''$. Then $A \cup S \cup A''$ is a spool. Let $\mc{B}$ be the spooling annulus as in the statement of the \href{spool lemma}{Spool Lemma}. The end of $\mc{B}$ not on the spool must be non-separating in $H$. To see this, let $W$ be the submanifold bounded by the spool and containing $\mc{B}$. Note that the spool is a torus and that each curve of $\boundary A$ bounds a disc to the complement of $W$. If the end of $\mc{B}$ not on the spool separated $H$, each curve of $\boundary A$ would bound a surface in both $W$ and its complement. This contradicts the fact that each closed surface in $M$ separates $M$. Thus, by the maximality of $\mc{B}$, the end of $\mc{B}$ not on the spool is an end of a BNS or VNS. However the fact that $A$ is a BNN obstructs this. Thus, $A''$ is a VNN. The case when $A'$ is contained in $X \cup Y$ is nearly identical.

Let $\mc{B}'$ be a maximal vertical long annulus having sharing an end with $\mc{A}$. If one of $A$ or $A'$ is contained in $X$ or $Y$, choose $\mc{B}'$ to contain $A''$ as above. If such is the case, then the other end of $\mc{B}'$ is not contained $\boundary_+ X \cup \boundary_+ Y$. A similar appeal\footnote{We elaborate on this argument in the proof of Lemma \ref{tower tales} below.} to the \href{spool lemma}{Spool Lemma}, shows that number of annuli in $\mc{B}'$ must be at least the number of annuli in $\mc{A} \setminus (A \cup A')$. Let $\mc{B} \subset \mc{B}'$ contain $A''$ and have $\ell(A, A')$ annuli. It follows, that $\mc{B} \cup (\mc{A} \setminus (A \cup A'))$ is an extended insulating set. Thus, $(A, A')$ is a cancellable matched pair, a contradiction.
\end{proof}

\begin{corollary}\label{spheres imply crushable}
    If $Q$ intersects a sphere component of $\mc{H}$, then $Q$ has a crushable handle.
\end{corollary}
\begin{proof}
 If $Q$ intersects a sphere component of $\mc{H}$, then there is an innermost or outermost VPC $(C, T_C)$ with $\boundary_+ C$ a sphere intersecting $Q$. The result follows as in Lemma \ref{spheres and crushable handles}.   
\end{proof}

Let $(X, T_X)$ and $(Y, T_Y)$ be the VPCs with $\boundary_+ X$ and $\boundary_+ Y$ tori and $\boundary_- X$ and $\boundary_- Y$ either empty or spheres. (These are the innermost and outermost VPCs with positive boundary a torus.)

\begin{definition}
    A tower is \defn{maximal} if its ends are not incident to VNNs.
\end{definition}

\begin{figure}[ht!]
\centering
\includegraphics[scale=0.1]{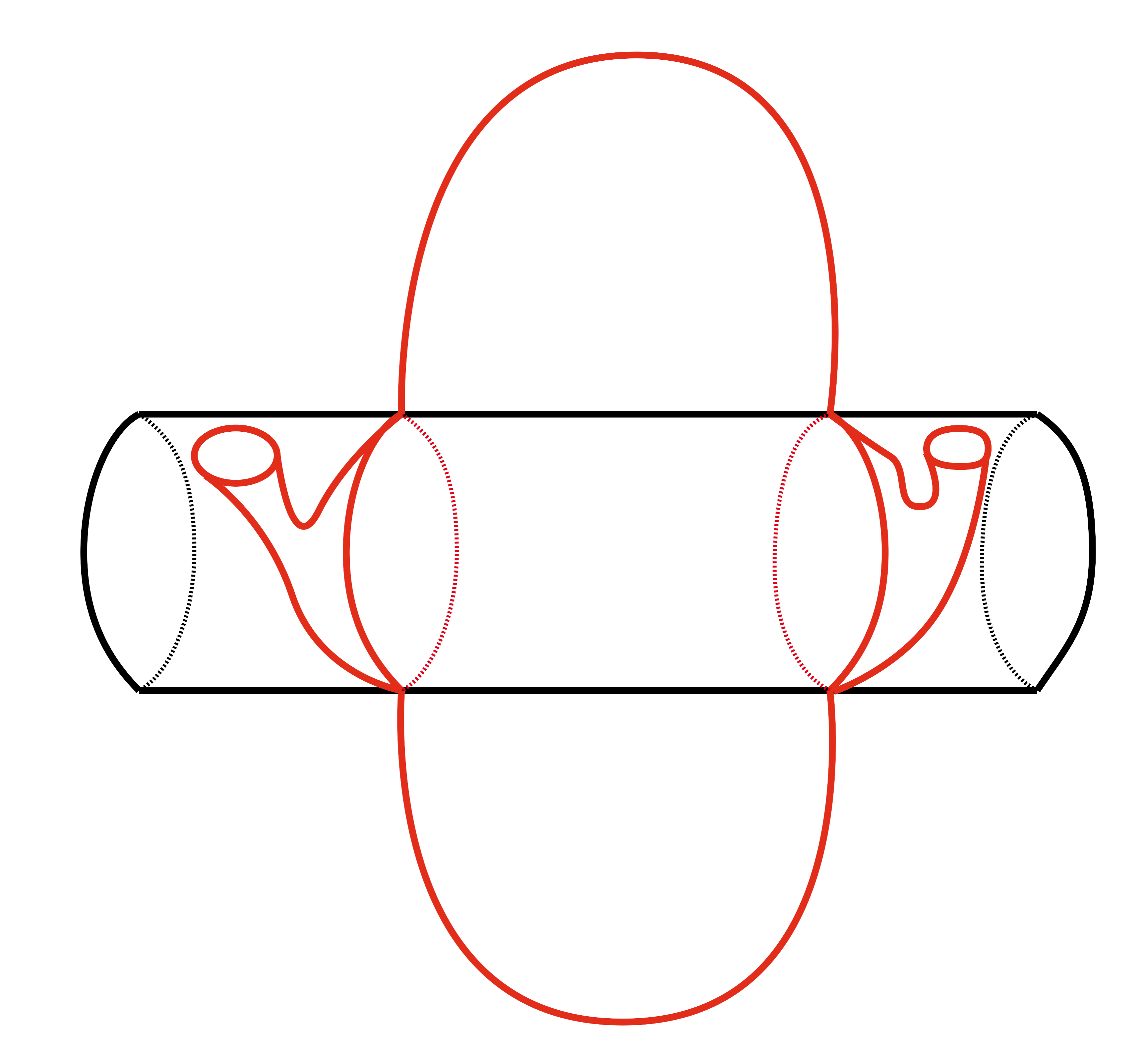}
\caption{The red surface is the union of the tower in Conclusion (2) of Lemma \ref{tower tales} with two BNS annuli. The black annulus is a portion of $\boundary_+ X$ or $\boundary_+ Y$, which is a torus.}
\label{fig: Tall tower}
\end{figure}

\begin{lemma}\label{tower tales}
Suppose that some component of $Q \cap \mc{H}$ is nonseparating in the component of $\mc{H}$ containing it. Then each component of $Q \setminus (X \cup Y)$ containing such a curve is a tower either with ends on $\boundary_+ X$ or with ends on $\boundary_+ Y$. Furthermore, one of the following occurs:
\begin{enumerate}
    \item $Q$ is the union of two towers, each with ends on $\boundary_+ X$ or each with ends on $\boundary_+ Y$, or
    \item Each tower with ends on $\boundary_+ X$ or with ends on $\boundary_+ Y$ and not in $X$ or $Y$ respectively, has its ends incident to BNS annuli lying in $X$ or $Y$ respectively.
\end{enumerate}
\end{lemma}

See Figure \ref{fig: Tall tower} for an example of the tower in Conclusion (2) of Lemma \ref{tower tales}.

\begin{proof}
As we traverse $Q$ examining the components of $Q \cap \mc{H}$, we may label each component as either S or N depending on whether or not it separates the component of $\mc{H}$ containing it. If no component is labelled N, the result is vacuous, so assume there is at least one N. Each component of $Q\setminus \mc{H}$ which has one end labeled N and one labeled S is a BNS that lies in either $X$ or $Y$. Furthermore, by Lemma \ref{lens space structure}, no long annulus that is the union of VNNs has ends at both $\boundary_+ X$ and $\boundary_+ Y$. Consequently each component of $Q \setminus \boundary_+ X$ that is disjoint from the interior of $X$ and has one end labelled $N$ and on $\boundary_+ X$ also has its other end labelled N and on $\boundary_+ X$. The analogous statement holds for $Y$.

Let $H = \boundary_+ X$. By Lemma \ref{no bendy towers}, each component $\mc{T} \cpt Q\setminus H$ with at least one end non-separating on $H$ must be a tower. Similarly for $H = \boundary_+ Y$. Suppose that $\mc{T}$ is one such tower. Without loss of generality, we may assume its ends are non-separating curves on $\boundary_+ X$ and that it is curved. Let $F\subset \boundary_+ X \setminus \boundary \mc{T}$ be the annulus that is outside $\boundary_+ \mc{T}$. Let $R$ be a component of $Q \cap \boundary_+ X$ incident to an end of $\mc{T}$. Since $Q$ is disjoint from sphere components of $\mc{H}$, $R$ is a BNN or BNS.

Suppose that $R$ is a BNN. If $\boundary R = \boundary \mc{T}$, then Conclusion (1) holds. So assume that an end $\gamma$ is not an end of $\mc{T}$. If $\gamma \subset F$, let $S \subset F$ be the annulus between $\gamma$ and the end of $\mc{T}$ not shared with $R$. In this case $\mc{T} \cup R \cup S$ is a spool. Let $\mc{B}$ be the maximal annulus containing $\mc{T} \cup R$, having $\boundary \mc{T} \setminus R$ as an end, and which is the union of VNN and BNN annuli. Observe that $\mc{B}$ does not contain a matched pair since if it did it would contain a cancellable matched pair by Lemma \ref{name it and claim it}, contradiction. Furthermore, $\mc{B}$ is contained in the spool room $W$ and so each component of $\mc{H}^+$ that intersects it is an annulus with ends on $\mc{T}$. Thus, each component of $\mc{B} \cap \mc{H}^+ \cap W$ separates the component of $\mc{H}^+ \cap W$ that it intersects. By the \href{spool lemma}{Spool Lemma}, $\mc{B}$ has an end in the interior of $F\setminus S$. Such an end $\gamma'$ must be incident to a BNS, however this is impossible since $\gamma'$ separates the ends of the annulus $R$. Thus, we encounter a contradiction. The situation when $\gamma$ does not lie in $F$ is similar, except that we reverse the roles of $\mc{T}$ and $R$ when defining the spool. We conclude that if Conclusion (1) does not hold each long annulus that is the maximal union of VNN and BNN annuli is a tower disjoint from the interior of $X \cup Y$ and either has its ends on $\boundary_+ X$ or has its ends on $\boundary_+ Y$. These ends can only be incident to BNS annuli in $X \cup Y$ since $Q$ is disjoint from the sphere components of $\mc{H}$.
\end{proof}

\begin{definition}
    Suppose $\mc{T}$ is a tower with ends on $\boundary_+ X$ or $\boundary_+ Y$. If $\mc{T}$ is curved (resp. nested), the annulus component of $\boundary_+ X \setminus \boundary \mc{T}$ or $\boundary_+ Y \setminus \boundary \mc{T}$ is the \defn{shadow annulus} for $\mc{T}$ if it is outside $\mc{T}$ (resp. inside $\mc{T}$).

    The \defn{height} $h(\mc{T})$ of $\mc{T}$ is $h(\mc{T}) = |\mc{T} \cap \mc{H}^+|/2$. It is the number of components of $\mc{H}^+$ it intersects.
\end{definition}

\begin{lemma}\label{nested shadows}
    Suppose that $\mc{T}_1$ and $\mc{T}_2$ are two towers, each with ends on $\boundary_+ X$ or each with ends on $\boundary_+ Y$ and not lying in $X$ or $Y$ respectively. Then either the shadow annulus for $\mc{T}_1$ is contained in the shadow annulus for $\mc{T}_2$ or the shadow annulus for $\mc{T}_2$ is contained in the shadow annulus for $\mc{T}_1$. 
  \end{lemma}
  \begin{proof}
      Without loss of generality, assume that $\mc{T}_1$ and $\mc{T}_2$ have their ends on $\boundary_+ X$ and that they do not lie in $X$. Let $\alpha_i$ be the shadow annulus for $\mc{T}_i$ with $i = 1,2$. Suppose, to establish a contradiction, that $\alpha_1 \not\subset \alpha_2$ and $\alpha_2 \not\subset \alpha_1$. Since $\alpha_i \cup \mc{T}_i$ is a separating surface, this implies that $\alpha_1 \cap \alpha_2 = \nil$. Recall from Lemma \ref{lem:disjoint vertical} that both $\alpha_1$ and $\alpha_2$ contain the boundaries of $Q$-discs for the VPC $(C, T_C) \neq (X, T_X)$ having $\boundary_+ C = \boundary_+ X$. Let $D \subset (X, T_X)$ be an c-disc for $\boundary_+ X$. We may find such that is disjoint from $Q$. In which case, $\boundary D$ is disjoint from either $\alpha_1$ or $\alpha_2$. This contradicts the $Q$-c-strong irreducibility of $\boundary_+ X$. Thus, either $\alpha_1 \subset \alpha_2$ or $\alpha_2 \subset \alpha_1$.
\end{proof}

\begin{lemma}\label{disjt discs imply crushable handle}
    Unless $Q$ has a crushable handle, no component of $Q \setminus \mc{H}$ is a BSS with ends bounding disjoint discs in $\mc{H}^+$.
\end{lemma}
\begin{proof}
    Suppose that $A$ is a BSS with ends bounding disjoint discs $D_1, D_2$ in $\mc{H}^+$. Consider all the components of $Q\setminus \mc{H}$ in the same VPC as $A$ and with at least one end in $D_1 \cup D_2$. By replacing $A$ with an innermost such annulus, we may assume that no such annulus has ends bounding disjoint discs in $D_1 \cup D_2$.  Since $Q$ is disjoint from spheres in $\mc{H}$ (Lemma \ref{spheres imply crushable}), no annulus with an end in $D_1 \cup D_2$ is vertical. Consequently, $A$ is crushable.    
\end{proof}

\begin{lemma}\label{all nonsep}
If some component of $Q \cap \mc{H}$ is nonseparating in $\mc{H}$, then either $Q$ is the union of two towers or $Q$ has a crushable handle.
\end{lemma}
\begin{proof}
    Suppose that $Q \cap \mc{H}$ contains a curve that does not separate $\mc{H}$. Thus, each component of $Q\setminus (X \cup Y)$ containing such a curve is a tower with ends on $\boundary_+ X$ or $\boundary_+ Y$ by Lemma \ref{tower tales}. If $Q$ is not the union of two towers, then the ends of a tower lying on $\boundary_+ X$ (resp. $\boundary_+ Y)$ but not in $X$ (resp. $Y$) are incident to BNS annuli in $X$ (resp. $Y$). There is at least one such tower. Without loss of generality, we may suppose there is a tower with ends on $\boundary_+ X$ and not contained in $X$. By Lemma \ref{nested shadows}, the shadow annuli for all such towers form a nested sequence of annuli in $\boundary_+ X$, each one properly contained in the next.

Let $\mc{T}$ be a tower with innermost shadow annulus $\alpha$; that is $\alpha$ does not contain the ends of any other tower. Indeed, it does not contain any component of $Q \cap \boundary_+ X$ that is nonseparating in $\boundary_+ X$. Let $\beta = \boundary_+ X \setminus \alpha$ be the complementary annulus and let $A_1$ and $A_2$ be the BNS annuli incident to the ends of $\mc{T}$. 

\textbf{Case 1:} The separating ends of $A_1$ and $A_2$ both lie in $\beta$.

Both $A_1$ and $A_2$ are $\boundary$-compressible in $(X, T_X)$. Let $D$ be a $\boundary$-compressing disc for one of them chosen so as to intersect $Q$ minimally among all such discs. Without loss of generality, suppose it is a $\boundary$-compressing disc for $A_1$. If $D \cap A_2 \neq \nil$, an outermost arc of intersection would cut off a $\boundary$-compressing disc for $A_2$ intersecting $Q$ fewer times than does $D$. Thus, $D \cap A_2 = \nil$. Since the interior of the arc $D \cap \boundary_+ X$ is then disjoint from $A_1 \cup A_2$, it lies in $\beta$. After $\boundary$-compressing $A_1$ or another component of $Q\cap X$ intersecting $D$ using a (sub)disc of $D$, we construct an $Q$-disc for $\boundary_+ X$ in $(X, T_X)$ with boundary contained in $\beta$. But since the VPC on the opposite side of $\boundary_+ X$ from $X$ has an $Q$-disc with boundary contained in $\alpha$, we contradict the $Q$-strong irreducibility of $\boundary_+ X$. Thus, this case cannot occur.

\textbf{Case 2:} One or both of the separating ends of $A_1$ and $A_2$ lie in $\alpha$.

The separating end $\boundary_s A_i$ of $A_i$ bounds a disc $D_i$ in $\boundary_+ X$. If both $D_1$ and $D_2$ are contained in $\alpha$, either they are disjoint or one is contained in the other, as in Figure \ref{fig:case 2}. Choose the numbering so that $D_1 \subset \alpha$ and if $D_2 \subset \alpha$ then $D_1 \not\subset D_2$.

\begin{figure}
\centering
\labellist
\small\hair 2pt
\pinlabel {$A_1$} [l] at 105 89
\pinlabel {$A_2$} [r] at 249 89
\pinlabel {$\beta$} [b] at 172 236
\pinlabel {$\alpha$} [t] at 172 27
\pinlabel {$A_1$} [l] at 482 89
\pinlabel {$A_2$} [b] at 589 195
\pinlabel {$\beta$} [b] at 554 236
\pinlabel {$\alpha$} [t] at 554 27
\endlabellist
 \includegraphics[scale=0.5]{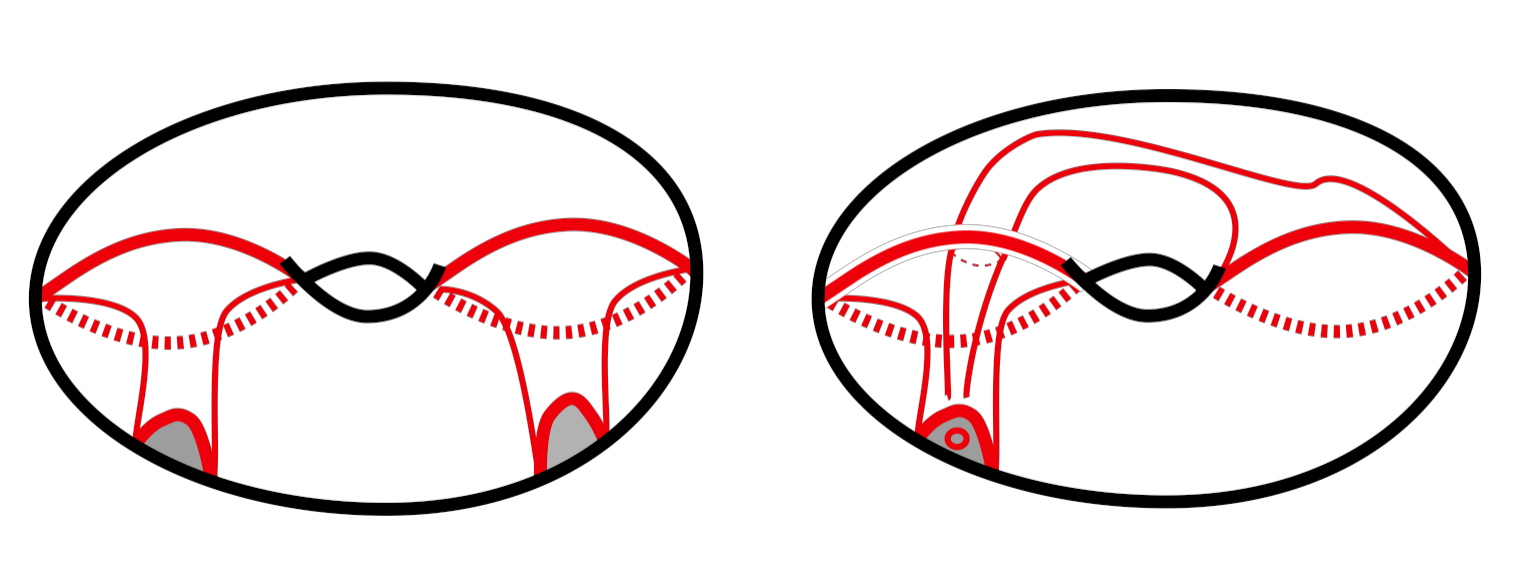}
\caption{This figure shows the possible configurations in Case 2 of the proof of Lemma \ref{all nonsep}.}
    \label{fig:case 2}
\end{figure}

Let $\mc{A}' \neq A_1$ be the component of $Q\setminus \boundary_+ X$ having the separating end $\boundary_s A_1$ of $A_1$ as one of its ends. Note that $\mc{A}'$ is the union of BSS and VSS annuli. Consequently, it does not contain a matched pair. Its ends are both separating and so it is a tube by the \href{tube lemma}{Tube Lemma}. If the end $\boundary \mc{A}' \setminus \boundary_s A_1$ lies in $D_1$, then $A_1 \cup \mc{A}'$ contains a cancellable matched pair as on the left side of Figure \ref{fig: case2b}, a contradiction. If the end $\boundary \mc{A}' \setminus \boundary_s A_1$ bounds a disc disjoint from $D_1$, then $\mc{A}'$ contains a BSS having ends bounding disjoint discs. This means that $Q$ contains a crushable handle by Lemma \ref{disjt discs imply crushable handle}, as on the right of Figure \ref{fig: case2b}. Thus, we may assume that the end $\boundary \mc{A}'\setminus \boundary_s A_1$ bounds a disc $D'_1 \subset \alpha$ containing $D_1$. Observe that the end $\boundary \mc{A}' \setminus \boundary_s A_1$ cannot be the end of a BNS since we chose $\mc{T}$ to be innermost. It is, therefore, the end of a BSS in $(X, T_X)$ (since any vertical annulus in $(X, T_X)$ would have an end on a sphere component of $\mc{H}^-$). The other end of that BSS bounds a disc $D''_1$ contained in $D'_1$ and containing $D_1$. Let $\mc{A}$ be the union of $\mc{A}'$ with that BSS. The surface $\mc{A} \cup (D'_1 \setminus D_1)$ is a spool surface. 

Let $\mc{B}$ be the long annulus containing $\mc{A}$ and sharing its initial end with $\mc{A}$ and which is the maximal union of BSS and VSS. The proof of the \href{spool lemma}{Spool Lemma} shows that the terminal end of $\mc{B}$ lies in $D'_1 \setminus D''_1$. However, this is impossible, for the terminal end of $\mc{B}$ must be the end of a BNS but $\mc{T}$ is innermost. Thus, if Case 2 occurs, $Q$ contains a crushable handle.
\end{proof}

\begin{figure}
  \centering
\labellist
\small\hair 2pt
\pinlabel {$A_1$} [b] at 88 273
\pinlabel {$\boundary_s A_1$} [l] at 114 163
\pinlabel {$\boundary_s A_1$} [l] at 490 170
\pinlabel {$\mc{A}'$} [l] at 136 85
\pinlabel {$A_1$} [b] at 467 282
\pinlabel {$\mc{A}'$} [t] at 524 53
\endlabellist
    \includegraphics[scale=0.35]{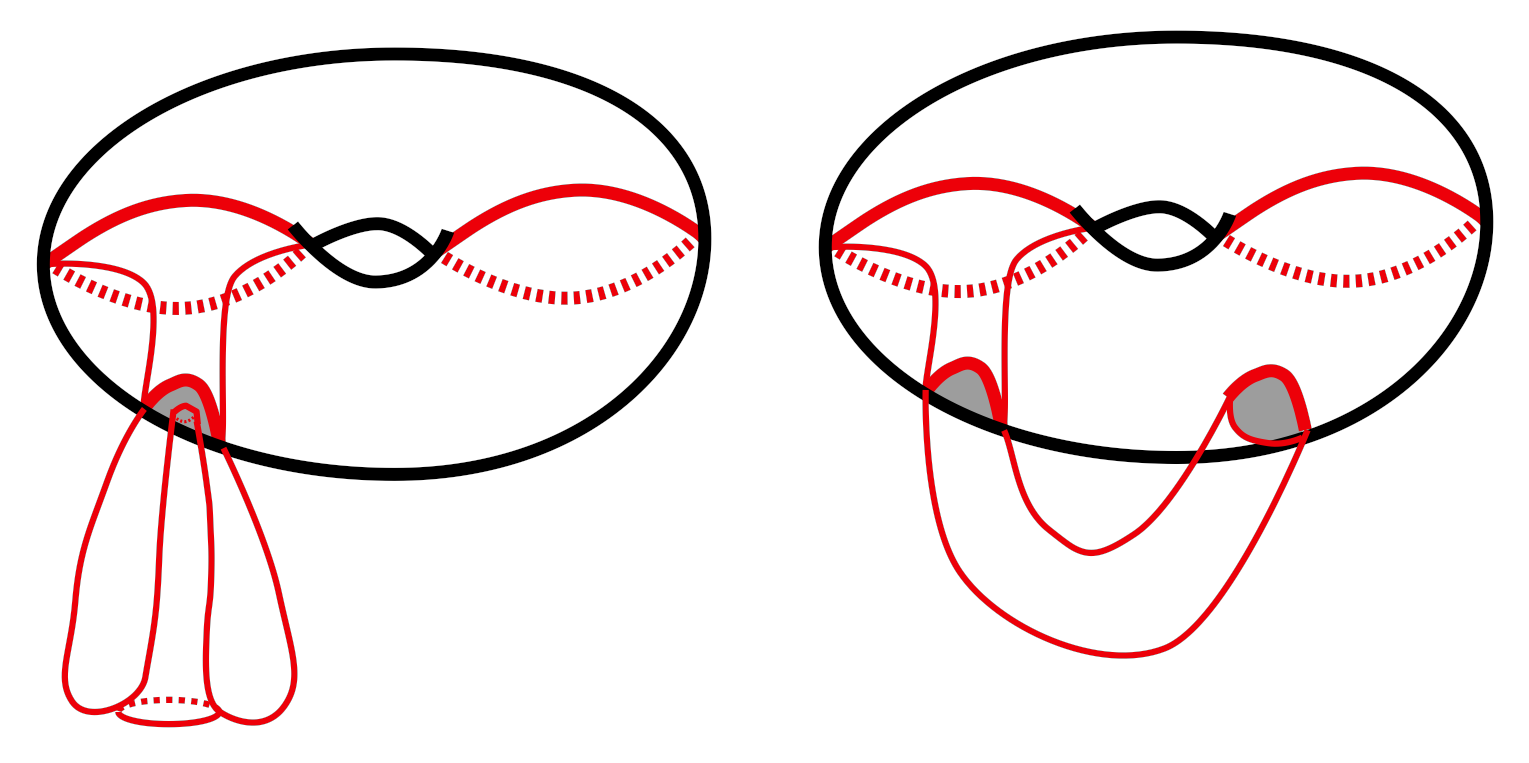}
    \caption{Case 2 of the proof of Lemma \ref{all nonsep}: we find either a cancellable matched pair or a crushable handle. On the left, since the tube $\mc{A}'$ incident to the separating end  of $A_1$ has its other end interior to the disc bounded by $\boundary_s A_1$, there will be a cancellable matched pair. On the right, we see an example of how if the other end of $\mc{A}'$ bounds a disc exterior to the disc bounded by $\boundary_s A_1$, there will be a crushable handle.}
    \label{fig: case2b}
\end{figure}

\begin{lemma}\label{all sep}
    Suppose that every curve of $Q \setminus \mc{H}$ separates $\mc{H}$ and that $Q$ does not have a crushable handle. Then $Q$ is the union of two bridge annuli $A_1$ and $A_2$, with VSS annuli. Each end of each of $A_1$ and $A_2$ bounds a disc in $\mc{H}$ and for each of them one of those discs contains the other.
\end{lemma}

The situation of Lemma \ref{all sep} is depicted in Figure \ref{fig: allsep case}.

\begin{figure}
  \centering
\labellist
\small\hair 2pt
\pinlabel {$Q$} [r] at 161 295
\pinlabel {$A_1$} [tl] at 268 194
\pinlabel{$A_2$} [bl] at 268 352
\endlabellist
    \includegraphics[scale=0.35]{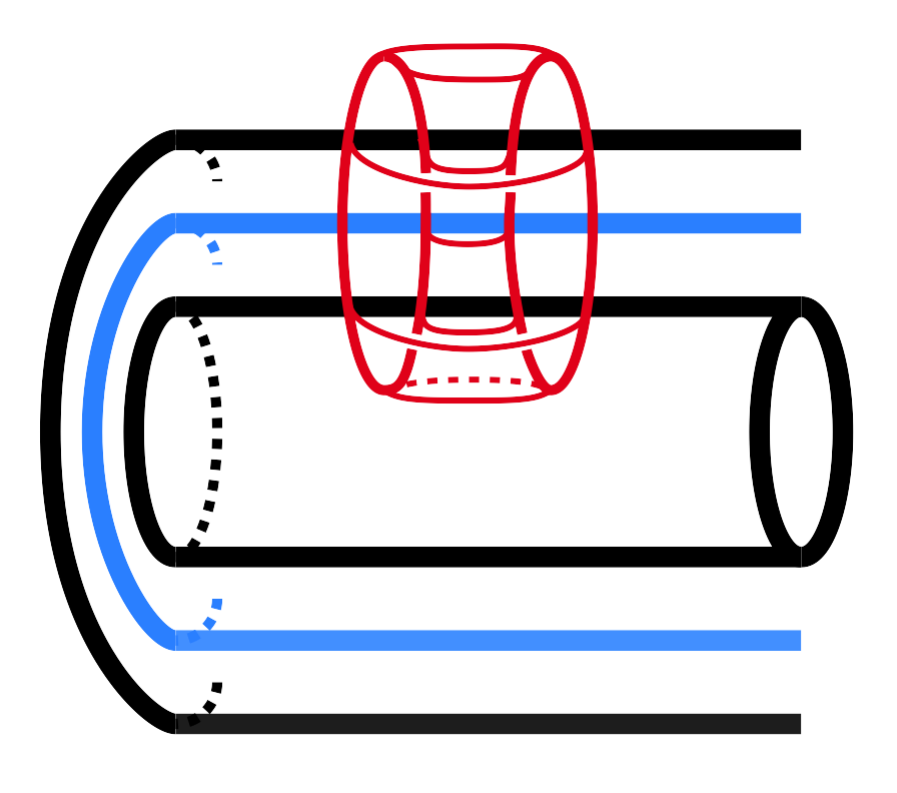}
    \caption{Lemma \ref{all sep} states that when all curves of $Q \cap \mc{H}$ are separating in $\mc{H}$ and there is no crushable handle, $Q$ is the union of two BSS and some VSS. Here we show an example where there are two VSS. The black cylinders represent portions of thick tori and the blue cylinder represents a portion of a thin torus.}
    \label{fig: allsep case}
\end{figure}

\begin{proof}
As in the proof of Lemma \ref{all nonsep}, this follows from the \href{spool lemma}{Spool Lemma} and the absence of cancellable matched pairs. Elaborating slightly: let $A_1$ be a bridge annulus. Its ends are both separating in $\mc{H}^+$, so one bounds a disc containing the other $\gamma_0$, as otherwise there would be a crushable handle. As we follow the annuli of $Q \setminus \mc{H}$ we eventually come to another bridge annulus $A_2$. Since there are no cancellable matched pairs, that bridge annulus has to turn into the disc in $\mc{H}^+$ bounded by its initial end. If we do not just follow other vertical annuli back to $A_1$, then we must end up spooling infinitely, as there are no cancellable matched pairs. 
\end{proof}

We now conclude the proof of the \href{main thm}{Main Theorem}. 

\begin{theorem}\label{Main Thm complicated}
Suppose that $T$ is a graph in a 3-manifold $M$, which is either $S^3$ or a lens space. Suppose that $Q \subset (M,T)$ is a c-essential unpunctured torus compressible to exactly one side $V$ in $M$ and that $T \cap V^c$ is a link. Then for any weighted companion $\weighted{L} = (T \cap V^c) \cup \weighted{K}$ for $T$ with respect to $V$, either $K$ is a torus knot or $\b_1(T) \geq \b_1(\weighted{L}) - \omega\delta$,
where $\delta = 1$ if $V$ is a lensed solid torus and $\delta = 0$ otherwise.
\end{theorem}
\begin{proof}
    By Theorem \ref{lem:Thinning invariance}, there exists $\mc{H} \in \H(Q)$ such that $\netg(\mc{H}) = 1$ and $\netw(\mc{H})/2 \leq \b_1(T)$. By Corollary \ref{adapted}, we may assume that $\mc{H}$ is adapted to $Q$. If $Q \cpt \mc{H}^-$, then by Proposition \ref{lens space structure}, it is a Heegaard torus for $M$ and compresses to both sides in $M$, contradicting our hypothesis.
    
    By Proposition \ref{no cancellable matched pairs}, we may assume there are no cancellable matched pairs. By Corollary \ref{name it and claim it}, $Q$ has no matched pairs $(A, A')$ with both of $A, A'$ BSS, or both BNN, or one BNS and the other BSS. 

     If $Q\cap \mc{H}$ contains a curve that is non-separating on $\mc{H}$, then by Lemma \ref{all nonsep}, we may assume that either $Q$ contains a crushable handle or $Q$ is the union of two BNN annuli and VNN annuli. On the other hand, if all curves of $Q \cap \mc{H}$ are separating in $\mc{H}$, then by Lemma \ref{all sep}, we may assume that either $Q$ contains a crushable handle or $Q$ is the union of two BSS annuli and VSS annuli. If $Q$ contains a crushable handle, the theorem holds by Corollary \ref{low g}. Thus, we have two possibilities two consider:
    \begin{enumerate}
         \item $Q$ is the union of two BSS annuli and some number of VSS annuli
        \item $Q$ is the union of two BNN annuli and some number of VNN annuli
    \end{enumerate}
    Let $H \cpt \mc{H}^+$ be a thick surface intersecting $Q$.  Let $\mc{A}_1$ and $\mc{A}_2$ be the components of $Q \setminus H$. Each of $\mc{A}_1$ and $\mc{A}_2$ is isotopic into $H$ (ignoring $T$). In both cases (1) and (2), they actually must be isotopic to each other, for otherwise $Q$ is isotopic to $H$, contradicting the the assumption that $Q$ is compressible in $M$ to a unique side. Using this parallelism, a spanning arc in $\mc{A}_1$ traces out a compressing disc $D$ for $Q$ in $M$. In Case (1), an innermost disc in $H$ with boundary a curve of $Q \cap H$ is a compressing disc for $Q$ in $M$ on the opposite side from $D$. This again contradicts the assumption that $Q$ compresses to a unique side in $M$. We are, therefore in Case (2). If the curves of $H \cap Q$ bound a disc to one side of $H$, there is such a disc with interior disjoint from $Q$ and on the opposite side of $Q$ from $D$. Again we contradict the assumption that $Q$ compresses to a unique side. Thus, the curves of $Q \cap H$ are torus knots in $M$ and intersect $D$ exactly once. Let one of them be called $K$; it is a torus knot companion for $Q$ relative to $V$.
     \end{proof}

\section{Potential generalizations}\label{pot gen}

In this section we reflect on the opportunities and obstructions to generalizing this work both to higher genus bridge surfaces and to higher genus satellites. In Remark \ref{generalizing history} we provided background on results from other authors in these directions. Here we content ourselves with the technical aspects of our approach. For what follows consider $Q$ to be a c-essential surface with (possibly empty) meridional boundary in the exterior of a knot, link, or spatial graph $T$ in a 3-manifold $M$. We let $g \geq 0$ be the genus of the bridge surfaces we are interested in.

The overall vision applies no matter what $\g(Q)$, $|\boundary Q|$, and $g$ are. We start with a bridge surface for $T$ of genus $g$, thin it to a locally thin multiple Heegaard surface $\mc{H}$, arrange for $Q \cap \mc{H}$ to consist of curves that are essential in both $Q$ and $\mc{H}$ (thought of as surfaces punctured by $T$). We then want to modify $\mc{H}$ further to control the nature of the intersections between $Q$ and $\mc{H}$ and use that control to bound $\netw(\mc{H};T)$ in terms of data coming from $\chi(Q)$ and some sort of generalized wrapping number of $T$ inside $Q$. We then amalgamate $\mc{H}$ back to a genus $g$ bridge surface and obtain some sort of inequality for $\b_g(T)$. 

The process of thinning and arranging for curves of $Q \cap \mc{H}$ to be mutually essential works no matter what $\g(Q)$, $|\boundary Q|$ or $g$ are, as we described in Sections \ref{bridge defs} and \ref{sec: thinning}.

In Remark \ref{pot gen} we observed that one obstruction to the program for $g \geq 2$ occurs at the amalgamation step, due to the presence of ghost arcs. If we replace the companion knot $K$ with a (weighted) spine $G$ for a genus $\g(Q)$ handlebody, we also may have ghost arcs which obstruct amalgamation. The obstruction occurs when there are ghost arcs sharing an endpoint on a thin surface. We can overcome this obstruction by isotoping a bit of a thick surface through one of the ghost arcs to convert it into two vertical arcs and a bridge arc, as in Figure \ref{fig: Kill ghost}. This increases $\netw(\mc{H})$ by 2. The number of ghost arcs is constrained by $\netchi(\mc{H})$ and $\chi(G)$, so it is possible to bound the total error created by this move. See the proof of \cite[Theorem 6.9]{Taylor-Equivariant} for an example of this type of argument.

\begin{figure}[ht!]
\centering
\includegraphics[scale=0.35]{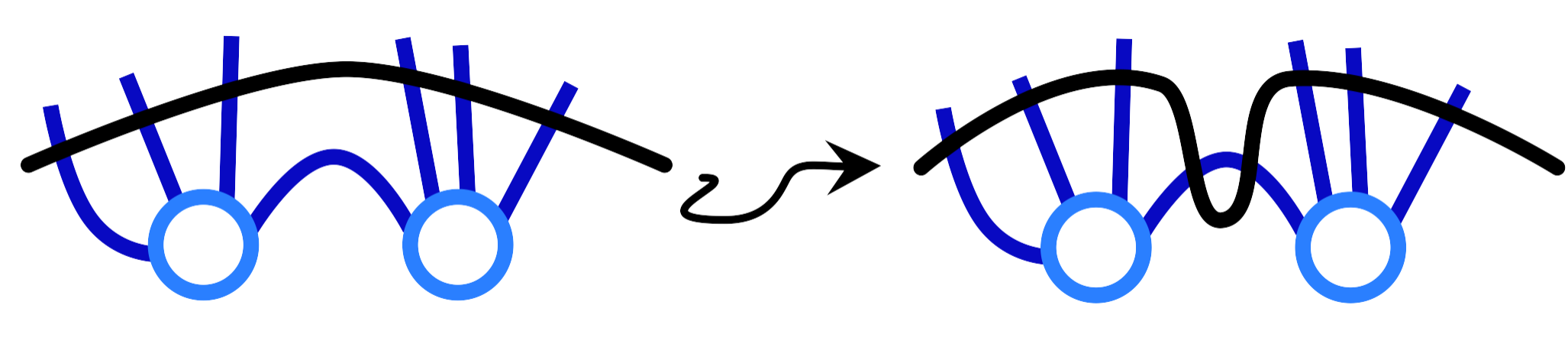}
\caption{Eliminating a ghost arc by an isotopy of a thick surface.}
\label{fig: Kill ghost}
\end{figure}

Crushing is a very generally applicable technique. To apply it when $\g(Q) \geq 2$, we need to find an annulus of $Q\setminus \mc{H}$ whose ends bound discs in a thick surface and where pieces of $Q$ do not pass the whole way through the annulus, but $T$ does. It is likely this technique can be useful in the settings where $g \geq 2$ or where $Q$ is not a torus.

Starting in Section \ref{sec:annuli} and throughout the rest of the paper, we frequently use the fact that in a VPC it is relatively easy to find sc-discs disjoint from annuli (as in Lemma \ref{Q disc exist}). We apply this most often to the case when those annuli are components of $Q\setminus \mc{H}$. When $Q$ is a torus intersecting $\mc{H}$ in mutually essential curves, every component of $Q \setminus \mc{H}$ is an annulus. This makes it easy to find $Q$-discs to ensure that $\mc{H}$ interacts nicely with $Q$. If $Q$ is a four-punctured sphere, the components will be annuli and twice-punctured discs and these are equally easy to work with. For example, the isotopies of a Conway sphere described by Blair in \cite{Blair1} can be intepreted in terms of multiple bridge surfaces. We apply our techniques in the case when $Q$ is a four-punctured sphere in a forthcoming paper. Once components of $Q\setminus \mc{H}$ are not annuli or twice-punctured discs, it becomes harder to find sc-discs disjoint from the pieces. This complicates the notion of matched pair and the ability to cancel matched pairs. When $-\chi(Q)$ is small, there will not be very many pieces that are not annuli or twice-punctured discs, and so it is conceivable that some conclusions can be drawn.

Throughout the paper, we rely heavily on the Jordan Curve Theorem/Sch\"onflies theorem for simple closed curves in spheres. In particular, in Section \ref{sec: crushing},  we reproved Schubert's Satellite Theorem by easily finding a crushable handle when $g = 0$. To prove our theorems in the case when $g = 1$, we repeatedly also use the fact that each pair of curves on a torus separates the torus. When $g \geq 2$, we would need to ensure that there are enough intersections between $Q$ and thick surfaces so that $Q$ separates the thick surfaces. When $Q$ is a torus, we could then apply the matched pair technology to try to force either the existence of a crushable handle or spooling behavior.

Finally, we note that if $M$ contains nonseparating closed surfaces, then it is possible that $Q$ could ``escape'' from spooling by wrapping through the 3-manifold. This is why in Section \ref{sec:noodles} we assumed that every (orientable) closed surface in $M$ separates $M$. This forces the surface $Q$ to turn around, thus creating either matched pairs or crushable handles.

Despite these challenges, it is conceivable that our techniques could be applied in more general settings.

\bibliographystyle{plain}
\bibliography{Graphsbib.bib}

\end{document}